\documentclass[12pt,reqno]{amsart} 
\usepackage{amssymb}
\usepackage{enumerate}
\usepackage[all]{xy}

\let\oldcirc=\circ
\renewcommand{\circ}{\mathchoice
    {\mathbin{\scriptstyle\oldcirc}}{\mathbin{\scriptstyle\oldcirc}}
    {\mathbin{\scriptscriptstyle\oldcirc}}
    {\mathbin{\scriptscriptstyle\oldcirc}}}

\newlength{\upto}\newlength{\dnto}
\newcommand{\I}[2]{\addtolength{\upto}{#1pt}\addtolength{\dnto}{#2pt}%
{\vrule height\upto depth\dnto width 0pt}}

\SelectTips{cm}{10} \UseTips   %% use computer modern tips in xy

\numberwithin{equation}{section}

\mathcode`\:="003A
\mathchardef\cdot="0201

\newcommand{\newsect}[1]{\bigskip\section{#1}\setcounter{table}{0}}
\newcommand{\newsub}[1]{\mbfx{\smallskip\subsection{#1}\leavevmode\noindent}}
\newcommand{\newsubb}[2]{\mbfx{\smallskip\subsection{#1}\label{#2}%
\leavevmode\noindent}}
\renewenvironment{enumerate}[1][]
{\begin{enumerat}[#1]\setlength{\itemsep}{6pt}}{\end{enumerat}}

\newenvironment{enuma}{\begin{enumerate}[{\rm(a) }]}{\end{enumerate}}

\renewenvironment{itemize}
{\begin{itemiz}\setlength{\itemsep}{6pt}\setlength{\itemindent}{-20pt}}
{\end{itemiz}}

\def\beq#1\eeq{\begin{equation*}#1\end{equation*}}
\def\beqq#1\eeqq{\begin{equation}#1\end{equation}}
\newcommand{\mbfx}[1]{{\boldmath #1\unboldmath}}

\let\emptyset=\varnothing

\newcommand{\longline}{\bigskip\centerline{\hbox to 5cm{\hrulefill}}\bigskip}

\newcommand{\widebar}[1]{\overset{\mskip3mu\hrulefill\mskip3mu}{#1}
		\vphantom{#1}}

\let\mynote=\rmk
\renewcommand{\:}{\colon}

\newcommand{\e}[2]{e_{#1#2}}

\newcommand{\mxtwo}[4]{\left(\begin{smallmatrix}#1&#2\\#3&#4
\end{smallmatrix}\right)}
\newcommand{\mxthree}[9]{\left(\begin{smallmatrix}#1&#2&#3\\#4&#5&#6\\
#7&#8&#9\end{smallmatrix}\right)}

\newcommand{\orb}{\mathcal{O}}

\DeclareMathAlphabet\EuR{U}{eur}{m}{n}
\SetMathAlphabet\EuR{bold}{U}{eur}{b}{n}
\newcommand{\curs}{\EuR}
\newcommand{\Ab}{\curs{Ab}}
\renewcommand{\mod}{\textup{-}\curs{mod}}

\newcommand{\higherlim}[2]{\displaystyle\setbox1=\hbox{\rm lim}
	\setbox2=\hbox to \wd1{\leftarrowfill} \ht2=0pt \dp2=-1pt
	\setbox3=\hbox{$\scriptstyle{#1}$}
	\def\test{#1}\ifx\test\empty
	\mathop{\mathop{\vtop{\baselineskip=5pt\box1\box2}}}\nolimits^{#2}
	\else
	\ifdim\wd1<\wd3
	\mathop{\hphantom{^{#2}}\vtop{\baselineskip=5pt\box1\box2}^{#2}}_{#1}
	\else
	\mathop{\mathop{\vtop{\baselineskip=5pt\box1\box2}}_{#1}}%
	\nolimits^{#2}
	\fi\fi}
\newcommand{\higherlimm}[2]{\setbox1=\hbox{\rm lim}
	\setbox2=\hbox to \wd1{\leftarrowfill} \ht2=0pt \dp2=-1pt
	\mathop{\mathop{\vtop{\baselineskip=5pt\box1\box2}}}\limits_{#1}
	\nolimits^{#2}}

\newcounter{let} 
\loop\stepcounter{let}
\expandafter\edef\csname cal\alph{let}\endcsname%
{\noexpand\mathcal{\Alph{let}}}
\ifnum\thelet<26\repeat

\newcommand{\tdef}[2][]{\expandafter\newcommand\csname#2\endcsname%
{#1\textup{#2}}}
\tdef{Iso}   \tdef{Aut}    \tdef{Out}    \tdef{Inn}    \tdef{Hom}
\tdef{End}   \tdef{Inj}    \tdef{map}    \tdef{Ker}    \tdef{Ob}
\tdef{Mor}   \tdef{Res}    \tdef{Spin}   \tdef{Id}     \tdef{Fr}
\tdef{rk}    \tdef{conj}   \tdef{incl}   \tdef{proj}   \tdef{diag} 
\tdef{trf}   \tdef{Sol}    \tdef{He}     \tdef{Sz}     \tdef{cj}
\tdef{supp}  \tdef[_]{typ}   \tdef[^]{op}    \tdef[^]{ab}

\newcommand{\fdef}[1]{\expandafter\newcommand\csname#1\endcsname%
{\mathfrak{#1}}}
\fdef{X}  \fdef{red}  \fdef{foc}  \fdef{hyp}

\newcommand{\bbdef}[1]{\expandafter\newcommand% 
\csname#1\endcsname{\mathbb{#1}}}
\bbdef{C} \bbdef{F} \bbdef{R} \bbdef{Z}

\newcommand{\3}[1]{\widebar{#1}}

\newcommand{\SFL}[1][]{(S#1,\calf#1,\call#1)}

\newcommand{\gen}[1]{\langle{#1}\rangle}
\newcommand{\Gen}[1]{\bigl\langle{#1}\bigr\rangle}

\let\nsg=\normal

\newcommand{\pr}{\operatorname{pr}\nolimits}

\newcommand{\E}{\widehat{\cale}}

\newcommand{\syl}[2]{\textup{Syl}_{#1}(#2)}
\newcommand{\sylp}[1]{\syl{p}{#1}}
\newcommand{\autf}{\Aut_{\calf}}

\newcommand{\outf}{\Out_{\calf}}

\newcommand{\homf}{\Hom_{\calf}}
\newcommand{\isof}{\Iso_{\calf}}
\newcommand{\sminus}{\smallsetminus}
\newcommand{\defeq}{\overset{\textup{def}}{=}}
\newcommand{\zploc}{\Z_{(p)}}

\renewcommand{\Im}{\textup{Im}}

\renewcommand{\2}[2]{\underset{#1}{#2}}
\newcommand{\pcom}{{}^\wedge_p}
\newcommand{\sd}[1]{\overset{{#1}}{\rtimes}}

\let\til=\widetilde

\newtheorem{Thm}{Theorem}[section]
\newtheorem{Prop}[Thm]{Proposition}

\newtheorem{Lem}[Thm]{Lemma}

\newtheorem{Defi}[Thm]{Definition}   %% mudei para o proximo groupo

\newtheorem{Th}{Theorem}

\newcommand{\longleft}[1]{\;{\leftarrow%
\count255=0 \loop \mathrel{\mkern-6mu}%
    \relbar\advance\count255 by1\ifnum\count255<#1\repeat}\;}
\newcommand{\longright}[1]{\;{\count255=0 \loop \relbar\mathrel{\mkern-6mu}%
    \advance\count255 by1\ifnum\count255<#1\repeat\rightarrow}\;}
\newcommand{\Right}[2]{\overset{#2}{\longright#1}}
\newcommand{\RIGHT}[3]{\mathrel{\mathop{\kern0pt\longright#1}
	\limits^{#2}_{#3}}}
\newcommand{\Left}[2]{{\buildrel #2 \over {\longleft#1}}}
\newcommand{\LEFT}[3]{\mathrel{\mathop{\kern0pt\longleft#1}\limits^{#2}_{#3}}
}
\newcommand{\longleftright}[1]{\;{\leftarrow\mathrel{\mkern-6mu}%
    \count255=0\loop\relbar\mathrel{\mkern-6mu}% 
    \advance\count255 by1\ifnum\count255<#1\repeat\rightarrow}\;} 
 
\newcommand{\onto}[1]{\;{\count255=0 \loop \relbar\joinrel
    \advance\count255 by1
    \ifnum\count255<#1 \repeat \twoheadrightarrow}\;}
\newcommand{\Onto}[2]{\overset{#2}{\onto#1}}
\newcommand{\RLEFT}[3]{\mathrel{%
   \mathop{\vcenter{\baselineskip=0pt\hbox{$\kern0pt\longright#1$}%
   \hbox{$\kern0pt\longleft#1$}}}\limits^{#2}_{#3}}}

\theoremstyle{definition}

\theoremstyle{remark}

\newcommand{\bfn}{{\mathbf n}}
\newcommand{\mm}{\mathbf{m}}
\title{Reduced, tame, and exotic fusion systems}

\author{Kasper Andersen}
\address{Matematisk Institut, Ny Munkegade, 8000 Aarhus C, Denmark}
\email{kksa@imf.au.dk}

\author{Bob Oliver}
\address{LAGA, Institut Galil\'ee, Av. J-B Cl\'ement, 93430
Villetaneuse, France}
\email{bobol@math.univ-paris13.fr}
\thanks{B. Oliver is partially supported by UMR 7539 of the CNRS, and by 
project ANR BLAN08-2\_338236, HGRT}

\author{Joana Ventura}
\address{Departamento de Matem\'atica, Instituto Superior T\'ecnico, Av.
Rovisco Pais, 1049--001 Lisboa, Portugal}
\email{jventura@math.ist.utl.pt}
\thanks{J. Ventura is partially supported by FCT/POCTI/FEDER and grant
PDCT/MAT/58497/2004.}

\subjclass[2000]{Primary 55R35. Secondary 55R40, 20D20}
\keywords{Classifying space, $p$-completion, finite groups, fusion.}

\begin{document}

\begin{abstract} 
We define here two new classes of saturated fusion systems, reduced fusion 
systems and tame fusion systems.  These are motivated by our attempts to 
better understand and search for exotic fusion systems:  fusion systems 
which are not the fusion systems of any finite group.  Our main theorems 
say that every saturated fusion system reduces to a reduced fusion system 
which is tame only if the original one is realizable, and that every 
reduced fusion system which is not tame is the reduction of some exotic 
(nonrealizable) fusion system.
\end{abstract}

\maketitle

When $G$ is a finite group and $S\in\sylp{G}$, the fusion category of $G$ 
is the category $\calf_S(G)$ whose objects consist of all subgroups of 
$S$, and where 
	\[ \Mor_{\calf_S(G)}(P,Q)=\Hom_G(P,Q)\defeq
	\{c_g\in\Hom(P,Q)\,|\, g\in{}G,\ gPg^{-1}\le{}Q \}~. \]
This provides a means of encoding the \emph{$p$-local structure} of $G$:  
the conjugacy relations among the subgroups of the Sylow $p$-subgroup $S$.  
An abstract 
``saturated fusion system'' over a finite $p$-group $S$ is a category 
whose objects are the subgroups of $S$, whose morphisms are certain 
monomorphisms of groups between the subgroups, and which satisfies certain 
conditions formulated by Puig \cite{Puig} and stated here in Definition 
\ref{sat.Frob.}.  In particular, for any finite $G$ as above, $\calf_S(G)$ 
is a saturated fusion system.  A saturated fusion system is called 
\emph{realizable} if it is isomorphic to the fusion system of some finite 
group $G$, and is called \emph{exotic} otherwise.

It turns out to be very difficult in general to construct exotic fusion 
systems, especially over $2$-groups.  This says something about how close 
Puig's definition is to the properties of fusion systems of finite groups.  

This paper is centered around the problem of identifying exotic fusion 
systems.  A first step towards doing this was taken in \cite{OV2}, where 
two of the authors developed methods for listing saturated fusion systems 
over any given $2$-group.  However, it quickly became clear that in order 
to have any chance of making a systematic search through all $2$-groups 
(or $p$-groups) of a given type, one must first find a way to limit the 
types of fusion systems under consideration, and do so without missing any 
possible exotic ones.

This leads to the concept of a \emph{reduced fusion system}.  A saturated 
fusion system is reduced if it contains no nontrivial normal 
$p$-subgroups, and also contains no proper normal subsystems of $p$-power 
index or of index prime to $p$.  These last concepts will be defined 
precisely in Definitions \ref{D:subgroups} and \ref{d:p-p'-index}; for 
now, we just remark that they are analogous to requiring a finite group to 
have no nontrivial normal $p$-subgroups and no proper normal subgroups of 
$p$-power index or of index prime to $p$.  Thus it is very far from 
requiring that the fusion system be simple in any sense, but it is 
adequate for our purposes.

The second concept which plays a central role in our results is that of a 
\emph{tame fusion system}.  Roughly, a fusion system $\calf$ is tame if it 
is realized by a finite group $G$ for which all automorphisms of $\calf$ 
are induced by automorphisms of $G$.  The precise (algebraic) definition is 
given in Definition \ref{d:tame}.  In terms of classifying spaces, $\calf$ 
is tame if it is realized by a finite group $G$ such that the natural map 
from $\Out(G)$ to $\Out(BG\pcom)$ is split surjective, where 
$\Out(BG\pcom)$ is the group of homotopy classes of self homotopy 
equivalences of the space $BG\pcom$.  

For any saturated fusion system $\calf$ over a finite $p$-group $S$, there 
is a canonical reduction $\red(\calf)$ of $\calf$ (Definition 
\ref{d:reduced}).  The analogy for a finite group $G$ with maximal normal 
$p$-subgroup $Q$ would be to set $G_0=C_G(Q)/Q$, and then let 
$\red(G)\nsg{}G_0$ be the smallest normal subgroup such that $G_0/\red(G)$ 
is $p$-solvable.  Our first main theorem is the following.

\begin{Th} \label{ThA}
For any saturated fusion system $\calf$ over a finite $p$-group $S$, if 
$\red(\calf)$ is tame, then $\calf$ is also tame, and in particular 
$\calf$ is realizable.
\end{Th}

Thus Theorem \ref{ThA} says that reduced fusion systems detect all possible 
exotic fusion systems.  If one wants to find all exotic fusion systems over 
$p$-groups of order $\le{}p^k$ for some $p$ and $k$, then one first 
searches for all reduced fusion systems over $p$-groups of order $\le{}p^k$ 
which are not tame, and then for all other fusion systems which reduce to 
them.  

The proof of Theorem \ref{ThA} uses the uniqueness of linking systems 
associated to the fusion system of a finite group, and through that 
depends on the classification of finite simple groups.  In order to make 
it clear exactly which part of the result depends on the classification 
theorem and which part is independent, we introduce another (more 
technical) concept, that of ``strongly tame'' fusion systems (Definition 
\ref{d:L(p)}).  Using the classification, together with results in 
\cite{limz-odd} and \cite{limz}, we prove that all tame fusion systems are 
strongly tame (Theorem \ref{t:L(p)}).  (In fact, the definition of 
``strongly tame'' is such that any tame fusion system which we're ever 
likely to be working with can be shown to be strongly tame without using 
the classification.)  Independently of that, and 
without using the classification theorem, we prove in Theorem 
\ref{T:reduce} that $\calf$ is tame whenever $\red(\calf)$ is strongly 
tame; and this together with Theorem \ref{t:L(p)} imply Theorem \ref{ThA}.

Alternatively, one can also avoid using the classification theorem by 
restating Theorem \ref{ThA} in terms of fusion systems together with 
associated linking systems.

Albert Ruiz has constructed examples \cite{Ruiz} which show that the 
reduction of a tame fusion system need not be tame, and in fact, can be 
exotic.  So there is no equivalence between the tameness of $\calf$ and 
tameness of $\red(\calf)$.  The next theorem does, however, provide a 
weaker converse to Theorem \ref{ThA}, by saying that for every non-tame 
reduced fusion system, there is some associated exotic fusion system in 
the background.

\begin{Th} \label{ThB}
Let $\calf$ be a reduced fusion system which is not tame.  Then there is 
an exotic fusion system whose reduction is isomorphic to $\calf$.
\end{Th}

As remarked above, reduced fusion systems can be very far from being 
simple in any sense.  For example, a product of reduced fusion systems is 
always reduced (Proposition \ref{F1xF2-red}).  The next theorem handles 
reduced fusion systems which factor as products.

\begin{Th} \label{ThC}
Each reduced fusion system $\calf$ over a finite $p$-group $S$ has a 
unique maximal factorization $\calf=\calf_1\times\cdots\times\calf_m$ as a 
product of indecomposable fusion systems $\calf_i$ over subgroups 
$S_i\nsg{}S$.  If $\calf_i$ is tame for each $i$, then $\calf$ is tame.
\end{Th}

Here, ``unique'' means that the indecomposable subsystems are unique as 
subcategories, not only up to isomorphism.  By Theorem \ref{ThC}, in 
order to find minimal reduced fusion systems which are not tame, it 
suffices to look at those which are indecomposable.  In practice, it seems 
that any reduced indecomposable fusion system which is not simple (which 
has no proper normal fusion subsystems in the sense of Definition 
\ref{d:F0<|F} or of \cite[\S\,6]{Asch}) has to be over a $p$-group of very 
large order.  The smallest example of this type we know of is the fusion 
system of $A_6\wr{}A_5$, over a group of order $2^{17}$.  

Using these results and those in \cite{OV2} as starting point, we have 
started to undertake a systematic computer search for reduced fusion 
systems over small $2$-groups.  So far, while details have yet to be 
rechecked carefully, we seem to have shown that each reduced fusion system 
over a $2$-group of order $\le512$ is the fusion system of a finite simple 
group, and is tame.  We hope to be able to extend this soon to $2$-groups 
of larger order.

What we really would like to find is an example of a realizable fusion 
system which is not tame.  It seems very likely that such a fusion system 
exists, but so far, our attempts to find one have been unsuccessful.

The theorems stated above will all be proven in Sections \ref{s:red-sfs} 
and \ref{s:factor}:  Theorems \ref{ThA} and \ref{ThB} as Theorems 
\ref{T:reduce} and \ref{untame->exotic}, and Theorem \ref{ThC} as 
Proposition \ref{uniq-decomp} and Theorem \ref{prod-tame}.  They are 
preceded by a first section containing mostly background definitions and 
results, and are followed by a fourth section with examples of how to 
prove certain fusion systems are tame.

All three authors would like to thank Copenhagen University for its 
hospitality, when letting us meet there for 2-week periods on two separate 
occasions.  We would also like to thank Richard Lyons for his help with 
automorphisms of certain sporadic groups.

\tableofcontents

\newsect{Fusion and linking systems}

We first collect the basic results about fusion and linking systems and 
their automorphisms which will be needed in the rest of the paper.  Most 
of this is taken directly from earlier papers, such as \cite{BLO2}, 
\cite{BCGLO1}, \cite{BCGLO2} and \cite{O3}.

%%%%%%%%%%%%%%%%%%%%%%%%%

\newsubb{Background on fusion systems}{s:fus}

We first recall very briefly the definition of a saturated fusion system, 
in the form given in \cite{BLO2}.  In general, for any group $G$ and any 
pair of subgroups $H,K\le{}G$, $\Hom_G(H,K)$ denotes the set of all 
homomorphisms from $H$ to $K$ induced by conjugation by some element of 
$G$.  When $G$ is finite and $S\in\sylp{G}$, $\calf_S(G)$ (the 
\emph{fusion category} of $G$) is the category whose objects are the 
subgroups of $S$, and where for each pair of objects 
$\Mor_{\calf_S(G)}(P,Q)=\Hom_G(P,Q)$.  

A \emph{fusion system} over a finite $p$-group $S$ is a category $\calf$, 
where $\Ob(\calf)$ is the set of all subgroups of $S$, such that for all 
$P,Q\le{}S$, 
	\[ \Hom_S(P,Q) \subseteq \homf(P,Q) \subseteq \Inj(P,Q); \]
and each $\varphi\in\homf(P,Q)$ is the composite of an isomorphism in 
$\calf$ followed by an inclusion.  Here, $\Inj(P,Q)$ denotes the set of 
injective homomorphisms from $P$ to $Q$.  If $\calf$ is a fusion system 
over a finite $p$-subgroup $S$, then two subgroups $P,Q\le S$ are 
\emph{$\calf$-conjugate} if they are isomorphic as objects of the category 
$\calf$. 

\begin{Defi}[{\cite{Puig}, see \cite[Definition~1.2]{BLO2}}] \label{sat.Frob.}
\renewcommand{\labelenumi}{\hskip-4pt$\bullet$}
Let $\calf$ be a fusion system over a finite $p$-group $S$.
\begin{itemize}
\item A subgroup $P\le{}S$ is \emph{fully centralized in $\calf$} if
$|C_S(P)|\ge|C_S(P^*)|$ for each $P^*\le{}S$ which is $\calf$-conjugate to
$P$.
\item A subgroup $P\le{}S$ is \emph{fully normalized in $\calf$} if
$|N_S(P)|\ge|N_S(P^*)|$ for each $P^*\le{}S$ which is $\calf$-conjugate to
$P$.
\item $\calf$ is a \emph{saturated fusion system} if the following
two conditions hold:
\begin{enumerate}[\rm(I) ]
\item For each $P\le{}S$ which is fully normalized in $\calf$, $P$ is fully
centralized in $\calf$ and $\Aut_S(P)\in\sylp{\autf(P)}$.
\item If $P\le{}S$ and $\varphi\in\homf(P,S)$ are such that $\varphi(P)$ is
fully centralized in $\calf$, and if we set
	\[ N_\varphi = \{ g\in{}N_S(P) \,|\, \varphi c_g\varphi^{-1} \in
	\Aut_S(\varphi(P)) \}, \]
then there is $\widebar{\varphi}\in\homf(N_\varphi,S)$ such that
$\widebar{\varphi}|_P=\varphi$.
\end{enumerate}
\end{itemize}
\end{Defi}

If $G$ is a finite group and $S\in\sylp{G}$, then the category 
$\calf_S(G)$ is a saturated fusion system (cf. \cite[Proposition 
1.3]{BLO2}).

We now list some classes of subgroups of $S$ which play an important role 
when working with fusion systems over $S$.  Here and elsewhere, for any 
fusion system $\calf$ over a finite $p$-group $S$, we write for each 
$P\le{}S$, 
	\[ \outf(P) = \autf(P)/\Inn(P) \le \Out(P)~. \]

\begin{Defi} \label{D:subgroups}
Fix a prime $p$, a finite $p$-group $S$, and a fusion system $\calf$ over 
$S$.  Let $P\le{}S$ be any subgroup.
\begin{itemize} 
\item $P$ is \emph{$\calf$-centric} if $C_S(P^*)=Z(P^*)$ for each $P^*$ 
which is $\calf$-conjugate to $P$.

\item $P$ is \emph{$\calf$-radical} if $O_p(\outf(P))=1$; i.e., if
$\outf(P)$ contains no nontrivial normal $p$-subgroups.

\item $P$ is \emph{central in $\calf$} if $P\nsg{}S$, and every morphism
$\varphi\in\homf(Q,R)$ in $\calf$ extends to a morphism
$\widebar{\varphi}\in\homf(PQ,PR)$ such that $\widebar{\varphi}|_P=\Id_P$.

\item $P$ is \emph{normal in $\calf$} ($P\nsg\calf$) if $P\nsg{}S$, and 
every morphism $\varphi\in\homf(Q,R)$ in $\calf$ extends to a morphism
$\widebar{\varphi}\in\homf(PQ,PR)$ such that $\widebar{\varphi}(P)=P$. 

\item $P$ is \emph{strongly closed in $\calf$} if no element of $P$ is 
$\calf$-conjugate to an element of $S{\sminus}P$.

\item $Z(\calf)\le{}Z(S)$ and $O_p(\calf)\le{}S$ denote the largest 
subgroups of $S$ which are central in $\calf$ and normal in $\calf$, 
respectively.
\end{itemize}
\end{Defi}

It follows directly from the definitions that if $P_1$ and $P_2$ are both 
central (normal) in $\calf$, then so is $P_1P_2$.  This is why there 
always is a largest central subgroup $Z(\calf)$, and a largest normal 
subgroup $O_p(\calf)$.

Several forms of Alperin's fusion theorem have been shown for saturated 
fusion systems, starting with Puig in \cite[\S\,5]{Puig}.  The following 
version suffices for what we need in most of this paper.  A stronger 
version will be given in Theorem \ref{AFT-E}.

\begin{Thm}[{\cite[Theorem A.10]{BLO2}}] \label{AFT}
For any saturated fusion system $\calf$ over a finite $p$-group $S$, each 
morphism in $\calf$ is a composite of restrictions of automorphisms in 
$\autf(P)$, for subgroups $P$ which are fully normalized in $\calf$, 
$\calf$-centric, and $\calf$-radical.
\end{Thm}

The following elementary result is useful for identifying subgroups which 
are centric and radical in a fusion system.

\begin{Lem} \label{centrad}
Let $\calf$ be a saturated fusion system over a finite $p$-group $S$.  If 
$P\le{}S$ is $\calf$-centric and $\calf$-radical, then there is no 
$g\in{}N_S(P){\sminus}P$ such that $c_g\in{}O_p(\autf(P))$.  Conversely, 
if $P\le{}S$ is fully normalized in $\calf$, and there is no 
$g\in{}N_S(P){\sminus}P$ such that $c_g\in{}O_p(\autf(P))$, then $P$ is 
$\calf$-centric and $\calf$-radical.
\end{Lem}

\begin{proof}  Assume $P$ is $\calf$-centric and $\calf$-radical.  Fix 
$g\in{}N_S(P)$ such that $c_g\in{}O_p(\autf(P))$.  Then 
$O_p(\outf(P))=1$ since $P$ is $\calf$-radical, so $c_g\in\Inn(P)$, and 
$g\in{}P{\cdot}C_S(P)=P$ since $P$ is $\calf$-centric.  This proves the 
first statement.

Now assume $P$ is fully normalized in $\calf$.  
If $P$ is not $\calf$-centric, then $C_S(P)\nleq{}P$ (since $P$ 
is fully centralized), and hence there is $g\in{}N_S(P){\sminus}P$ with 
$c_g=1$.  If $P$ is not $\calf$-radical, then $O_p(\outf(P))\ne1$.  This 
subgroup is contained in each Sylow $p$-subgroup of $\outf(P)$, and in 
particular is contained in $\Out_S(P)$.  Thus each nontrivial element of 
$O_p(\outf(P))$ is induced by conjugation by some element of 
$N_S(P){\sminus}P$.
\end{proof}

\begin{Prop} \label{norm<=>}
Let $\calf$ be a saturated fusion system over a finite $p$-group $S$.  For 
any normal subgroup $Q\nsg{}S$, $Q$ is normal in $\calf$ if and only if $Q$ 
is strongly closed and contained in all subgroups which are centric and 
radical in $\calf$.
\end{Prop}

\begin{proof} This is shown in \cite[Proposition 1.6]{BCGLO1}.  Note, 
however, that wherever ``$\calf$-radical'' appears in the statement and 
proof of that proposition, it should be replaced by ``$\calf$-centric and 
$\calf$-radical''.  
\end{proof}

Lemma \ref{centrad} shows the importance of being able to identify 
elements of the subgroup $O_p(\autf(P))$.  The following, very well known 
property of automorphisms of $p$-groups is useful in many cases when doing 
this.

\begin{Lem} \label{mod-Fr}
Fix a prime $p$, a finite $p$-group $P$, and a group $\cala\le\Aut(P)$ of 
automorphisms of $P$.  Assume $1=P_0\nsg{}P_1\nsg\cdots\nsg{}P_m=P$ is a 
sequence of normal subgroups such that $\alpha(P_i)=P_i$ for all 
$\alpha\in\cala$ and all $i$.  For $1\le{}i\le{}m$, let 
$\Psi_i\:\cala\Right2{}\Aut(P_i/P_{i-1})$ be the homomorphism which sends 
$\alpha\in\cala$ to the induced automorphism of $P_i/P_{i-1}$.  
Then for all $\alpha\in\cala$, $\alpha\in{}O_p(\cala)$ if and only if
$\Psi_i(\alpha)\in{}O_p(\Psi_i(\cala))$ for all $i=1,\dots,m$.
\end{Lem}

\begin{proof}  Set $\Psi=(\Psi_1,\ldots,\Psi_m)$, as a homomorphism from 
$\cala$ to $\prod_{i=1}^m\Aut(P_i/P_{i-1})$.  Then $\Ker(\Psi)$ is a 
$p$-group (cf. \cite[Corollary 5.3.3]{Gorenstein}).  If 
$\Psi_i(\alpha)\in{}O_p(\Psi_i(\cala))$ for all $i=1,\dots,m$, then 
$\Psi(\alpha)\in{}O_p(\Psi(\cala))$, and so $\alpha\in{}O_p(\cala)$.  
Conversely, if $\alpha\in{}O_p(\cala)$, then clearly 
$\Psi_i(\alpha)\in{}O_p(\Psi_i(\cala))$ for all $i$.
\end{proof}

Another elementary lemma which is frequently useful when working with 
centric and radical subgroups is the following:

\begin{Lem} \label{QnleqP}
Let $P$ and $Q$ be $p$-subgroups of a finite group $G$ such that 
$P\le{}N_G(Q)$ and $Q\nleq{}P$.  Then $N_{QP}(P)\gneqq{}P$, and 
$(Q\cap{}N_G(P))\nleq{}P$.
\end{Lem}

\begin{proof}  Since $P$ normalizes $Q$, $QP$ is also a $p$-group, and 
$QP\gneqq{}P$ by assumption.  Hence $N_{QP}(P)\gneqq{}P$ (cf. 
\cite[Theorem 2.1.6]{Sz1}).  Since 
$N_{QP}(P)=P{\cdot}(Q\cap{}N_{QP}(P))$, we have 
$(Q\cap{}N_{QP}(P))\nleq{}P$.
\end{proof}

We also need to work with certain quotient fusion systems.  When $\calf$ 
is a saturated fusion system over $S$ and $Q\nsg{}S$ is strongly closed in 
$\calf$, we define the quotient fusion system $\calf/Q$ over $S/Q$ by 
setting 
	\[ \Hom_{\calf/Q}(P/Q,R/Q) = \Im\bigl[ \homf(P,R) \Right2{} 
	\Hom(P/Q,R/Q) \bigr] \]
for all $P,R\le{}S$ containing $Q$.  

\begin{Prop} \label{F/Q}
Let $\calf$ be a saturated fusion system over a finite $p$-group $S$, and 
let $Q\nsg{}S$ be a strongly closed subgroup.  Then $\calf/Q$ is a 
saturated fusion system.  For each $P\le{}S$ containing $Q$, $P$ is fully 
normalized in $\calf$ if and only if $P/Q$ is fully normalized in 
$\calf/Q$.  If $Q$ is central in $\calf$, then $P\nsg\calf$ if and only if 
$P/Q\nsg\calf/Q$.  
\end{Prop}

\begin{proof}  By \cite[Lemma 2.6]{limz-odd}, $\calf/Q$ is a saturated 
fusion system, and $P$ is fully normalized if and only if $P/Q$ is.  So it 
remains only to prove the last statement.

By Proposition \ref{norm<=>}, $P$ is normal in $\calf$ if and only if it 
is strongly closed in $\calf$, and contained in each subgroup which is 
$\calf$-centric and $\calf$-radical.  For $P\le{}S$ containing $Q$, 
clearly $P$ is strongly closed in $\calf$ if and only if $P/Q$ is strongly 
closed in $\calf/Q$.  

We apply the criterion in Lemma \ref{centrad} for detecting subgroups 
which are centric and radical.  Let 
$\rho\:\autf(P)\Right2{}\Aut_{\calf/Q}(P/Q)$ be the homomorphism induced 
by projection.  Then $\rho$ is surjective by definition of $\calf/Q$.  
For $\alpha\in\autf(P)$, we have $\alpha|_Q=\Id_Q$ since $Q$ is central in 
$\calf$, so by Lemma \ref{mod-Fr}, $\alpha\in O_p(\autf(P))$ if and only if 
$\rho(\alpha)\in O_p(\Aut_{\calf/Q}(P/Q))$.

If $Q\le{}R\le{}S$, and $R^*$ is $\calf$-conjugate to $R$ and fully 
normalized in $\calf$, then $R^*/Q$ is $\calf/Q$-conjugate to $R/Q$ and 
fully normalized in $\calf/Q$.  So by Lemma \ref{centrad}, $R$ and $R^*$ 
are centric and radical in $\calf$ if and only if $R/Q$ and $R^*/Q$ are 
centric and radical in $\calf/Q$.  Upon combining this with the above 
criterion for normality, we see that $P\nsg\calf$ if and only if 
$P/Q\nsg\calf/Q$.
\end{proof}

\newsubb{Background on linking systems}{s:link}

We next define abstract linking systems associated to a fusion system 
$\calf$.  We use the definition given in \cite{O3}, which is more flexible 
in the choice of objects than the earlier definitions in \cite{BLO2} and 
\cite{BCGLO1}.  This definition also differs slightly from the one given 
in \cite[Definition 3.3]{BCGLO1}, in that we include a choice of inclusion 
morphisms as part of the data in the linking system.  All of these 
definitions are, however, equivalent, aside from having greater freedom in 
the choice of objects.

For any finite group $G$ and any $S\in\sylp{G}$, let $\calt_S(G)$ denote 
the \emph{transporter category} of $G$:  the category whose objects are the 
subgroups of $S$, and where for all $P,Q\le{}S$,
	\[ \Mor_{\calt_S(G)}(P,Q) = N_G(P,Q) \defeq \bigl\{ g\in{}G 
	\,\big|\, gPg^{-1}\le Q \bigr\} ~. \]
If $\calh$ is a set of subgroups of $S$, then 
$\calt_\calh(G)\subseteq\calt_S(G)$ denotes the full subcategory with 
object set $\calh$.

\begin{Defi}[{\cite[Definition 3]{O3}}] \label{d:L}
Let $\calf$ be a fusion system over a finite p-group $S$.  A \emph{linking 
system} associated to $\calf$ is a finite category $\call$, together with 
a pair of functors 
	\[ \calt_{\Ob(\call)}(S) \Right5{\delta} \call
	\Right5{\pi} \calf\,, \]
satisfying the following conditions:
\begin{enumerate}[\rm(A) ]
\item $\Ob(\call)$ is a set of subgroups of $S$ closed under 
$\calf$-conjugacy and overgroups, and includes all subgroups which are 
$\calf$-centric and $\calf$-radical.  Each object in $\call$ is 
isomorphic (in $\call$) to one which is fully centralized in $\calf$.  Also, 
$\delta$ is the identity on objects, and $\pi$ is the inclusion on 
objects.  For each $P,Q\in\Ob(\call)$ such that $P$ is fully centralized 
in $\calf$, $C_S(P)$ acts freely on $\Mor_{\call}(P,Q)$ via $\delta_{P}$ 
and right composition, and $\pi_{P,Q}$ induces a bijection 
	\[ \Mor_\call(P,Q)/C_S(P) \Right5{\cong} \homf(P,Q) ~. \]

\item For each $P,Q\in\Ob(\call)$ and each $g\in{}N_S(P,Q)$, $\pi_{P,Q}$ 
sends $\delta_{P,Q}(g)\in\Mor_\call(P,Q)$ to $c_g\in\homf(P,Q)$. 

\item For all $\psi\in\Mor_{\call}(P,Q)$ and all $g\in P$, the diagram
	\[ \xymatrix@C+2ex{
	\I02P \ar[r]^{\psi} \ar[d]_{\delta_{P}(g)} & \I02Q 
	\ar[d]^{\delta_{Q}(\pi(\psi)(g))} \\
	\I02P \ar[r]^{\psi} & \I02Q } \]
commutes in $\call$.
\end{enumerate}
If $\call^*$ is another linking system associated to $\calf$ with the same 
set of objects as $\call$, then an \emph{isomorphism} of linking systems is 
an isomorphism of categories $\call\Right2{\cong}\call^*$ which commutes 
with the structural functors:  those coming from $\calt_{\Ob(\call)}(S)$ and 
those going to $\calf$.
\end{Defi}

A \emph{$p$-local finite group} is defined to be a triple $\SFL$, where 
$\calf$ is a saturated fusion system over a finite $p$-group $S$, and where 
$\call$ is a \emph{centric linking system} associated to $\calf$ (i.e., one 
whose objects are the $\calf$-centric subgroups of $S$).  

For $P\le{}Q$ in $\Ob(\call)$, we usually write 
$\iota_P^Q=\delta_{P,Q}(1)$, and regard this as the ``inclusion'' of $P$ 
into $Q$.  The definition in \cite[Definition 3.3]{BCGLO1} of a 
(quasicentric) linking system does not include these inclusions, but it is 
explained there how to choose inclusions in a way so that a functor 
$\delta\:\calt_{\Ob(\call)}(S)\Right2{}\call$ can be defined in a unique 
way with the above properties (\cite[Lemma 3.7]{BCGLO1}).  Note also that 
because the above definition includes a choice of inclusions, condition 
(D)$_q$ in \cite[Definition 3.3]{BCGLO1} is not needed here.  

We have defined linking systems to be as flexible as possible in the choice 
of objects, but one cannot avoid completely discussing quasicentric 
subgroups in this context.

\begin{Defi} \label{d:qcentric}
\begin{enuma}  
\item For any finite group $G$, a $p$-subgroup $P\le{}G$ is 
\emph{$G$-quasicentric} if $O^p(C_G(P))$ has order prime to $p$.
\item For any fusion system $\calf$ over a finite $p$-group $S$, a subgroup 
$P\le{}S$ is \emph{$\calf$-quasicentric} if for each $P^*$ which is 
fully centralized in $\calf$ and $\calf$-conjugate to $P$, $C_\calf(P^*)$ 
is the fusion system of the $p$-group $C_S(P^*)$.  Equivalently, for each 
$Q\le{}P^*{\cdot}C_S(P^*)$ containing $P^*$, 
$\{\alpha\in\autf(Q)\,|\,\alpha|_{P^*}=\Id\}$ is a $p$-group.
\end{enuma}
\end{Defi}

For any saturated fusion system $\calf$, the set of $\calf$-quasicentric 
subgroups is closed under $\calf$-conjugacy and overgroups.  So a
quasicentric linking system as defined in \cite[\S\,3]{BCGLO1} is a linking 
system in the sense defined here.

Fix a finite group $G$ and $S\in\sylp{G}$, and set $\calf=\calf_S(G)$.  
Then a subgroup $P\le{}S$ is $G$-quasicentric if and only if it is 
$\calf$-quasicentric.  For any set $\calh$ of $G$-quasicentric subgroups 
of $S$, define $\call_S^\calh(G)$ to be the category with object set 
$\calh$, and where for each $P,Q\in\calh$,
	\[ \Mor_{\call_S^\calh(G)}(P,Q) = N_G(P,Q)/O^p(C_G(P)). \]
Composition is well defined, since for each $g\in{}N_G(P,Q)$, 
$g^{-1}Qg\ge{}P$, so $g^{-1}C_G(Q)g\le{}C_G(P)$, and thus 
$g^{-1}O^p(C_G(Q))g\le{}O^p(C_G(P))$.  When $\calh$ is closed under 
$\calf$-conjugacy and overgroups and contains all subgroups of $S$ which 
are $\calf$-centric and $\calf$-radical, then $\call_S^\calh(G)$ is a 
linking system associated to $\calf$.  When $\calh$ is the set of 
$\calf$-centric subgroups of $S$, we write $\call_S^c(G)=\call_S^\calh(G)$.

\begin{Prop} \label{L-prop}
The following hold for any linking system $\call$ associated to a saturated 
fusion system $\calf$ over a finite $p$-group $S$.
\begin{enuma}  
\item For each $P,Q\in\Ob(\call)$, the subgroup 
$E(P)\defeq\Ker[\Aut_\call(P)\Right2{}\autf(P)]$ acts freely on 
$\Mor_{\call}(P,Q)$ via right composition, and $\pi_{P,Q}$ induces a 
bijection 
	\[ \Mor_\call(P,Q)/E(P) \Right5{\cong} \homf(P,Q) ~. \]

\item For every morphism $\psi\in\Mor_\call(P,Q)$, and every 
$P_0,Q_0\in\Ob(\call)$ such that $P_0\le{}P$, $Q_0\le{}Q$, and 
$\pi(\psi)(P_0)\le{}Q_0$, there is a unique morphism 
$\psi|_{P_0,Q_0}\in\Mor_\call(P_0,Q_0)$ (the ``restriction'' of $\psi$) 
such that $\psi\circ\iota_{P_0}^{P}=\iota_{Q_0}^Q\circ\psi|_{P_0,Q_0}$.  

\item[\rm(b$'$) ] For each $P,Q\in\Ob(\call)$ and each 
$\psi\in\Mor_\call(P,Q)$, if we 
set $Q_0=\pi(\psi)(P)$, then there is a unique $\psi_0\in\Iso_\call(P,Q_0)$ 
such that $\psi=\iota_{Q_0}^Q\circ\psi_0$.

\item The functor $\delta$ is injective on all morphism sets.

\item If $P\in\Ob(\call)$ is fully normalized in $\calf$, then 
$\delta_P(N_S(P))\in\sylp{\Aut_\call(P)}$.

\item Let $P,Q,\widebar{P},\widebar{Q}\in\Ob(\call)$ and 
$\psi\in\Mor_\call(P,Q)$ be such that $P\nsg\widebar{P}$, 
$Q\le\widebar{Q}$, and for each $g\in\widebar{P}$ there is 
$h\in\widebar{Q}$ such that 
$\iota_Q^{\widebar{Q}}\circ\psi\circ\delta_P(g)=
\delta_{Q,\widebar{Q}}(h)\circ\psi$.
Then there is a unique morphism 
$\widebar{\psi}\in\Mor_\call(\widebar{P},\widebar{Q})$ such that 
$\widebar{\psi}|_{P,Q}=\psi$. 

\item All morphisms in $\call$ are monomorphisms and epimorphisms in the 
categorical sense.

\item All objects in $\call$ are $\calf$-quasicentric.  
\end{enuma}
\end{Prop}

\begin{proof}  See \cite[Proposition 4]{O3}.  Point (b$'$) is a special 
case of (b), where $\psi_0\defeq\psi|_{P,Q_0}$ is an isomorphism by (a).
\end{proof}

We will also have use for 
the following ``linking system version'' of Alperin's fusion theorem. 

\begin{Thm} \label{AFT-L}
For any saturated fusion system $\calf$ over a finite $p$-group $S$ and 
any linking system $\call$ associated to $\calf$, each morphism in $\call$ 
is a composite of restrictions of automorphisms in $\Aut_\call(P)$, where 
$P$ is fully normalized in $\calf$, $\calf$-centric, and $\calf$-radical.
\end{Thm}

\begin{proof}  Using Theorem \ref{AFT} together with Proposition 
\ref{L-prop}(a), we are reduced to proving the theorem for automorphisms 
in $E(P)=\Ker[\Aut_\call(P)\Right2{}\autf(P)]$ for $P\in\Ob(\call)$.  If 
$P$ is fully centralized, then $E(P)=\{\delta_P(g)\,|\,g\in{}C_S(P)\}$ by 
axiom (A), and each element $\delta_P(g)$ is the restriction of 
$\delta_S(g)\in\Aut_\call(S)$.  If $P$ is arbitrary, and $Q$ is fully 
centralized in $\calf$ and $\calf$-conjugate to $P$, then there is some 
$\psi\in\Iso_\call(P,Q)$ which satisfies the conclusion of the theorem 
(choose any $\varphi\in\isof(P,Q)$, write it as a composite of 
restrictions of automorphisms, and lift each of those automorphisms to 
$\call$).  Then each element of $E(P)$ has the form 
$\psi^{-1}\delta_Q(g)\psi$ for some $g\in{}C_S(Q)$, and hence satisfies 
the conclusion of the theorem.
\end{proof}

\newsubb{Automorphisms of fusion and linking systems}{s:Aut(L)}

\newcommand{\Equ}{\textup{Equ}}

Recall that for any linking system $\call$ associated to a fusion system 
$\calf$ over $S$, and any pair $P\le{}Q$ of objects in $\call$, the 
\emph{inclusion} of $P$ into $Q$ is the morphism 
$\iota_P^Q=\delta_{P,Q}(1)\in\Mor_\call(P,Q)$.  By Proposition 
\ref{L-prop}(b$'$), each morphism in $\call$ splits uniquely as the 
composite of an isomorphism followed by an inclusion. 

As usual, an \emph{equivalence} of small categories is a functor 
$\Phi\:\calc\Right2{}\cald$ which induces a bijection between the sets of 
isomorphism classes of objects and bijections between each pair of 
morphism sets.  It is not hard to see that for each such equivalence, 
there is an ``inverse'' $\Psi\:\cald\Right2{}\calc$ such that both 
composites $\Phi\circ\Psi$ and $\Psi\circ\Phi$ are naturally isomorphic to 
the identities.  In particular, the quotient monoid $\Out(\calc)$ of all 
self equivalences of $\calc$ modulo natural isomorphisms of functors is a 
group.

\begin{Defi}[{\cite[\S\,8]{BLO2}}] \label{Out(F,L)}
Let $\calf$ be a saturated fusion system over a finite $p$-group $S$, and 
let $\call$ be a linking system associated to $\calf$.
\begin{enuma}  
\item An automorphism $\beta\in\Aut(S)$ is \emph{fusion preserving} if 
for each $P,Q\le{}S$ and each $\varphi\in\homf(P,Q)$, 
$(\beta|_{Q,\beta(Q)})\varphi(\beta|_{P,\beta(P)})^{-1}$ lies in 
$\homf(\beta(P),\beta(Q))$.  In particular, each such $\beta$ 
normalizes $\autf(S)$.  Let $\Aut(S,\calf)$ be the group of all fusion 
preserving automorphisms of $S$, and set 
$\Out(S,\calf)=\Aut(S,\calf)/\Aut_{\calf}(S)$.  Note that $\Out(S,\calf)$ 
is a \emph{subquotient} of $\Out(S)$.

\item An equivalence of categories $\alpha\:\call\Right2{}\call$ is 
\emph{isotypical} if $\alpha(\delta_P(P))=\delta_{\alpha(P)}(\alpha(P))$ 
for each $P\in\Ob(\call)$.

\item Let $\Out\typ(\call)$ be the group of classes of isotypical self 
equivalences of $\call$ modulo natural isomorphisms of functors.  

\item Let $\Aut\typ^I(\call)$ be the group of isotypical equivalences of 
$\call$ which send inclusions to inclusions.  
\end{enuma}
\end{Defi}

Since $\Out(\call)$ is a group by the above remarks, and is finite since 
$\Mor(\call)$ is finite, $\Out\typ(\call)$ is a submonoid of a finite 
group and hence itself a group.  Another proof of this, as well as a 
proof that $\Aut\typ^I(\call)$ is a group, will be given in Lemma 
\ref{OutI1}. 

One of the main results in \cite{BLO2} (Theorem 8.1) says that for any 
$p$-local finite group $\SFL$, $\Out\typ(\call)\cong\Out(|\call|\pcom)$:  
the group of homotopy classes of self homotopy equivalences of 
$|\call|\pcom$.  This helps to explain the importance of $\Out\typ(\call)$, 
among other groups of automorphisms of $\SFL$ which we might have chosen.

The next lemma gives an alternative description of $\Out\typ(\call)$, and 
also of $\Out(G)$ --- descriptions which will be useful later.  For each 
$\call$ associated to $\calf$ over $S$, and each $\gamma\in\Aut_\call(S)$, 
let $c_\gamma\in\Aut\typ^I(\call)$ be the automorphism which sends 
$P\in\Ob(\call)$ to $\gamma(P)=\pi(\gamma)(P)$, and sends 
$\psi\in\Mor_\call(P,Q)$ to 
$(\gamma|_{Q,\gamma(Q)})\circ\psi\circ(\gamma|_{P,\gamma(P)})^{-1}$.  This 
is clearly isotypical, since for $g\in{}P\in\Ob(\call)$, 
$c_\gamma(\delta_P(g))=\delta_{\gamma(P)}(\pi(\gamma)(g))$ by axiom (C).  
For $P\le{}Q$ in $\Ob(\call)$, $c_\gamma$ sends $\iota_P^Q$ to 
$\iota_{\gamma(P)}^{\gamma(Q)}$ by definition of restriction, and thus 
$c_\gamma\in\Aut\typ^I(\call)$.

\begin{Lem} \label{OutI1} 
\begin{enuma}  
\item For any saturated fusion system $\calf$ over a finite $p$-group $S$, 
and any linking system $\call$ associated to $\calf$, the sequence 
	\[ 1 \Right2{} Z(\calf) \Right4{\delta_S} \Aut_{\call}(S) 
	\Right5{\gamma\mapsto c_\gamma} 
	\Aut\typ^I(\call) \Right4{} \Out\typ(\call) \Right2{} 1 \]
is exact.  All elements of $\Aut\typ^I(\call)$ are automorphisms of 
$\call$, and hence $\Aut\typ^I(\call)$ and $\Out\typ(\call)$ are both 
groups.

\item For any finite group $G$ and any $S\in\sylp{G}$, the sequence
	\[ 1 \Right2{} Z(G) \Right4{\incl} N_G(S) \Right5{g\mapsto c_g} 
	\Aut(G,S) \Right4{} \Out(G) \Right2{} 1 \]
is exact, where $\Aut(G,S)=\{\alpha\in\Aut(G)\,|\,\alpha(S)=S\}$.
\end{enuma}
\end{Lem}

\begin{proof} \noindent\textbf{(a) }  Each equivalence of $\call$ 
(isotypical or not) sends $S$ to itself, since $S$ is the only object 
which is the target of morphisms from all other objects.  

If $\alpha\in\Aut\typ^I(\call)$, then for each $P\in\Ob(\call)$, 
$\alpha_{P,S}$ sends $\iota_P^S$ to $\iota_{\alpha(P)}^S$, and $\alpha_P$ 
sends $\delta_P(P)$ to $\delta_{\alpha(P)}(\alpha(P))$.  Hence $\alpha_S$ 
sends $\delta_S(P)$ to $\delta_S(\alpha(P))$, and thus determines the 
action of $\alpha$ on $\Ob(\call)$.  In particular, $\alpha$ permutes the 
objects of $\call$ bijectively, and hence is an automorphism of $\call$. 
This proves that $\Aut\typ^I(\call)$ is a group; and that 
$\Out\typ(\call)$ is also a group if the above sequence is exact.

We next show that each isotypical equivalence 
$\alpha\:\call\Right2{}\call$ is naturally isomorphic to an isotypical 
equivalence which sends inclusions to inclusions.  For each 
$P\in\Ob(\call)$, let $\alpha(\iota_P^S)=\iota_{\beta(P)}^S\circ\omega(P)$ 
be the unique decomposition of $\alpha(\iota_P^S)$ as a composite of an 
isomorphism $\omega(P)\in\Iso_{\call}(\alpha(P),\beta(P))$ followed by an 
inclusion (Proposition \ref{L-prop}(b$'$)).  In particular, 
$\omega(S)=\Id$. Let $\beta$ be the automorphism of $\call$ which on 
objects sends $P$ to $\beta(P)$, and which on morphisms sends 
$\varphi\in\Mor_{\call}(P,Q)$ to 
$\omega(Q)\circ\alpha(\varphi)\circ\omega(P)^{-1}$ in 
$\Mor_{\call}(\beta(P),\beta(Q))$.  Then $\beta$ is isotypical by axiom 
(C) (and since $\alpha$ is isotypical); it sends inclusions to inclusions 
by construction (and since $\omega(S)=\Id$); and $\omega(-)$ defines a 
natural isomorphism from $\alpha$ to $\beta$.

This proves that the natural homomorphism from $\Aut\typ^I(\call)$ to 
$\Out\typ(\call)$ is onto.  If $\alpha\in\Aut\typ^I(\call)$ is in the 
kernel, then it is naturally isomorphic to the identity, via some 
$\omega(-)$ which consists of isomorphisms 
$\omega(P)\in\Iso_\call(P,\alpha(P))$ such that for each 
$\psi\in\Mor_\call(P,Q)$, $\alpha(\psi)\circ\omega(P)=\omega(Q)\circ\psi$. 
Since $\alpha$ sends $\iota_P^S$ to $\iota_{\alpha(P)}^S$, 
$\omega(P)=\omega(S)|_{P,\alpha(P)}$, and thus $\alpha$ is conjugation by 
$\omega(S)\in\Aut_\call(S)$.

Conversely, if $\gamma\in\Aut_\call(S)$, then $c_\gamma$ is naturally 
isomorphic to $\Id_\call$, by the natural isomorphism which sends 
$P\in\Ob(\call)$ to $\gamma|_{P,\pi(\gamma)(P)}$.  This finishes the proof 
that the above sequence is exact at $\Aut\typ^I(\call)$.

It remains to show, for $\gamma\in\Aut_\call(S)$, that $c_\gamma=\Id_\call$ 
if and only if $\gamma\in\delta_S(Z(\calf))$.  If $c_\gamma=\Id$, then since 
$\gamma\delta_S(g)\gamma^{-1}=\delta_S(g)$ for all $g\in{}S$, 
$\pi(\gamma)=\Id_S$ by axiom (C), and $\gamma\in\delta_S(Z(S))$ by 
(A).  So assume $\gamma=\delta_S(a)$ for some $a\in{}Z(S)$.  By Proposition 
\ref{L-prop}(e) (and axiom (C) again), $\gamma$ commutes with a morphism 
$\psi\in\Mor_\call(P,Q)$ if and only if $\psi$ extends to some 
$\widebar{\psi}\in\Mor_\call(\gen{P,a},\gen{Q,a})$ such that 
$\pi(\widebar{\psi})(a)=a$.  Thus $c_\gamma=\Id_\call$ if and only if 
$a\in{}Z(\calf)$.  This finishes the proof that the sequence in (a) is 
exact.

\smallskip

\noindent\textbf{(b) } The natural homomorphism from $\Aut(G,S)$ to 
$\Out(G)$ is onto by the Frattini argument (the Sylow $p$-subgroups of $G$ 
are permuted transitively by inner automorphisms).  The kernel of that map 
clearly consists of conjugation by elements of $N_G(S)$.
\end{proof}

In particular, the group $\Aut\typ^I(\call)$ defined here is the same as 
that defined in \cite{O3}, where it was defined explicitly as a group of 
automorphisms of $\call$ rather than of equivalences.

The next lemma describes how elements of $\Aut\typ^I(\call)$ induce 
automorphisms of the associated fusion system.  For 
$\beta\in\Aut(S,\calf)$, let $c_\beta\in\Aut(\calf)$ be the 
automorphism of the category $\calf$ which sends $P\le{}S$ to $\beta(P)$, 
and sends $\varphi\in\Mor(\calf)$ to $\beta\varphi\beta^{-1}$.  

\begin{Lem}[{\cite[Proposition 6]{O3}}] \label{AutI}
Let $\call$ be a linking system associated to a saturated fusion system 
$\calf$ over a finite $p$-group $S$, with structure functors 
$\calt_{\Ob(\call)}(S)\Right2{\delta}\call\Right2{\pi}\calf$.  Fix 
$\alpha\in\Aut\typ^I(\call)$.  Let $\beta\in\Aut(S)$ be such that 
$\alpha(\delta_S(g))=\delta_S(\beta(g))$ for all $g\in{}S$.  Then 
$\beta\in\Aut(S,\calf)$, $\alpha(P)=\beta(P)$ for $P\in\Ob(\call)$, and 
$\pi\circ\alpha=c_\beta\circ\pi$. 
\end{Lem}

\begin{proof}  See \cite[Proposition 6]{O3}.  The relation 
$\alpha(P)=\beta(P)$ is not in the statement of the proposition, but it is 
shown in its proof.  It is really part of the statement 
$\pi\circ\alpha=c_\beta\circ\pi$ (since $\pi$ is the inclusion 
on objects). 
\end{proof}

Lemma \ref{AutI} motivates the following definition.  
For any saturated fusion system $\calf$ over a finite $p$-group $S$, and 
any linking system $\call$ associated to $\calf$, define 
	\[ \til\mu_\call\: \Aut\typ^I(\call) \Right5{} \Aut(S,\calf) \]
by setting $\til\mu_\call(\alpha) = 
\delta_S^{-1}\circ\alpha_S\circ\delta_S\in\Aut(S)$ for 
$\alpha\in\Aut\typ^I(\call)$.  By Lemma \ref{AutI}, 
$\Im(\til\mu_\call)\le\Aut(S,\calf)$.  For $\gamma\in\Aut_\call(S)$, 
$\til\mu_\call(c_\gamma)=\pi(\gamma)\in\autf(S)$ by axiom (C) in 
Definition \ref{d:L}.  So by Lemma  \ref{OutI1}(a), $\til\mu_\call$ 
induces a homomorphism 
	\[ \mu_\call\: \Out\typ(\call) \Right5{} \Out(S,\calf) \]
by sending the class of $\alpha$ to that of $\til\mu_\call(\alpha)$.  When 
$\call=\call_S^\calh(G)$ for some finite group $G$ and some set of objects 
$\calh$, we write $\til\mu_G^\calh=\til\mu_\call$ and 
$\mu_G^\calh=\mu_\call$ for short.  When $\call=\call_S^c(G)$ is the 
centric linking system, we write $\til\mu_G=\til\mu_\call$ and 
$\mu_G=\mu_\call$.  

\begin{Lem} \label{Ker(mu)-p-gp}
For any linking system $\call$ associated to a saturated fusion system 
$\calf$, $\Ker(\mu_\call)$ is a finite $p$-group.
\end{Lem}

\begin{proof}  Assume $\calf$ is a fusion system over the finite $p$-group 
$S$.  Since $\call$ is a finite category, $\Aut\typ^I(\call)$ and 
$\Out\typ(\call)$ are finite groups.  So it suffices to prove that each 
element of $\Ker(\mu_\call)$ has $p$-power order.

Fix $\alpha\in\Aut\typ^I(\call)$ such that $[\alpha]\in\Ker(\mu_\call)$.  
Thus $\til\mu_\call(\alpha)\in\autf(S)$, and 
$\til\mu_\call(\alpha)=\pi(\gamma)$ for some $\gamma\in\Aut_\call(S)$.  So
upon replacing $\alpha$ by $c_\gamma^{-1}\circ\alpha$, we can assume 
$\alpha\in\Ker(\til\mu_\call)$.  

Thus $\alpha_S|_{\delta_S(S)}=\Id$.  Since $\alpha$ sends inclusions to 
inclusions, $\alpha_P|_{\delta_P(P)}=\Id$ for all $P\in\Ob(\call)$.  For 
each $P$ and each $\psi\in\Aut_\call(P)$, $\psi$ and $\alpha(\psi)$ have 
the same conjugation action on $\delta_P(P)$, so 
$\psi^{-1}\alpha(\psi)\in 
C_{\Aut_\call(P)}(\delta_P(P))\le\delta_P(Z(P))$.  Hence 
$\alpha(\psi)=\psi\delta_P(g)$ for some $g\in{}Z(P)$, and 
$\alpha^k(\psi)=\psi\delta_P(g^k)$ for all $k$ since $\alpha$ is the 
identity on $\delta_P(P)$.  

Choose $m\ge0$ such that $g^{p^m}=1$ for all $g\in{}S$.  Then $\alpha^{p^m}$ 
is the identity on $\Aut_\call(P)$ for each $P\in\Ob(\call)$, and hence (by 
Theorem \ref{AFT-L}) is the identity on $\call$.
\end{proof}

The kernel of $\mu_\call$ will be studied much more closely in 
Proposition \ref{Ker(mu)}.  

Since we will need to work with linking systems with different sets of 
objects associated to the same fusion system, it will be important to know 
they have the same automorphisms.

\begin{Lem} \label{OutI2}
Fix a saturated fusion system $\calf$ over a finite $p$-group $S$.  Let 
$\call_0\subseteq\call$ be a pair of linking systems associated to 
$\calf$.  Set $\calh_0=\Ob(\call_0)$ and $\calh=\Ob(\call)$, and assume 
$\calh_0\subseteq\calh$ are both $\Aut(S,\calf)$-invariant.  Then 
restriction defines an isomorphism
	\[ \Out\typ(\call) \RIGHT4{R}{\cong} \Out\typ(\call_0). \]
\end{Lem}

\begin{proof}  Using Proposition \ref{L-prop}(a), one sees that $\call_0$ 
must be a full subcategory of $\call$.  Set $\calp=\calh{\sminus}\calh_0$. 
We can assume, by induction on $|\calh|-|\calh_0|$, that all subgroups in 
$\calp$ have the same order.  Thus all morphisms in $\call$ between 
subgroups in $\calp$ are isomorphisms.  

Since $\calh_0$ is $\Aut(S,\calf)$-invariant and $\call_0$ is a full 
subcategory, there is a well defined restriction homomorphism
	\[ \Aut\typ^I(\call) \Right4{\Res} \Aut\typ^I(\call_0). \]
By assumption, $\calh_0$ contains all subgroups which are $\calf$-centric 
and $\calf$-radical.  Hence Theorem \ref{AFT-L} implies that all morphisms 
in $\call$ are composites of restrictions of morphisms in $\call_0$.  
Since each $\alpha\in\Aut\typ^I(\call)$ sends inclusions to inclusions, it 
also sends restrictions to restrictions, and hence 
$\alpha|_{\call_0}=\Id_{\call_0}$ only if $\alpha=\Id_\call$.  Thus $\Res$ 
is injective.  We next show it is surjective, and hence an isomorphism.

Let $\calp_*\subseteq\calp$ be a subset consisting of one 
fully normalized subgroup from each $\calf$-conjugacy class in $\calp$.  
For each $P\in\calp_*$, $\delta_P(N_S(P))\in\sylp{\Aut_\call(P)}$ by 
Proposition \ref{L-prop}(d), so there is a unique $\widehat{P}\le{}N_S(P)$ 
such that $\delta_P(\widehat{P})=O_p(\Aut_\call(P))$.  Since 
$P\notin\calh_0$, $P$ is either not $\calf$-centric or not 
$\calf$-radical.  In either case, $\widehat{P}\gneqq{}P$ by Lemma 
\ref{centrad}.  By Proposition \ref{L-prop}(e), each $\psi\in\Aut_\call(P)$ 
extends to a unique automorphism $\widehat{\psi}\in\Aut_\call(\widehat{P})$. 

Let $\nu\:\calp\Right2{}\calp_*$ be the map which sends $P$ to the unique 
subgroup $\nu(P)\in\calp_*$ which is $\calf$-conjugate to $P$.  For each 
$P\in\calp$, $\delta_{\nu(P)}(N_S(\nu(P)))\in\sylp{\Aut_\call(\nu(P))}$ by 
Proposition \ref{L-prop}(d) (and since $\nu(P)$ is fully normalized), and 
hence there is $\lambda_P\in\Iso_\call(P,\nu(P))$ such that 
	\[ \lambda_P\delta_P(N_S(P))\lambda_P^{-1}
	\le\delta_{\nu(P)}(N_S(\nu(P)))~. \]
By Proposition \ref{L-prop}(e) again, $\lambda_P$ extends to a unique 
$\widehat{\lambda}_P\in\Mor_\call(N_S(P),N_S(\nu(P)))$.  When 
$P\in\calp_*$ (so $\nu(P)=P$), we set $\lambda_P=\Id_P$, and hence 
$\widehat{\lambda}_P=\Id_{N_S(P)}$.

Fix any $\alpha_0\in\Aut\typ^I(\call_0)$; we want to extend $\alpha_0$ to 
$\call$.  By Lemma \ref{AutI}, $\alpha_0$ induces some 
$\beta\in\Aut(S,\calf)$, and $\alpha_0(P)=\beta(P)$ for all $P\in\calh_0$. 
So define $\alpha(P)=\beta(P)$ for $P\in\calh$; this is possible since 
$\calh$ is $\Aut(S,\calf)$-invariant by assumption.  By Lemma \ref{AutI} 
again, for each $P,Q\in\calh_0$ and each $\psi\in\Mor_{\call_0}(P,Q)$, 
$\pi(\alpha_0(\psi))=c_\beta(\pi(\psi))=\beta(\pi(\psi))\beta^{-1}$.  In 
other words, 
	\beqq \textup{$\psi\in\Mor_{\call_0}(P,Q)$, $g\in{}P$, 
	$\pi(\psi)(g)=h\in{}Q$ $\Longrightarrow$ 
	$\pi(\alpha_0(\psi))(\beta(g))=\beta(h)$~.} \label{e:2.8} \eeqq

We next define $\alpha$ on isomorphisms between subgroups in $\calp$.  Fix 
$P_1,P_2\in\calp$ and $\psi\in\Iso_\call(P_1,P_2)$, and set 
$P_*=\nu(P_1)=\nu(P_2)$.  There is a unique $\psi_*\in\Aut_\call(P_*)$ such 
that $\psi=\lambda_{P_2}^{-1}\circ\psi_*\circ\lambda_{P_1}$, and we set
	\[ \alpha(\psi) = 
	\bigl(\alpha_0(\widehat{\lambda}_{P_2})
	|_{\alpha(P_2),\alpha(P_*)} \bigr)^{-1} 
	\circ \bigl(\alpha_0(\widehat{\psi}_*)|_{\alpha(P_*),\alpha(P_*)} 
	\bigr) \circ 
	\bigl(\alpha_0(\widehat{\lambda}_{P_1})|_{\alpha(P_1),\alpha(P_*)}
	\bigr) ~. \]
Note that $\widehat{\lambda}_{P_1}$, $\widehat{\lambda}_{P_2}$, and 
$\widehat{\psi}_*$ are all in $\Mor(\call_0)$, since all subgroups strictly 
containing subgroups in $\calp$ are in $\calh_0=\Ob(\call_0)$ by 
assumption.  Also, the restrictions are well defined (for example, 
$\pi(\alpha_0(\widehat{\lambda}_{P_i}))(\alpha(P_i))=\alpha(P^*)$) by 
\eqref{e:2.8}. 

Each morphism in $\call$ not in $\call_0$ factors uniquely as an 
isomorphism between subgroups in $\calp$ followed by an inclusion 
(Proposition \ref{L-prop}(b$'$)), and thus these definitions extend to 
define $\alpha$ as a map from $\Mor(\call)$ to itself.  This clearly 
preserves composition of isomorphisms between subgroups in $\calp$.  To 
prove that $\alpha$ is a functor, it remains to show it preserves 
composites of inclusions followed by isomorphisms in $\call_0$.  This 
means showing, for each $P_1,P_2\in\calp$, each $P_i\lneqq{}Q_i$, and each 
$\psi\in\Iso_\call(P_1,P_2)$ which extends to 
$\varphi\in\Mor_{\call_0}(Q_1,Q_2)$, that 
$\alpha(\psi)=\alpha_0(\varphi)|_{\alpha(P_1),\alpha(P_2)}$.  Since 
$N_{Q_i}(P_i)\gneqq{}P_i$, we can assume $P_i\nsg{}Q_i$ for $i=1,2$.  Set 
$P_*=\nu(P_1)=\nu(P_2)$ again, and set 
$R_i=\pi(\widehat{\lambda}_{P_i})(Q_i) \gneqq{}P_*$.  Then $P_*\nsg{}R_i$ 
since $P_i\nsg{}Q_i$.  We saw that $\psi$ factors in a unique way 
$\psi=\lambda_{P_2}^{-1}\circ\psi_*\circ\lambda_{P_1}$ for 
$\psi_*\in\Aut_\call(P_*)$.  We also have 
$\varphi=\widebar{\lambda}_{P_2}^{-1}\circ\varphi_*\circ 
\widebar{\lambda}_{P_1}$, where $\widebar{\lambda}_{P_i}= 
\widehat{\lambda}_{P_i}|_{Q_i,R_ i}$ and 
$\varphi_*\in\Mor_{\call_0}(R_1,R_2)$.  Thus 
$\alpha_0(\widebar{\lambda}_{P_i})$ is a restriction of 
$\alpha_0(\widehat{\lambda}_{P_i})$ ($i=1,2$), and hence an extension of 
$\alpha(\lambda_{P_i})$.  

It remains to show $\alpha(\psi_*)$ is the restriction of 
$\alpha_0(\varphi_*)$.  By definition, $\alpha(\psi_*)$ is the 
restriction to $\alpha(P_*)$ of $\alpha_0(\widehat{\psi}_*)$, where 
$\widehat{\psi}_*\in\Aut_\call(\widehat{P_*})$.  Set 
$T_i=\gen{\widehat{P_*},R_i}$.  By Proposition \ref{L-prop}(e), since 
$\psi_*\in\Aut_\call(P_*)$ extends to 
$\widehat{\psi}_*\in\Aut_\call(\widehat{P_*})$ and to 
$\varphi_*\in\Mor_\call(R_1,R_2)$, there is 
$\widebar{\varphi}_*\in\Mor_\call(T_1,T_2)$ which extends both 
$\widehat{\psi}_*$ and $\varphi_*$.  Hence $\alpha_0(\widebar{\varphi}_*)$ 
extends both $\alpha_0(\widehat{\psi}_*)$ and $\alpha_0(\varphi_*)$ (all 
of these are in $\call_0$), and thus $\alpha(\psi_*)$ is a restriction of 
each of the latter.  This finishes the proof that $\alpha$ is a functor.  
By construction, $\alpha$ is isotypical, sends inclusions to inclusions, 
and extends $\alpha_0$; and thus $\Res$ is surjective.  

We have now shown that restriction defines an isomorphism from 
$\Aut\typ^I(\call)$ to $\Aut\typ^I(\call_0)$.  By Lemma \ref{OutI1}(a), 
the outer automorphism groups of $\call$ and $\call_0$ are defined by 
dividing out by conjugation by elements of $\Aut_\call(S)$.  Hence the 
induced homomorphism 
	\[ \Out\typ(\call) \Right4{R} \Out\typ(\call_0) \]
is also an isomorphism.
\end{proof}

\newsubb{Normal fusion subsystems}{s:normfus}

Let $\calf$ be a saturated fusion system over a finite $p$-group $S$.  By 
a \emph{(saturated) fusion subsystem} of $\calf$ 
over a subgroup $S_0\le{}S$, we mean a subcategory $\calf_0\subseteq\calf$ 
whose objects are the subgroups of $S_0$, and which is itself a 
(saturated) fusion system over $S_0$.  

The following definition of a normal fusion subsystem is equivalent to that 
of Puig \cite[\S\,6.4]{Puig}, and also to Aschbacher's definition of an 
$\calf$-invariant subsystem \cite[\S\,3]{Asch} (except that they do 
not require the subsystem to be saturated).  See \cite[Proposition 
6.6]{Puig} and \cite[Theorem 3.3]{Asch} for proofs of those equivalences.

\begin{Defi} \label{d:F0<|F}
Let $\calf$ be a saturated fusion system over a finite $p$-group $S$, and 
let $\calf_0\subseteq\calf$ be a saturated fusion subsystem over 
$S_0\le{}S$.  Then $\calf_0$ is \emph{normal} in $\calf$ 
($\calf_0\nsg\calf$) if 
\begin{enumerate}[\rm(i) ]  
\item $S_0$ is strongly closed in $\calf$;

\item for each $P,Q\le{}S_0$ and each $\varphi\in\homf(P,Q)$, there are
$\alpha\in\autf(S_0)$ and $\varphi_0\in\Hom_{\calf_0}(\alpha(P),Q)$ such 
that $\varphi=\varphi_0\circ\alpha|_{P,\alpha(P)}$; and

\item for each $P,Q\le{}S_0$, each $\varphi\in\Hom_{\calf_0}(P,Q)$, and 
each $\beta\in\autf(S_0)$, $\beta\varphi\beta^{-1} 
\in\Hom_{\calf_0}(\beta(P),\beta(Q))$.
\end{enumerate}
\end{Defi}

We next list some of the basic properties of normal fusion subsystems,   
starting with the following technical result.  

\begin{Lem} \label{F0<|F(d)}
Fix a saturated fusion system $\calf$ over a finite $p$-group $S$.  Let 
$\calf_0\subseteq\calf$ be a fusion subsystem (not necessarily saturated) 
over the subgroup $S_0\nsg{}S$, which satisfies conditions (i--iii) in 
Definition \ref{d:F0<|F}.  Assume $P_0\le{}S_0$ is $\calf_0$-centric and 
fully normalized in $\calf$, and $\Out_{S_0}(P_0)\cap 
O_p(\Out_{\calf_0}(P_0))=1$.  Then there is $P\le{}S$ which is 
$\calf$-centric and $\calf$-radical and such that $P\cap S_0=P_0$. 
\end{Lem}

\begin{proof}  Set 
	\[ P = \{x\in{}N_S(P_0) \,|\, c_x\in O_p(\autf(P_0)) \} ~. \]
If $x\in{}P\cap{}S_0$, then 
$c_x\in{}O_p(\autf(P_0))\cap\Aut_{\calf_0}(P_0)=O_p(\Aut_{\calf_0}(P_0))$ 
($\Aut_{\calf_0}(P_0)$ is normal in $\autf(P_0)$ by 
\ref{d:F0<|F}(ii--iii)), so $c_x\in\Aut_{S_0}(P_0)\cap 
O_p(\Aut_{\calf_0}(P_0))=\Inn(P_0)$, and $x\in{}P_0$ since $P_0$ is 
$\calf_0$-centric.  Thus $P\cap{}S_0=P_0$.

By construction, $N_S(P)=N_S(P_0)$.  So if $Q$ if $\calf$-conjugate 
to $P$ and $Q_0=Q\cap{}S_0$, then 
$|N_S(Q)|\le|N_S(Q_0)|\le|N_S(P_0)|=|N_S(P)|$ since $P_0$ is 
fully normalized in $\calf$ and $\calf$-conjugate to $Q_0$.  This proves 
that $P$ is fully normalized in $\calf$.  

Now, $\Aut_S(P_0)\ge{}O_p(\autf(P_0))$ since $P_0$ is fully normalized in 
$\calf$.  So $\Aut_P(P_0)=O_p(\autf(P_0))$, and hence this is normal in 
$\autf(P_0)$.  By the extension axiom, the restriction homomorphism 
$\autf(P)\Right2{}\autf(P_0)$ is surjective, and thus sends 
$O_p(\autf(P))$ into $O_p(\autf(P_0))$.  So for all $x\in{}N_S(P)$ such 
that $c_x\in{}O_p(\autf(P))$, $c_x\in{}O_p(\autf(P_0))$, and hence 
$x\in{}P$.  Since $P$ is fully normalized, it is $\calf$-centric and 
$\calf$-radical by Lemma \ref{centrad}.  
\end{proof}

The following is our main lemma listing properties of normal pairs of 
fusion systems.  Recall that $O_p(\calf)$ is the largest normal 
$p$-subgroup of the fusion system $\calf$.  

\begin{Lem} \label{F0<|F}
Fix a saturated fusion system $\calf$ over a finite $p$-group $S$, and let 
$\calf_0\nsg\calf$ be a normal fusion subsystem over the normal subgroup 
$S_0\nsg{}S$.  Then the following hold.
\begin{enuma}  
\item Each $\calf$-conjugacy class contains a subgroup $P\le{}S$ such that 
$P$ and $P\cap{}S_0$ are both fully normalized in $\calf$, and $P\cap{}S_0$ 
is fully normalized in $\calf_0$.

\item For each $P,Q\le{}S_0$ and each $\varphi\in\isof(P,Q)$, 
$\varphi\Aut_{\calf_0}(P)\varphi^{-1}=\Aut_{\calf_0}(Q)$.

\item The set of $\calf_0$-centric subgroups of $S_0$, and the set of 
$\calf_0$-radical subgroups of $S_0$, are both invariant under 
$\calf$-conjugacy. 

\item If $P\le{}S$ is $\calf$-centric and $\calf$-radical, then 
$P\cap{}S_0$ is $\calf_0$-centric and $\calf_0$-radical.  
Conversely, if $P_0\le{}S_0$ is $\calf_0$-centric, $\calf_0$-radical, and 
fully normalized in $\calf$, then there is $P\le{}S$ which is 
$\calf$-centric and $\calf$-radical, and such that $P\cap{}S_0=P_0$.  

\item $O_p(\calf_0)$ is normal in $\calf$, and 
$O_p(\calf_0)=O_p(\calf)\cap{}S_0$.

\end{enuma}
\end{Lem}

\begin{proof}  Throughout the proof, whenever $P\le{}S$, we write 
$P_0=P\cap{}S_0$ for short.

\smallskip

\noindent\textbf{(a) }  Fix $Q\le{}S$.  By \cite[Proposition 
A.2(b)]{BLO2}, there are subgroups $R\le{}S$ and $P_0\le{}S_0$ which are 
fully normalized in $\calf$, and morphisms 
$\varphi\in\homf(N_S(Q),N_S(R))$ and $\psi\in\homf(N_S(R_0),N_S(P_0))$ 
such that $\varphi(Q)=R$ and $\psi(R_0)=P_0$.  Set $P=\psi(R)$ (note that 
$P\cap{}S_0=P_0$ since $S_0$ is strongly closed).  Since 
$N_S(R)\le{}N_S(R_0)$, $P$ is also fully normalized in $\calf$.  Also, 
$P=\psi\circ\varphi(Q)$ is $\calf$-conjugate to $Q$.  

By \cite[Proposition A.2(b)]{BLO2} again, if $P^*_0$ is $\calf$-conjugate 
to $P_0$, then there is a morphism in $\calf$ from $N_S(P^*_0)$ to 
$N_S(P_0)$ which sends $P^*_0$ to $P_0$.  In particular, 
$|N_{S_0}(P^*_0)|\le|N_{S_0}(P_0)|$, and hence $P_0$ is also fully 
normalized in $\calf_0$.

\smallskip

\noindent\textbf{(b) } Fix $P,Q\le{}S_0$ and $\varphi\in\isof(P,Q)$.  By 
condition (ii) in Definition \ref{d:F0<|F}, there are $\alpha\in\autf(S_0)$ 
and $\varphi_0\in\Iso_{\calf_0}(\alpha(P),Q)$ such that 
$\varphi=\varphi_0\circ\alpha|_{P,\alpha(P)}$.  Hence 
	\[ \varphi\Aut_{\calf_0}(P)\varphi^{-1} = 
	\varphi_0\Aut_{\calf_0}(\alpha(P))\varphi_0^{-1} = 
	\Aut_{\calf_0}(Q)~, \]
where the first equality holds by condition (iii) in Definition 
\ref{d:F0<|F}. 

\smallskip

\noindent\textbf{(c) }  Fix $P\le{}S_0$, let $\calp$ be the 
$\calf$-conjugacy class of $P$, and let $\calp_0$ be its 
$\calf_0$-conjugacy class.  

If $P$ is $\calf_0$-centric, then $C_{S_0}(P^*)=Z(P^*)$ for all 
$P^*\in\calp_0$.  For all $R\in\calp$, there is $\alpha\in\autf(S_0)$ such 
that $\alpha(R)\in\calp_0$ (condition (ii) in Definition \ref{d:F0<|F}), 
and hence $C_{S_0}(R)=Z(R)$.  Since this holds for all subgroups in 
$\calp$, all of these subgroups are $\calf_0$-centric.

Now assume $P$ is $\calf_0$-radical; then $O_p(\Out_{\calf_0}(P^*))=1$ for 
all $P^*\in\calp_0$.  If $R\in\calp$, and $\alpha\in\autf(S_0)$ is such 
that $\alpha(R)\in\calp_0$, then by condition (iii) in Definition 
\ref{d:F0<|F}, conjugation by $\alpha$ sends 
$\Out_{\calf_0}(R)\le\outf(R)$ isomorphically to 
$\Out_{\calf_0}(\alpha(R))$.  Since $O_p(\Out_{\calf_0}(\alpha(R)))=1$, 
$O_p(\Out_{\calf_0}(R))=1$.  So all subgroups in $\calp$ are 
$\calf_0$-radical.

\smallskip

\noindent\textbf{(d) } The second statement was shown in Lemma 
\ref{F0<|F(d)}.  It remains to prove the first.

Assume $P$ is $\calf$-centric and $\calf$-radical. 
We must show that $P_0$ is $\calf_0$-centric and $\calf_0$-radical.  By 
(c), this is independent of the choice of $P$ in its $\calf$-conjugacy 
class, and hence by (a), it suffices to prove it when $P_0$ is fully 
normalized in $\calf_0$.  By (b), $\Aut_{\calf_0}(P_0)$ is normal in 
$\autf(P_0)$, and hence 
	\[ O_p(\Aut_{\calf_0}(P_0)) \le O_p(\autf(P_0))~. \]

Let $T$ be the subgroup of all $x\in{}N_{S_0}(P_0)$ such that 
$c_x\in{}O_p(\Aut_{\calf_0}(P_0))$.  If $x\in{}T\cap{}N_{S}(P)$, then 
$c_x\in{}O_p(\autf(P_0))$, and $c_x$ induces the identity on $P/P_0$ since 
$[x,P]\le{}P\cap[S_0,P]\le{}P_0$.  Thus 
$c_x\in{}O_p(\autf(P))$ by Lemma \ref{mod-Fr}, and $x\in{}P$ by Lemma 
\ref{centrad} since $P$ is centric and radical in $\calf$.  

Thus $T\cap{}N_{S}(P)\le{}P_0$.  Also, $P$ normalizes $T$ by construction, 
so $T\le{}P_0$ by Lemma \ref{QnleqP}.  Hence $P_0$ is centric and radical 
in $\calf_0$ by Lemma \ref{centrad} again.  

\smallskip

\noindent\textbf{(e) }  Set $Q=O_p(\calf_0)$ and $R=O_p(\calf)$ for short. 
To prove that $Q\nsg\calf$ and $R_0=Q$, it suffices to show that 
$Q\nsg\calf$ and $R_0\nsg\calf_0$.  We apply Proposition \ref{norm<=>}, 
which says that a subgroup is normal in a saturated fusion system if and 
only if it is strongly closed and contained in all subgroups which are 
centric and radical.  Since an intersection of strongly closed subgroups 
is strongly closed, $R_0$ is strongly closed in $\calf$ and hence in 
$\calf_0$.  

If $P\le{}S$ is $\calf$-centric and $\calf$-radical, then $P_0$ is 
$\calf_0$-centric and $\calf_0$-radical by (d), so $P\ge{}P_0\ge{}Q$.  If 
$P_0$ is $\calf_0$-centric and $\calf_0$-radical, then the same holds for 
each subgroup in its $\calf$-conjugacy class by (c).  So by (d), there is 
$P^*\le{}S$ which is $\calf$-centric and $\calf$-radical with $P^*_0$ 
$\calf$-conjugate to $P_0$; $P^*\ge{}R$, and hence $P^*_0$ and $P_0$ both 
contain $R_0$.  

It remains to prove that $Q$ is strongly closed in $\calf$.  Fix 
$\calf$-conjugate elements $g,h\in{}S$ such that $g\in{}Q$; we must show 
$h\in{}Q$.  Since $S_0$ is strongly closed in $\calf$ (since 
$\calf_0\nsg\calf$), $h\in{}S_0$.  Fix 
$\varphi\in\isof(\gen{g},\gen{h})$ with $\varphi(g)=h$.  Since 
$\calf_0\nsg\calf$, there are morphisms $\chi\in\autf(S_0)$ and 
$\varphi_0\in\Iso_{\calf_0}(\gen{g},\gen{\chi^{-1}(h)})$ such that 
$\varphi=\chi\circ\varphi_0$.  Then $g'\defeq\varphi_0(g)\in{}Q$, and 
$h=\chi(g')$.  The invariance condition (iii) in Definition \ref{d:F0<|F} 
implies that $\chi$ sends a normal subgroup of $\calf_0$ to another normal 
subgroup.  Thus $\chi(Q){\cdot}Q$ is also normal in $\calf_0$, so 
$\chi(Q)=Q$ since $Q$ is the largest subgroup of $S_0$ normal in 
$\calf_0$, and thus $h=\chi(g')\in{}Q$. 
\end{proof}

We now turn to the specific examples of normal fusion subsystems which we 
work with in this paper.  We first look at those of $p$-power 
index and of index prime to $p$.  Two other definitions are first needed.  
For any saturated fusion system $\calf$, the \emph{focal subgroup} 
$\foc(\calf)$ and the \emph{hyperfocal subgroup} $\hyp(\calf)$ are 
defined by 
	\begin{align*}
	\foc(\calf) &=
	\gen{s^{-1}t\,|\,s,t\in{}S \textup{ and $\calf$-conjugate} }
	= \gen{s^{-1}\alpha(s) \,|\, s\in{}P\le{}S,\ \alpha\in\autf(P) } \\
	\hyp(\calf) &= \gen{s^{-1}\alpha(s) \,|\, 
	s\in{}P\le{}S,\ \alpha\in O^p(\autf(P))}~.
	\end{align*}
Note that in \cite{BCGLO2}, we wrote $O^p_\calf(S)=\hyp(\calf)$.  

The following definition also includes many fusion subsystems which are 
not normal.

\begin{Defi}[{\cite[Definition 3.1]{BCGLO2}}] \label{d:p-p'-index}
Let $\calf$ be a saturated fusion system over a finite $p$-group $S$, and 
let $\calf_0\subseteq\calf$ be a saturated fusion subsystem over a subgroup 
$S_0\le{}S$.
\begin{enuma}  
\item $\calf_0$ \emph{has $p$-power index} in $\calf$ if 
$\hyp(\calf)\le{}S_0\le{}S$, and $\Aut_{\calf_0}(P)\ge O^p(\autf(P))$ 
for all $P\le{}S_0$.  

\item $\calf_0$ \emph{has index prime to $p$} in $\calf$ if $S_0=S$, and 
$\Aut_{\calf_0}(P)\ge O^{p'}(\autf(P))$ for all $P\le{}S$.  
\end{enuma}
\end{Defi}

Recall that despite the terminology, these are not analogous to subgroups 
of a finite group of $p$-power index or index prime to $p$.  Instead, they 
are analogous to subgroups which contain a normal subgroup having 
appropriate index.

The following theorem gives a complete description of all such fusion 
subsystems.

\begin{Thm}[{\cite[Theorems 4.3 \& 5.4]{BCGLO2}}] \label{t:O^p(F)}
The following hold for any saturated fusion system $\calf$ over a finite 
$p$-group $S$.  
\begin{enuma}
\item For each subgroup $S_0\le{}S$ containing the hyperfocal subgroup 
$\hyp(\calf)$, there is a unique fusion subsystem $\calf_0$ over $S_0$ 
of $p$-power index in $\calf$.  Thus $\calf$ contains a proper fusion 
subsystem of $p$-power index if and only if $\hyp(\calf)\lneqq{}S$, or 
equivalently $\foc(\calf)\lneqq{}S$.

\item There is a subgroup $\Gamma\nsg\outf(S)$ with the following 
properties.  For each subsystem $\calf_0\subseteq\calf$ of index prime to 
$p$, $\Out_{\calf_0}(S)\ge\Gamma$.  Conversely, for each $H\le\outf(S)$ 
containing $\Gamma$, there is a unique subsystem $\calf_0\subseteq\calf$ 
of index prime to $p$ with $\Out_{\calf_0}(S)=H$.
\end{enuma}
\end{Thm}

\begin{proof}  The only part not shown in \cite{BCGLO2} is that 
$\hyp(\calf)\lneqq{}S$ implies $\foc(\calf)\lneqq{}S$.  By Theorem 
\ref{AFT}, 
	\[ \foc(\calf) = \gen{ s^{-1}\alpha(s)\,|\, s\in{}P\le{}S,\ P 
	\textup{ fully normalized in $\calf$},\ \alpha\in\autf(P)}~. \]
Since $\autf(P)=O^p(\autf(P)){\cdot}\Aut_S(P)$ when $P$ is fully 
normalized, and since $s^{-1}\alpha(s)\in[S,S]$ when $s\in{}P$ and 
$\alpha\in\Aut_S(P)$, we have $\foc(\calf)=\hyp(\calf){\cdot}[S,S]$.  
Also, $\hyp(\calf)\lneqq{}S$ implies there is $Q\nsg{}S$ such that 
$[S:Q]=p$ and $\hyp(\calf)\le{}Q$.  Then $[S,S]\le{}Q$ since $S/Q$ is 
abelian, and hence $\foc(\calf)\le{}Q\lneqq{}S$.
\end{proof}

One can show, in the situation of Theorem \ref{t:O^p(F)}, that a fusion 
subsystem of $p$-power index is normal in $\calf$ exactly when its 
underlying $p$-group is normal in $S$, and that a fusion subsystem 
$\calf_0\subseteq\calf$ of index prime to $p$ is normal in $\calf$ exactly 
when $\Aut_{\calf_0}(S)$ is normal in $\autf(S)$.  But in fact, we will 
only be concerned here (in Proposition \ref{fus-ex}(a,b)) with the minimal 
such fusion subsystems, defined as follows.

\begin{Defi} \label{d:Op-p'}
For any saturated fusion system $\calf$ over a finite $p$-group $S$, 
$O^p(\calf)$ and $O^{p'}(\calf)$ denote the unique minimal saturated 
fusion subsystems of $p$-power index over $\hyp(\calf)$, or of index 
prime to $p$ over $S$, respectively.
\end{Defi}

We next recall the definitions of the normalizer fusion systems 
$N_\calf^K(Q)$ (cf. \cite[\S\,2.8]{Puig} or \cite[Definitions A.1, 
A.3]{BLO2}).  For any group $G$, any subgroup $Q\le{}G$, and any 
$K\le\Aut(Q)$, define 
	\[ N_G^K(Q) = \big\{g\in N_G(Q) \,\big|\, c_g|_Q\in K\big\}\,. \]
For example, $N_G^{\Aut(Q)}(Q)=N_G(Q)$ is the usual normalizer, and 
$N_G^{\{\Id\}}(Q)=C_G(Q)$ is the centralizer.  

Let $\calf$ be a saturated fusion system over a finite $p$-group $S$, and 
fix $Q\le{}S$ and $K\le\Aut(Q)$.  We say $Q$ is \emph{fully 
$K$-normalized} if for each $Q^*$ which is $\calf$-conjugate to $Q$ and each 
$\varphi\in\isof(Q,Q^*)$, $|N_S^K(Q)|\ge|N_S^{\varphi{}K\varphi^{-1}}(Q^*)|$.  
Let $N_\calf^K(Q)$ be the fusion system over $N_S^K(Q)$ defined by 
setting, for all $P,R\le{}N_S^K(Q)$,
	\begin{multline*}  
	\Hom_{N_\calf^K(Q)}(P,R) = \bigl\{ \varphi\in\homf(P,R) 
	\,\big|\, \exists\, \widebar{\varphi}\in\homf(PQ,RQ),\\ 
	\widebar{\varphi}|_P=\varphi,\ \widebar{\varphi}(Q)=Q,\ 
	\widebar{\varphi}|_Q\in{}K \bigr\}~. 
	\end{multline*}
As special cases, $C_\calf(Q)=N_\calf^{\{\Id\}}(Q)$ and 
$N_\calf(Q)=N_\calf^{\Aut(Q)}(Q)$.  By \cite[Proposition 2.15]{Puig} or 
\cite[Proposition A.6]{BLO2}, if $Q$ is fully $K$-normalized in $\calf$, 
then $N_\calf^K(Q)$ is a saturated fusion system.  If $K\ge\Inn(Q)$, then 
$Q$ is normal in $N_\calf^K(Q)$ by definition.

This construction is motivated by the following proposition.

\begin{Prop} \label{N_F=F(N_G)}
Fix a finite group $G$ and $S\in\sylp{G}$, and set $\calf=\calf_S(G)$.  For 
$Q\le{}S$ and $K\le\Aut(Q)$, $Q$ is fully $K$-normalized in $\calf$ if and 
only if $N_S^K(Q)\in\sylp{N_G^K(Q)}$.  When this is the case, then 
$N_\calf^K(Q)=\calf_{N_S^K(Q)}(N_G^K(Q))$.  
\end{Prop}

\begin{proof}  Fix any $T\in\sylp{N_G^K(Q)}$.  Since $T$ normalizes $Q$, 
$TQ$ is a $p$-group, and hence there is $a\in{}G$ such that 
$a(TQ)a^{-1}\le{}S$.  Set $Q^*=aQa^{-1}$, $\varphi=c_a\in\isof(Q,Q^*)$, and 
$K^*=\varphi{}K\varphi^{-1}\le\Aut(Q^*)$.  Then 
$aTa^{-1}\in\sylp{N_G^{K^*}(Q^*)}$ and $aTa^{-1}\le{}S$, so 
$aTa^{-1}=N_S^{K^*}(Q^*)$.  It follows that $Q^*$ is fully $K^*$-normalized 
in $\calf$; and also that $Q$ is fully $K$-normalized if and only if 
$|N_S^K(Q)|=|T|$, if and only if $N_S^K(Q)\in\sylp{N_G^K(Q)}$.  

Now assume $Q$ is fully $K$-normalized.  
Set $N_\calf=N_\calf^K(Q)$, 
$N_G=N_G^K(Q)$, and $N_S=N_S^K(Q)$ for 
short, and fix $P,R\le{}N_S$.  If $\varphi\in\Hom_{N_\calf}(P,R)$, then it 
extends to some $\widebar{\varphi}=c_g\in\homf(PQ,RQ)$ such that 
$\widebar{\varphi}(Q)=Q$ and $\widebar{\varphi}|_Q\in{}K$, so 
$g\in{}N_G$, and $\varphi=c_g\in\Hom_{N_G}(P,R)$.  This proves that 
$\Hom_{N_\calf}(P,R)\subseteq\Hom_{N_G}(P,R)$.  The opposite inclusion 
is clear, and thus $N_\calf=\calf_{N_S}(N_G)$. 
\end{proof}

We now give some examples of normal fusion subsystems:  
examples which will be important later in the paper.  The most obvious 
example is the inclusion $\calf_{S_0}(G_0)\subseteq\calf_S(G)$ when 
$G_0\nsg{}G$ are finite groups and $S_0=S\cap{}G_0$, but this case will be 
handled later (Proposition \ref{L(G0)<|L(G)}).  

\begin{Prop} \label{fus-ex}
The following hold for any saturated fusion system $\calf$ over a finite 
$p$-group $S$.
\begin{enuma}  
\item $O^p(\calf)\nsg\calf$.
\item $O^{p'}(\calf)\nsg\calf$.
\item For each $Q\nsg\calf$ and each $K\nsg\Aut(Q)$, 
$N_\calf^K(Q)\nsg\calf$.
\end{enuma}
\end{Prop}

\begin{proof}  \textbf{(a,b) } The subsystem $O^{p'}(\calf)$ is a fusion 
system over $S$, which is clearly strongly closed.  The subsystem 
$O^p(\calf)$ is a fusion system over the hyperfocal subgroup
	\[ \hyp(\calf) = \Gen{s^{-1}\alpha(s) 
	\,\big|\, s\in{}P\le{}S,\alpha\in O^p(\autf(P))}. \]
We claim $\hyp(\calf)$ is also strongly closed in $\calf$.  By Theorem 
\ref{AFT}, it suffices to show, for each $P\le{}S$ fully normalized 
in $\calf$, that $P\cap\hyp(\calf)$ is $\autf(P)$-invariant.  Since 
$\Aut_S(P)\in\sylp{\autf(P)}$, $\autf(P)$ is generated by $\Aut_S(P)$ and 
$O^p(\autf(P))$.  Finally, $P\cap\hyp(\calf)$ is $\Aut_S(P)$-invariant 
since $\hyp(\calf)\nsg{}S$ and is $O^p(\autf(P))$-invariant by 
definition of $\hyp(\calf)$, and this proves the claim.

Condition (ii) in Definition \ref{d:F0<|F} holds for these two subsystems 
by Lemma 3.4 and Proposition 3.8(b,c) in \cite{BCGLO2}.  Condition 
(iii) (invariance under conjugation by elements of $\autf(S_0)$) follows 
from \cite[Theorems 4.3 \& 5.4]{BCGLO2}:  $O^p(\calf)$ is the unique 
saturated fusion subsystem over $\hyp(\calf)$ of $p$-power index in 
$\calf$, and $O^{p'}(\calf)$ is the unique saturated fusion subsystem over 
$S$ with its given $\Aut_{O^{p'}(\calf)}(S)$, of index prime to $p$ in 
$\calf$.  

\smallskip

\noindent\textbf{(c) } Assume $g,h\in{}S$ are $\calf$-conjugate.  Since 
$Q\nsg\calf$, there is $\varphi\in\homf(\gen{Q,g},\gen{Q,h})$ such that 
$\varphi(g)=h$ and $\varphi(Q)=Q$.  Set $\varphi_0=\varphi|_Q\in\autf(Q)$. 
Then $c_h=\varphi_0{}c_g\varphi_0^{-1}$ in $\autf(Q)$.  Since 
$K\nsg\Aut(Q)$, $g\in{}N_S^K(Q)$ (i.e., $c_g\in{}K$) if and only if 
$h\in{}N_S^K(Q)$.  This proves that $N_S^K(Q)$ is strongly closed in 
$\calf$.

Set $\Aut_S^K(Q)=K\cap\Aut_S(Q)$ and $\autf^K(Q)=K\cap\autf(Q)$.  
Fix $P,R\le{}N_S^K(Q)$ and $\varphi\in\homf(P,R)$.  Since $Q\nsg\calf$, 
there is $\widehat{\varphi}\in\homf(QP,QR)$ such that 
$\widehat{\varphi}|_P=\varphi$ and $\widehat{\varphi}(Q)=Q$.  Set 
$\varphi_0=\widehat{\varphi}|_Q\in\autf(Q)$.  Since $K\nsg\Aut(Q)$ and 
$\Aut_S(Q)\in\sylp{\autf(Q)}$, $\Aut_S^K(Q)\in\sylp{\autf^K(Q)}$.  Since 
$\varphi_0\Aut_S^K(Q)\varphi_0^{-1}$ is contained in $\autf^K(Q)$ 
(again since $K$ is normal), there is $\chi\in\autf^K(Q)$ such that 
	\[ (\chi\varphi_0)\Aut_S^K(Q)(\chi\varphi_0)^{-1} =
	\Aut_S^K(Q)~. \]
By the extension axiom, there is 
$\widebar{\varphi}\in\homf(N_S^K(Q){\cdot}Q,S)$ such that 
$\widebar{\varphi}|_Q=\chi\varphi_0$.  Furthermore, 
$\widebar{\varphi}(N_S^K(Q))=N_S^K(Q)$ since $\chi\varphi_0$ normalizes 
$\Aut_S^K(Q)$.  

Set $P_1=\widebar{\varphi}(P)$, $\widehat{\psi}=
\widehat{\varphi}\circ(\widebar{\varphi}|_{PQ,P_1Q})^{-1} 
\in\homf(P_1Q,RQ)$, and $\psi=\widehat{\psi}|_{P_1,R}$.  Then 
$\widehat{\psi}|_Q=\chi^{-1}$, so $\psi\in\Hom_{N_\calf^K(Q)}(P_1,R)$, and 
$\varphi=\psi\circ\widebar{\varphi}|_{P,P_1}$.  This proves condition (ii) 
in Definition \ref{d:F0<|F}.  The last condition --- the subsystem 
$N_\calf^K(Q)$ is invariant under conjugation by elements of 
$\autf(N_S^K(Q))$ --- is clear.
\end{proof}

We just showed that $O^{p'}(\calf)$ is normal in $\calf$ for any $\calf$.  
The following lemma can be thought of as a ``converse'' to this.

\begin{Lem} \label{F0<|F,S0=S}
Assume $\calf_0\nsg\calf$ is a normal pair of fusion systems over the same 
finite $p$-group $S$.  Then $\calf_0$ has index prime to $p$ in $\calf$, 
and thus $\calf_0\supseteq{}O^{p'}(\calf)$.
\end{Lem}

\begin{proof}  If $P,Q\le{}S$ are $\calf$-conjugate, then by condition 
(ii) in Definition \ref{d:F0<|F}, $P$ is $\calf_0$-conjugate to 
$\alpha(Q)$ for some $\alpha\in\autf(S)$.  Since 
$|N_S(Q)|=|N_S(\alpha(Q))|$, this shows that $P$ is fully normalized in 
$\calf_0$ if and only if it is fully normalized in $\calf$.

If $P\le{}S$ is fully normalized in $\calf_0$ (and hence in $\calf$), then 
$\Aut_{\calf_0}(P)$ contains $\Aut_S(P)\in\sylp{\autf(P)}$.  Also, since 
$\Aut_{\calf_0}(P)$ is normal in $\autf(P)$ by Lemma \ref{F0<|F}(b), 
$\Aut_{\calf_0}(P)$ contains $O^{p'}(\autf(P))$.  Since this 
property depends only on the isomorphism class of $P$ in $\calf_0$, it holds 
for all $P\le{}S$.  So $\calf_0$ has index prime to $p$ in $\calf$ by 
Definition \ref{d:p-p'-index}(b).  
\end{proof}

\newsubb{Normal linking subsystems}{s:normlink}

The following definition of a normal linking subsystem seems to be the most 
appropriate one for our needs here; it is also the one used in \cite{O3}.  
In the following definition (and elsewhere), whenever we say that 
$\call_0\subseteq\call$ is a \emph{pair of linking systems} associated to 
$\calf_0\subseteq\calf$ (or $\call_0$ is a \emph{linking subsystem}), it is 
understood not only that $\call_0$ is a subcategory of $\call$, but also 
that the structural functors for $\call_0$ are the restrictions of the 
structural functors 
$\calt_{\Ob(\call)}(S)\Right2{\delta}\call\Right2{\pi}\calf$ for $\call$.

\begin{Defi} \label{d:L0<|L}
Fix a pair of saturated fusion systems $\calf_0\subseteq\calf$ over 
finite $p$-groups $S_0\nsg{}S$ such that $\calf_0\nsg\calf$, and let 
$\call_0\subseteq\call$ be a pair of associated linking systems.  Then 
$\call_0$ is \emph{normal} in $\call$ ($\call_0\nsg\call$) if 
\begin{enumerate}[\rm(i) ]
\item $\Ob(\call)=\{P\le{}S \,|\, P\cap{}S_0\in\Ob(\call_0) \}$; 

\item for all $P,Q\in\Ob(\call_0)$ and $\psi\in\Mor_{\call}(P,Q)$, 
there are morphisms $\gamma\in\Aut_\call(S_0)$ and 
$\psi_0\in\Mor_{\call_0}(\gamma(P),Q)$ such that 
$\psi=\psi_0\circ\gamma|_{P,\gamma(P)}$; and 

\item for all $\gamma\in\Aut_\call(S_0)$ and $\psi\in\Mor(\call_0)$, 
$\gamma\psi\gamma^{-1}\in\Mor(\call_0)$.  
\end{enumerate}
Here, for $P,Q\in\Ob(\call_0)$, and $\psi\in\Mor_{\call_0}(P,Q)$, we write
$\gamma(P)=\pi(\gamma)(P)$, $\gamma(Q)=\pi(\gamma)(Q)$, and
	\[ \gamma\psi\gamma^{-1} = 
	\gamma|_{Q,\gamma(Q)}\circ\psi\circ (\gamma|_{P,\gamma(P)})^{-1} 
	\in \Mor_{\call}(\gamma(P),\gamma(Q)) \]
for short.  For any such pair $\call_0\nsg\call$, the quotient group 
$\call/\call_0$ is defined by setting
	\[ \call/\call_0 = \Aut_\call(S_0) / \Aut_{\call_0}(S_0). \]
Also, $\call_0$ is \emph{centric} in $\call$ if for each 
$\gamma\in\Aut_\call(S_0){\sminus}\Aut_{\call_0}(S_0)$, there is 
$\psi\in\Mor(\call_0)$ such that $\gamma\psi\gamma^{-1}\ne\psi$.
\end{Defi}

In the situation of Definition \ref{d:L0<|L}, we will sometimes say that 
$\call_0\nsg\call$ is a \emph{normal pair of linking systems} associated 
to $\calf_0\nsg\calf$, or just that $\SFL[_0]\nsg\SFL$ is a \emph{normal 
pair}.  

One source of normal pairs of linking systems is a normal pair of finite 
groups; at least, under certain conditions.

\begin{Prop} \label{L(G0)<|L(G)}
Fix a pair $G_0\nsg{}G$ of finite groups, choose $S\in\sylp{G}$, and set 
$S_0=S\cap{}G_0\in\sylp{G_0}$.  Then $\calf_{S_0}(G_0)\nsg\calf_S(G)$.  
Assume in addition that  
$\calh_0$ and $\calh$ are sets of subgroups of $S_0$ and $S$, respectively, 
such that $\call_{S_0}^{\calh_0}(G_0)$ and $\call_S^\calh(G)$ are linking 
systems associated to $\calf_{S_0}(G_0)$ and $\calf_S(G)$, and such that 
$\calh=\{P\le{}S\,|\,P\cap{}S_0\in\calh_0\}$.  Then 
$\call_{S_0}^{\calh_0}(G_0)\nsg\call_S^\calh(G)$.
\end{Prop}

\begin{proof}  Fix $P,Q\le{}S_0$ and $g\in{}N_G(P,Q)$.  Then $gS_0g^{-1}$ 
is another Sylow $p$-subgroup of $G_0$, so there is some $h\in{}G_0$ such 
that $(h^{-1}g)S_0(h^{-1}g)^{-1}=S_0$.  Set $a=h^{-1}g$; thus $g=ha$ where 
$a\in{}N_G(S_0)$ and $h\in{}G_0$.  Thus $c_g=c_h\circ{}c_a\in\Hom_G(P,Q)$, 
where $c_a\in\Aut_G(S_0)$ and $c_h\in\Hom_{G_0}(aPa^{-1},Q)$.  This proves 
condition (ii) in the definition of a normal fusion system; and condition 
(ii) in Definition \ref{d:L0<|L} follows in a similar way.  The other 
conditions clearly hold.
\end{proof}

%%%%%%%%%%%%%%%%%%%%%

When $\SFL[_0]\nsg\SFL$ is a normal pair, then for 
each $\gamma\in\Aut_\call(S_0)$, we let $c_\gamma\in\Aut(\call_0)$ 
denote the automorphism which sends $P$ to $\gamma(P)=\pi(\gamma)(P)$ and 
sends $\psi\in\Mor_{\call_0}(P,Q)$ to 
$(\gamma|_{Q,\gamma(Q)})\circ\psi\circ(\gamma|_{P,\gamma(P)})^{-1}$.  
The next lemma describes how to tell, in terms only of the 
fusion system $\calf$, whether or not $c_{\delta(g)}=\Id_{\call_0}$ 
for $g\in{}S$ ($\delta=\delta_{S_0}$).

When $\call$ is a linking system associated to $\calf$, and $A\nsg\calf$, we 
say that an automorphism $\alpha$ of $\call$ is the \emph{identity modulo $A$} 
if for each $P,Q\in\Ob(\call)$ which contain $A$ and each 
$\psi\in\Mor_\call(P,Q)$, $\alpha(P)=P$, $\alpha(Q)=Q$, and 
$\alpha(\psi)=\psi\circ\delta_P(a)$ for some $a\in{}A$.  

\begin{Lem} \label{g-on-L0}
Let $\SFL[_0]\nsg\SFL$ be a normal pair such that all objects in $\call$ 
are $\calf$-centric.  Fix $A\nsg\calf_0$.  Then for $g\in{}S$, 
$c_{\delta(g)}\in\Aut(\call_0)$ is the identity modulo $A$ if and only if 
$[g,S_0]\le{}A$, and for each $P,Q\le{}S_0$ and 
$\varphi\in\Mor_{\calf_0}(P,Q)$, $\varphi$ extends to some 
$\widebar{\varphi}\in\Mor_\calf(\gen{PA,g},\gen{QA,g})$ such that 
$\widebar{\varphi}(g)\in{}gA$.
\end{Lem}

\iffalse
\begin{Lem} \label{g-on-L0}
Let $\SFL[_0]\nsg\SFL$ be a normal pair such that all objects in $\call$ 
are $\calf$-centric.  Fix $A\nsg\calf_0$.  Then for $g\in{}S$, 
$c_{\delta(g)}\in\Aut(\call_0)$ is the identity modulo $A$ if and only if 
$[g,S_0]\le{}A$, and each $\varphi\in\Mor(\calf_0)$ 
extends to some $\widebar{\varphi}\in\Mor(\calf)$ between subgroups 
containing $A$ and $g$ such that $\widebar{\varphi}(g)\in{}gA$.
\end{Lem}
\fi

\begin{proof}  Fix $g\in{}S$.  Set $\gamma=\delta_{S_0}(g)$ and 
$B=\gen{g,A}$ for short. 

Assume $c_\gamma\in\Aut(\call_0)$ is the identity modulo $A$.  Then 
$[g,S_0]\le{}A$ since $[\gamma,\delta_{S_0}(s)]\in \delta_{S_0}(A)$ for 
$s\in{}S_0$ (and $\delta_{S_0}$ is injective by Proposition 
\ref{L-prop}(c)).  Since $\calf_0$ is generated by morphisms between 
objects of $\call_0$ which contain $A$ (by Theorem \ref{AFT} and 
Proposition \ref{norm<=>}, and since $A\nsg\calf_0$), it suffices to prove 
the extension property for such morphisms.  Fix $P,Q\in\Ob(\call_0)$ such 
that $A\le{}P$ and $A\le{}Q$, fix $\varphi\in\Hom_{\calf_0}(P,Q)$, and 
choose a lifting of $\varphi$ to $\psi\in\Mor_{\call_0}(P,Q)$.  By 
assumption, 
$\delta_Q(g)\circ\psi\circ\delta_P(g)^{-1}=\psi\circ\delta_P(a)$ for some 
$a\in{}A$.  So by Proposition \ref{L-prop}(e), there is a unique morphism 
$\widebar{\psi}\in\Mor_\call(PB,QB)$ such that 
$\widebar{\psi}|_{P,Q}=\psi$, and 
$\delta_{QB}(g)\circ\widebar{\psi}\circ\delta_{PB}(ag)^{-1}= 
\widebar{\psi}$ by the uniqueness of the extension.  Set 
$\widebar{\varphi}=\pi(\widebar{\psi})$; then 
$\widebar{\varphi}\in\homf(PB,QB)$, $\widebar{\varphi}|_P=\varphi$, and 
$\widebar{\varphi}(ag)=g$ (so $\widebar{\varphi}(g)\in{}gA$) by axiom (C). 

Now assume $[g,S_0]\le{}A$, and $g$ has the above extension property:  
each $\varphi\in\Hom_{\calf_0}(P,Q)$ extends to 
$\widebar{\varphi}\in\homf(PB,QB)$ such that 
$\widebar{\varphi}(g)\in{}gA$.  We claim $c_\gamma\in\Aut(\call_0)$ is the 
identity modulo $A$.  Since $[g,S_0]\le{}A$, $gPg^{-1}=P$ for all 
$P\in\Ob(\call_0)$ which contain $A$. Fix $\psi\in\Mor_{\call_0}(P,Q)$, 
where $A\le{}P,Q$.  By assumption, $\pi(\psi)$ extends to some 
$\widebar{\varphi}\in\homf(PB,QB)$ such that 
$\widebar{\varphi}(g)\in{}gA$, and this lifts to 
$\widehat{\psi}\in\Mor_\call(PB,QB)$.  Since $P$ is $\calf$-centric, 
$\widehat{\psi}|_{P,Q}=\psi\circ\delta_P(x)$ for some $x\in{}Z(P)$.  Upon 
replacing $\widehat{\psi}$ by $\widehat{\psi}\circ\delta_{PB}(x)^{-1}$ and 
$\widebar{\varphi}$ by $\widebar{\varphi}\circ{}c_x^{-1}$, we can assume 
$\widehat{\psi}|_{P,Q}=\psi$.  By axiom (C), the conjugation action of 
$\delta_S(g)$ fixes $\widehat{\psi}$ modulo $\delta_{PB}(A)$, and hence 
$c_\gamma\in\Aut(\call_0)$ sends $\psi$ into $\psi\circ\delta_P(A)$.  
\end{proof}

%%%%%%%%%%%%%%%%%%%%%

The next lemma describes another way to construct normal pairs of 
linking systems.

\begin{Lem} \label{link-pb}
Fix a normal pair of fusion systems $\calf_0\nsg\calf$ over $p$-groups 
$S_0\nsg{}S$.  Let $\calh_0$ be a set of subgroups of $S_0$ such that 
\begin{itemize}  
\item $\calh_0$ is closed under $\calf$-conjugacy and overgroups, and 
contains all subgroups of $S_0$ which are $\calf_0$-centric and 
$\calf_0$-radical; and 
\item $\calh\defeq\{P\le{}S\,|\,P\cap{}S_0\in\calh_0\}$ is contained in 
the set of $\calf$-centric subgroups.
\end{itemize}
Assume $\calf$ has an associated centric linking system $\call^c$.  Let 
$\call\subseteq\call^c$ be the full subcategory with object set $\calh$.  
Let $\call_0\subseteq\call$ be the subcategory with object set 
$\calh_0$, where for $P,Q\in\calh_0$, 
	\beqq \Mor_{\call_0}(P,Q) = \{\psi\in\Mor_\call(P,Q) \,|\, 
	\pi(\psi)\in\Hom_{\calf_0}(P,Q) \}~. \label{e:link-pb} \eeqq
Then $\call_0\nsg\call$ is a normal pair of linking systems associated to 
$\calf_0\nsg\calf$.  For any such pair $\call_0\nsg\call$ with 
$\Ob(\call_0)=\calh_0$ and $\Ob(\call)=\calh$, $\call_0$ is centric in 
$\call$.
\end{Lem}

\begin{proof}  Since $\Ob(\call)$ is closed under $\calf$-conjugacy and 
under overgroups, and contains all subgroups which are $\calf$-centric and 
$\calf$-radical by Lemma \ref{F0<|F}(d) and the assumptions on $\calh_0$, 
$\call$ is a linking system associated to $\calf$. Since all objects in 
$\call_0$ are $\calf$-centric, they are also $\calf_0$-centric, and hence 
fully centralized in $\calf_0$.  Axiom (A) for $\call_0$ thus follows from 
axiom (A) for $\call$, together with the assumptions on 
$\calh_0=\Ob(\call_0)$.  Axioms (B) and (C) for $\call_0$ follow 
immediately from those for $\call$, and $\call_0$ is thus a linking system 
associated to $\calf_0$.

Condition (i) in Definition \ref{d:L0<|L} holds by assumption, while 
conditions (ii) and (iii) follow from \eqref{e:link-pb} and since $\calf_0$ 
is normal in $\calf$.  Thus $\call_0\nsg\call$.

Fix any such pair $\call_0\nsg\call$ associated to $\calf_0\nsg\calf$.  
Assume $\gamma\in\Aut_\call(S_0)$ is such that 
$\gamma\psi\gamma^{-1}=\psi$ for each $\psi\in\Mor(\call_0)$.  Since 
$\gamma(\delta_{S_0}(g))\gamma^{-1}=\delta_{S_0}(\pi(\gamma)(g))$ for 
$g\in{}S_0$ by axiom (C) for the linking system $\call$, 
$\pi(\gamma)=\Id_{S_0}$.  Since $S_0\in\calh_0$ is $\calf$-centric, this 
means that $\gamma=\delta_{S_0}(z)$ for some $z\in{}Z(S_0)$, and in 
particular, that $\gamma\in\Aut_{\call_0}(S_0)$.  So $\call_0$ is centric 
in $\call$.
\end{proof}

We now list the examples of normal pairs of linking systems which motivated 
Definition \ref{d:L0<|L}, and which we need to refer to later.  

\begin{Prop} \label{link-ex}
Let $\calf$ be a saturated fusion system over the finite $p$-group $S$, 
let $\calf_0\nsg\calf$ be a normal fusion subsystem over $S_0\nsg{}S$, and 
let $\calh_0$ be a set of subgroups of $S_0$. Assume that $\calf$ has an 
associated centric linking system $\call^c$, and that one of the following 
three conditions holds.
\begin{enuma}  
\item $\calf_0=O^p(\calf)$, $S_0=\hyp(\calf)$, and $\calh_0$ is the set of 
$\calf_0$-centric subgroups of $S_0$.

\item $\calf_0=O^{p'}(\calf)$, $S_0=S$, and $\calh_0$ is the set of 
$\calf_0$-centric subgroups of $S_0$.

\item For some normal $p$-subgroup $Q\nsg\calf$ and some normal subgroup 
$K\nsg\Aut(Q)$ containing $\Inn(Q)$, $\calf_0=N_\calf^K(Q)$, 
$S_0=N_S^K(Q)$, and $\calh_0$ is the set of all $\calf_0$-centric subgroups 
of $S_0$ which contain $Q$.
\end{enuma}
Set $\calh=\{P\le{}S\,|\,P\cap{}S_0\in\calh_0\}$.  Then there is a normal 
pair of linking systems $\call_0\nsg\call$ associated to 
$\calf_0\nsg\calf$ with $\Ob(\call_0)=\calh_0$ and $\Ob(\call)=\calh$.  
For any such normal pair $\call_0\nsg\call$, $\call_0$ is centric in 
$\call$ in cases (b) and (c), and in case (a) if $Z(\calf)=1$. 
Furthermore, in cases (a) and (b), and in case (c) if $Q=O_p(\calf)$ and 
$K=\Inn(Q)$, $\calh_0$ is $\Aut(S_0,\calf_0)$-invariant, $\calh$ is 
$\Aut(S,\calf)$-invariant, and $\call_0\nsg\call$ can be chosen such that 
$\call_0$ is $\Aut\typ^I(\call)$-invariant.  
\end{Prop}

\begin{proof} In all cases, $\calf_0\nsg\calf$ by Proposition 
\ref{fus-ex}.  Also, $\calh_0$ is $\Aut(S_0,\calf_0)$-invariant and 
$\calh$ is $\Aut(S,\calf)$-invariant:  this is clear in cases (a) and (b), 
and holds in case (c) when $Q=O_p(\calf)$ (since $Q=O_p(\calf_0)$ by Lemma 
\ref{F0<|F}(e)).

\smallskip

\noindent\textbf{(a) }  Set $S_0=\hyp(\calf)$, $\calf_0=O^p(\calf)$, and 
$\calh_0=\Ob(\calf_0^c)$.  By \cite[Theorem 4.3(a)]{BCGLO2}, 
a subgroup of $S_0$ is $\calf_0$-quasicentric if and only if it is 
$\calf$-quasicentric.  In particular, every $\calf_0$-centric subgroup 
of $S_0$ is $\calf$-quasicentric and hence all subgroups in $\calh$ are 
$\calf$-quasicentric.  By Lemma \ref{F0<|F}(d), $\calh$ contains all 
subgroups which are $\calf$-centric and $\calf$-radical.  By Lemma 
\ref{F0<|F}(c), $\calh_0$ is closed under $\calf$-conjugacy, so $\calh$ is 
closed under $\calf$-conjugacy (and it is clearly closed under 
overgroups).  Hence if $\call^q\supseteq\call^c$ is the quasicentric 
linking system which contains $\call^c$ constructed in \cite[Proposition 
3.4]{BCGLO1}, then the full subcategory $\call\subseteq\call^q$ with 
object set $\calh$ is also a linking system associated to $\calf$.

\iffalse
Hence if $\call^q\supseteq\call^c$ is a quasicentric linking 
system which contains $\call^c$ (see \cite[Proposition 3.4]{BCGLO1}), then 
the full subcategory $\call\subseteq\call^q$ with object set $\calh$ is 
also a linking system associated to $\calf$.  
\fi

By \cite[Proposition 2.4]{BCGLO2},  there is a unique map 
$\lambda\:\Mor(\call^q)\Right2{}S/S_0$ which sends composites to 
products and inclusions in $\call^q$ to the identity, and such that 
$\lambda(\delta_S(g))=[g]$ for all $g\in{}S$.  By \cite[Theorem 
3.9]{BCGLO2}, there is a $p$-local finite group $(S_0,\calf'_0,\call_0)$ 
where for $P,Q\in\Ob(\call_0)$, 
	\beqq \Mor_{\call_0}(P,Q) = \{\psi\in\Mor_{\call^q}(P,Q) \,|\, 
	\lambda(\psi)=1\}~. \label{e:L0=Ker} \eeqq
Furthermore, $\calf'_0$ is constructed using \cite[Proposition 
3.8]{BCGLO2} (cf. the proof of \cite[Theorem 3.9]{BCGLO2}), and hence (by 
part (b) of that proposition) it has $p$-power index in $\calf$.  
Thus $\calf'_0=\calf_0$ by Theorem \ref{t:O^p(F)}(a).  

Now, $\Ob(\call_0)=\calh_0$ since $\call_0$ is a 
centric linking system.  Condition (i) in Definition \ref{d:L0<|L} holds 
for $\call_0\subseteq\call$ by definition of $\calh$, 
condition (iii) ($\gamma\call_0\gamma^{-1}=\call_0$ for 
$\gamma\in\Aut_\call(S_0)$) holds by construction, and condition (ii) 
holds since $\lambda|_{\delta_{S_0}(S)}$ is surjective.  So 
$\call_0\nsg\call$.  

We next check that $\call_0$ is $\Aut\typ^I(\call)$-invariant.  Fix 
$\alpha\in \Aut\typ^I(\call)$ and set $\beta=\til\mu_\call(\alpha)\in 
\Aut(S,\calf)$.  Then $\beta(S_0)=S_0$ since $S_0=\hyp(\calf)$, and 
$\beta|_{S_0}\in\Aut(S_0,\calf_0)$ by the uniqueness of $\calf_0$ (Theorem 
\ref{t:O^p(F)}(a) again).  Since $\alpha(P)=\beta(P)$ for 
$P\in\Ob(\call_0)$ (Lemma \ref{AutI}), $\alpha$ sends 
$\Ob(\call_0)=\calh_0$ to itself.  By Lemma \ref{OutI2}, 
$\alpha=\widebar{\alpha}|_\call$ for some 
$\widebar{\alpha}\in\Aut\typ^I(\call^q)$, 
$\lambda\circ\widebar{\alpha}=\widebar{\beta}\circ\lambda$ (where 
$\widebar{\beta}\in \Aut(S/S_0)$ is induced by $\beta$) by the uniqueness 
of $\lambda$, and hence $\alpha(\Mor(\call_0))=\Mor(\call_0)$ by 
\eqref{e:L0=Ker}.

Now let $\call_0\nsg\call$ be any normal pair of linking systems 
associated to $\calf_0\nsg\calf$ with these objects.  Assume $Z(\calf)=1$; 
we must show $\call_0$ is centric in $\call$.  Assume 
$\gamma\in\Aut_\call(S_0)$ is such that $\gamma\psi\gamma^{-1}=\psi$ for 
each $\psi\in\Mor(\call_0)$.  Since 
$\gamma(\delta_{S_0}(g))\gamma^{-1}=\delta_{S_0}(\pi(\gamma)(g))$ for 
$g\in{}S_0$ by axiom (C) for the linking system $\call$, 
$\pi(\gamma)=\Id_{S_0}$.  So by axiom (A) (and since $S_0$ is fully 
centralized in $\calf$), $\gamma=\delta_{S_0}(h)$ for some 
$h\in{}C_S(S_0)$.  

Let $H\le{}C_S(S_0)$ be the subgroup of all $h$ such that the conjugation 
action of $\delta_{S_0}(h)$ on $\call_0$ is trivial.  The $p$-group 
$\call/\call_0=\Aut_\call(S_0)/\Aut_{\call_0}(S_0)$ acts on 
$\delta_{S_0}(H) \cong H$ by conjugation. Let $H_0$ be the fixed subgroup 
of this action.  Note that $H_0\leq Z(S)$ since $H_0$ is fixed by 
$\delta_{S_0}(S)\leq \Aut_\call(S_0)$.  Fix $h\in{}H_0$, and set 
$\bar\gamma=\delta_S(h)$.  Let $\bar{\gamma}(P)$ and 
$\bar{\gamma}\psi\bar{\gamma}^{-1}$ be as in Definition 
\ref{d:L0<|L}, but this time for all $P\in\Ob(\call)$ and 
$\psi\in\Mor(\call)$.  For $P\in\Ob(\call)$, 
$\bar\gamma(P)=hPh^{-1}=P$.  Also, 
$\bar{\gamma}\psi\bar{\gamma}^{-1}=\psi$ for all 
$\psi\in\Aut_\call(S_0)$ by definition of $H_0$, 
$\bar\gamma\psi\bar{\gamma}^{-1}=\psi$ for $\psi\in\Mor(\call_0)$ 
by definition of $H$, and hence conjugation by $\bar{\gamma}$ is the 
identity on morphisms in $\call$ between subgroups in $\calh_0$ by 
condition (ii) in Definition \ref{d:L0<|L}.  By Proposition 
\ref{L-prop}(f), for each $P,Q\in\calh$, the restriction map from 
$\Mor_{\call}(P,Q)$ to $\Mor_{\call}(P\cap{}S_0,Q\cap{}S_0)$ is injective, 
and hence $\bar\gamma\psi\bar{\gamma}^{-1}=\psi$ for all 
$\psi\in\Mor_{\call}(P,Q)$.  Thus conjugation by 
$\bar{\gamma}=\delta_S(h)$ is the identity on $\call$, and so 
$h\in{}Z(\calf)=1$ by Lemma \ref{OutI1}(a).  Since $H_0=1$ is the fixed 
subgroup of an action of the $p$-group $\call/\call_0$ on the $p$-group 
$H$, $H=1$, and so $\call_0$ is centric in $\call$.  

\smallskip

\noindent\textbf{(b) }  Set $\calf_0=O^{p'}(\calf)$.  
By \cite[Proposition 3.8(c)]{BCGLO2}, a subgroup 
of $S$ is $\calf_0$-centric if and only if it is $\calf$-centric.  So upon 
letting $\calh_0=\calh$ be the set of all $\calf$-centric subgroups of $S$, 
the hypotheses of Lemma \ref{link-pb} are satisfied.  By the 
lemma, there is a normal pair $\call_0\nsg\call$ of linking systems 
associated to $\calf_0\nsg\calf$ with object set $\calh_0=\calh$; and for 
any such pair, $\call_0$ is centric in $\call$.  By the explicit 
description of $\call_0$ (formula \eqref{e:link-pb} in Lemma 
\ref{link-pb}), $\call_0$ is $\Aut\typ^I(\call)$-invariant.

\smallskip

\noindent\textbf{(c) }  Fix $Q\nsg\calf$ and $\Inn(Q)\le{}K\nsg\Aut(Q)$, 
and set $S_0=N_S^K(Q)$ and $\calf_0=N_\calf^K(Q)$.  Let $\calh_0$ be the 
set of all $\calf_0$-centric subgroups of $S_0$ which contain $Q$.  We 
first check that all subgroups in 
$\calh=\{P\le{}S\,|\,P\cap{}S_0\in\calh_0\}$ are $\calf$-centric; it 
suffices to show this for subgroups in $\calh_0$.  By Lemma \ref{F0<|F}(c) 
(and since $\calf_0\nsg\calf$), the set of $\calf_0$-centric subgroups, 
and hence also the set $\calh_0$, are closed under $\calf$-conjugacy.  
For each $P\in\calh_0$, $C_S(P)\le{}C_S(Q)\le{}S_0$ since $P\ge{}Q$, and 
hence $C_S(P)=C_{S_0}(P)=Z(P)$ since $P$ is $\calf_0$-centric.  Since this 
holds for all subgroups $\calf$-conjugate to $P$, we conclude that $P$ is 
$\calf$-centric.

We just saw that $\calh_0$ is closed under $\calf$-conjugacy, and it is 
clearly closed under overgroups.  Since $Q\nsg\calf_0$, each subgroup of 
$S_0$ which is $\calf_0$-centric and $\calf_0$-radical contains $Q$ by 
Proposition \ref{norm<=>}, and thus lies in $\calh_0$. So by Lemma 
\ref{link-pb}, there is a normal pair $\call_0\nsg\call$ of linking 
systems associated to $\calf_0\nsg\calf$ with object sets $\calh_0$ and 
$\calh$, and for any such pair, $\call_0$ is centric in $\call$.  If 
$Q=O_p(\calf)$ and $K=\Inn(Q)$, then $Q=O_p(\calf_0)$ by Lemma 
\ref{F0<|F}(e), and so $\calf_0$ is $\Aut(S,\calf)$-invariant.  Hence 
$\call_0$ is $\Aut\typ^I(\call)$-invariant by the explicit description of 
$\call_0$ in Lemma \ref{link-pb}. 
\end{proof}

%%%%%%%%%%%%%%%%%%%%%%%%%%%%%%%%%%%%%%%%%%%%%%%

\newsect{Reduced fusion systems and tame fusion systems}
\label{s:red-sfs}

Throughout this section, $p$ denotes a fixed prime, and we work with 
fusion systems over finite $p$-groups.  We first define reduced fusion 
systems and the reduction of a fusion system.  We then define tame fusion 
systems, and prove that a reduced fusion system is tame if every saturated 
fusion system which reduces to it is realizable (Theorem \ref{ThB}).  We 
then make a digression to look at the existence of linking systems in 
certain situations, before proving that all fusion systems whose reduction 
is tame are realizable (Theorem \ref{ThA}).  We thus end up with a way to 
``detect'' exotic fusion systems in general while looking only at reduced 
fusion systems.

\newsubb{Reduced fusion systems and reductions of fusion systems}{s:red}

We begin with the definition of a reduced fusion system, and the reduction 
of an (arbitrary) fusion system.  See Proposition \ref{F/Q} and the 
discussion before that for the definition and properties of quotient fusion 
systems.

\begin{Defi} \label{d:reduced}
A \emph{reduced fusion system} is a saturated fusion system $\calf$ such 
that 
\begin{itemize}  
\item $\calf$ has no nontrivial normal $p$-subgroups, 
\item $\calf$ has no proper normal subsystem of $p$-power index, and 
\item $\calf$ has no proper normal subsystem of index prime to $p$.  
\end{itemize}
Equivalently, $\calf$ is reduced if $O_p(\calf)=1$, $O^p(\calf)=\calf$, 
and $O^{p'}(\calf)=\calf$.  

For any saturated fusion system $\calf$, the \emph{reduction} of $\calf$ 
is the fusion system $\red(\calf)$ defined as follows.  Set 
$\calf_0=C_\calf(O_p(\calf))/Z(O_p(\calf))$, and let 
$\calf_0 \supseteq \calf_1 \supseteq \calf_2 \supseteq \cdots 
	\supseteq \calf_m $
be such that $\calf_i=O^p(\calf_{i-1})$ if $i$ is odd, 
$\calf_i=O^{p'}(\calf_{i-1})$ if $i$ is even, and 
$O^p(\calf_m)=O^{p'}(\calf_m)=\calf_m$.  Then $\red(\calf)=\calf_m$.  
\end{Defi}

Fix any $\calf$, and set $Q=O_p(\calf)$ for short.  By definition of 
centralizer fusion systems, every morphism in $C_\calf(Q)$ extends to a 
morphism in $\calf$ which is the identity on $Q$, and hence to a morphism 
in $C_\calf(Q)$ which is the identity on $Z(Q)$.  This proves that $Z(Q)$ 
is always central in $C_\calf(Q)$, and hence that $\calf_0=C_\calf(Q)/Z(Q)$ 
is well defined as a fusion system.  

What is important in the last part of the definition of $\red(\calf)$ is 
that we give an explicit procedure for successively applying $O^p(-)$ and 
$O^{p'}(-)$, starting with $\calf_0$, until neither makes the fusion 
system any smaller.  It seems likely that the final result $\red(\calf)$ 
is independent of the order in which we apply these reductions, but we 
have not shown this, and do not need to know it when proving the results 
in this section.

Clearly, for these definitions to make sense, we want $\red(\calf)$ to 
always be reduced.

\begin{Prop} \label{p:red->red}
The reduction of any saturated fusion system is reduced.  
\end{Prop}

For later reference, we also state the following, more technical 
result, which will be proven together with Proposition \ref{p:red->red}.

\begin{Lem} \label{l:red->red}
Let $\calf$ be a saturated fusion system.  Set $Q=O_p(\calf)$ and 
$\calf_0=C_\calf(Q)/Z(Q)$.  Let 
$\calf_0\supseteq\calf_1\supseteq\cdots\supseteq\calf_m=\red(\calf)$ 
be such that for each $i$, $\calf_i=O^p(\calf_{i-1})$ or 
$\calf_i=O^{p'}(\calf_{i-1})$.  Then $O_p(\calf_i)=1$ for each 
$0\le{}i\le{}m$.
\end{Lem}

\begin{proof}  Fix $\calf$, and let $Q\nsg\calf$ and the $\calf_i$ be as 
above.  Since $C_\calf(Q)\nsg\calf$ by Proposition \ref{fus-ex}(c), 
$O_p(C_\calf(Q))\le{}O_p(\calf)=Q$ by Lemma \ref{F0<|F}(e).  Hence 
$O_p(C_\calf(Q))=Z(Q)$.  We just saw that $Z(Q)$ is central in 
$C_\calf(Q)$.  So by Proposition \ref{F/Q}, a subgroup 
$P/Z(Q)\le{}C_S(Q)/Z(Q)$ is normal in $C_\calf(Q)/Z(Q)$ only if 
$P\nsg{}C_\calf(Q)$.  Thus $O_p(\calf_0)=Z(Q)/Z(Q)=1$.  

By definition, $O^p(\red(\calf))=O^{p'}(\red(\calf))=\red(\calf)$.  By 
Proposition \ref{fus-ex}(a,b), $\calf_i\nsg\calf_{i-1}$ for each $i\ge1$.  
So by Lemma \ref{F0<|F}(e) again, $O_p(\calf_i)=1$ if $O_p(\calf_{i-1})=1$.  
Since $O_p(\calf_0)=1$, this proves that $O_p(\calf_i)=1$ for each $i$.  In 
particular, $O_p(\red(\calf))=1$, and hence $\red(\calf)$ is reduced.
\end{proof}

A saturated fusion system $\calf$ is \emph{constrained} if there is a normal 
subgroup $Q\nsg\calf$ which is $\calf$-centric (cf. \cite[\S\,4]{BCGLO1}).  

\begin{Prop} \label{red(constr)}
For any saturated fusion system $\calf$, $\red(\calf)=1$ (the fusion 
system over the trivial group) if and only if $\calf$ is constrained.
\end{Prop}

\begin{proof}  If $\calf$ is constrained, then clearly $\red(\calf)=1$.  
Conversely, assume $\calf$ is a fusion system over a finite $p$-group $S$ 
such that $\red(\calf)=1$.  Set $Q=O_p(\calf)$ and 
$\calf_0=C_\calf(Q)/Z(Q)$.  If $\calf_0=1$, then $C_\calf(Q)$ is a fusion 
system over $Z(Q)$, and hence $C_S(Q)=Z(Q)$.  So $Q$ is $\calf$-centric, 
and hence $\calf$ is constrained in this case.  

If $\calf_0\ne1$, then there is a sequence of fusion subsystems 
$1=\calf_m\subsetneqq\calf_{m-1}\subsetneqq\cdots\subsetneqq\calf_0$ such 
that for each $i$, $\calf_{i+1}=O^p(\calf_i)$ or 
$\calf_{i+1}=O^{p'}(\calf_i)$.  By Lemma \ref{l:red->red}, 
$O_p(\calf_i)=1$ for each $0\le{}i\le{}m$.  Since $\calf_{m-1}\ne1$, it is 
a fusion system over a $p$-group $S_{m-1}\ne1$, so 
$O^{p'}(\calf_{m-1})\ne1$ (it is over the same $p$-group), which implies 
$O^p(\calf_{m-1})=1$.  Thus $\hyp(\calf_{m-1})=1$ by Definition 
\ref{d:p-p'-index}(a), so there are no nontrivial automorphisms of order 
prime to $p$ in $\calf_{m-1}$, and $\calf_{m-1}$ is the fusion system of 
the $p$-group $S_{m-1}$.  This is impossible, since it would imply 
$O_p(\calf_{m-1})=S_{m-1}\ne1$, and we conclude $\calf_0=1$.  
\end{proof}

\newsubb{Tame fusion systems and the proof of Theorem \ref{ThB}}{s:tamedef}

Assume $\calf=\calf_S(G)$ for some finite group $G$ with $S\in\sylp{G}$.   
Let $\calh$ be an $\Aut(G,S)$-invariant set of $G$-quasicentric subgroups  
of $S$ such that $\call\defeq\call^\calh_S(G)$ is a linking system  
associated to $\calf$ (i.e. $\calh$ is closed under overgroups and  
contains all $\calf$-centric $\calf$-radical subgroups).  Define  the 
homomorphism 
	\[ \til\kappa_G^\calh: \Aut(G,S) \Right5{} \Aut\typ^I(\call) \]
as follows.  For $\beta\in\Aut(G,S)$, $\til\kappa_G^\calh(\beta)$ sends 
$P$ to  $\beta(P)$ and sends $[a]\in\Mor_\call(P,Q)$ (for 
$a\in{}N_G(P,Q)$) to  $[\beta(a)]$. 

For any $g\in{}N_G(S)$, $\til\kappa_G^\calh$ sends $c_g\in\Aut(G,S)$ to 
$c_{[g]}\in\Aut\typ^I(\call)$, where $[g]\in\Aut_\call(S)$ is the class of 
$g$.  Thus by Lemma \ref{OutI1}, $\til\kappa_G^\calh$ induces a 
homomorphism 
	\[ \kappa_G^\calh\: \Out(G) \Right5{} \Out\typ(\call) \]
by sending the class of $\beta$ to the class of $\til\kappa_G^\calh(\beta)$.
When $\call=\call_S^c(G)$ is the centric linking system of $G$, we write
$\til\kappa_G=\til\kappa_G^\calh$ and $\kappa_G=\kappa_G^\calh$ for short.  

Note that when $\calf=\calf_S(G)$ and $\call=\call_S^\calh(G)$ as above, 
$\til\mu_G^\calh\circ\til\kappa_G^\calh\:\Aut(G,S)\Right2{}\Aut(S,\calf)$ 
is the restriction homomorphism.  

\begin{Defi} \label{d:tame}
A saturated fusion system $\calf$ over $S$ is \emph{tame} if 
there is a finite group $G$ which satisfies:
\begin{itemize}  
\item $S\in\sylp{G}$ and $\calf\cong\calf_S(G)$; and 
\item $\kappa_G\:\Out(G)\Right3{}\Out\typ(\call_S^c(G))$ is split 
surjective.
\end{itemize}
In this situation, we say $\calf$ is \emph{tamely realized} by $G$.
\end{Defi}

The condition that $\kappa_G$ be split surjective was chosen since, as we 
will see shortly, that is what is needed in the proof of Theorem 
\ref{ThB}.  In contrast, Theorem \ref{ThA} would still be true (with 
essentially the same proof) if we replaced ``split surjective'' by ``an 
isomorphism'' in the above definition.  

By Lemma \ref{OutI2}, 
$\Out\typ(\call_S^c(G))\cong\Out\typ(\call_S^\calh(G))$ for any 
$\Aut(S,\calf)$-invariant set of objects $\calh$ (which satisfies the 
conditions for $\call_S^\calh(G)$ to be a linking system).  Hence 
$\Ker(\kappa_G^\calh)=\Ker(\kappa_G)$, and $\kappa_G^\calh$ is (split) 
surjective if and only if $\kappa_G$ is.

By \cite[Theorem B]{BLO1}, $\kappa_G$ is split 
surjective if and only if the natural map from $\Out(G)$ to 
$\Out(BG\pcom)$ is split surjective, where $\Out(BG\pcom)$ is the group of 
homotopy classes of self equivalences of $BG\pcom$.  So this gives another 
way to formulate the definition of tameness.

It is natural to ask whether a tame fusion system $\calf$ can always be 
realized by a finite group $G$ such that $\kappa_G$ is an isomorphism.  We 
know of no counterexamples to this, but do not know how to prove it either.

We are now ready to prove Theorem \ref{ThB}:  every reduced fusion system 
which is not tame is the reduction of some exotic fusion system.  This is 
basically a consequence of the definition of tameness, together with 
\cite[Theorem 9]{O3} which gives a general procedure for constructing 
extensions of linking systems.

\begin{Thm} \label{untame->exotic}
Let $\calf$ be a reduced fusion system which is not tame.  Then there is 
an exotic fusion system $\3\calf$ whose reduction is isomorphic to 
$\calf$.
\end{Thm}

\begin{proof}  If $\calf$ is itself exotic, then we take $\3\calf=\calf$.  
So assume $\calf$ is the fusion system of a finite group, 
and hence that $\calf$ has at least one associated centric 
linking system $\call$.  Assume $\call$ is chosen such that 
$|\Out\typ(\call)|$ is maximal among all centric linking systems 
associated to $\calf$.  (All such linking systems are isomorphic to each 
other by \cite[Proposition 3.1]{BLO2} and Theorem A in 
\cite{limz-odd,limz}, but since those results use the classification of 
finite simple groups, we will not use them here.)  Since $\calf$ is not 
tame, it is not the fusion system of any finite group $H$ for which 
$\kappa_H$ is split surjective.  

Since $Z(\calf)=1$ ($\calf$ is reduced), we can identify $\Aut_\call(S)$ 
as a normal subgroup of $\Aut\typ^I(\call)$ via its conjugation action on 
$\call$ (Lemma \ref{OutI1}(a)).  Thus 
$\Out\typ(\call)=\Aut\typ^I(\call)/\Aut_\call(S)$.  Let $A$ 
be any finite abelian $p$-group on which $\Out\typ(\call)$ acts 
faithfully, and let 
	\[ \nu\:\Aut\typ^I(\call)\Right2{}\Aut(A) \]
denote the given action.  Thus $\Ker(\nu)=\Aut_\call(S)$.  

Set $S_0=A\times{}S$ and $\calf_0=A\times\calf$ 
($=\calf_A(A)\times\calf$).  We refer to the beginning of Section 
\ref{s:factor}, or to \cite[\S\,1]{BLO2}, for the definition of the 
product of two fusion systems.  Set $\call_0=A\times\call$:  the centric 
linking system associated to $\calf_0$ whose objects are the subgroups 
$A\times{}P\le{}S_0$ for $P\in\Ob(\call)$, and where 
$\Mor_{\call_0}(A\times{}P,A\times{}Q)= A\times\Mor_\call(P,Q)$.  

Set $\Gamma_0=\Aut_{\call_0}(S_0)=A\times\Aut_\call(S)$.  Set 
$\Gamma=A\sd{}\Aut\typ^I(\call)$:  the semidirect product taken with 
respect to the action $\nu$ of $\Aut\typ^I(\call)$ on $A$.  Thus 
$\Gamma_0$ embeds as a normal subgroup of $\Gamma$, and 
$\Gamma/\Gamma_0\cong\Out\typ(\call)$.  To avoid confusion, an element 
$\psi\in\Aut_\call(S)$ will be written $c_\psi$ when regarded as an element 
of $\Gamma_0\nsg\Gamma$.

We claim the given $\Gamma$-action on $\call_0$ satisfies the hypotheses of 
\cite[Theorem 9]{O3}.  This means checking that the following 
diagram commutes:
	\beq \vcenter{\xymatrix@C=30pt{
	\llap{$A\times\Aut_\call(S)=\,$} \Gamma_0 \ar[r]^-{\conj} 
	\ar[d]_{\incl} & \Aut\typ^I(\call_0) 
	\rlap{\,$=\Aut\typ^I(A\times\call)$} 
	\ar[d]^{(\alpha\mapsto\alpha_{A\times{}S})} \\
	\llap{$A \sd{} \Aut\typ^I(\call)=\,$}
	\Gamma \ar[r]^(.45){\conj} \ar[ur]^{\tau} & \Aut(\Gamma_0) 
	\rlap{~,}
	}} \eeq
where $\tau$ sends $(a,\gamma)\in{}A\rtimes\Aut\typ^I(\call)$ to 
$(\nu(\gamma),\gamma)\in\Aut\typ^I(A\times\call)$.  For 
$\psi\in\Aut_\call(S)$, $\nu(c_\psi)=\Id_A$, so 
$\tau(a,c_\psi)=(\Id,c_\psi)$, which shows that the upper triangle 
commutes.  As for the lower triangle, for $(a,\gamma)\in\Gamma$ and 
$(b,c_\psi)\in\Gamma_0$, 
	\[ \tau(a,\gamma)(b,c_\psi) = (\nu(\gamma)(b),c_{\gamma(\psi)}) = 
	(\nu(\gamma)(b)\cdot a,c_{\gamma(\psi)}\circ\gamma)(a,\gamma)^{-1}
	= (a,\gamma)(b,c_\psi)(a,\gamma)^{-1} \]
(since $c_{\gamma(\psi)}=\gamma\circ{}c_\psi\circ\gamma^{-1}$);
and thus the lower triangle commutes.

Fix $\3S\in\sylp{\Gamma}$.  We identify $S_0=O_p(\Gamma_0)$ 
(via $\delta_{S_0}$), and hence $S_0\nsg\3S$. Since
 	\[ C_\Gamma(S_0) = C_\Gamma(A\times{}S) = C_{\Gamma_0}(A\times{}S)
 	= A \times C_{\Aut_\call(S)}(S) = A\times{}Z(S) \]
is a $p$-group, and since all objects in $\call_0$ are $\calf_0$-centric 
by construction, \cite[Theorem 9]{O3} shows that there exists a saturated 
fusion system $\3\calf$ over $\3S$ and an associated linking system 
$\3\call$ such that $(S_0,\calf_0,\call_0)\nsg (\3S,\3\calf,\3\call)$ and 
$\Aut_{\3\call}(S_0)=\Gamma$ with the action on $\call_0$ given by $\tau$. 
In particular,
 	\beqq \Aut_{\3\calf}(S_0)=\Aut_\Gamma(S_0)
 	= \bigl\{ (\nu(\gamma),\til\mu_\call(\gamma)) \,\big|\,
 	\gamma\in\Aut\typ^I(\call) \bigr\} ~. \label{e:2.6a} \eeqq

Assume $\3\calf$ is realizable: the fusion system of a finite group 
$\3G$. Since $A$ is central in $\calf_0$, $O_p(\calf_0)/A$
is normal in $\calf_0/A\cong\calf$ by Proposition \ref{F/Q}.  Since 
$O_p(\calf)=1$ ($\calf$ is reduced), this shows that $O_p(\calf_0)=A$.
By Lemma \ref{F0<|F}(e) we then get $A\nsg\3\calf$. By Proposition
\ref{N_F=F(N_G)} we have $\3\calf \cong \calf_{\3S}(\3G) =
\calf_{\3S}(N_{\3G}(A))$.  Upon replacing $\3G$ by $N_{\3G}(A)$, we may 
assume $A\nsg \3G$.  

Set $G_0=C_{\3G}(A)$ and $G=G_0/A$.  Assume the following two statements 
hold:
\begin{enumerate}[(i) ]
\item \label{e:ii} $A=O_p(\3\calf)$ and $C_{\3\calf}(A)=\calf_0$. 

\item \label{e:iii} The composite
	\[ \xi\: %%\Out\typ(\call) \cong 
	\Aut_{\3G}(A) \cong \3G/G_0 \Right4{\conj} \Out(G) 
	\Right4{\kappa_G} \Out\typ(\call_S^c(G)) \]
is injective.
\end{enumerate}
We now finish the proof of the theorem, assuming \eqref{e:ii} and 
\eqref{e:iii}.

By \eqref{e:ii}, $C_{\3\calf}(O_p(\3\calf))/Z(O_p(\3\calf))=
\calf_0/A\cong \calf$. Since $O^p(\calf)=O^{p'}(\calf)=\calf$, this shows 
$\red(\3\calf)\cong\calf$.  Also, $S_0=C_{\3S}(A)\in\sylp{C_{\3G}(A)}$ 
since $\calf_0=C_{\3\calf}(A)$.  Hence by Proposition \ref{N_F=F(N_G)} 
(applied with $K=1$), $\calf_0=C_{\3\calf}(A)=\calf_{S_0}(C_{\3G}(A))$, 
and so 
	\[ \calf \cong \calf_0/A \cong \calf_{S_0/A}(G_0/A) 
	\cong \calf_S(G)~. \]
By condition (ii) in Definition \ref{d:F0<|F} (applied to 
$\calf_0\nsg\3\calf$), and since $A\nsg\3\calf$ and $A\le{}S_0$, each 
$\varphi\in\Aut_{\3\calf}(A)$ has the form 
$\varphi=\varphi_0\circ\alpha|_{A,A}$ for some 
$\alpha\in\Aut_{\3\calf}(S_0)=\Aut_\Gamma(S_0)$ and some 
$\varphi_0\in\Aut_{\calf_0}(A)$.  
Also, $\Aut_{\calf_0}(A)=1$ by definition, and thus 
	\[ \Aut_{\3G}(A) = \Aut_{\3\calf}(A) = \Aut_\Gamma(A) 
	\cong \Gamma/\Gamma_0 \cong\Out\typ(\call)~. \]  
So by \eqref{e:iii}, there is a homomorphism $s$ from $\Out\typ(\call)$ to 
$\Out(G)$ 
such that $\kappa_G\circ{}s$ is injective.  Since $\call$ was chosen with 
$|\Out\typ(\call)|$ maximal, $\kappa_G\circ{}s$ is an isomorphism, so 
$\kappa_G$ is split surjective, contradicting the assumption that $\calf$ 
is not tame.  We conclude that $\3\calf$ is exotic (and 
$\red(\3\calf)\cong\calf$).  

It remains to prove \eqref{e:ii} and \eqref{e:iii}.

\smallskip

%%\noindent\textbf{Proof of \eqref{e:ii}: } 
\noindent\textbf{Proof of (i): }
For each $P\in\Ob(\3\call)$, $P_0=P\cap{}S_0\in\Ob(\call_0)$, so 
$P_0=A\times{}Q$ for some $Q\le{}S$ which is $\calf$-centric.  Then 
$C_{\3S}(P)\le{}C_{\3S}(A)=\3S\cap C_\Gamma(A)=\3S\cap\Gamma_0=S_0$, so 
$C_{\3S}(P)\le{}A\times{}Z(Q)\le{}P$.  
Since this holds for all subgroups $\3\calf$-conjugate to $P$, 
all objects in $\3\call$ are $\3\calf$-centric.  

Set $B=O_p(\3\calf)$.  We already saw that $A=O_p(\calf_0)$.  Hence 
$B\cap{}S_0=A$ by Lemma \ref{F0<|F}(e).  Since $B\nsg\3S$ and 
$S_0\nsg\3S$, it follows that $[B,S_0]\le{}A$.  

Since $A\nsg\calf_0$ and $B\nsg\3\calf$, each 
$\varphi\in\Hom_{\calf_0}(P,Q)$ can be extended to a morphism 
$\widebar{\varphi}\in\Hom_{\3\calf}(PB,QB)$ such that 
$\widebar{\varphi}|_{PA}\in\Hom_{\calf_0}(PA,QA)$ and 
$\widebar{\varphi}(B)=B$.  Then $\widebar{\varphi}|_A=\Id_A$.  Hence for 
each $g\in{}B$, $g$ and $\widebar{\varphi}(g)$ have the same conjugation 
action on $A$, and $g^{-1}\widebar{\varphi}(g)\in 
C_B(A)=B\cap{}C_{\3S}(A)=B\cap{}S_0=A$.  By Lemma \ref{g-on-L0}, 
$c_{\delta(g)}=\tau(\delta_{S_0}(g))\in\Aut\typ^I(\call_0)$ is the identity 
modulo $A$, and thus $g\in{}A$ by definition of $\tau$.  So 
$B=O_p(\3\calf)=A$.

Since $C_{\3S}(A)=S_0$, $\calf_0$ and $C_{\3\calf}(A)$ are both fusion 
systems over $S_0$.  Also, $\calf_0\subseteq{}C_{\3\calf}(A)$ since $A$ is 
central in $\calf_0$.  To see that $C_{\3\calf}(A)=\calf_0$, fix $P, Q\leq 
S_0$ and $\varphi\in \Hom_{C_{\3\calf}(A)}(P,Q)$. By definition, $\varphi$ 
extends to $\3\varphi\in \Hom_{\3\calf}(AP,AQ)$ with $\3\varphi(A)=A$ and 
$\3\varphi|_A = \Id_A$. Since $\calf_0\nsg \3\calf$, condition (ii) in 
Definition \ref{d:F0<|F} shows that there are $\alpha\in 
\Aut_{\3\calf}(S_0)$ and $\varphi_0\in \Hom_{\calf_0}(\alpha(AP),AQ)$ such 
that $\3\varphi = \varphi_0\circ \alpha|_{AP,\alpha(AP)}$.  For each $a\in 
A$, $\alpha(a)=\varphi_0^{-1}(\3\varphi(a))=a$, and thus 
$\alpha|_A=\Id_A$. Hence
  	\[ \alpha \in \bigl\{\beta\in\Aut_{\3\calf}(S_0) \,\big|\, 
  	\beta|_A=\Id_A\bigr\}
 	= \{\Id_A\}\times\autf(S) = \Aut_{\calf_0}(S_0) ~, \]
where the first equality holds by \eqref{e:2.6a} (and since 
$\til\mu_\call(\Aut_\call(S))=\autf(S)$).  Thus 
$\alpha\in\Aut_{\calf_0}(S_0)$, so $\3\varphi\in\Mor(\calf_0)$, and hence 
also $\varphi\in\Mor(\calf_0)$. This proves that $C_{\3\calf}(A)=\calf_0$.

\smallskip

%%\noindent\textbf{Proof of \eqref{e:iii}: } 
\noindent\textbf{Proof of (ii): }
Set 
	\[ \call^*=\call_S^c(G), \qquad
	\call_0^*=\call_{S_0}^{\calh_0}(G_0)=\call_{S_0}^c(G_0), 
	\qquad\textup{and}\qquad 
	\3\call^*=\call_{\3S}^{\3\calh}(\3G)~, \]
where $\calh_0=\Ob(\call_0)$ and $\3\calh=\Ob(\3\call)$.  Note that 
$(S_0,\calf_0,\call_0^*)\nsg(\3S,\3\calf,\3\call^*)$ by Proposition 
\ref{L(G0)<|L(G)}.  Let $\cj\:\3G\Right2{}\Aut(G)$ denote the conjugation 
action of $\3G$ on $G$.  Set 
	\[ H = \bigl\{g\in N_{\3G}(S_0) 
	\,\big|\, \til\kappa_G(\cj(g))=\Id_{\call^*} \bigr\}
	\qquad\textup{and}\qquad T=H\cap\3S~. \]

We first claim $T=A$.  By \cite[Theorem 6.8]{BCGLO2}, $\call_0^*/A$ is a 
centric linking system associated to $\calf_0/A\cong\calf$. Hence the 
natural functor $\call_0^*/A\Right2{}\call^*$ (induced by the projection 
$G_0\Right2{}G$) is an isomorphism, since it commutes with the structure 
functors. So for $g\in\3S$, $g\in{}T$ if and only if 
$c_{\delta(g)}\in\Aut\typ^I(\call_0^*)$ is the identity modulo $A$, in the 
sense of Lemma \ref{g-on-L0}.  We showed in the proof of (i) that each 
$P\in\Ob(\3\call^*)=\Ob(\3\call)$ is $\3\calf$-centric.  Hence by Lemma 
\ref{g-on-L0}, applied to both normal pairs 
$(S_0,\calf_0,\call_0)\nsg(\3S,\3\calf,\3\call)$ and 
$(S_0,\calf_0,\call_0^*)\nsg(\3S,\3\calf,\3\call^*)$, $g\in{}T$ if and 
only if $c_{\delta(g)}\in\Aut\typ^I(\call_0)$ is the identity modulo $A$; 
i.e., induces the identity on $\call$.  By definition of 
$\3S\le\Gamma=A\rtimes\Aut\typ^I(\call)$, this is the case exactly when 
$g\in{}A$.  

Thus $H$ is a normal subgroup of $N_{\3G}(S_0)$ whose 
intersection with its Sylow $p$-subgroup $\3S$ is $A$.  It follows that 
$H/A$ has order prime to $p$.  We claim that $H\le{}G_0$.  Fix $h\in{}H$ 
of order prime to $p$.  Then $\cj(h)\in\Aut(G)$ acts via the identity on 
$S=S_0/A$, so $[h,S_0]\le{}A$.  Hence by \eqref{e:2.6a}, 
$c_h=(\nu(\gamma),\Id_S)\in\Aut_{\3G}(S_0)$ for some 
$\gamma\in\Aut\typ^I(\call)$ such that $\gamma\in\Ker(\til\mu_\call)$.  
Since $\Ker(\mu_\call)$ is a $p$-group by Lemma \ref{Ker(mu)-p-gp} and $h$ 
has order prime to $p$, $\gamma\in\Aut_\call(S)$, so 
$\nu(\gamma)=1\in\Aut(A)$, and $h\in{}G_0$.  Thus 
$H=O^p(H){\cdot}A\le{}G_0$.

Fix $g\in{}\3G$ such that $c_g\in\Ker(\xi)$.  Recall we are only 
interested in $g$ modulo $C_{\3G}(A)=G_0$.  Since 
$\3G=G_0{\cdot}N_{\3G}(S_0)$ by the Frattini argument, 
we can assume $g$ normalizes $S_0$.  Thus 
$\kappa_G([\cj(g)])=1$ in $\Out\typ(\call^*)$, so 
$\til\kappa_G(\cj(g))=c_\gamma$ for some $\gamma\in\Aut_{\call^*}(S)$.  
Let $h\in{}N_G(S)$ be such that $\gamma=[h]$ and lift $h$ to 
$\til{h}\in{}N_{G_0}(S_0)$.  Upon replacing $g$ by $\til{h}^{-1}g$, we can 
assume $\til\kappa_G(\cj(g))=\Id_{\call^*}$, and thus $g\in{}H\le{}G_0$.  
Hence $c_g=\Id_A$, $\xi$ is injective, and this finishes the proof of 
\eqref{e:iii}.
\end{proof} 

%%%%%%%%%%%%%%%%%%%%%%%%%%%%%%%

\newsubb{Strongly tame fusion systems and linking systems for extensions}
{s:lim2}

\newcommand{\simpclass}[1]{\widehat{\mathfrak{L}}(#1)}
\newcommand{\gpclass}[1]{\mathfrak{G}(#1)}

We are now ready to start working on the proof of Theorem \ref{ThA}.  
As stated in the introduction, this proof uses the 
vanishing of certain higher limit groups, and through that depends on the 
classification of finite simple groups.  In order to have a clean 
statement which does not depend on the classification (Theorem 
\ref{T:reduce}), we first define a certain class of finite groups which in 
fact (using the classification) is shown to include all finite groups.  

The obstruction groups for the existence and uniqueness of centric 
linking systems associated to a given saturated fusion system are higher 
derived functors for inverse limits taken over the centric orbit category 
of the fusion system.  We begin by defining this category.

\begin{Defi} \label{d:orb}
Let $\calf$ be a fusion system over a finite $p$-group $S$, and let 
$\calf^c\subseteq\calf$ be the full subcategory whose objects are the 
$\calf$-centric subgroups of $S$.  The \emph{centric orbit category} 
$\orb(\calf^c)$ of $\calf$ is the category with 
$\Ob(\orb(\calf^c))=\Ob(\calf^c)$, and where
	\[ \Mor_{\orb(\calf^c)}(P,Q) = \Inn(Q){\backslash}\homf(P,Q) \]
for any pair of objects $P,Q\le{}S$.  In particular, 
$\Aut_{\orb(\calf^c)}(P)=\outf(P)$ for each $P$.  If 
$\calf_0\subseteq\calf^c$ is any full subcategory, then $\orb(\calf_0)$ 
denotes the full subcategory of $\orb(\calf^c)$ with the same objects as 
$\calf_0$. 
\end{Defi}

We need the following technical result about higher limits over these 
orbit categories.

\begin{Lem} \label{lim*(orb)}
Let $\calf$ be a saturated fusion system over a finite $p$-group $S$.  Let 
$\calh\subseteq\Ob(\calf^c)$ be any subset which is closed under 
$\calf$-conjugacy and overgroups, and let $\calf^\calh\subseteq\calf^c$ be 
the full subcategory with object set $\calh$.  Fix a functor
$F\:\orb(\calf^c)\op\Right2{}\zploc\mod$.  Assume, for each 
$P\in\Ob(\calf^c){\sminus}\calh$, that either $O_p(\outf(P))\ne1$, or some 
element of order $p$ in $\outf(P)$ acts trivially on $F(P)$.  Let 
$F_0\:\orb(\calf^c)\op\Right2{}\zploc\mod$ be the functor where 
$F_0(P)=F(P)$ if $P\in\calh$ and $F_0(P)=0$ otherwise.  Then 
	\[ \higherlim{\orb(\calf^c)}*(F) \cong 
	\higherlim{\orb(\calf^c)}*(F_0) \cong 
	\higherlim{\orb(\calf^\calh)}*(F|_{\orb(\calf^\calh)\op})~. \]
\end{Lem}

\begin{proof}  Let $F_1\subseteq{}F$ be the subfunctor defined by setting 
$F_1(P)=F(P)$ if 
$P\notin\calh$ and $F_1(P)=0$ otherwise.  Thus $F_0\cong{}F/F_1$.  By 
\cite[Lemma 2.3]{limz-odd}, $\higherlimm{\orb(\calf^c)}*(F_1)=0$ if certain 
graded groups $\Lambda^*(\outf(P);F(P))$ vanish for each 
$P\in\Ob(\calf^c){\sminus}\calh$.  By \cite[Proposition 6.1(ii)]{JMO}, 
this is the case whenever $O_p(\outf(P))\ne1$, or some element of order $p$ 
in $\outf(P)$ acts trivially on $F(P)$.  This proves the first 
isomorphism.

For any category $\calc$, let $\calc\mod$ be the category of contravariant 
functors from $\calc$ to abelian groups.  Let $E$ be the functor 
``extension by zero'' from $\orb(\calf^\calh)\mod$ to $\orb(\calf^c)\mod$. 
Since $\calh\subseteq\Ob(\calf^c)$ is closed under $\calf$-conjugacy and 
overgroups, $E$ is right adjoint to the restriction functor.  Thus $E$ 
sends injectives to injectives.  So for $\Phi$ in $\orb(\calf^\calh)\mod$, 
$\higherlimm{\orb(\calf^\calh)}*(\Phi)\cong
\higherlimm{\orb(\calf^c)}*(E(\Phi))$.  Since 
$F_0=E(F|_{\orb(\calf^\calh)\op})$, the second isomorphism now follows.  
\end{proof}

For any saturated fusion system $\calf$ over a finite $p$-group $S$, let 
	\[ \calz_\calf\: \orb(\calf^c)\op \Right5{} \zploc\mod \]
be the functor which sends an object $P$ of $\calf^c$ to $Z(P)=C_S(P)$.  
For each $\varphi\in\homf(P,Q)$ in $\calf^c$, 
	\[ \calz_\calf([\varphi]) = \varphi^{-1}|_{Z(Q)} \: Z(Q) \Right4{} 
	Z(P). \]
When $\calf=\calf_S(G)$ for some finite group $G$ and some $S\in\sylp{G}$, 
and $H\nsg{}G$ is a normal subgroup, we let 
$\calz_\calf^H\subseteq\calz_\calf$ be the subfunctor which sends $P$ to 
$Z(P)\cap{}H=\calz_\calf(P)\cap{}H$.

The obstruction to the existence of a centric linking system associated to 
$\calf$ lies in $\higherlimm{\orb(\calf^c)}3(\calz_\calf)$, and the 
obstruction to its uniqueness lies in 
$\higherlimm{\orb(\calf^c)}2(\calz_\calf)$ \cite[Proposition 3.1]{BLO2}.  
The main results in \cite{limz-odd} and \cite{limz} state that these 
groups vanish whenever $\calf$ is the fusion system of a finite group $G$.

\begin{Defi} \label{d:L(p)}
Fix a prime $p$.
\begin{enuma}  
\item Let $\simpclass{p}$ be the class of all nonabelian finite simple 
groups $L$ with the following property.  For any finite group $G$, and any 
pair of subgroups $H\nsg{}K\nsg{}G$ both normal in $G$ such that 
$K/H\cong{}L^m$ for some $m\ge1$, if we set $\calf=\calf_S(G)$ for some 
$S\in\sylp{G}$, then 
$\higherlimm{\orb(\calf^c)}i(\calz_\calf^K/\calz_\calf^H)=0$ for $i\ge2$.

\item Let $\gpclass{p}$ be the class of all finite groups $G$ all of 
whose nonabelian composition factors lie in $\simpclass{p}$.

\item A saturated fusion system $\calf$ over a finite $p$-group $S$ is 
\emph{strongly tame} if it is tamely realizable by a group 
$G\in\gpclass{p}$.  
\end{enuma}
\end{Defi}

We first show that by results in \cite{limz-odd} and \cite{limz}, for each 
prime $p$, all nonabelian finite simple groups are in $\simpclass{p}$.  
Hence all finite groups are in $\gpclass{p}$, and all tame fusion systems 
are strongly tame.

Fix a finite group $G$ with $S\in\sylp{G}$, and set $\calf=\calf_S(G)$.  
Let $\orb_p(G)$ be the $p$-subgroup orbit category of $G$ as defined in 
\cite{limz} and \cite{limz-odd}.  Let $\calz_G\:\orb_p(G)\op\Right2{}\Ab$ 
be the functor $\calz_G(P)=Z(P)$ if $P$ is $\calf$-centric and 
$\calz_G(P)=0$ otherwise.  For $H\nsg{}G$, let $\calz_G^H\subseteq\calz_G$ 
be the subfunctor $\calz_G^H(P)=Z(P)\cap{}H$ if $P$ is $\calf$-centric.  
By \cite[Lemma 2.1]{limz-odd}, for $K\nsg{}H\nsg{}G$ both normal in $G$, 
$\higherlimm{\orb(\calf^c)}*(\calz_\calf^H/\calz_\calf^K)\cong 
\higherlimm{\orb_p(G)}*(\calz_G^H/\calz_G^K)$.  In particular, 
$\calz_\calf$ and $\calz_G$ have the same higher limits.  This allows us, 
in the proofs of the next theorem and lemma, to apply the results in 
\cite{limz-odd} and \cite{limz}, which are all stated in terms of limits 
taken over $\orb_p(G)$.  

\begin{Thm} \label{t:L(p)}
For each prime $p$, the class $\simpclass{p}$ contains all nonabelian 
finite simple groups, and the class $\gpclass{p}$ contains all finite 
groups.  Hence all tame fusion systems are strongly tame.
\end{Thm}

\begin{proof}  The last two statements follow immediately from the first 
one and the definitions.  So we need only show that $\simpclass{p}$ 
contains all nonabelian finite simple groups.

Assume $p$ is odd.  By \cite[Proposition 4.1]{limz-odd} (and its proof), a 
nonabelian finite simple group $L$ with $S\in\sylp{L}$ lies in 
$\simpclass{p}$ if there is a subgroup $Q\le\mathfrak{X}_L(S)$ which is 
centric in $S$ (i.e., $C_S(Q)\le{}Q$) and not $\Aut(L)$-conjugate to any 
other subgroup of $S$.  Here, $\mathfrak{X}_L(S)$ is a certain subgroup of 
$S$ defined in \cite[\S\S\,3--4]{limz-odd}.  By \cite[Propositions 
4.2--4.4]{limz-odd} (and the classification theorem), all nonabelian 
finite simple groups have this property, and thus they all lie in 
$\simpclass{p}$.

If $p=2$, then by \cite[Proposition 2.7]{limz}, a nonabelian finite simple 
group $L$ is contained in $\simpclass{2}$ if it is contained in the class 
$\mathfrak{L}^{\ge2}(2)$ defined in \cite[Definition 2.8]{limz}.  By 
\cite[Theorems 5.1, 6.2, 7.5, 8.13, \& 9.1]{limz} and the classification 
theorem, all nonabelian finite simple groups are contained in 
$\mathfrak{L}^{\ge2}(2)$. 
\end{proof}

Theorem \ref{t:L(p)} together with Theorem \ref{T:reduce} will imply Theorem 
\ref{ThA}.  From now on, for the rest of the section, we avoid using the 
classification theorem by assuming whenever necessary that our groups are 
in $\gpclass{p}$ and applying the following lemma.

\begin{Lem} \label{lim2=0}
Fix a finite group $G$ with Sylow subgroup $S\in\sylp{G}$, and set 
$\calf=\calf_S(G)$.  Assume $G\in\gpclass{p}$.  Then the following hold.
\begin{enuma} 
\item Assume $\widehat{G}$ is a finite group with $G\nsg\widehat{G}$.  Fix 
$\widehat{S}\in\sylp{\widehat{G}}$ such that $S=\widehat{S}\cap{}G$, and 
set $\widehat{\calf}=\calf_{\widehat{S}}(\widehat{G})$.  Then 
$\higherlimm{\orb(\widehat{\calf}^c)}i(\calz_{\widehat{\calf}}^G)=0$ for 
$i\ge2$.

\item If $G=G_1\times{}G_2$, then $G\in\gpclass{p}$ if and only if 
$G_1,G_2\in\gpclass{p}$.  If $H\nsg{}G$ and $G/H$ is $p$-solvable, then 
$H\in\gpclass{p}$ if and only if $G\in\gpclass{p}$.

\item Let $\call$ be a linking system associated to $\calf$ such that all 
subgroups in $\calh\defeq\Ob(\call)$ are $\calf$-centric.  Then 
$\call\cong\call_S^{\calh}(G)$.

\item The homomorphism 
$\mu_G\:\Out\typ(\call_S^c(G))\Right2{}\Out(S,\calf)$ defined in 
Section \ref{s:Aut(L)} is surjective.
\end{enuma}
\end{Lem}

\begin{proof}  \textbf{(a) } 
Let $1=G_0\nsg{}G_1\nsg\cdots\nsg{}G_m=G$ be a sequence of subgroups, all 
normal in $\widehat{G}$, such that for each $r$, $G_{r+1}/G_r$ is a 
minimal nontrivial normal subgroup of $\widehat{G}/G_r$.  By 
\cite[Theorem 2.1.5]{Gorenstein}, each quotient $G_{r+1}/G_r$ is a 
product of simple groups isomorphic to each other.  By \cite[Proposition 
2.2]{limz}, if $G_{r+1}/G_r$ is abelian, then 
$\higherlimm{\orb(\widehat{\calf}^c)}i
(\calz_{\widehat\calf}^{G_{r+1}}/\calz_{\widehat\calf}^{G_r}) =0$ for all 
$i\ge1$.  Thus (a) follows immediately from the definition of 
$\simpclass{p}$, together with the exact sequences 
	\[ \higherlim{\orb(\widehat{\calf}^c)}i(\calz_{\widehat\calf}
	^{G_s}/\calz_{\widehat\calf}^{G_r}) \Right4{}
	\higherlim{\orb(\widehat{\calf}^c)}i(\calz_{\widehat\calf}
	^{G_t}/\calz_{\widehat\calf}^{G_r}) \Right4{}
	\higherlim{\orb(\widehat{\calf}^c)}i(\calz_{\widehat\calf}
	^{G_t}/\calz_{\widehat\calf}^{G_s}) \]
for all $0\le{}r<s<t\le{}m$ and all $i\ge2$.  

\smallskip

\noindent\textbf{(b) }  The first statement is immediate, since a simple 
group is a composition factor of $G=G_1\times{}G_2$ if and only if 
it is a composition factor of $G_1$ or of $G_2$.  When $H\nsg{}G$ and 
$G/H$ is $p$-solvable, then the only simple groups which could be 
composition factors of $G$ but not of $H$ are $C_p$ and simple groups 
of order prime to $p$.  So we need only show that every nonabelian simple 
group of order prime to $p$ lies in $\simpclass{p}$.  

Fix such a simple group $L$, and assume $H\nsg{}K\nsg{}G$ (where 
$H\nsg{}G$), and $K/H\cong{}L^m$ for some $m$.  Set $\calf=\calf_S(G)$ for 
some $S\in\sylp{G}$.  Then $\calz_\calf^H=\calz_\calf^K$ since $K/H$ 
has order prime to $p$; and thus $L\in\simpclass{p}$.

\smallskip

\noindent\textbf{(c,d) }  By (a), applied with $\widehat{G}=G$, 
$\higherlimm{\orb(\calf^c)}2(\calz_\calf)=0$.  So by \cite[Theorem E]{BLO1}, 
$\mu_G$ is onto.  This proves (d).

Now let $\call$ be a linking system associated to $\calf$, set 
$\calh=\Ob(\call)$, and assume $\calh\subseteq\Ob(\calf^c)$.  Since 
$\calh$ contains all subgroups of $S$ which are $\calf$-centric and 
$\calf$-radical, $O_p(\outf(P))\ne1$ for $P\in\Ob(\calf^c){\sminus}\calh$. 
By Lemma \ref{lim*(orb)}, 
$\higherlimm{\orb(\calf^\calh)}2(\calz_\calf|_{\orb(\calf^\calh)\op})=0$ 
for $i\ge2$ since $\higherlimm{\orb(\calf^c)}2(\calz_\calf)=0$.  So by the 
same argument as that used in the proof of \cite[Proposition 3.1]{BLO2}, 
all linking systems associated to $\calf$ with object set $\calh$ are 
isomorphic.  In particular, $\call\cong\call_S^\calh(G)$, and this proves 
(c).
\end{proof}

The following is the main technical result in this subsection, and will be 
needed in the proof of Theorem \ref{ThA}.  Given $\calf_0\nsg\calf$ 
satisfying certain technical assumptions, and given a linking system 
$\call_0$ associated to $\calf_0$, we want to find $\call$ associated to 
$\calf$ such that $\call_0\nsg\call$.  It is natural to ask why this 
cannot be done using \cite[Theorem 9]{O3}, where conditions are explicitly 
set up to construct extensions of fusion and linking systems.  There seem 
to be two difficulties with that approach.  First, the hypotheses of 
Proposition \ref{F0/A=F(G)} are very different from those in \cite{O3}, 
and it is not clear how to convert from the one to the other.  But more 
seriously, even if one does manage to do that and construct an extension 
$(\calf',\call')$ of $(\calf_0,\call_0)$, it is not clear how to prove 
that $\calf'\cong\calf$; i.e., that $\call'$ really is a linking system 
associated to $\calf$.

\begin{Prop} \label{F0/A=F(G)}
Let $\calf$ be a saturated fusion system over a finite $p$-group $S$, and 
let $\calf_0\nsg\calf$ be a normal fusion subsystem over 
$S_0\nsg{}S$.  Assume $\calf_0$ is strongly tame; and either 
\begin{enuma}  
\item $\calf_0=O^p(\calf)$, or
\item $\calf_0=O^{p'}(\calf)$, or
\item $\calf_0=N_\calf^K(Q)$ for some $Q\nsg\calf$ and some 
$K\nsg\Aut(Q)$ containing $\Inn(Q)$.  
\end{enuma}
Then there is a centric linking system associated to $\calf$.
\end{Prop}

\begin{proof}  In all cases (a), (b), and (c), $\calf_0\nsg\calf$ 
by Proposition \ref{fus-ex}.  

Since $\calf_0$ is strongly tame, we can choose a finite group 
$G_0\in\gpclass{p}$ such that $S_0\in\sylp{G_0}$, 
$\calf_0=\calf_{S_0}(G_0)$, and $\kappa_{G_0}$ is split surjective.  
We first claim that
	\beqq \til\mu_{G_0}\circ\til\kappa_{G_0}
	\: \Aut(G_0,S_0)\Right5{}\Aut(S_0,\calf_0) 
	\quad\textup{is onto.} \label{e:onto} \eeqq
As noted in Section \ref{s:tamedef}, this composite is defined 
by restriction.  Since $G_0\in\gpclass{p}$, $\mu_{G_0}$ is surjective by 
Lemma \ref{lim2=0}(d).  Also, $\kappa_{G_0}$ is split surjective by 
assumption.  Thus every element of 
$\Out(S_0,\calf_0)=\Aut(S_0,\calf_0)/\Aut_{G_0}(S_0)$ extends to an 
element of $\Out(G_0)=\Aut(G_0,S_0)/\Aut_{N_{G_0}(S_0)}(G_0)$, and hence 
the map in \eqref{e:onto} is onto. 

Define
	\[ \Delta = \bigl\{ \alpha\in\Aut(G_0,S_0) \,\big|\, 
	\alpha|_{S_0}\in\autf(S_0) \bigr\} ~. \]
We just showed that every element in $\autf(S_0)$ is the restriction of 
some element in $\Delta$.  Fix $S_\Delta\in\sylp{\Delta}$ which surjects 
onto $\Aut_S(S_0)$ under the restriction map to $\autf(S_0)$.  Set 
	\[ \widehat{G}=G_0\sd{}\Delta \qquad\textup{and}\qquad
	\widehat{S}=S_0\sd{}S_\Delta. \]
Thus $\widehat{S}\in\sylp{\widehat{G}}$.

Now set $S_1=S$, $\calf_1=\calf$, $S_2=\widehat{S}$, and 
$\calf_2=\calf_{\widehat{S}}(\widehat{G})$.  We claim, for each 
$P_0,Q_0\le{}S_0$, that 
	\beqq \begin{split}  
	\Hom_{\calf_2}(P_0,Q_0) &= \Hom_{\calf_1}(P_0,Q_0) 
	\defeq \homf(P_0,Q_0) \\
	\Hom_{S_2}(P_0,Q_0) &= \Hom_{S_1}(P_0,Q_0) 
	\defeq \Hom_S(P_0,Q_0)~.
	\end{split}
	\label{F1_0=F2_0} \eeqq
We have already remarked that $\calf_0\nsg\calf=\calf_1$, and 
$\calf_0\nsg\calf_2$ by Proposition \ref{L(G0)<|L(G)} since they are the 
fusion systems of $G_0\nsg\widehat{G}$.  Hence by condition (ii) in 
Definition \ref{d:F0<|F}, each $\varphi\in\Hom_{\calf_i}(P,Q)$ (for 
$i=1,2$ and $P,Q\le{}S_0$) is the composite of a morphism in $\calf_0$ and 
the restriction of a morphism in $\Aut_{\calf_i}(S_0)$.  Furthermore,
	\[ \Aut_{\calf_2}(S_0)= \Aut_{\widehat{G}}(S_0) = 
	\Gen{ \Aut_{G_0}(S_0) \,,\, \Res^{G_0}_{S_0}(\Delta) } =
	\Aut_{\calf_1}(S_0) \]
by \eqref{e:onto}, and the first line in \eqref{F1_0=F2_0} now follows.  
The second holds since $\Aut_{S_2}(S_0)=\Aut_S(S_0)$ by definition of 
$S_2=S_0\rtimes{}S_\Delta$.  

We next claim that for all $P_0\le{}S_0$, 
	\beqq \textup{$P_0$ is fully centralized in $\calf_2$} 
	\quad\Longleftrightarrow\quad
	\textup{$P_0$ is fully centralized in $\calf_1=\calf$~.} 
	\label{e:same-fc} \eeqq
By \eqref{F1_0=F2_0}, the $\calf_1$- and $\calf_2$-conjugacy classes of 
$P_0$ are the same, and $\Aut_{S_2}(S_0)=\Aut_S(S_0)$.  Hence for each 
$Q_0$ which is $\calf_i$-conjugate to $P_0$, 
	\[ \frac{|C_{S_1}(Q_0)|}{|C_{S_1}(S_0)|} = 
	\bigl|\bigl\{\alpha\in\Aut_S(S_0)\,\big|\,
	\alpha|_{Q_0}=\Id \bigr\}\bigr| = 
	\frac{|C_{S_2}(Q_0)|}{|C_{S_2}(S_0)|}~, \]
and so $|C_{S_1}(P_0)|$ is maximal if and only if $|C_{S_2}(P_0)|$ is 
maximal.  

We want to compute $\higherlimm{\orb(\calf_1^c)}*(\calz_{\calf_1})$ by 
comparing it with $\higherlimm{\orb(\calf_2^c)}*(\calz_{\calf_2})$.  To do 
this, we first define in Step 1 certain full subcategories 
$\calf_i^*\subseteq\calf_i^c$, and an intermediate category $\calc$ which 
can be used to compare $\orb(\calf_1^*)$ with $\orb(\calf_2^*)$.  Certain 
properties of the ``comparison functors'' 
$\Phi_i\:\orb(\calf_i^*)\Right2{}\calc$ are stated and proven in Step 2.  
In Step 3, we define certain subfunctors $\calz_i\subseteq\calz_{\calf_i}$ 
on $\orb(\calf_i^c)$, and prove that 
$\higherlimm{\orb(\calf_1^c)}*(\calz_1)\cong 
\higherlimm{\orb(\calf_2^c)}*(\calz_2)$ using the intermediate categories 
$\orb(\calf_i^*)$ and $\calc$ to compare them.  Finally, in Step 4, we 
prove that $\higherlimm{\orb(\calf_2^c)}*(\calz_2)=0$ for $*\ge2$, and 
then show that $\higherlimm{\orb(\calf^c)}*(\calz_\calf)\cong 
\higherlimm{\orb(\calf_1^c)}*(\calz_1)$ for $*\ge1$ by analyzing 
individually the three cases (a)--(c).  

Throughout the rest of the proof, whenever $P\le{}S_1$ or $P\le{}S_2$, we 
write $P_0=P\cap{}S_0$. 

\smallskip

\noindent\textbf{Step 1: }  Let $\calf_i^*\subseteq\calf_i$ ($i=1,2$) be 
the full subcategories with objects
	\[ \Ob(\calf_i^*) = \bigl\{ P\le S_i \,\big|\, 
	C_{S_i}(Q_0)\le{}Q \textup{ for all $Q$ $\calf_i$-conjugate 
	to $P$}  \bigr\}. \]
All objects in $\calf_i^*$ are $\calf_i$-centric; i.e., 
$\calf_i^*\subseteq\calf_i^c$.  Also, if $P\le{}S_i$ is $\calf_i$-conjugate 
to an object in $\calf_i^*$, then $P\in\Ob(\calf_i^*)$.  

We next construct a category $\calc$ which acts as intermediary between 
the orbit categories of $\calf_1^*$ and $\calf_2^*$.  It will be 
a subcategory of a larger category $\widehat{\calc}$, defined by setting 
	\begin{multline*}  
	\Ob(\widehat{\calc}) = \bigl\{ (P_0,K) \,\big|\, 
	P_0\le{}S_0 \textup{ is $\calf_0$-centric and fully centralized 
	in $\calf$, } \\ \Inn(P_0)\le K\le\Aut_S(P_0) \bigr\}
	\end{multline*}
and
	\[ \Mor_{\widehat{\calc}}\bigl((P_0,K),(Q_0,L)\bigr) = 
	L{\big\backslash} \bigl\{\varphi\in\homf(P_0,Q_0) 
	\,\big|\, \varphi{}K\subseteq L\varphi \bigr\}. \]
Here, we regard $\varphi{}K$ and $L\varphi$ as subsets of $\homf(P_0,Q_0)$.  

Define functors
	\[ \orb(\calf_1^*) \Right4{\Phi_1} \widehat{\calc} \Left4{\Phi_2} 
	\orb(\calf_2^*)~, \]
by setting $\Phi_i(P)=(P_0,\Aut_P(P_0))$ and 
$\Phi_i([\varphi])=[\varphi|_{P_0}]$ for $P,Q\in\Ob(\calf_i^*)$ and 
$\varphi\in\Hom_{\calf_i^*}(P,Q)$.  
Since restriction sends $\Inn(Q)$ to $\Aut_Q(Q_0)$ for 
$Q\in\Ob(\calf_i^*)$, $\Phi_i$ is well defined on morphisms if it is 
defined on objects.

%%Then $\Phi_i$ is defined on objects by 
%%\eqref{e:fc3}, and since $\Aut_S(P_0)=\Aut_{S_i}(P_0)$ by 
%%\eqref{F1_0=F2_0}. 

To see that $\Phi_i$ is well defined on objects, fix $P\in\Ob(\calf_i^*)$, 
and set $K=\Aut_P(P_0)$.  Then $\Inn(P_0)\le\Aut_P(P_0)\le\Aut_S(P_0)$, 
since $P\ge{}P_0$, and since $\Aut_S(P_0)=\Aut_{S_i}(P_0)$ by 
\eqref{F1_0=F2_0}.  By Lemma \ref{F0<|F}(a), $P$ is $\calf_i$-conjugate to 
some $Q$ such that $Q_0$ is fully normalized in $\calf_i$ and in 
$\calf_0$.  Hence $P_0$ is $\calf_0$-centric by Lemma \ref{F0<|F}(c).  
Also, $|C_{S_i}(P_0)|=|C_P(P_0)|=|C_Q(Q_0)|=|C_{S_i}(Q_0)|$ by definition 
of $\Ob(\calf_i^*)$, and so $P_0$ is fully centralized in $\calf_i$ (hence 
in $\calf$ by \eqref{e:same-fc}) since $Q_0$ is fully centralized in 
$\calf_i$.  Thus $(P_0,K)\in\Ob(\widehat{\calc})$. 

We claim that $\Im(\Phi_1)=\Im(\Phi_2)$.  In what follows, when 
$P_0\le{}S_0$ and $K\le\Aut(P_0)$, we set 
$N_S^K(P_0)=\{g\in{}N_S(P_0)\,|\,c_g\in{}K\}$.  Then 
	\beqq P\in\Ob(\calf_i^*) \textup{ and } \Phi_i(P)=(P_0,K) 
	\quad\Longrightarrow\quad P=N_{S_i}^K(P_0) \label{e:uniq} \eeqq
since $P\ge{}C_{S_i}(P_0)$ by definition of $\Ob(\calf_i^*)$.

Assume $(P_0,K)\in\Ob(\widehat{\calc})$, and set $P_i=N_{S_i}^K(P_0)$ 
for $i=1,2$.  Then $P_i\ge{}P_0$ and $K=\Aut_{P_i}(P_0)$, since 
$\Inn(P_0)\le{}K\le\Aut_S(P_0)$ by assumption.  Also, 
$P_1\cap{}S_0=N_{S_0}^K(P_0)=P_2\cap{}S_0$, so $P_1\cap{}S_0=P_0$ if and 
only if $P_2\cap{}S_0=P_0$, and we assume this is the case since otherwise 
$(P_0,K)$ is in the image of neither functor $\Phi_i$ by \eqref{e:uniq}.  
By assumption, $P_0$ is fully centralized in $\calf$, and hence in 
$\calf_i$ by \eqref{e:same-fc}.  So for each $Q$ which is 
$\calf_i$-conjugate to $P_i$, 
$|C_{S_i}(Q_0)|\le|C_{S_i}(P_0)|=|C_{P_i}(P_0)|=|C_Q(Q_0)|$, where the last 
equality holds since any $\varphi\in\Iso_{\calf_i}(P_i,Q)$ induces an 
isomorphism of pairs $(P_i,P_0)\cong(Q,Q_0)$.  Thus 
$C_{S_i}(Q_0)\le{}Q$.  This proves that $P_i\in\Ob(\calf_i^*)$, 
and hence that $\Phi_i(P_i)=(P_0,K)$ for $i=1,2$.

Now fix objects $(P_0,K)$ and $(Q_0,L)$ in $\Im(\Phi_i)$, and choose 
$\varphi_0\in\homf(P_0,Q_0)$ such that $\varphi_0K\subseteq{}L\varphi_0$.  
Thus $[\varphi_0]\in\Mor_{\widehat{\calc}}((P_0,K),(Q_0,L))$.  If 
$[\varphi_0]=\Phi_i([\varphi])$ for some $\varphi\in\Hom_{\calf_i^*}(P,Q)$ 
($i=1$ or $2$), then $\varphi(P)\in\Ob(\calf_i^*)$, so $\varphi_0(P_0)$ 
is fully centralized in $\calf$, and hence in $\calf_1$ and $\calf_2$ by 
\eqref{e:same-fc}.  So we assume this from now on.  

Set $P_i=N_{S_i}^K(P_0)$ and $Q_i=N_{S_i}^L(Q_0)$:  these are both in 
$\Ob(\calf_i^*)$ by \eqref{e:uniq}.  Set $R_0=\varphi_0(P_0)$, let 
$\dot\varphi_0\in\isof(P_0,R_0)$ be the restriction of $\varphi_0$, and 
set $M=\dot\varphi_0K\dot\varphi_0^{-1}\le \autf(R_0)$.  Then 
$\varphi_0K\dot\varphi_0^{-1}\subseteq{}L|_{R_0}\subseteq 
\Hom_{S}(R_0,Q_0)$, and so $M\le\Aut_S(R_0)=\Aut_{S_i}(R_0)$.  By the 
extension axiom for $\calf_i$, $\varphi_0$ extends to some 
$\varphi_i\in\Hom_{\calf_i}(P_i,S_i)$, and $[\varphi_0]\in\Im(\Phi_i)$ if 
and only if $\varphi_i$ can be chosen with $\varphi_i(P_i)\le{}Q_i$.  Now, 
$M=\Aut_{\varphi_i(P_i)}(R_0)$ since $K=\Aut_{P_i}(P_0)$, so 
$\Phi_i(\varphi_i(P_i))=(R_0,M)$, and $\varphi_i(P_i)=N_{S_i}^M(R_0)$ by 
\eqref{e:uniq}.  Hence $\varphi_i(P_i)\le{}Q_i$ if and only if for all 
$\alpha\in\Aut_{S_i}(S_0)$, 
	\[ \alpha(R_0)=R_0 \textup{ and } \alpha|_{R_0}\in{}M 
	\quad\Longrightarrow\quad
	\alpha(Q_0)=Q_0 \textup{ and } \alpha|_{Q_0}\in{}L ~. \]
Since $\Aut_{S_1}(S_0)=\Aut_{S_2}(S_0)$ by \eqref{F1_0=F2_0}, 
$[\varphi_0]\in\Im(\Phi_1)$ if and only if $[\varphi_0]\in\Im(\Phi_2)$.

Now set $\calc=\Im(\Phi_1)=\Im(\Phi_2)\subseteq\widehat{\calc}$.  Since the 
$\Phi_i$ are injective on objects by \eqref{e:uniq}, this is a subcategory 
of $\widehat{\calc}$.  From now on, we regard the $\Phi_i$ as functors to 
$\calc$.

\smallskip

\noindent\textbf{Step 2: }  For each $i=1,2$, and each 
$P\in\Ob(\calf_i^*)$, set 
	\begin{align*}  
	\Gamma_i(P)&=\Ker\bigl[ \Out_{\calf_i}(P) \Right2{\Phi_i} 
	\Aut_{\calc}(P_0,\Aut_P(P_0)) \bigr] \\
	&= \Ker\bigl[ \Out_{\calf_i}(P) \Right2{R} 
	N_{\Aut_{\calf}(P_0)}(\Aut_P(P_0))/\Aut_P(P_0) \bigr] 
	\end{align*}
where $R$ is induced by restriction.  We claim that, for each $i=1,2$,
\begin{enumerate}[(i)]
\item $\Phi_i\:\orb(\calf_i^*)\Right2{}\calc$ is bijective on objects and 
surjective on morphism sets; 

\item $\Gamma_i(P)$ has order prime to $p$ for all $P$; and

\item whenever $\psi,\psi'\in\Mor_{\orb(\calf_i^*)}(P,Q)$ are such that 
$\Phi_i(\psi)=\Phi_i(\psi')$, there is $\chi\in\Gamma_i(P)$ such 
that $\psi'=\psi\circ\chi$.  
\end{enumerate}
Point (i) follows from \eqref{e:uniq} and the definition of $\calc$ in Step 
1. 

When proving (ii), it suffices to consider the case where $P$ is fully 
normalized in $\calf_i$.  If $g\in{}N_{S_i}(P)$ is such that 
$[c_g]\in\Gamma_i(P)$, then $c_g|_{P_0}\in\Aut_P(P_0)$; and since 
$C_{S_i}(P_0)\le{}P$, this implies $g\in{}P$ and 
$[c_g]=1\in\Gamma_i(P)\le\Out_{\calf_i}(P)$.  Thus $\Out_{S_i}(P)$ is a Sylow 
$p$-subgroup of $\Out_{\calf_i}(P)$ and intersects trivially with 
$\Gamma_i(P)\nsg\Out_{\calf_i}(P)$, so $|\Gamma_i(P)|$ is prime to $p$.

It remains to prove (iii).  Assume 
$\psi,\psi'\in\Mor_{\orb(\calf_i^*)}(P,Q)$ are such that 
$\Phi_i(\psi)=\Phi_i(\psi')$.  Fix 
$\varphi,\varphi'\in\Hom_{\calf_i^*}(P,Q)$ such that $\psi=[\varphi]$ and 
$\psi'=[\varphi']$.  Then $\varphi|_{P_0}=c_g\circ\varphi'|_{P_0}$ for 
some $c_g\in\Aut_Q(Q_0)$; i.e., for some $g\in{}Q$.  So upon replacing 
$\varphi'$ by $c_g\circ\varphi'$ (this time with $c_g\in\Inn(Q)$), we can 
assume $\varphi_0\defeq\varphi|_{P_0}=\varphi'|_{P_0}$.  Since $P\in 
\Ob(\calf_i^*)$, we have $C_{S_i}(\varphi_0(P_0))\le \varphi(P)$, so
	\[ \varphi(P) = \bigl\{g\in{}N_{S_i}(\varphi_0(P_0)) \,\big|\, 
	c_g\in\varphi_0\Aut_P(P_0)\varphi_0^{-1} \bigr\} = \varphi'(P). \]
Hence there is a unique $\beta\in\Aut_{\calf_i}(P)$ such that 
$\varphi'=\varphi\circ\beta$; and also $\beta|_{P_0}=\Id$.  So 
$\psi'=\psi\circ[\beta]$ in $\Mor_{\orb(\calf_i^*)}(P,Q)$, where 
$[\beta]\in\Gamma_i(P)$.  

\smallskip

\noindent\textbf{Step 3: }  Define functors
	\[ \calz_i\:\orb(\calf_i^c)\op \Right4{} 
	\zploc\mod \qquad\textup{and}\qquad 
	\calz_\calc\: \calc\op \Right4{} \zploc\mod \]
by setting $\calz_i(P)=Z(P)\cap{}S_0=C_{Z(P_0)}(P)$ and 
$\calz_\calc(Q,K)=C_{Z(Q)}(K)$ (the 
subgroup of elements of $Z(Q)$ fixed pointwise by $K$).  
Morphisms are sent in the obvious way.  Set 
$\calz_{i*}=\calz_i|_{\orb(\calf_i^*)\op}$.  

We claim that 
	\beqq \higherlim{\orb(\calf_1^c)}*(\calz_1) \cong
	\higherlim{\orb(\calf_1^*)}*(\calz_{1*}) \cong
	\higherlim{\calc}*(\calz_\calc) \cong
	\higherlim{\orb(\calf_2^*)}*(\calz_{2*}) \cong
	\higherlim{\orb(\calf_2^c)}*(\calz_2) . \label{5lims} \eeqq
Since $\calz_{i*}=\calz_\calc\circ\Phi_i$ by definition, the second and 
third isomorphisms follow from points (i--iii) in Step 2 and \cite[Lemma 
1.3]{BLO1}.

We prove the other isomorphisms in \eqref{5lims} using Lemma 
\ref{lim*(orb)}.  Fix $i=1,2$.  We already saw in Step 1 that 
$\Ob(\calf_i^*)$ is closed (inside $\Ob(\calf_i^c)$) with respect to 
$\calf_i$-conjugacy.  If $P\le{}Q\le{}S_i$ and $P\in\Ob(\calf_i^*)$, then 
for each $Q^*$ which is $\calf_i$-conjugate to $Q$, if we set 
$P^*=\varphi(P)\le{}Q^*$ for some $\varphi\in\Iso_{\calf_i}(Q,Q^*)$, then 
$C_{S_i}(P^*_0)\le{}P^*$ implies $C_{S_i}(Q^*_0)\le{}Q^*$, and so 
$Q\in\Ob(\calf_i^*)$.  Thus $\Ob(\calf_i^*)$ is closed with respect to 
overgroups.  

For each object $P$ in $\calf_i^c$ not in $\calf_i^*$, there is $P^*$ 
$\calf_i$-conjugate to $P$ such that $C_{S_i}(P^*_0)\nleq{}P^*$.  By Lemma 
\ref{QnleqP}, there is $g\in{}N_{S_i}(P^*){\sminus}P^*$ which centralizes 
$P^*_0$.  Thus $[c_g]\in\Out_{\calf_i}(P^*)$ is a nontrivial element of 
$p$-power order which acts trivially on $\calz_i(P^*)$.  So by Lemma 
\ref{lim*(orb)}, $\higherlimm{\orb(\calf_i^c)}*(\calz_i)\cong
\higherlimm{\orb(\calf_i^*)}*(\calz_{i*})$ 
for each $i=1,2$; and this finishes the proof of \eqref{5lims}.

\smallskip

\noindent\textbf{Step 4: }  By Lemma \ref{lim2=0}(a), 
$\higherlimm{\orb(\calf_2^c)}j(\calz_2)=
\higherlimm{\orb(\calf_2^c)}j(\calz_{\calf_2}^{G_0})=0$ for 
$j\ge2$.  Hence by \eqref{5lims}, 
$\higherlimm{\orb(\calf^c)}j(\calz_1)=0$ for $j\ge2$ (recall 
$\calf=\calf_1$). 

We claim that for $j\ge2$,
	\beqq \higherlim{\orb(\calf^c)}j(\calz_\calf)\cong
	\higherlim{\orb(\calf^c)}j(\calz_1). \label{ZF/Z1} \eeqq
Set $\widehat{\calz}=\calz_\calf/\calz_1$ for short; thus 
$\widehat{\calz}(P)=Z(P)/(Z(P)\cap{}S_0)$ for each $P$.  If 
$\calf_0=O^{p'}(\calf)$, then \eqref{ZF/Z1} holds since $S_0=S$ and 
hence $\calz_\calf=\calz_1$.  If $\calf_0=N_\calf^K(Q)$ for some 
$Q\nsg\calf$ and some $K\nsg\Aut(Q)$, then $\widehat{\calz}(P)=0$ for each 
$P\in\Ob(\calf^c)$ which contains $Q$, in particular for each subgroup 
which is $\calf$-centric and $\calf$-radical (Proposition \ref{norm<=>}); 
and \eqref{ZF/Z1} holds by Lemma \ref{lim*(orb)}.

Assume $\calf_0=O^p(\calf)$.  For each $P\in\Ob(\calf^c)$, let 
$H_P\nsg\outf(P)$ be the kernel of the $\outf(P)$-action on 
$\widehat{\calz}(P)=Z(P)/(Z(P)\cap{}S_0)$.  By definition of 
$S_0=\hyp(\calf)$, $H_P$ contains 
$O^p(\outf(P))$, and thus $\outf(P)/H_P$ is a $p$-group.  So for $j\ge1$,
	\[ \Lambda^j(\outf(P);\widehat{\calz}(P)) \cong
	\begin{cases} 
	0 & \textup{if $p\big||H_P|$} \\
	0 & \textup{if $p\nmid|\outf(P)|$} \\
	\Lambda^j(\outf(P)/H_P;\widehat{\calz}(P)) = 0 & 
	\textup{otherwise}
	\end{cases} \]
by \cite[Proposition 6.1]{JMO}:  by point (ii) of the proposition in the 
first case, by point (i) in the second, and by points (iii) and (ii) in 
the third.  So by \cite[Lemma 2.3]{limz-odd}, 
$\higherlimm{\orb(\calf^c)}j(\widehat{\calz})=0$ for all $j\ge1$, and 
\eqref{ZF/Z1} also holds in this case.  

We now conclude that $\higherlimm{\orb(\calf^c)}j(\calz_\calf)=0$ for all 
$j\ge2$.  So by \cite[Proposition 3.1]{BLO2}, there is a (unique) centric 
linking system $\call^c$ associated to $\calf$.  
\end{proof}

\newsubb{Proof of Theorem \ref{ThA}}{s:ThA}

We want to show that if $\red(\calf)$ is tame, then so is $\calf$.  The 
proof splits naturally into two parts.  We first show, under certain 
additional hypotheses, that if $\calf_0\nsg\calf$ and $\calf_0$ is tame, 
then $\calf$ is tame.  Afterwards, we show (again under additional 
hypotheses) that $\calf$ is tame if 
$\calf/Z(\calf)$ is tame.  In both cases, this means proving that certain 
homomorphisms are split surjective, by first constructing an appropriate 
pullback square of automorphism groups, and then applying the following 
elementary lemma.  

\begin{Lem} \label{p.b.split}
If the following square of groups and homomorphisms 
	\beq \xymatrix@C=40pt{
	A_1 \ar[r]^{\alpha} \ar[d] & A_2 \ar[d] \\
	B_1 \ar[r]^{\beta} & B_2 
	} \eeq
is a pullback square, and $\beta$ is split surjective, then $\alpha$ is 
split surjective.  \qed
\end{Lem}

We first work with normal subsystems.  We first recall some convenient 
notation.  When $P$ is a $p$-centric subgroup of a finite group $G$ (i.e., 
an $\calf_S(G)$-centric subgroup when $P\le{}S\in\sylp{G}$), we set 
$C'_G(P)=O^p(C_G(P))$.  Thus $C'_G(P)$ has order prime to $p$, and 
$C_G(P)=Z(P)\times{}C'_G(P)$.  

For any normal pair $\SFL[_0]\nsg\SFL$, let
	\[ \til\rho = \til\rho^\call_{\call_0} \: \Aut_\call(S_0) \Right4{} 
	\Aut\typ^I(\call_0) \]
be the homomorphism which sends $\gamma\in\Aut_\call(S_0)$ to $c_\gamma$.  
Here, $c_\gamma\in\Aut\typ^I(\call_0)$ sends an object $P$ to 
$\pi(\gamma)(P)$ and sends $\psi\in\Mor_{\call_0}(P,Q)$ to 
$(\gamma|_{Q,\pi(\gamma)(Q)})\circ\psi\circ 
(\gamma|_{P,\pi(\gamma)(P)})^{-1}$ (well defined by Definition 
\ref{d:L0<|L}).  Let
%%	\[ \rho = \rho^\call_{\call_0} \: \call/\call_0 =
%%	\Aut_\call(S_0)/\Aut_{\call_0}(S_0) \Right4{} 
%%	\Out\typ(\call_0) \]
	\[ \rho = \rho^\call_{\call_0} \: 
	\2{=\Aut_\call(S_0)/\Aut_{\call_0}(S_0)}{\call/\call_0} \Right6{} 
	\2{=\Aut\typ^I(\call_0)/
	\{c_\gamma\,|\,\gamma\in\Aut_{\call_0}(S_0)\}}
	{\Out\typ(\call_0)} \]
be the homomorphism induced by $\til\rho$, which sends $[\gamma]$ to the 
class of $c_\gamma$.  This is analogous to the conjugation homomorphism 
$G/G_0\Right2{}\Out(G_0)$ for a pair of groups $G_0\nsg{}G$.  For example, 
$\call_0$ is centric in $\call$ (see Definition \ref{d:L0<|L}) if and only 
if $\rho^\call_{\call_0}$ is injective.

We next show that when $\calf_0\nsg\calf$ have associated linking systems 
$\call_0\nsg\call$, where $\call_0$ is centric in $\call$ and $\calf_0$ is 
realizable, then under some extra conditions, $\calf$ is also realizable.  

\begin{Lem} \label{G0<|G}
Fix a normal pair $\SFL[_0]\nsg\SFL$ such that $\call_0$ is centric in 
$\call$.  Set $\calh_0=\Ob(\call_0)$ and $\calh=\Ob(\call)$, and assume 
$\calh_0$ is $\Aut(S_0,\calf_0)$-invariant.  Assume there 
is a finite group $G_0$ such that
\begin{enuma}  
\item $S_0\in\sylp{G_0}$, $\calf_0=\calf_{S_0}(G_0)$, and
$\call_0\cong\call_{S_0}^{\calh_0}(G_0)$;
\item $Z(G_0)=Z(\calf_0)$; and 
\item there is a homomorphism 
$\widehat{\rho}\:\call/\call_0\Right2{}\Out(G_0)$ such that 
	\[ \kappa_{G_0}^{\calh_0}\circ\widehat{\rho} = 
	\rho^\call_{\call_0}\: \call/\call_0\Right2{}\Out\typ(\call_0) ~. \]
\end{enuma}
Then $\calf=\calf_S(G)$ and $\call\cong\call_S^\calh(G)$ for some 
finite group $G$ such that $S\in\sylp{G}$, $G_0\nsg{}G$, 
$G/G_0\cong\call/\call_0$, and such that the extension realizes the given 
outer action $\widehat{\rho}$ of $G/G_0\cong\call/\call_0$ on $G_0$.
\end{Lem}

\begin{proof}  We construct the group $G$ in Step 1, and prove that 
$\call_S^{\calh}(G)\cong\call$ and $\calf_S(G)=\calf$ in Step 2.  
Throughout the proof, we identify $\call_0$ with 
$\call_{S_0}^{\calh_0}(G_0)$.

\smallskip

\noindent\textbf{Step 1: }  Consider the following diagram whose rows are 
exact by Lemma \ref{OutI1}:
	\beqq \vcenter{\xymatrix@C=30pt{
	1 \ar[r] & Z(G_0) \ar[r] \ar@{=}[d] & N_{G_0}(S_0) \ar[r]^-{\conj} 
	\ar@{->>}[d]_{\lambda_0} & \Aut(G_0,S_0) \ar[r]^-{\pr_1} 
	\ar[d]_{\til\kappa}^{=\til\kappa_{G_0}^{\calh_0}} & 
	\Out(G_0) \ar[r] \ar[d]_{\kappa}^{=\kappa_{G_0}^{\calh_0}} & 1 \\
	1 \ar[r] & Z(\calf_0) \ar[r] & \Aut_{\call_0}(S_0) \ar[r]^{\conj} & 
	\Aut\typ^I(\call_0) 
	\ar[r]^-{\pr_2} & \Out\typ(\call_0) \ar[r] & 1 \rlap{~.}
	}} \label{e:2.10a} \eeqq
Here, $\lambda_0$ sends $g\in{}N_{G_0}(S_0)$ to its class in 
$\Aut_{\call_0}(S_0)=N_{G_0}(S_0)/C'_{G_0}(S_0)$.  
The first and third squares clearly commute.  The second square commutes 
since for $g\in{}N_{G_0}(S_0)$, $\til\kappa$ sends $c_g$ to the 
automorphism $[a]\mapsto[gag^{-1}]=c_{\lambda_0(g)}(a)$.  By definition of 
$\til\kappa=\til\kappa_{G_0}^{\calh_0}$,
	\beqq \til\kappa(\beta)(\lambda_0(g)) = \lambda_0(\beta(g))
	\qquad \textup{for all $\beta\in\Aut(G_0,S_0)$, 
	$g\in{}N_{G_0}(S_0)$~.} \label{e:2.10aa} \eeqq

Set $\Aut(G_0,S_0)_{\widehat{\rho}}= 
\pr_1^{-1}(\widehat{\rho}(\call/\call_0))$.  Since $\call_0$ is centric in 
$\call$, $\rho=\rho_{\call_0}^\call$ sends $\call/\call_0$ injectively 
into $\Out\typ(\call_0)$.  Hence $\kappa$ sends 
$\widehat{\rho}(\call/\call_0)$ injectively into $\Out\typ(\call_0)$.  So 
by a diagram chase in \eqref{e:2.10a}, $N_{G_0}(S_0)$ is the pullback of 
$\Aut(G_0,S_0)_{\widehat{\rho}}$ and $\Aut_{\call_0}(S_0)$ over 
$\Aut\typ^I(\call_0)$.

Let $H$ be the group which makes the following square a pullback:
	\beqq \vcenter{\xymatrix@C=30pt{
	H \ar[r]^-{\varphi} \ar@{->>}[d]^-{\lambda} & 
	\Aut(G_0,S_0)_{\widehat{\rho}} 
	\ar[d]^{\til\kappa} \\
	\Aut_{\call}(S_0) \ar[r]^-{\til\rho} & \Aut\typ^I(\call_0) 
	\rlap{~.} }} \label{e:2.10b} \eeqq
For each $\alpha\in\Aut_\call(S_0)$, 
$\til\rho(\alpha)\in\pr_2^{-1}(\rho(\call/\call_0))$ by definition, and 
hence lifts to an element of $\Aut(G_0,S_0)_{\widehat{\rho}}$.  This 
proves that $\lambda$ is onto.  By comparison with the middle square in 
\eqref{e:2.10a}, we can identify 
$N_{G_0}(S_0)$ with $\lambda^{-1}(\Aut_{\call_0}(S_0))\nsg{}H$.  Thus 
	\beqq H/N_{G_0}(S_0) = H/H_0 \cong 
	\Aut_\call(S_0)/\Aut_{\call_0}(S_0) = \call/\call_0~, 
	\label{e:2.10c} \eeqq
where we set $H_0=N_{G_0}(S_0)$, regarded as a subgroup of $G_0$ and of 
$H$.

We claim that for all $h\in{}H$ and $a\in{}H_0$, 
	\beqq \varphi(h)(a) = hah^{-1} \in H_0~. \label{e:2.10d} \eeqq
Since \eqref{e:2.10b} is a pullback, it suffices to prove \eqref{e:2.10d}
after applying $\varphi$ and after applying $\lambda$.  It holds 
after applying $\lambda$ (or $\lambda_0$) since 
	\[ \lambda_0(\varphi(h)(a)) = \til\kappa(\varphi(h))(\lambda_0(a)) 
	= \til\rho(\lambda(h))(\lambda_0(a)) = 
	\lambda(h)\lambda_0(a)\lambda(h)^{-1} = \lambda_0(hah^{-1}) ~: \]
the first equality by \eqref{e:2.10aa}, the second by the commutativity of 
\eqref{e:2.10b}, and the third since $\til\rho$ is defined by conjugation 
in $\call$.  Since $\varphi|_{H_0}$ is also defined to be conjugation,
	\[ \varphi(\varphi(h)(a)) = c_{\varphi(h)(a)} = 
	\varphi(h)\circ c_a \circ\varphi(h)^{-1} = 
	\varphi(h)\circ \varphi(a) \circ\varphi(h)^{-1} = 
	\varphi(hah^{-1})~. \]
This finishes the proof of \eqref{e:2.10d}.

We want to construct a group $G$ with $G_0\nsg{}G$, 
$G/G_0\cong\call/\call_0$, and $N_G(S_0)=H$.  To do this, first set 
$\Gamma=G_0\sd{}H$:  the semidirect product with the action of $H$ on $G_0$ 
given by $\varphi$ as defined in \eqref{e:2.10b}.  Elements of $\Gamma$ are 
written as pairs $(g,h)$ for $g\in{}G_0$ and $h\in{}H$.  Thus 
$(g,h)(g',h')=(g\cdot\varphi(h)(g'),hh')$.  Set
$N=\{(a,a^{-1})\,|\,a\in{}H_0\}$.  For $a,b\in{}H_0$, 
	\[ (a,a^{-1})(b,b^{-1}) = (a{\cdot}\varphi(a^{-1})(b),a^{-1}b^{-1}) 
	= (a{\cdot}a^{-1}ba,a^{-1}b^{-1}) = (ba,(ba)^{-1}) \in N, \]
where the second equality holds by \eqref{e:2.10d}.  Thus $N$ is a 
subgroup.  For $g\in{}G_0$ and $a\in{}H_0$, 
	\begin{align*}  
	(g,1)(a,a^{-1})(g,1)^{-1} &= (ga,a^{-1})(g^{-1},1) =
	(ga{\cdot}\varphi(a^{-1})(g^{-1}),a^{-1}) \\ 
	&= (ga{\cdot}a^{-1}g^{-1}a,a^{-1}) = (a,a^{-1}) ~;
	\end{align*}
where $\varphi(a^{-1})(g^{-1})=a^{-1}g^{-1}a$ since by construction, 
$\varphi|_{H_0}$ is the conjugation homomorphism of \eqref{e:2.10a}.  
Thus $(g,1)$ normalizes (centralizes) $N$.  For $h\in{}H$ and $a\in{}H_0$, 
	\[ (1,h)(a,a^{-1})(1,h)^{-1} = (\varphi(h)(a),ha^{-1})(1,h^{-1})
	= (\varphi(h)(a),(hah^{-1})^{-1})\in{}N \]
by \eqref{e:2.10d}, and thus $(1,h)$ also normalizes $N$.  This proves that 
$N\nsg\Gamma$.

Now set $G=\Gamma/N$, and regard $G_0$ and $H$ as subgroups of $G$.  By 
construction, $G=G_0H$, $G_0\cap{}H=H_0=N_{G_0}(S_0)$, $G_0\nsg{}G$, and 
$G/G_0\cong{}H/H_0\cong\call/\call_0$ (the last isomorphism by 
\eqref{e:2.10c}).  Also, $H\le{}N_G(S_0)$, and since 
$[H:N_{G_0}(S_0)]=[G:G_0]\ge[N_G(S_0):N_{G_0}(S_0)]$, we have 
$H=N_G(S_0)$.  The outer conjugation action of $G/G_0$ on $G_0$ is induced 
by $\varphi$.  So by \eqref{e:2.10b} and the definition of 
$\Aut(G_0,S_0)_{\widehat{\rho}}$, and since $\kappa$ sends 
$\Im(\widehat{\rho})$ isomorphically to $\Im(\rho_{\call_0}^\call)$ (since 
$\rho_{\call_0}^\call$ is injective), this outer action is equal to 
$\widehat{\rho}$ via our identification $G/G_0\cong\call/\call_0$.  

By comparison of \eqref{e:2.10a} and \eqref{e:2.10b}, we see that 
	\beq \Ker(\lambda)=\Ker(\lambda_0)=C'_{G_0}(S_0)~. \eeq
In particular, $\Ker(\lambda)$ has order prime to 
$p$.  Also, $\delta_{S_0}(S)$ is a 
Sylow $p$-subgroup of $\Aut_\call(S_0)$ by Proposition \ref{L-prop}(d).  
Fix any Sylow $p$-subgroup of $\lambda^{-1}(\delta_{S_0}(S))$, and 
identify it with $S$ via $\delta_{S_0}^{-1}\circ\lambda$.  Since 
$[G:H]=[G_0:H_0]$ is prime to $p$, we also have $S\in\sylp{G}$.  

\smallskip

\noindent\textbf{Step 2: }  Set $\calf'=\calf_S(G)$ 
for short.  By Proposition \ref{L(G0)<|L(G)}, 
$\calf_0=\calf_{S_0}(G_0)$ is normal in $\calf'$.  So by Lemma 
\ref{F0<|F}(d), $\calh=\Ob(\call)$ contains all subgroups of $S$ which are 
$\calf'$-centric and $\calf'$-radical.  

We next show that all subgroups in $\calh$ are $G$-quasicentric.  
Since overgroups of quasicentric subgroups are quasicentric, it suffices 
to prove this for $P\in\calh_0$.  Fix such $P$, and assume it is fully 
centralized in $\calf'$.  We must show that $O_{p'}(C_G(P))=O^p(C_G(P))$; 
i.e., that $C_G(P)$ contains a normal subgroup of order prime to $p$ and 
of $p$-power index.  Define 
	\[ \Phi_P\: N_G(P) \Right5{} \Aut_\call(P) \]
as follows.  Fix $g\in{}N_G(P)$, write $g=g_0h$ for some $g_0\in{}G_0$ and 
$h\in{}H=N_G(S_0)$, and set $\Phi_P(g)=[g_0]\circ\lambda(h)|_{P,hPh^{-1}}$, 
where $[g_0]\in\Mor_{\call_0}(hPh^{-1},P)$ is induced by the 
identification $\call_0=\call_{S_0}^{\calh_0}(G_0)$.  If 
$g=g_0h=g'_0h'$ where $g_0,g'_0\in{}G_0$ and $h,h'\in{}H$, then 
$a\defeq{}g_0^{-1}g'_0=hh'{}^{-1}\in{}H_0$, so $g'_0=g_0a$, $h=ah'$, and
	\[ [g_0]\circ\lambda(h)|_{P,hPh^{-1}} = 
	[g_0]\circ\lambda(a)|_{h'Ph'{}^{-1},hPh^{-1}}\circ
	\lambda(h')|_{P,h'Ph'{}^{-1}} = 
	[g_0a]\circ\lambda(h')|_{P,h'Ph'{}^{-1}}. \]
Thus $\Phi_P(g)$ is well defined, independently of the choice of $g_0$ and 
$h$, and $\Phi_{S_0}=\lambda$.  Moreover, $\Phi_P|_{N_S(P)} = \delta_P$, 
since $\lambda|_S = \delta_{S_0}\:S\Right2{}\Aut_\call(P)$ by the 
identification of $S$ as a subgroup of $H$. To see that $\Phi_P$ is a 
homomorphism, it suffices to check that 
	\beqq [hg_0h^{-1}] = \lambda(h)\circ [g_0] \circ \lambda(h)^{-1} 
	\label{e:2.10f} \eeqq
for each $g_0\in{}G_0$ and $h\in{}H$, and this follows from the 
commutativity of \eqref{e:2.10b}.  

We next claim that the composite 
$\pi_P\circ\Phi_P\:N_G(P)\Right2{}\autf(P)$ sends $g\in{}N_G(P)$ to 
$c_g\in\Aut(P)$.  Set $g=g_0h$ as above.  By definition of the linking 
system $\call_{S_0}^{\calh_0}(G_0)$, 
$\pi_P(\Phi_P(g_0))=\pi_P([g_0])=c_{g_0}$.  By \eqref{e:2.10f} and axiom 
(C) for the linking system $\call$, $\pi_{S_0}(\lambda(h))\in\autf(S_0)$ 
is conjugation by  $h$, and hence it is also conjugation by $h$ on 
$P\le{}S_0\nsg{}H$.  This proves the claim.  

Since $\calf_0\nsg\calf$, $\calf_0\nsg\calf'$, and $\Aut_\calf(S_0) =
\Aut_{\calf'}(S_0)$, the $\calf$- and $\calf'$-conjugacy classes of
any subgroup $Q\leq S_0$ are the same. It follows that $\calh$ is closed
under $\calf'$-conjugacy, and that $P$ is fully centralized in
$\calf$.  
Hence 
\begin{itemize}  
\item $\Ker[\Aut_\call(P)\Right2{\pi_P}\autf(P)]=\delta_P(C_S(P))$;
\item $\Phi_P|_{N_S(P)}=\delta_P$ is injective by Proposition \ref{L-prop}(c);
\item $\Ker(\pi_P\circ\Phi_P)=C_G(P)$ since $\pi_P\circ\Phi_P(g)=c_g$; and 
\item $C_S(P)\in\sylp{C_G(P)}$ by \cite[Proposition 1.3]{BLO2}.
\end{itemize}
Hence $\Ker(\Phi_P)$ is a normal subgroup of $C_G(P)$ of order prime to 
$p$, and $C_G(P)/\Ker(\Phi_P)\cong{}C_S(P)$ is a $p$-group.  It follows 
that $\Ker(\Phi_P)=O^p(C_G(P))$, and thus that $P$ is $G$-quasicentric. 

Set $\call'=\call_S^{\calh}(G)$.  We have now shown that $\calh$ satisfies 
the conditions which ensure that $\call'$ is a linking system associated 
to $\calf_S(G)$.  By Proposition \ref{L(G0)<|L(G)} again, $\call'$ 
contains $\call_0$ as a normal linking subsystem.  Also, 
$\Aut_{\call'}(S_0)=H/\Ker(\lambda)\cong\Aut_\call(S_0)$ since 
$O^p(C_G(S_0))=\Ker(\Phi_{S_0})=\Ker(\lambda)$, and they 
have the same action on $\call_0$ (under this identification) by the 
commutativity of \eqref{e:2.10b}.  

By the remarks at the end of \cite{O3} (after the proof of Theorem 9), the 
existence of the linking systems $\call$ and $\call'$ imply that conditions 
(2) and (3) in the statement of \cite[Theorem 9]{O3} hold.  So by the 
uniqueness statement in that theorem, $\calf=\calf'$ and 
$\call\cong\call'$.  
\end{proof}

\iffalse
Also, 
$\lambda(C_H(S_0))\le{}C_{\Aut_\call(S_0)}(\delta_{S_0}(S_0))=
\delta_{S_0}(C_S(S_0))$ by axioms (A) and (C), it is 
thus a $p$-group, and hence $O^p(C_H(S_0))=\Ker(\lambda)$ since 
$\Ker(\lambda)$ has order prime to $p$.  So 
$\Aut_{\call'}(S_0)=H/\Ker(\lambda)\cong\Aut_\call(S_0)$, 
\fi

In order to compare tameness in $\calf_0$ and in $\calf$ when 
$\SFL[_0]\nsg\SFL$, we need to compare the automorphisms of $\call_0$ 
with those of $\call$.  This is done in the following lemma.  For any 
normal pair $\call_0\nsg\call$ of linking systems, we set
	\begin{align*}  
	\Aut_\call(\call_0) &= \til\rho_{\call_0}^\call(\Aut_\call(S_0)) 
	= \{c_\gamma\,|\,\gamma\in\Aut_\call(S_0)\}\le\Aut\typ^I(\call_0) \\
	\Out_\call(\call_0) &= \rho_{\call_0}^\call(\call/\call_0) = 
	\Aut_\call(\call_0)/\Aut_{\call_0}(\call_0)\le\Out\typ(\call_0) ~. 
	\end{align*}

\begin{Lem} \label{norm-pb} 
Fix a pair of finite groups $G_0\nsg{}G$, let $S_0\nsg{}S$ be Sylow 
$p$-subgroups of $G_0\nsg{}G$, and set $\calf_0=\calf_{S_0}(G_0)$ and 
$\calf=\calf_S(G)$.  Assume $Z(G_0)=Z(\calf_0)$.  Let $\calh_0$ and 
$\calh$ be sets of subgroups such that 
	\[ \call_0\defeq\call_{S_0}^{\calh_0}(G_0)
	\qquad\textup{and}\qquad \call\defeq\call_{S}^{\calh}(G) \]
are linking systems associated to $\calf_0$ and $\calf$, respectively.  
Assume 
	\[ \call_0 \nsg \call~, \qquad \call_0 \textup{ is centric in } 
	\call~, \qquad\textup{and}\qquad 
	\call/\call_0\cong{}G/G_0~. \]
Assume also $\calh_0$ is $\Aut(S_0,\calf_0)$-invariant, and $\calh$ is 
$\Aut(S,\calf)$-invariant.  Then the following square
	\beqq \vcenter{\xymatrix@C=40pt{
	\Out(G,G_0) \ar[r]^-{\kappa} \ar[d]_{R_1} & 
	\Out\typ(\call,\call_0) \ar[d]_{R_2} \\
	N_{\Out(G_0)}(\Out_G(G_0))/\Out_G(G_0) \ar[r]^-{\kappa^*} & 
	N_{\Out\typ(\call_0)}(\Out_\call(\call_0))/\Out_\call(\call_0)
	}} \label{e:2.12a} \eeqq
is a pullback.  Here, $\Out(G,G_0)\le\Out(G)$ and 
$\Out\typ(\call,\call_0)\le\Out\typ(\call)$ are the subgroups of classes of 
automorphisms which leave $G_0$ and $\call_0$ invariant, respectively, 
$\kappa$ is the restriction of $\kappa_G^\calh$, $\kappa^*$ 
is induced by $\kappa_{G_0}^{\calh_0}$, and $R_1$ and $R_2$ are induced by 
restriction.
\end{Lem}

\begin{proof}  By the Frattini argument, $G=G_0{\cdot}N_G(S_0)$ (all 
subgroups $G$-conjugate to $S_0$ are $G_0$-conjugate to $S_0$).  Hence 
$G/G_0\cong{}N_G(S_0)/N_{G_0}(S_0)$, while 
	\[ \call/\call_0 \defeq \Aut_\call(S_0)/\Aut_{\call_0}(S_0) =
	\bigl(N_G(S_0)/O^p(C_G(S_0))\bigr)\big/ 
	\bigl(N_{G_0}(S_0)/C'_{G_0}(S_0) \bigr) \]
(and $S_0$ is $G$-quasicentric since it is an object of the linking system 
$\call=\call_S^\calh(G)$).  Since $G/G_0\cong\call/\call_0$, it follows 
that $O^p(C_G(S_0))=C'_{G_0}(S_0)$.  Also, for each 
$g\in{}C_G(G_0)\le{}N_G(S_0)$, $[g]\in\Aut_\call(S_0)$ acts trivially on 
$\call_0$ under conjugation, so $[g]\in\Aut_{\call_0}(S_0)$ since 
$\call_0$ is centric in $\call$, and hence $g\in{}G_0$.  We have now shown 
that 
	\beqq O^p(C_G(S_0))=C'_{G_0}(S_0) \qquad\textup{and}\qquad
	C_G(G_0)=Z(G_0)~. \label{e:2.12aa} \eeqq

\smallskip

\textbf{Step 1: }
We first show the following square is a pullback:
 	\beqq \vcenter{\xymatrix@C=40pt{
 	\Aut(G,G_0,S\cdot C'_{G_0}(S_0)) \ar[r]^-{\widetilde{\kappa}}
 	\ar[d]_{\Res_1} &
 	\Aut\typ^I(\call,\call_0) \ar[d]_{\Res_2} \\
 	N_{\Aut(G_0,S_0)}(\Aut_G(G_0,S_0)) \ar[r]^-{\widetilde{\kappa}_0} &
 	N_{\Aut\typ^I(\call_0)}(\Aut_\call(\call_0)) \rlap{~.}
 	}} \label{e:2.12b} \eeqq
Here, $\Aut(G,G_0,S\cdot{}C'_{G_0}(S_0))$ is the group of automorphisms of 
$G$ which send both $G_0$ and $S\cdot{}C'_{G_0}(S_0)$ to themselves
and $\Aut\typ^I(\call,\call_0)\leq \Aut\typ^I(\call)$ is the subgroup
of elements which leave $\call_0$ invariant.

Both $\Res_1$ and $\Res_2$ are defined by restriction. Each $\alpha\in 
\Aut(G,G_0,S\cdot C'_{G_0}(S_0))$ leaves $S_0\times 
C'_{G_0}(S_0) = G_0\cap (S\cdot C'_{G_0}(S_0))$ invariant, and hence also 
leaves $S_0$ invariant.  Clearly, $\alpha|_{\Aut(G_0,S_0)}$ normalizes 
$\Aut_G(G_0,S_0)$. To see that $\Res_2$ maps to the normalizer, fix
$\sigma\in\Aut\typ^I(\call,\call_0)$ and $\gamma\in\Aut_\call(S_0)$,
and set $\sigma_0=\sigma|_{\call_0}\in \Aut\typ^I(\call_0)$.  Then
 	\beqq \sigma_0{}c_\gamma\sigma_0^{-1}=c_{\sigma(\gamma)},
 	\label{e:2.12bb} \eeqq
(using Lemma \ref{AutI} to show this holds on objects), and thus
$\sigma_0$ normalizes $\Aut_\call(\call_0)$. 

The homomorphism $\widetilde{\kappa}_0$ is the restriction of
$\til\kappa_{G_0}^{\calh_0}$, which is defined since $\calh_0$ is
$\Aut(S_0,\calf_0)$-invariant.  Since $\til\kappa_{G_0}^{\calh_0}$ maps 
$\Aut_G(G_0,S_0)$ onto $\Aut_\call(\call_0)$, it sends the normalizer of 
$\Aut_G(G_0,S_0)$ into the normalizer of $\Aut_\call(\call_0)$.

Defining $\til\kappa$ requires more explanation. For
$\alpha\in\Aut(G,G_0,S\cdot{}C'_{G_0}(S_0))$, $\alpha(S)$ is a Sylow
$p$-subgroup of $S\cdot{}C'_{G_0}(S_0)$, so $\alpha(S)=hSh^{-1}$ for
some $h\in{}C'_{G_0}(S_0)$. Hence
$c_h^{-1}\circ\alpha\in\Aut(G,G_0,S)$ and we define
$\til\kappa(\alpha) = \til\kappa_G^\calh(c_h^{-1}\circ\alpha)\in
\Aut\typ^I(\call,\call_0)$. If $h'\in C'_{G_0}(S_0)$ with
$\alpha(S) = h' S {h'}^{-1}$, then $h^{-1} h'\in C'_{G_0}(S_0)\cap
N_G(S)$. Since $S_0$ is strongly closed in $\calf$, the
restriction homomorphism 
	\[ N_G(S)/C'_G(S) = \Aut_\call(S) \Right5{} 
	\Aut_\call(S_0) = N_G(S_0)/C'_{G_0}(S_0) \] 
is injective by Proposition \ref{L-prop}(f). It follows that $h^{-1} h'\in 
C'_G(S)$, so $\til\kappa_G^\calh(c_{h^{-1} h'}) = 1$ since $C'_G(S)\leq 
O^p(C_G(P))$ for each $P\leq S$. Thus $\til\kappa$ is well defined, and it 
is easily seen to be a homomorphism.  Since conjugation by any element of 
$C'_{G_0}(S_0)$ induces the identity in $\Aut\typ^I(\call_0)$ (and since 
$\Res_2\circ\til\kappa_G^\calh= \til\kappa_{G_0}^{\calh_0}\circ\Res_1$ as 
maps from $\Aut(G,G_0,S)$ to $\Aut\typ^I(\call_0)$), square 
\eqref{e:2.12b} commutes.

Next consider the following commutative diagram:
	\beqq \vcenter{\xymatrix@C=40pt{
	1 \ar[r] & Z(G_0) \ar[r] \ar@{=}[d] & N_G(S_0) \ar[r]^-{\cj_1}
	\ar@{->>}[d]_{\lambda_0}^{g\mapsto[g]} & \Aut_G(G_0,S_0) 
	\ar@{->>}[d]^{\til\kappa_1} \ar[r] & 1 \\
	1 \ar[r] & Z(\calf_0) \ar[r]^{\delta_{S_0}} & \Aut_\call(S_0) 
	\ar[r]^-{\cj_2} & \Aut_\call(\call_0) \ar[r] & 1 \rlap{~.}
	}} \label{e:2.12c} \eeqq
Here, $\cj_1$ and $\cj_2$ are induced by conjugation, and $\til\kappa_1$ 
is the restriction of $\til\kappa_0$.  Both rows in 
\eqref{e:2.12c} are exact:  the first since 
$\Ker(\cj_1)=C_G(G_0)=Z(G_0)$ by \eqref{e:2.12aa}; and the second since 
$\Ker(\cj_2)\le\Aut_{\call_0}(S_0)$ ($\call_0$ is centric in $\call$) and 
hence $\Ker(\cj_2)=Z(\calf_0)$ by Lemma \ref{OutI1}(a).  Thus the right 
hand square in \eqref{e:2.12c} is a pullback square.  

Fix automorphisms 
	\[ \alpha\in N_{\Aut(G_0,S_0)}(\Aut_G(G_0,S_0)) 
	\qquad\textup{and}\qquad
	\chi\in\Aut\typ^I(\call,\call_0) \] 
such that $\chi|_{\call_0}=\widetilde{\kappa}_0(\alpha)$.  Then 
$\chi(S_0)=S_0$, so $\chi_{S_0}$ is an automorphism of 
$\Aut_\call(S_0)=N_G(S_0)/C'_{G_0}(S_0)$ by \eqref{e:2.12aa}.  

We first construct $\beta\in\Aut(N_G(S_0))$ such 
that for each $g\in{}N_G(S_0)$, $c_{\beta(g)}=\alpha{}c_g\alpha^{-1}$ in 
$\Aut(G_0)$ and $\chi_{S_0}([g])=[\beta(g)]$ in $\Aut_\call(S_0)$.  
Consider the following automorphisms
	\[ c_\alpha\in\Aut\bigl(\Aut_G(G_0,S_0)\bigr), \qquad
	\chi_{S_0}\in\Aut\bigl(\Aut_\call(S_0)\bigr), \qquad
	c_{\til\kappa_0(\alpha)}=c_\chi
	\in\Aut\bigl(\Aut_\call(\call_0)\bigr) \]
of groups in the pullback square in \eqref{e:2.12c}.  We want to define 
$\beta$ as the pullback of $c_\alpha$ and $\chi_{S_0}$ over $c_\chi$.  
For $\gamma\in\Aut_\call(S_0)$, $c_\chi(\cj_2(\gamma))=
\chi c_\gamma\chi^{-1}= c_{\chi(\gamma)}=\cj_2(\chi_{S_0}(\gamma))$
(using \eqref{e:2.12bb}) and thus $\cj_2\circ\chi_{S_0}=c_\chi\circ\cj_2$.  
By a similar (but simpler) computation, $\til\kappa_1\circ{}c_\alpha= 
c_{\til\kappa_0(\alpha)}\circ\til\kappa_1$; and hence these three 
automorphisms pull back (via the pullback square in \eqref{e:2.12c}) to a 
unique $\beta\in\Aut(N_G(S_0))$.  Thus for $g\in{}N_G(S_0)$,
	\beqq   [\beta(g)]=\chi_{S_0}([g])\in\Aut_\call(S_0)
	\quad\textup{and}\quad
	\cj_1(\beta(g))=c_\alpha\circ\cj_1(g)=\alpha{}c_g\alpha^{-1}
	\in\Aut(G_0) ~.
	\label{e:2.12d} \eeqq

Now, $\chi_{S_0}(\delta_{S_0}(S_0))=\delta_{S_0}(S_0)$ and 
$\chi_{S_0}(\delta_{S_0}(S))=\delta_{S_0}(S)$ since $\chi$ is isotypical 
and sends inclusions to inclusions (and hence restrictions to restrictions).  
Since $\Aut_\call(S_0)=N_G(S_0)/C'_{G_0}(S_0)$ by \eqref{e:2.12aa}, 
\eqref{e:2.12d} implies that $\beta$ sends $S_0\times{}C'_{G_0}(S_0)$ to 
itself and sends $S\cdot{}C'_{G_0}(S_0)$ to itself.  In particular, 
$\beta(S_0)=S_0$.  

Now, for all $g\in{}N_{G_0}(S_0)$, 
	\beq \lambda_0(\alpha(g)) = [\alpha(g)]=\til\kappa_0(\alpha)([g])= 
	\chi_{S_0}([g])\in\Aut_\call(S_0) 
	\tag{$\til\kappa_0(\alpha)=\chi|_{\call_0}$} \eeq
and
	\[ \cj_1\circ\alpha(g)=c_{\alpha(g)} =\alpha{}c_g\alpha^{-1}
	\in\Aut_G(G_0,S_0)~. \]
Thus $\lambda_0(\alpha(g))=\lambda_0(\beta(g))$ and 
$\cj_1(\alpha(g))=\cj_1(\beta(g))$ by comparison with \eqref{e:2.12d}; and 
hence $\alpha(g)=\beta(g)$ by the pullback square in \eqref{e:2.12c}.  This 
proves that $\alpha|_{N_{G_0}(S_0)}=\beta|_{N_{G_0}(S_0)}$.  

We already saw that $G=G_0{\cdot}N_G(S_0)$.  Define 
$\widehat{\alpha}\in\Aut(G,G_0,S\cdot{}C'_{G_0}(S_0))$ by setting 
$\widehat{\alpha}(g_0h)=\alpha(g_0)\beta(h)$ for 
$g_0\in{}G_0$ and $h\in{}N_G(S_0)$.  Since 
$\alpha|_{N_{G_0}(S_0)}=\beta|_{N_{G_0}(S_0)}$, this is well defined as a 
bijective map of sets.  For all $g_0,g'_0\in{}G_0$ and $h,h'\in{}N_G(S_0)$,
	\begin{align*}  
	\widehat{\alpha}(g_0h \cdot g'_0h') &= \widehat{\alpha}(g_0\cdot 
	c_h(g'_0) \cdot hh') 
	= \alpha(g_0)\alpha(c_h(g'_0))\beta(hh') \\
	&= \alpha(g_0) c_{\beta(h)}(\alpha(g'_0)) \beta(hh') 
	= \alpha(g_0) \beta(h) \alpha(g'_0) \beta(h') 
	= \widehat{\alpha}(g_0h)\widehat{\alpha}(g'_0h'),
	\end{align*}
where the third equality follows from the condition 
$c_{\beta(h)}=\alpha{}c_h\alpha^{-1}$.  It now follows that
$\widehat{\alpha}\in\Aut(G,G_0)$.  Also, $\widehat{\alpha}$ sends 
$S\cdot{}C'_{G_0}(S_0)$ to itself since $\beta$ does.

By construction, $\Res_1(\widehat{\alpha}) = \widehat{\alpha}|_{G_0} = 
\alpha$. We claim that $\til\kappa(\widehat{\alpha}) = \chi$. Since 
$\widehat{\alpha}|_{G_0} = \alpha$ and $\chi|_{\call_0} = 
\til\kappa_0(\alpha)$, $\til\kappa(\widehat{\alpha})$ and $\chi$ define 
the same action on $\call_0$ (by the commutativity of \eqref{e:2.12b}). 
Choose $h\in C'_{G_0}(S_0)=O^p(C_G(S_0))$ with $\widehat{\alpha}(S) = 
hSh^{-1}$ and let $\tau\in\Aut(S)$ be given by 
$\tau(s)=c_h^{-1}(\widehat{\alpha}(s))$. For $g\in N_G(S_0)$, 
$\til\kappa(\widehat{\alpha})([g]) = [h^{-1}\widehat{\alpha}(g)h] = 
[\widehat{\alpha}(g)] = [\beta(g)] = \chi([g])$ in $\Aut_\call(S_0)$ by 
\eqref{e:2.12d} and since $\widehat{\alpha}|_{N_G(S_0)}=\beta$. Hence 
$\til\kappa(\widehat{\alpha})$ and $\chi$ define the same action on 
$\Aut_\call(S_0)$. Since $\call_0$ and $\Aut_\call(S_0)$ generate the  
full subcategory $\call|_{\le{}S_0}$, 
$\widetilde{\kappa}(\widehat{\alpha})$ and $\chi$ are equal after 
restriction to this subcategory. 

\iffalse
Now, $\til\kappa(\widehat{\alpha})$ and $\chi$ both send inclusions to 
inclusions, and hence commute with restriction of morphisms.  If 
$P,Q\in\calh$ are such that $\til\kappa(\widehat{\alpha})(P)=\chi(P)$ and 
$\til\kappa(\widehat{\alpha})(Q)=\chi(Q)$, then since the restriction map 
from $\Mor_\call(\chi(P),\chi(Q))$ to 
$\Mor_\call(\chi(P)\cap{}S_0,\chi(Q)\cap{}S_0)$ is injective by 
Proposition \ref{L-prop}(f), 
$\til\kappa(\widehat{\alpha})_{P,Q}=\chi_{P,Q}$.  In particular, 
$\til\kappa(\widehat{\alpha})_S=\chi_S$, and so 
$\til\kappa(\widehat{\alpha})(P)=\chi(P)$ for each $P\in\calh$ by Lemma 
\ref{AutI}.  Thus $\til\kappa(\widehat{\alpha})=\chi$.
\fi

We just showed that 
$\chi_{S_0}([s])=\til\kappa(\widehat{\alpha})_{S_0}([s])$ for $s\in{}S$.  
Hence $\chi_S([s]) = \til\kappa(\widehat{\alpha})_S([s])$ in 
$\Aut_\call(S)$. Lemma \ref{AutI} now implies that $\chi(P) = 
\til\kappa(\widehat{\alpha})(P)$ for $P\in \Ob(\call)$. Since both 
$\til\kappa(\widehat{\alpha})$ and $\chi$ send inclusions to inclusions, 
and since the restriction map from $\Mor_\call(P,Q)$ to 
$\Mor_\call(P\cap{}S_0,Q\cap{}S_0)$ is injective for all $P,Q\in\calh$ by 
Proposition \ref{L-prop}(f), it now follows that 
$\widetilde{\kappa}(\widehat{\alpha}) = \chi$.

To prove \eqref{e:2.12b} is a pullback, it remains to show 
$\til\kappa\times\Res_1$ is injective.  So assume 
$\widehat{\alpha}\in\Aut(G,G_0,S\cdot{}C'_{G_0}(S_0))$ is such that 
$\widehat{\alpha}|_{G_0}=\Id_{G_0}$ and 
$\widetilde{\kappa}(\widehat{\alpha})=\Id_\call$.  For each $g\in{}G$, 
$c_{\widehat{\alpha}(g)}=c_g\in\Aut(G_0)$, and hence 
$g^{-1}\widehat{\alpha}(g)\in{}C_G(G_0)=Z(G_0)$ by \eqref{e:2.12aa}.  
Since $\widetilde{\kappa}(\widehat{\alpha})=\Id_\call$, $\widehat{\alpha}$ 
induces the identity on $\Aut_\call(S_0)=N_G(S_0)/C'_{G_0}(S_0)$ (see 
\eqref{e:2.12aa} again).  Since $G=G_0{\cdot}N_G(S_0)$ and 
$\widehat{\alpha}|_{G_0}=\Id$, 
$g^{-1}\widehat{\alpha}(g)\in{}C'_{G_0}(S_0)$ for all $g\in{}G$.  Finally, 
$C'_{G_0}(S_0)\cap{}Z(G_0)=1$ because $Z(G_0)=Z(\calf_0)\le{}S_0$ is a 
$p$-group, and we conclude that $\widehat{\alpha}=\Id_G$.

\smallskip

\textbf{Step 2: }  We are now ready to prove \eqref{e:2.12a} is a pullback.  
Fix elements
	\[ [\alpha]\in{}N_{\Out(G_0)}(\Out_G(G_0))/\Out_G(G_0) 
	\qquad\textup{and}\qquad
	[\chi]\in\Out\typ(\call,\call_0) \]
such that $\kappa^*([\alpha])=R_2([\chi])$, and choose liftings 
$\alpha\in\Aut(G_0,S_0)$ and $\chi\in\Aut\typ^I(\call,\call_0)$.  Then 
$\alpha$ normalizes $\Aut_G(G_0)$, and hence also normalizes 
$\Aut_G(G_0,S_0)$.  

%%Also, $\Aut_S(G_0)\cong{}S/Z(G_0)$ is a 
%%Sylow $p$-subgroup of $\Aut_G(G_0,S_0)\cong{}N_G(S_0)/Z(G_0)$ (recall 
%%$C_G(G_0)=Z(G_0)$ by \eqref{e:2.12aa}), so 
%%$\alpha\Aut_S(G_0)\alpha^{-1}=\beta\Aut_S(G_0)\beta^{-1}$ for some 
%%$\beta\in\Aut_G(G_0,S_0)$.  Upon replacing $\alpha$ by $\beta^{-1}\alpha$ 
%%(without changing the class $[\alpha]$), we can assume $\alpha$ also 
%%normalizes $\Aut_S(G_0)$.

Since $\kappa^*([\alpha])=R_2([\chi])$,
$\chi|_{\call_0}=\widetilde{\kappa}_0(\alpha)\circ{}c_{[x]}$ for some 
element $x\in{}N_G(S_0)$ (where $[x]\in\Aut_\call(S_0)$ is the class of 
$x$).  Upon replacing $\alpha$ by $\alpha\circ{}c_x\in\Aut(G_0)$, we can 
arrange that $\chi|_{\call_0}=\widetilde{\kappa}_0(\alpha)$.  Hence 
$\alpha$ and $\chi$ pull back to an element of 
$\Aut(G,G_0,S\cdot{}C'_{G_0}(S_0))$ by Step 1, 
and so $[\alpha]$ and $[\chi]$ pull back to an element of $\Out(G,G_0)$.  

To see that this pullback is unique, fix $[\gamma]\in\Out(G,G_0)$ such 
that $R_1([\gamma])=1$ and $\kappa([\gamma])=1$, and choose 
$\gamma\in\Aut(G,G_0)$ which represents $[\gamma]$.  Then 
$\gamma(S)=gSg^{-1}$ for some $g\in{}G$, and upon replacing $\gamma$ by 
$c_g^{-1}\circ\gamma$, we can assume $\gamma(S)=S$.  Also, 
$\widetilde{\kappa}(\gamma)=c_{[y]}$ for some $y\in{}N_G(S)$; and upon 
replacing $\gamma$ by $\gamma\circ{}c_y^{-1}$, we can assume 
$\widetilde{\kappa}(\gamma)=\Id_\call$.  Now, $\gamma|_{G_0}=c_h$ for some 
$h\in{}N_G(S_0)$, and $c_{[h]}=\Id_{\call_0}$.  Hence $h\in{}G_0$ since 
$\call_0$ is centric in $\call$, and so 
$h\in{}C_{G_0}(S_0)=Z(S_0)\times{}C'_{G_0}(S_0)$.  

Write $h=h_1h_2$, where $h_1\in{}Z(S_0)$ and $h_2\in{}C'_{G_0}(S_0)$.  
Thus $[h]=[h_1]\in\Aut_{\call_0}(S_0)$, and $h_1\in{}Z(\calf_0)=Z(G_0)$ 
since $c_{[h]}=\Id_{\call_0}$ (see Lemma \ref{OutI1}(a)).  Thus 
$\gamma|_{G_0}=c_h=c_{h_2}$ in $\Aut(G_0)$.  Since 
$[S,h_2]\le[S,C'_{G_0}(S_0)]\le{}C'_{G_0}(S_0)$, 
$c_{h_2}\in\Aut(G,G_0,S{\cdot}C'_{G_0}(S_0))$.  Also, 
$\til\kappa(c_{h_2})=\Id$ by definition of $\til\kappa$ (and since 
$h_2\in{}C'_{G_0}(S_0)$).  Thus $\gamma=c_{h_2}$ since \eqref{e:2.12b} is a 
pullback, and so $[\gamma]=1$ in $\Out(G,G_0)$.  
\end{proof}

We are finally ready to prove:

\begin{Prop} \label{c:tame->tame}
Let $\SFL[_0]\nsg\SFL$ be a normal pair such that $\call_0$ is centric in 
$\call$, $\Ob(\call_0)$ and $\Ob(\call)$ are $\Aut(S_0,\calf_0)$- and 
$\Aut(S,\calf)$-invariant, respectively, and $\call_0$ is 
$\Aut\typ^I(\call)$-invariant.  Assume $\calf_0$ is tamely realized by 
some finite group $G_0$ such that $S_0\in\sylp{G_0}$, $Z(G_0)=Z(\calf_0)$, 
and $\call_0\cong\call_{S_0}^{\Ob(\call_0)}(G_0)$.  Then $\calf$ is tamely 
realized by a finite group $G$ such that $S\in\sylp{G}$, $G_0\nsg{}G$ and 
$G/G_0\cong\call/\call_0$.  
\end{Prop}

%%Set 
%%	\[ \calh_0=\{P\in\Ob(\call_0)\,|\,P \textup{ is 
%%	$\calf_0$-centric}\}
%%	\quad\textup{and}\quad 
%%	\calh=\{P\le{}S\,|\,P\cap{}S_0\in\calh_0\}\subseteq\Ob(\call)~. \]

\begin{proof}  Set $\calh=\Ob(\call)$ and $\calh_0=\Ob(\call_0)$.  
By assumption, $\calf_0=\calf_{S_0}(G_0)$, and 
$\kappa_{G_0}$ is split surjective.  
Also, $\call_0\cong\call_{S_0}^{\calh_0}(G_0)$ by assumption, and we 
identify these two linking systems.  By Lemma \ref{OutI2}, 
$\Out\typ(\call_0)\cong \Out\typ(\call_{S_0}^{\calh_0^c}(G_0)) 
\cong\Out\typ(\call_{S_0}^c(G_0))$, where $\calh_0^c$ is the set of 
$\calf_0$-centric subgroups in $\calh_0$.  Choose a splitting 
	\[ s\:\Out\typ(\call_0) \cong \Out\typ(\call_{S_0}^c(G_0))
	\Right5{}\Out(G_0) \]
for $\kappa_{G_0}^{\calh_0}$, and set
	\[ \widehat{\rho} = s \circ \rho_{\call_0}^\call \: \call/\call_0 
	\Right4{} \Out\typ(\call_0) \Right4{} \Out(G_0) ~. \]

By Lemma \ref{G0<|G}, there is a finite group $G$ such that 
$S\in\sylp{G}$, $G_0\nsg{}G$, $\calf=\calf_S(G)$, 
$\call\cong\call_S^{\calh}(G)$, $G/G_0\cong\call/\call_0$, and such that 
the outer action of $G/G_0$ on $G_0$ is equal to $\widehat{\rho}$ via this 
last isomorphism.  In particular, $s$ sends 
$\Out_\call(\call_0)=\Im(\rho_{\call_0}^\call)$ isomorphically to 
$\Out_G(G_0)=\Im(\widehat{\rho})$.  

Since $\call_0$ is $\Aut\typ^I(\call)$-invariant by assumption, 
$\Out\typ(\call,\call_0)=\Out\typ(\call)$.  So by Lemma \ref{norm-pb}, the 
following is a pullback square:
	\beqq \vcenter{\xymatrix@C=40pt{
	\Out(G,G_0) \ar[r]^-{\kappa} \ar[d]_{R_1} & 
	\Out\typ(\call) \ar[d]_{R_2} \\
	N_{\Out(G_0)}(\Out_G(G_0))/\Out_G(G_0) \ar[r]^-{\kappa^*} & 
	N_{\Out\typ(\call_0)}(\Out_\call(\call_0))/\Out_\call(\call_0)
	}} \label{e:2.13a} \eeqq
where $\kappa^*$ is induced by $\kappa_{G_0}^{\calh_0}$.  Since the 
splitting $s$ of $\kappa_{G_0}^{\calh_0}$ sends $\Out_\call(\call_0)$ 
isomorphically to $\Out_G(G_0)$, it induces a splitting $s^*$ of 
$\kappa^*$.  Since \eqref{e:2.13a} is a pullback, $s^*$ induces a 
splitting of $\kappa=\kappa_G^\calh|_{\Out(G,G_0)}$ (Lemma 
\ref{p.b.split}). By Lemma \ref{OutI2}, $\Out\typ(\call)\cong 
\Out\typ(\call_S^c(G))$, and so $\calf$ is tamely realized by $G$.
\end{proof}

We next turn to central extensions of fusion and linking systems.  In the 
following lemma, when $\call$ is a linking system associated to $\calf$ 
over the $p$-group $S$, and $A\le{}S$, we set 
	\[ \Aut\typ^I(\call,A)=\{\alpha\in\Aut\typ^I(\call) \,|\, 
	\alpha_S(\delta_S(A))=\delta_S(A) \}~, \]
and let $\Out\typ(\call,A)$ be its image in $\Out\typ(\call)$. 

%%\newpage

\begin{Lem} \label{cent-pb}
Fix a finite group $G$ and a central $p$-subgroup $A\le{}Z(G)$.  Choose 
$S\in\sylp{G}$, and set $\widebar{G}=G/A$ and 
$\widebar{S}=S/A\in\sylp{\widebar{G}}$.  Set $\calf=\calf_{S}(G)$,
$\widebar{\calf}=\calf_{\widebar{S}}(\widebar{G})$ and
 	\[ \calh = \{P\le S \,|\, P\ge A, \ P/A \textup{ is
 	$\widebar{\calf}$-centric} \}~. \]
Then $\calh$ contains all subgroups of $S$ which are 
$\calf$-centric and $\calf$-radical, all subgroups in $\calh$ are 
$\calf$-centric, and hence $\call\defeq\call_{S}^\calh(G)$ is a linking
system associated to $\calf$. If, furthermore, $O_{p'}(\widebar{G})=1$ and 
$Z(\widebar{G})=Z(\widebar{\calf})$, then the 
following square is a pullback:
 	\beqq \vcenter{\xymatrix@C=40pt{
 	\Out(G,A) \ar[d]^{\nu_1} \ar[r]^-{\kappa_{G,A}^\calh} &
 	\Out\typ(\call,A) \ar[d]^{\nu_2} \\
 	\Out(\widebar{G}) \ar[r]^-{\kappa_{\widebar{G}}} &
 	\Out\typ(\widebar{\call}) \rlap{~,}
 	}} \label{e:2.14a} \eeqq
where $\widebar{\call}=\call_{\widebar{S}}^c(\widebar{G})$, 
$\kappa_{G,A}^\calh$ is defined analogously to $\kappa_G$, and $\nu_1$
and $\nu_2$ are induced by the projections  $G\Onto2{}\widebar{G}$ and
$\call\Onto2{}\widebar{\call}$.
\end{Lem}

\begin{proof}  
We first prove the statements about $\calh=\Ob(\call)$.  If
$P\in\calh$, then $P$ is $\calf$-centric since $A\le{}P$ and $P/A$ is
$\widebar{\calf}$-centric (cf.\ \cite[Lemma 6.4(a)]{BCGLO2}).  Now
assume $P\le{}S$ is $\calf$-centric and $\calf$-radical; we must show
$P\in\calh$. It suffices to do this when $P$ is fully normalized in
$\calf$.  Since $P$ is $\calf$-centric, $A\le C_S(P)\le P$.  
For $x\in S$ with $xA\in{}C_{\widebar{S}}(P/A)$, $c_x$
induces the identity on $A$ and on $P/A$. Hence
$c_x\in{}O_p(\Aut_{\calf}(P))$ by Lemma \ref{mod-Fr}, so $x\in P$ by Lemma 
\ref{centrad}. This proves $C_{\widebar{S}}(P/A)\leq P/A$. By
Proposition \ref{F/Q}, $P/A$ is fully normalized and hence fully
centralized in $\widebar{\calf}$. We conclude that $P/A$ is
$\widebar{\calf}$-centric, so $P\in\calh$.

Consider the following diagram (with homomorphisms defined below):
	\beqq \vcenter{\xymatrix@C=40pt{
	\I221 \ar[r] & \I22\Hom(\widebar{G},A) \ar[r]^-{\lambda_1} 
	\ar[d]^{\tau}_{\cong} & 
	\I22\Aut(G,S,A) \ar[r]^-{(\til\nu_1,r_1)} 
	\ar[d]^{\til\kappa_1} 
	& \I22\Aut(\widebar{G},\widebar{S}) \times \Aut(A) 
	\ar[d]^{\til\kappa_2\times\Id} \\
	\I221 \ar[r] & \I22\Hom(\pi_1(|\widebar{\call}|),A) 
	\ar[r]^-{\lambda_2} & 
	\I22\Aut\typ^I(\call,A) \ar[r]^-{(\til\nu_2,r_2)} & 
	\I22\Aut\typ^I(\widebar{\call}) \times \Aut(A) \rlap{~.}
	}} \label{e:2.14b} \eeqq
Here, $\til\nu_1$ and $\til\nu_2$ are induced by the projection 
$G\Right2{}\widebar{G}$ and $r_1$ and $r_2$ by restriction to $A$, and 
$\Aut(G,S,A)\le\Aut(G)$ is the subgroup of automorphisms which leave both 
$S$ and $A$ invariant.  Also, $\til\kappa_1=\til\kappa_{G,A}^\calh$ 
(defined analogously to $\til\kappa_G$), and 
$\til\kappa_2=\til\kappa_{\widebar{G}}$.  The right hand square clearly 
commutes.  

For $\beta\in\Hom(\widebar{G},A)$ and $g\in{}G$, 
$\lambda_1(\beta)(g)=g{\cdot}\beta(gA)$.  For any morphism 
$\widebar\psi\in\Mor_{\widebar{\call}}(P,Q)$, let 
$[\widebar\psi]\in\pi_1(|\widebar{\call}|)$ be the class of the loop based 
at the vertex $\widebar{S}$, formed by the edges $\iota_P^{\widebar{S}}$, 
$\widebar\psi$, and $\iota_Q^{\widebar{S}}$ (in that order).  For 
$\beta\in\Hom(\pi_1(|\widebar{\call}|),A)$, $\lambda_2(\beta)$ is the 
automorphism of $\call$ which is the identity on objects, and sends 
$\psi\in\Mor_\call(P,Q)$ (with
image $\widebar{\psi}\in\Mor_{\widebar{\call}}(P/A,Q/A)$) to
$\psi\circ\delta_P(\beta([\widebar{\psi}]))$.  It follows immediately 
from these definitions that for $i=1,2$, $\lambda_i$ is injective and 
$(\til\nu_i,r_i)\circ\lambda_i$ is trivial. 

Since $A$ is a finite abelian $p$-group, $\Hom(\pi_1(X),A)\cong 
H^1(X;A)\cong H^1(X\pcom;A)$ for any ``$p$-good'' space $X$ (the second 
isomorphism by \cite[Definition I.5.1]{BK}).  Also, $|\widebar{\call}|$ is 
$p$-good by \cite[Proposition 1.12]{BLO2}, $B\widebar{G}$ is $p$-good 
since it has finite fundamental group (cf. \cite[Proposition 
VII.5.1]{BK}), and $B\widebar{G}\pcom\simeq|\widebar{\call}|\pcom$ by 
\cite[Proposition 1.1]{BLO1}.  We thus get an isomorphism
	\[ \tau\: \Hom(\widebar{G},A) \Right3{\cong} 
	H^1(B\widebar{G}\pcom;A) \Right3{\cong} 
	H^1(|\widebar{\call}|\pcom;A) \Right3{\cong}
	\Hom(\pi_1(|\widebar{\call}|),A) ~. \]
Alternatively, by \cite[Theorem B]{BCGLO2}, 
$\pi_1(|\widebar{\call}|)/O^p(\pi_1(|\widebar{\call}|))\cong 
\widebar{S}/\hyp(\widebar{\calf})$, where for an infinite group $\Gamma$, 
$O^p(\Gamma)$ denotes the intersection of all normal subgroups of $p$-power 
index.  By the hyperfocal subgroup theorem for groups 
\cite[\S\,1.1]{Puig-hyper}, $\widebar{G}/O^p(\widebar{G})\cong
\widebar{S}/\hyp(\widebar{\calf})$; and these isomorphisms induce an 
isomorphism
	\[ \tau\: \Hom(\widebar{G},A) \Right4{\cong}
	\Hom(\widebar{S}/\hyp(\widebar{\calf}),A) \Right4{\cong}
	\Hom(\pi_1(|\widebar{\call}|),A)~. \]
By either construction, $\tau$ makes the left hand square in 
\eqref{e:2.14b} commute.

\iffalse
\mynote{Keep both of these constructions of $\tau$??  If not, which?  Note 
that each uses a reference not needed elsewhere in the paper.  (Joana 
11/12:  not a strong preference; maybe a slight preference for the second.  
Kasper??}
\fi

An element $\alpha\in\Ker(\til\nu_1,r_1)$ is an automorphism of $G$ which 
induces the identity on $A$ and on $\widebar{G}=G/A$, and since 
$A\le{}Z(G)$, any such automorphism has the form 
$\alpha(g)=g\cdot\beta(gA)$ for some unique $\beta\in\Hom(\widebar{G},A)$. 
Thus the top row in \eqref{e:2.14b} is exact.  

Similarly, an element $\alpha\in\Ker(\til\nu_2,r_2)$ is an isotypical 
automorphism of $\call$ which sends inclusions to inclusions and 
induces the identity on $\widebar{\call}$ and on $A$.  Since 
$\call\Right2{}\widebar{\call}$ is bijective on objects (by definition), 
$\alpha$ induces the identity on objects in $\call$, and on morphisms 
it has the form $\alpha(\psi)=\psi\circ\beta(\widebar\psi)$ for some 
$\beta\:\Mor(\widebar{\call})\Right2{}A$ which preserves composition and sends 
inclusions to the identity.  Such a $\beta$ is equivalent to a 
homomorphism from $\pi_1(|\widebar{\call}|)$ to $A$ (cf. \cite[Proposition 
A.3(a)]{OV1}), so $\alpha=\lambda_2(\beta)$, and thus the second row in 
\eqref{e:2.14b} is exact.  

We are now ready to prove that \eqref{e:2.14a} is a pullback.  Fix 
automorphisms $\alpha\in\Aut(\widebar{G},\widebar{S})$ and 
$\beta\in\Aut\typ^I(\call,A)$ such that 
$\kappa_{\widebar{G}}([\alpha])=\nu_2([\beta])$.  Then 
$\til\nu_2(\beta)=\til\kappa_2(\alpha)\circ{}c_{[x]}$ for some 
$x\in{}N_{\widebar{G}}(\widebar{S})$ which induces 
$[x]\in\Aut_{\widebar{\call}}(\widebar{S})$.  So upon replacing $\alpha$ 
by $\alpha\circ{}c_x$, we can assume 
$\til\kappa_2(\alpha)=\til\nu_2(\beta)$.  Consider the following diagram:
	\beq \vcenter{\xymatrix@C=30pt{
	1 \ar[r] & \I21A \ar[r] \ar[d]_{\cong}^{r_2(\beta)} & \I21G 
	\ar[r] \ar@{-->}[d]^{\widehat\alpha} & \I21\widebar{G} \ar[r] 
	\ar[d]_{\cong}^{\alpha} & 1 \\
	1 \ar[r] & \I21A \ar[r] & \I21G 
	\ar[r] & \I21\widebar{G} \ar[r] & 1 \\
	}} \eeq
We want to find $\widehat{\alpha}\in\Aut(G)$ which makes the two squares 
commute.  This means showing that the class $[G]\in{}H^2(\widebar{G};A)$ 
is invariant under the automorphism of $H^2(\widebar{G};A)$ induced by 
$r_2(\beta)$ and $\alpha$.  But $\beta\in\Aut\typ^I(\call)$ induces an 
automorphism $\gamma=\beta_S|_{\delta_S(S)}\in\Aut(S,\calf)$ (see Lemma 
\ref{AutI}).  Also, $\gamma|_A=\beta_S|_A=r_2(\beta)$, $\gamma$ induces 
the automorphism $(\til\nu_2(\beta))_{\widebar{S}}|_{\widebar{S}}=
\alpha|_{\widebar{S}}$ on $\widebar{S}$, and thus 
$[S]\in{}H^2(\widebar{S};A)$ is invariant under these automorphisms of 
$\widebar{S}$ and $A$.  Since $H^2(\widebar{G};A)$ injects into 
$H^2(\widebar{S};A)$ under restriction, this proves that $[G]$ is also 
invariant, and hence that there is an automorphism 
$\widehat{\alpha}\in\Aut(G,S,A)$ as desired.  

Thus $(\til\nu_1,r_1)(\widehat{\alpha})=(\alpha,r_2(\beta))$.  By the 
commutativity of \eqref{e:2.14b}, 
	\[ (\til\nu_2,r_2)(\til\kappa_1(\widehat{\alpha})) = 
	(\til\kappa_2(\alpha),r_2(\beta))=(\til\nu_2,r_2)(\beta). \]
Hence there is $\chi\in\Hom(\widebar{G},A)$ such that 
$\lambda_2(\tau(\chi))=\til\kappa_1(\widehat{\alpha})^{-1}\circ\beta$, and 
the element $\widehat{\alpha}\circ\lambda_1(\chi)\in\Aut(G,S,A)$ pulls back 
$\alpha\in\Aut(\widebar{G},\widebar{S})$ and $\beta\in\Aut\typ^I(\call,A)$.  

This proves that $\Out(G)$ surjects onto the pullback in square 
\eqref{e:2.14a}.  To prove that it injects into the pullback, fix 
$\widehat{\alpha}\in\Aut(G,S,A)$ such that 
$\kappa_{G}^\calh([\widehat{\alpha}])=1$ and $\nu_1([\widehat{\alpha}])=1$.  
Upon composing $\widehat{\alpha}$ by an appropriate inner automorphism, we 
can assume it induces the identity on $\widebar{G}$.  Thus 
$\til\kappa_1(\widehat{\alpha})=c_{[x]}\in\Aut\typ^I(\call)$ for some 
$x\in{}N_{G}(S)$ inducing $[x]\in\Aut_{\call}(S)$, where $c_{[x]}$ induces the 
identity on $\widebar{\call}$.  This means that 
$xA\in{}Z(\widebar{\calf})$ (Lemma \ref{OutI1}(a)), and hence 
$xA\in{}Z(\widebar{G})$ by 
assumption.  So upon replacing $\widehat{\alpha}$ by 
$\widehat{\alpha}\circ{}c_x^{-1}\in\Aut(G)$ we have an automorphism which 
induces the identity on $\call$ and on $\widebar{G}$.  By the exactness of 
the rows in \eqref{e:2.14b} again, $\widehat{\alpha}=\Id$, and this 
finishes the proof.
\end{proof}

Lemma \ref{cent-pb} now implies the result we need about tameness.

\begin{Prop} \label{F/Z-tame}
Fix a saturated fusion system $\calf$ over a finite $p$-group $S$.  Assume 
$\calf/Z(\calf)$ is tamely realized by the finite group $\widebar{G}$ such 
that $O_{p'}(\widebar{G})=1$ and $Z(\widebar{G})=Z(\calf/Z(\calf))$.  
Then $\calf$ is tamely realized by a finite group $G$ such that 
$Z(G)=Z(\calf)$ and $G/Z(G)\cong\widebar{G}$, and hence 
$O_{p'}(G)=1$.  If $\widebar{G}\in\gpclass{p}$, then $G\in\gpclass{p}$.
\end{Prop}

\begin{proof}  Set $A=Z(\calf)$ and $\widebar{S}=S/A$ for short.  By 
assumption, $\widebar{S}\in\sylp{\widebar{G}}$, 
$\calf/A\cong\calf_{\widebar{S}}(\widebar{G})$, $\kappa_{\widebar{G}}$ is 
split surjective, $O_{p'}(\widebar{G})=1$, and 
$Z(\widebar{G})=Z(\calf/A)$.  

By \cite[Corollary 6.14]{BCGLO2}, the fusion system $\calf$ is realizable, 
and by the proof of that corollary, it is realizable by a finite group $G$ 
such that $S\in\sylp{G}$, $A\le{}Z(G)$, and $G/A\cong\widebar{G}$.  
Hence $O_{p'}(G)=1$, so $Z(G)$ is a $p$-group which is central in 
$\calf$. Thus $Z(G) = Z(\calf)$.

Let $\call\subseteq\call_S^c(G)$ be the full subcategory whose objects are 
the subgroups $P\le{}S$ such that $P\ge{}A$ and $P/A$ is 
$\calf/A$-centric, and set 
$\widebar{\call}=\call_{\widebar{S}}^c(\widebar{G})$.  Then $\call$ is a 
linking system associated to $\calf$ by Lemma \ref{cent-pb}, and 
$A=Z(\calf)$ is invariant under all automorphisms in $\Aut\typ^I(\call)$ 
by Lemma \ref{AutI}.  Lemma \ref{cent-pb} now implies that the following 
is a pullback square:
	\beq \vcenter{\xymatrix@C=40pt{
	\Out(G,A) \ar[r]^-{\kappa} \ar[d] & \Out\typ(\call) \ar[d] \\
	\Out(\widebar{G}) \ar[r]^-{\kappa_{\widebar{G}}} 
	& \Out\typ(\widebar{\call}) \rlap{~.}
	}} \eeq
By assumption, $\kappa_{\widebar{G}}$ is split surjective.  Hence 
$\kappa=\kappa_G^\calh|_{\Out(G,A)}$ ($\calh=\Ob(\call)$) is also split 
surjective by Lemma \ref{p.b.split}, so $\kappa_G^\calh$ is split surjective.  
Since 
$\Out\typ(\call)\cong\Out\typ(\call_S^c(G))$ by Lemma \ref{OutI2}, this 
finishes the proof that $\calf$ is tame.

By construction, $G$ and $\widebar{G}$ have the same nonabelian 
composition factors.  Hence $G\in\gpclass{p}$ if 
$\widebar{G}\in\gpclass{p}$.
\end{proof}

One more technical lemma is needed before we can prove Theorem \ref{ThA}.

\begin{Lem} \label{Op'-Z(G)}
Let $\calf$ be a saturated fusion system over a finite $p$-group $S$.  If 
$\calf$ is tame, then there is a finite group $G$ such that $O_{p'}(G)=1$ 
and $\calf$ is tamely realized by $G$.  
If $\calf$ is strongly tame, then $G$ can be chosen such 
that in addition, $G\in\gpclass{p}$.
\end{Lem}

\begin{proof}  Fix any $\widehat{G}$ which tamely realizes $\calf$.  If 
$\calf$ is strongly tame, we assume $\widehat{G}\in\gpclass{p}$.  Thus 
$S\in\sylp{\widehat{G}}$, $\calf\cong\calf_S(\widehat{G})$, and 
$\kappa_{\widehat{G}}$ is split surjective. Set 
$G=\widehat{G}/O_{p'}(\widehat{G})$, and identify $S$ with its image in 
$G$.  Since $G$ is a quotient group of $\widehat{G}$, $G\in\gpclass{p}$ if 
$\widehat{G}\in\gpclass{p}$.

By construction, $\calf_S(G)\cong\calf_S(\widehat{G})\cong\calf$, and 
$O_{p'}(G)=1$.  The natural homomorphism from $\widehat{G}$ onto $G$ 
induces a homomorphism between their outer automorphism groups and an 
isomorphism between their linking systems, and the resulting square 
	\beq \xymatrix@C=30pt{
	\Out(\widehat{G}) \ar[r] \ar[d]_-{\kappa_{\widehat{G}}} &
	\Out(G) \ar[d]^-{\kappa_{G}} \\
	\Out\typ(\call_S^c(\widehat{G})) \ar[r]^{\cong} & 
	\Out\typ(\call_S^c(G)) 
	} \eeq
commutes.  Since $\kappa_{\widehat{G}}$ is split surjective, so is 
$\kappa_{G}$.  
\end{proof}

We are now ready to prove Theorem \ref{ThA}.  Recall that 
$\red(\calf)$ denotes the reduction of a fusion system $\calf$ (see 
Definition \ref{d:reduced}).

\begin{Thm} \label{T:reduce}
For any saturated fusion system $\calf$ over a finite $p$-group $S$, if 
$\red(\calf)$ is strongly tame, then $\calf$ is tame.
\end{Thm}

\begin{proof}  Set $Q=O_p(\calf)$, $S_0=C_S(Q)/Z(Q)$, and 
$\calf_0=C_\calf(Q)/Z(Q)$.  Let $\red(\calf)=\calf_m\subseteq 
\calf_{m-1}\subseteq \cdots\subseteq\calf_0$ be a sequence of fusion 
subsystems, where for each $i$, $\calf_i=O^p(\calf_{i-1})$ or 
$\calf_i=O^{p'}(\calf_{i-1})$.  Let $S_m\nsg\cdots\nsg{}S_0$ be the 
corresponding sequence of $p$-groups:  each $\calf_i$ is a fusion system 
over $S_i$.  By Lemma \ref{l:red->red}, $O_p(\calf_i)=1$ for each 
$i$, and hence $Z(\calf_i)=1$ for each $i$. 

We first show inductively that each of the $\calf_i$ is strongly tame.  
Fix $1\le{}i\le{}m$, and assume $\calf_i$ is tamely realized by 
$G_i\in\gpclass{p}$.  By Lemma \ref{Op'-Z(G)}, we can assume 
$O_{p'}(G_i)=1$.  Thus $Z(G_i)$ is a $p$-group central in the fusion 
system $\calf_i$, and hence $Z(G_i)=1$ since $Z(\calf_i)=1$.  By 
Proposition \ref{F0/A=F(G)}(a,b), there is a centric linking system 
associated to $\calf_{i-1}$.  Hence by Proposition \ref{link-ex}(a,b), 
there are linking systems $\call_i\nsg\call_{i-1}$ associated to 
$\calf_i\nsg\calf_{i-1}$ such that $\call_i$ is a centric linking system 
(so $\Ob(\call_i)$ is $\Aut(S_i,\calf_i)$-invariant), $\Ob(\call_{i-1})$ 
is $\Aut(S_{i-1},\calf_{i-1})$-invariant, and $\call_i$ is 
$\Aut\typ^I(\call_{i-1})$-invariant.  Also, $\call_i$ is centric in 
$\call_{i-1}$ by Proposition \ref{link-ex}(a,b) again (and since 
$Z(\calf_{i-1})=1$).  By Lemma \ref{lim2=0}(c), 
$\call_i\cong\call_{S_i}^c(G_i)$.  The hypotheses of Proposition 
\ref{c:tame->tame} are thus satisfied, and hence $\calf_{i-1}$ is tamely 
realized by some $G_{i-1}$ such that $G_i\nsg{}G_{i-1}$ and 
$G_{i-1}/G_i\cong\call_{i-1}/\call_i$.  In particular, $G_{i-1}/G_i$ is 
$p$-solvable, and so $G_{i-1}\in\gpclass{p}$ by Lemma \ref{lim2=0}(b).  

Since $\calf_m$ was assumed to be tamely realized by some 
$G_m\in\gpclass{p}$, we now conclude that $\calf_0$ is tamely realized by 
$G_0\in\gpclass{p}$.  By Lemma \ref{Op'-Z(G)} again, we can assume 
$O_{p'}(G_0)=1$, and $Z(G_0)=1$ since $Z(\calf_0)=1$.  Next consider the 
saturated fusion system $\calf^*\defeq N_\calf^{\Inn(Q)}(Q)$
over $S^*\defeq{}Q{\cdot}C_S(Q)$.  
Since $\calf^*\nsg\calf$ by Proposition \ref{fus-ex}(c), $O_p(\calf^*)=Q$ by 
Lemma \ref{F0<|F}(e).  Let $Z(Q)=Z_1(Q)\le{}Z_2(Q)\le\cdots\le{}Q$ be the 
upper central series for $Q$.  Since $\Aut_{\calf^*}(Q)=\Inn(Q)$, 
$Z_{i+1}(Q)/Z_i(Q)$ is central in $\calf^*/Z_i(Q)$ for each $i$.  
Also, by repeated application of Proposition \ref{F/Q}, if 
$P/Z_i(Q)=Z(\calf^*/Z_i(Q))$, then $P\nsg\calf^*$, and hence $P\le{}Q$.  
Thus $Z(\calf^*/Z_i(Q))\le Z(Q/Z_i(Q))=Z_{i+1}(Q)/Z_i(Q)$, and these two 
subgroups are equal.

Thus $\calf_0=C_\calf(Q)/Z(Q)\cong\calf^*/Q$ is obtained from 
$\calf^*$ by sequentially dividing out by its center until the fusion 
system is centerfree.  By repeated application of Proposition 
\ref{F/Z-tame}, $\calf^*$ is tamely realizable by some finite group 
$G^*\in\gpclass{p}$ such that $O_{p'}(G^*)=1$ and $Z(G^*)=Z(\calf^*)$.  

By Proposition \ref{F0/A=F(G)}(c), there is a centric linking system 
associated to $\calf$.  Hence by Proposition \ref{link-ex}(c), there are 
linking systems $\call^*\nsg\call$ associated to $\calf^*\nsg\calf$, where 
all objects in $\call^*$ are $\calf^*$-centric, 
$\Ob(\call^*)$ is $\Aut(S^*,\calf^*)$-invariant, 
$\Ob(\call)$ is $\Aut(S,\calf)$-invariant, 
$\call^*$ is $\Aut\typ^I(\call)$-invariant, 
and $\call^*$ is centric in $\call$.  By Lemma \ref{lim2=0}(c) (and since 
$G^*\in\gpclass{p}$), $\call^*\cong\call_{S^*}^{\calh^*}(G^*)$.  Hence by 
Proposition \ref{c:tame->tame}, $\calf$ is tamely realized by a finite 
group $G$.
\end{proof}

%%%%%%%%%%%%%%%%%%%%%%%%%%%%%%%%%%%%

\newsect{Decomposing reduced fusion systems as products}
\label{s:factor}

If $\calf_1$ and $\calf_2$ are fusion systems over finite $p$-groups $S_1$ 
and $S_2$, respectively, then $\calf_1\times\calf_2$ is the fusion system 
over $S_1\times{}S_2$ defined as follows.  For all 
$P,Q\le{}S_1\times{}S_2$, if $P_i,Q_i\le{}S_i$ denote the images of $P$ 
and $Q$ under projection to $S_i$, then
	\[ \Hom_{\calf_1\times\calf_2}(P,Q) = \bigl\{(\varphi_1,\varphi_2)|_P 
	\,\big|\, \varphi_i\in\Hom_{\calf_i}(P_i,Q_i),\ 
	(\varphi_1,\varphi_2)(P)\le{}Q \bigr\}~. \]
Here, we regard $P$ and $Q$ as subgroups of $P_1\times{}P_2$ and 
$Q_1\times{}Q_2$, respectively.  
Thus $\calf_1\times\calf_2$ is the smallest fusion system over 
$S_1\times{}S_2$ for which 
	\[ \Hom_{\calf_1\times\calf_2}(P_1\times{}P_2,Q_1\times{}Q_2) = 
	\Hom_{\calf_1}(P_1,Q_1) \times \Hom_{\calf_2}(P_2,Q_2) \]
for each $P_i,Q_i\le{}S_i$.  By \cite[Lemma 1.5]{BLO2}, 
$\calf_1\times\calf_2$ is saturated if $\calf_1$ and $\calf_2$ are 
saturated.  We leave it as an easy exercise to check, for any pair of 
finite groups $G_1,G_2$ with Sylow subgroups $S_i\in\sylp{G_i}$, that
$\calf_{S_1\times{}S_2}(G_1\times{}G_2) 
=\calf_{S_1}(G_1)\times\calf_{S_2}(G_2)$.

We say that a nontrivial fusion system $\calf$ is \emph{indecomposable} if 
it has no decomposition as a product of fusion systems over nontrivial 
$p$-groups.  The main result in this section is Theorem \ref{ThC}:  every 
reduced fusion system has a unique decomposition as a product of reduced 
indecomposable fusion systems, and the product is tame if each of the 
indecomposable factors is tame.  The first statement will be proven as 
Proposition \ref{uniq-decomp}, and the second as Theorem \ref{prod-tame}.

We first prove the following easy lemma about fusion systems over products 
of finite $p$-groups.

\begin{Lem} \label{rad<S1xS2}
Let $S_1,S_2$ be a pair of finite $p$-groups, and set $S=S_1\times{}S_2$.  
For each subgroup $P\le{}S$ which does \emph{not} split as a product 
$P=P_1\times{}P_2$ for $P_i\le{}S_i$, there is $x\in{}N_S(P){\sminus}P$ 
such that $c_x\in{}O_p(\Aut(P))$.  Hence for each saturated fusion system 
$\calf$ over $S$, and each subgroup $P\le{}S$ which is $\calf$-centric and 
$\calf$-radical, $P=P_1\times{}P_2$ for some pair of subgroups 
$P_i\le{}S_i$.  
\end{Lem}

\begin{proof}  We prove the first statement; the last then follows by 
Lemma \ref{centrad}.  

Fix $P\le{}S$.  For $i=1,2$, let $P_i\le{}S_i$ be the image of $P$ under 
projection.  Thus $P\le{}P_1\times{}P_2$.  Let $Z_k(P)$ and 
$Z_k(P_i)$ be the $k$-th terms in the upper central series for $P$ and 
$P_i$; i.e., $Z_1(P)=Z(P)$ and $Z_{k+1}(P)/Z_k(P)=Z(P/Z_k(P))$.  We claim 
that for each $k$, 
	\beqq Z_k(P)=P\cap(Z_k(P_1)\times{}Z_k(P_2))~. \label{e:3.1a} \eeqq
This is clear for $k=1$: an element of $P$ is central only if it commutes 
with all elements in $P_1$ and all elements in $P_2$.  If \eqref{e:3.1a} 
holds for $k$, then $P/Z_k(P)$ can be identified as a subgroup of 
$(P_1/Z_k(P_1))\times(P_2/Z_k(P_2))$ (a subgroup which projects onto each 
factor), and the result for $Z_{k+1}(P)$ then follows immediately.

If $P\lneqq{}P_1\times{}P_2$, then choose 
$x\in{}N_{P_1\times{}P_2}(P){\sminus}P$ (see \cite[Theorem 2.1.6]{Sz1}).  
By \eqref{e:3.1a}, conjugation by $x$ acts via 
the identity on each quotient $Z_{k+1}(P)/Z_k(P)$.  So 
$c_x\in{}O_p(\Aut(P))$ by Lemma \ref{mod-Fr}.
\end{proof}

The next lemma gives some basic properties of product fusion systems.

\begin{Lem} \label{F'xF'<FxF}
Assume $\calf_1$ and $\calf_2$ are saturated fusion systems over finite 
$p$-groups $S_1$ and $S_2$.  For each $i=1,2$, let 
$\calf'_i\subseteq\calf_i$ be a saturated fusion subsystem over 
$S'_i\le{}S_i$.  
\begin{enuma}  
\item If $\calf'_i\nsg\calf_i$ for $i=1,2$, then $\calf'_1\times\calf'_2$ 
is normal in $\calf_1\times\calf_2$.

\item If $\calf'_i$ has index prime to $p$ in $\calf_i$ for $i=1,2$, then 
$\calf'_1\times\calf'_2$ has index prime to $p$ in $\calf_1\times\calf_2$.
\end{enuma}
\end{Lem}

\begin{proof}  Set $S=S_1\times{}S_2$, $S'=S'_1\times{}S'_2$, 
$\calf=\calf_1\times\calf_2$, and $\calf'=\calf'_1\times\calf'_2$.

\textbf{(a) } 
Since $S'_i$ is strongly closed in $\calf_i$, $S'$ is strongly closed in 
$\calf$.

Fix $P,Q\le{}S'$ and $\varphi\in\homf(P,Q)$.  Let $P_i,Q_i\le{}S'_i$ be the 
images of $P$ and $Q$ under projection to $S'_i$.  Then 
$\varphi=(\varphi_1,\varphi_2)|_P$ for some 
$\varphi_i\in\Hom_{\calf_i}(P_i,Q_i)$.  By condition (ii) in Definition 
\ref{d:F0<|F}, there are morphisms $\alpha_i\in\Aut_{\calf_i}(S'_i)$ and 
$\varphi'_i\in\Hom_{\calf'_i}(\alpha_i(P_i),Q_i)$ such that 
$\varphi_i=\varphi'_i\circ\alpha_i|_{P_i,\alpha_i(P_i)}$.  Set 
$\alpha=(\alpha_1,\alpha_2)\in\autf(S')$, and set 
$\varphi'=(\varphi'_1,\varphi'_2)|_{\alpha(P)}$.  Then 
$\varphi'(\alpha(P))\le{}Q$, so $\varphi'\in\Hom_{\calf'}(\alpha(P),Q)$ 
and $\varphi=\varphi'\circ\alpha|_{P,\alpha(P)}$.  This proves condition 
(ii) for the pair $\calf'\subseteq\calf$.

Let $P,Q\le{}S'$ and $P_i,Q_i\le{}S'_i$ be as above, and fix 
$\varphi=(\varphi_1,\varphi_2)|_P\in\Hom_{\calf'}(P,Q)$ and 
$\beta=(\beta_1,\beta_2)\in\autf(S')$.  Then 
$\beta_i\varphi_i\beta_i^{-1}\in\Hom_{\calf'_i}(\beta_i(P_i),\beta_i(Q_i))$
by condition (iii) for the normal pair $\calf'_i\nsg\calf_i$.  Also, 
$\beta\varphi\beta^{-1}(\beta(P))\le\beta(Q)$, and hence 
$\beta\varphi\beta^{-1}\in\Hom_{\calf'}(\beta(P),\beta(Q))$.  This proves 
condition (iii) for the pair $\calf'\subseteq\calf$, and finishes the proof 
that $\calf'$ is normal in $\calf$.

\smallskip

\noindent\textbf{(b) }  Note that $S'_i=S_i$, since $\calf'_i$ has index 
prime to $p$ in $\calf_i$.  Since $\calf'_i\supseteq{}O^{p'}(\calf_i)$, it 
suffices to prove this point when $\calf'_i=O^{p'}(\calf_i)$, and thus when 
$\calf'_i\nsg\calf_i$ (Proposition \ref{fus-ex}(b)).  Hence 
$\calf'_1\times\calf'_2$ is normal in $\calf_1\times\calf_2$ by (a).  
Since they are fusion systems over the same $p$-group, 
the result now follows by Lemma \ref{F0<|F,S0=S}.
\end{proof}

We next prove the following criterion for a reduced fusion system to 
decompose:  $\calf$ factors as a product of fusion subsystems whenever $S$ 
factors as a product of subgroups which are strongly closed in $\calf$.

\begin{Prop} \label{F=F1xF2}
Let $\calf$ be a saturated fusion system over a finite $p$-group 
$S=S_1\times\cdots\times{}S_m$, where $S_1,\ldots,S_m$ are all strongly 
closed in $\calf$.  Set $\calf_i=\calf|_{S_i}$ ($i=1,\ldots,m$):  the full 
subcategory of $\calf$ with objects the subgroups of $S_i$, regarded as a 
fusion system over $S_i$.  For each $i$, 
let $S_i^*=\prod_{j\ne{}i}S_j$, identify $S=S_i\times{}S_i^*$, and let 
$\calf'_i\subseteq\calf_i$ be the fusion subsystem over $S_i$ where for 
$P,Q\le{}S_i$, 
	\[ \Hom_{\calf'_i}(P,Q) = \bigl\{\varphi\in\Hom_{\calf_i}(P,Q) 
	\,\big|\, (\varphi,\Id_{S_i^*})\in
	\homf(P\times{}S_i^*,Q\times{}S_i^*) \bigr\}~. \]
Then $\calf'_i$ and $\calf_i$ are saturated fusion 
systems for each $i$, $O^{p'}(\calf_i)\subseteq\calf'_i$; and 
	\[ \calf'_1\times\cdots\times\calf'_m \subseteq \calf \subseteq 
	\calf_1 \times\cdots\times \calf_m~. \]
If $O^{p'}(\calf)=\calf$, then $\calf'_i=\calf_i$ for each $i$, and hence 
$\calf=\calf_1\times\cdots\times\calf_m$.
\end{Prop}

\begin{proof}  Fix $i\in\{1,\ldots,m\}$.  We first claim that
	\beqq \begin{split}  
	\textup{$\forall$ $P,Q\le S_i$ and $\varphi\in\Hom_{\calf_i}(P,Q)$, 
	there are $\psi\in\Aut_{\calf}(S_i^*)$ and 
	$\chi\in\Aut_{\calf_i}(S_i)$}& \\
	\textup{such that}\quad 
	(\varphi,\psi)\in\homf(P\times{}S_i^*,Q\times{}S_i^*) 
	\quad\textup{and}\quad
	\chi|_Q\circ\varphi\in\Hom_{\calf'_i}(P,S_i)~.& 
	\end{split} \label{e:3.2-a} \eeqq 
If $\varphi(P)$ is fully centralized in $\calf$, the existence of $\psi$ 
follows by the extension axiom, and since the $S_i$ are all strongly 
closed in $\calf$.  The general case then follows upon choosing 
$\alpha\in\isof(\varphi(P),R)$ where $R\le{}S_i$ is fully centralized in 
$\calf$, and applying the extension axiom to $\alpha\circ\varphi$ and to 
$\alpha$.  By the extension axiom again, this time applied to $\psi$, 
there is $\chi$ such that $(\chi^{-1},\psi)\in\autf(S)$, and hence 
$\chi|_Q\circ\varphi\in\Hom_{\calf'_i}(P,S_i)$.  This finishes the proof 
of \eqref{e:3.2-a}.

Two subgroups of $S_i$ are $\calf_i$-conjugate if and only 
if they are $\calf$-conjugate; and they cannot be $\calf$-conjugate to any 
other subgroups of $S$ since $S_i$ is strongly closed.  Also, for 
$P\le{}S_i$, $|N_S(P)|=|N_{S_i}(P)|\cdot|S_i^*|$ and 
$|C_S(P)|=|C_{S_i}(P)|\cdot|S_i^*|$.  Hence $P$ is fully normalized 
(centralized) in $\calf_i$ if and only if it is fully normalized 
(centralized) in $\calf$.  By \eqref{e:3.2-a}, $P,Q\le{}S_i$ are 
$\calf_i$-conjugate only if $P$ is $\calf'_i$-conjugate to a subgroup in 
the $\Aut_{\calf_i}(S_i)$-orbit of $Q$, and hence $P$ is fully normalized 
(centralized) in $\calf_i$ if and only if it is fully normalized 
(centralized) in $\calf'_i$.  Also, in the context of axiom (II), 
$N_\varphi^\calf=N_\varphi^{\calf_i}\times{}S^*_i$ for all 
$\varphi\in\Mor(\calf_i)$, and $N_\varphi^{\calf'_i}=N_\varphi^{\calf_i}$ 
for all $\varphi\in\Mor(\calf'_i)$.  Axioms (I) and (II) for $\calf_i$ and 
for $\calf'_i$ now follow easily from the same axioms applied to $\calf$; 
and thus $\calf_i$ and $\calf'_i$ are saturated.  

Fix $P\le{}S_i$, and choose $\varphi\in\Aut_{\calf_i}(P)$ and 
$\alpha\in\Aut_{\calf'_i}(P)$.  By \eqref{e:3.2-a}, there is 
$\psi\in\Aut_{\calf}(S_i^*)$ such that 
$(\varphi,\psi),(\alpha,\Id)\in\autf(P\times{}S_i^*)$.  Hence 
$(\varphi\alpha\varphi^{-1},\Id)\in\autf(P\times{}S_i^*)$, 
$\varphi\alpha\varphi^{-1}\in\Aut_{\calf'_i}(P)$, and so
$\Aut_{\calf'_i}(P)$ is normal in $\Aut_{\calf_i}(P)$.  When $P$ is fully 
normalized, $\Aut_{\calf'_i}(P)$ contains 
$\Aut_{S_i}(P)\in\sylp{\Aut_{\calf_i}(P)}$, and thus $\Aut_{\calf'_i}(P)\ge
O^{p'}(\Aut_{\calf_i}(P))$.  Hence $\calf'_i$ has index 
prime to $p$ in $\calf_i$ (see Definition \ref{d:p-p'-index}), and so
$\calf'_i\supseteq{}O^{p'}(\calf_i)$.  

Clearly, $\calf$ contains $\calf'_1\times\cdots\times\calf'_m$.  By 
Lemma \ref{rad<S1xS2} together with Alperin's fusion theorem (Theorem 
\ref{AFT}), each morphism in $\calf$ is a composite of restrictions 
of automorphisms of subgroups of the form $P_1\times\ldots\times{}P_m$ for 
$P_i\le{}S_i$.  Since the $S_i$ are strongly closed in $\calf$, each such 
automorphism has the form $(\varphi_1,\ldots,\varphi_m)$ for some 
$\varphi_i\in\autf(P_i)=\Aut_{\calf_i}(P_i)$.  Hence for arbitrary 
$P,Q\le{}S$, if $P_i,Q_i\le{}S_i$ denote the images of $P$ and $Q$ under 
projection, then each $\varphi\in\homf(P,Q)$ extends to some morphism 
$(\varphi_1,\ldots,\varphi_m)$ where $\varphi_i\in\homf(P_i,Q_i)$.  Since 
$\homf(P_i,Q_i)=\Hom_{\calf_i}(P_i,Q_i)$, this shows that 
$\calf\subseteq\calf_1\times\cdots\times\calf_m$.

Since $\calf'_i$ has index prime to $p$ in $\calf_i$ for each $i$, 
$\calf'_1\times\cdots\times\calf'_m$ has index prime to $p$ in 
$\calf_1\times\cdots\times\calf_m$ by Lemma \ref{F'xF'<FxF}(b), and hence 
has index prime to $p$ in $\calf$.  So if $O^{p'}(\calf)=\calf$, then 
$\calf=\calf'_1\times\cdots\times\calf'_m$; and $\calf_i=\calf'_i$ for 
each $i$ by definition of $\calf_i$.
\end{proof}

Note that if $\calf$ is any fusion system (saturated or not) over a 
finite $p$-group $S=S_1\times{}S_2$, and $\calf$ factors as a product of 
fusion systems over $S_1$ and $S_2$, then the factors must be the fusion 
subsystems $\calf_i=\calf'_i$ as defined in Proposition \ref{F=F1xF2}.  In 
other words, if there is any such factorization, it must be unique.

We next show that a product of reduced fusion systems is reduced.

\begin{Prop} \label{F1xF2-red}
Fix finite $p$-groups $S_1$ and $S_2$ and saturated fusion systems 
$\calf_i$ over $S_i$.  Set $\calf=\calf_1\times\calf_2$.  Then 
	\beq 
	O_p(\calf)=O_p(\calf_1)\times{}O_p(\calf_2), \quad
	O^p(\calf)=O^p(\calf_1)\times{}O^p(\calf_2), \quad
	O^{p'}(\calf)=O^{p'}(\calf_1)\times{}O^{p'}(\calf_2)~. \eeq
In particular, $\calf$ is reduced if and only if $\calf_1$ and $\calf_2$ 
are both reduced.
\end{Prop}

\begin{proof}  Set $S=S_1\times{}S_2$.  
The decomposition of $O_p(\calf)$ is clear:  if $P\le{}S$ is 
normal in $\calf$, then so are its projections into $S_1$ and $S_2$, 
and $P_i\nsg\calf_i$ implies $P_1\times P_2\nsg\calf$.

The relation ``of index prime to $p$'' among fusion systems is transitive 
(see Definition \ref{d:p-p'-index}), and hence 
$O^{p'}(O^{p'}(\calf))=O^{p'}(\calf)$.  So by Proposition \ref{F=F1xF2}, 
$O^{p'}(\calf)=\calf'_1\times\calf'_2$ for some pair of fusion systems 
$\calf'_i$ over $S_i$.  Also, 
$O^{p'}(\calf)\subseteq{}O^{p'}(\calf_1)\times{}O^{p'}(\calf_2)$ by Lemma 
\ref{F'xF'<FxF}(b), so $\calf'_i\subseteq{}O^{p'}(\calf_i)$, and 
$\calf'_i$ has index prime to $p$ in $\calf_i$ since 
$\calf'_1\times\calf'_2$ has index prime to $p$ in $\calf$.  Thus 
$\calf'_i=O^{p'}(\calf_i)$.

By definition,
	\[ \hyp(\calf) = \gen{s^{-1}\alpha(s) \,|\, 
	s\in{}P\le{}S,\ \alpha\in O^p(\autf(P))}
	= \hyp(\calf_1) \times \hyp(\calf_2)~. \]
Since $O^p(\calf)$ is the unique fusion subsystem over $\hyp(\calf)$ of 
$p$-power index in $\calf$ (Theorem \ref{t:O^p(F)}(a)), we have 
$O^p(\calf)=O^p(\calf_1)\times{}O^p(\calf_2)$.

The last statement is now immediate.
\end{proof}

By definition, every fusion system $\calf$ factors as a product of 
indecomposable fusion systems.  The following lemma is the key step when 
showing that this factorization is unique (not only up to isomorphism) 
when $\calf$ is reduced.

\begin{Lem} \label{dbl-decomp}
Let $\calf$ be a reduced fusion system over a finite $p$-group $S$.  Assume 
$\calf=\calf_1\times\calf_2=\calf_3\times\calf_4$, where each $\calf_i$ is 
a saturated fusion system over some $S_i\le{}S$.  Set 
$S_{ij}=S_i\cap{}S_j$ for $i=1,2$ and $j=3,4$.  Then 
$\calf=\calf_{13}\times{}\calf_{14}\times{}\calf_{23}\times{}\calf_{24}$, 
where $\calf_{ij}$ is a reduced fusion system over $S_{ij}$.
\end{Lem}

\begin{proof}  By assumption, the subgroups $S_i$ for $i\in\{1,2,3,4\}$ are 
all strongly closed in $\calf$, and $S=S_1\times{}S_2=S_3\times{}S_4$.  
Fix $x,y\in{}S_1$ which are $\calf$-conjugate, and choose 
$\varphi\in\homf(\gen{x},\gen{y})$ which sends $x$ to $y$.  Write 
$x=x_3x_4$ and $y=y_3y_4$, where $x_3,y_3\in{}S_3$ and $x_4,y_4\in{}S_4$.  
There are homomorphisms $\varphi_i\in\Hom_{\calf_i}(\gen{x_i},\gen{y_i})$ 
for $i=3,4$ which send $x_i$ to $y_i$, and such that $\varphi$ is the 
restriction of $(\varphi_3,\varphi_4)$.  Hence 
$(\varphi_3,\Id_{S_4})(x)=y_3x_4$, $y_3x_4\in{}S_1$ since $S_1$ is 
strongly closed, and thus $x_3^{-1}y_3\in{}S_{13}$.  By a similar 
argument, $x_4^{-1}y_4\in{}S_{14}$, and thus 
$x^{-1}y\in{}S_{13}\times{}S_{14}$.  This proves that 
$\foc(\calf_1)\le{}S_{13}\times{}S_{14}$.

By a similar argument, $\foc(\calf_2)\le{}S_{23}\times{}S_{24}$.  Since 
$\calf=\calf_1\times\calf_2$, it follows that
	\[ \foc(\calf) = \foc(\calf_1)\times\foc(\calf_2) \le 
	S_{13}\times{}S_{14}\times{}S_{23}\times{}S_{24} \le S ~. \]
Also, $\foc(\calf)=S$ since $\calf$ is reduced (Theorem 
\ref{t:O^p(F)}(a)), so $S$ is the product of the $S_{ij}$.  Since the 
intersection of two subgroups which are strongly closed in $\calf$ is 
strongly closed in $\calf$, $\calf$ splits as a product of reduced fusion 
systems $\calf_{ij}$ over $S_{ij}$ by Propositions \ref{F=F1xF2} and 
\ref{F1xF2-red} (recall $O^{p'}(\calf)=\calf$ since $\calf$ is reduced). 
\end{proof}

This now implies the uniqueness of any decomposition of a reduced fusion 
system as a product of indecomposables.

\begin{Prop} \label{uniq-decomp}
Each reduced fusion system $\calf$ over a finite $p$-group $S$ has a 
unique factorization $\calf=\calf_1\times\cdots\times\calf_m$ as a product 
of indecomposable fusion systems $\calf_i$ over subgroups $S_i\nsg{}S$.  
Moreover, the $\calf_i$ are all reduced, and each fusion preserving 
automorphism $\alpha\in\Aut(S,\calf)$ permutes the factors $S_i$.
\end{Prop}

\begin{proof}  Let $\calf=\calf_1\times\cdots\times\calf_m 
=\calf'_1\times\cdots\times\calf'_n$ be two decompositions as products of 
indecomposable fusion systems.  By Lemma \ref{dbl-decomp} applied to the 
decompositions $\calf=\calf_1\times\prod_{i\ge2}\calf_i
=\calf'_1\times\prod_{i\ge2}\calf'_i$, and since $\calf_1$ and $\calf'_1$ 
are indecomposable, either $\calf_1=\calf'_1$ and 
$\prod_{i\ge2}\calf_i=\prod_{i\ge2}\calf'_i$, or 
$\calf_1$ is a direct factor in $\prod_{i\ge2}\calf'_i$.  In the latter 
case, we can assume by induction on $|S|$ that the decomposition of 
$\prod_{i\ge2}\calf'_i$ is unique, 
and hence that for some $j$, $\calf_1=\calf'_j$ and so 
$\prod_{i\ne1}\calf_i=\prod_{i\ne{}j}\calf'_i$ (Lemma \ref{dbl-decomp} 
again).  By the same induction 
hypothesis, this proves that the two decompositions are equal up to 
permutation of the factors.  The factors $\calf_i$ are all reduced by 
Proposition \ref{F1xF2-red}.

Fix $\alpha\in\Aut(S,\calf)$.  Since $S=\prod_{i=1}^m\alpha(S_i)$ is a 
product of subgroups which are strongly closed in $\calf$, $\calf$ factors 
as a product of saturated fusion systems over the $\alpha(S_i)$ by 
Proposition \ref{F=F1xF2} (and since $O^{p'}(\calf)=\calf$).  So $\alpha$ 
permutes the factors $S_i$ by the uniqueness of the decomposition.
\end{proof}

We are now ready to prove that a product of reduced, 
\emph{indecomposable}, tame fusion systems is tame.  Together with Theorem 
\ref{T:reduce}, this shows that any ``minimal'' exotic fusion system is 
indecomposable as well as reduced.

\begin{Thm} \label{prod-tame}
Fix a reduced fusion system $\calf$ over a finite $p$-group $S$, and let 
$\calf=\calf_1\times\cdots\times\calf_m$ be its unique factorization as a 
product of indecomposable fusion systems.  If $\calf_i$ is tame (strongly 
tame) for each $i$, then $\calf$ is tame (strongly tame).
\end{Thm}

\begin{proof}  Let $S=S_1\times\cdots\times{}S_m$ be the corresponding 
decomposition of $p$-groups; i.e., $\calf_i$ is a fusion system over 
$S_i$. Assume each $\calf_i$ is tame, and let $G_i$ be a finite group 
which tamely realizes $\calf_i$.  Assume also that these are chosen so 
that $G_i\cong{}G_j$ if $\calf_i\cong\calf_j$.  Set 
$\call_i=\call_{S_i}^c(G_i)$.  Set $G=G_1\times\cdots\times{}G_m$, 
$\call=\call_S^c(G)$, and 
$\widehat{\call}=\call_1\times\cdots\times\call_m$.  We identify 
$\widehat{\call}$ with the full subcategory of $\call$ having as objects 
those $P=P_1\times\cdots\times{}P_m$ where $P_i\in\Ob(\call_i)$.  Note that 
$\widehat{\call}$ is not a linking system, since $\Ob(\widehat{\call})$ is 
not closed under overgroups.

Set $\mm=\{1,\ldots,m\}$.  Define 
	\[ \Aut\typ^0(\call) = \bigl\{ \alpha\in\Aut\typ^I(\call) \,\big|\, 
	\alpha_S(\delta_S(S_i))=\delta_S(S_i) \textup{ for each $i\in\mm$} 
	\bigr\}~. \]
We first construct a monomorphism
	\beq \Psi\: \Aut\typ^0(\call) \Right4{} 
	\Aut\typ^I(\call_1) \times \cdots 
	\times \Aut\typ^I(\call_m) \eeq
such that for each $\alpha\in\Aut\typ^0(\call)$, if 
$\Psi(\alpha)=(\alpha_1,\dots,\alpha_m)$, then 
$\alpha|_{\widehat{\call}}=\prod_{i\in\mm}\alpha_i$. 

To define $\Psi$, fix $\alpha\in\Aut\typ^0(\call)$, and let 
$\beta\in\Aut(S,\calf)$ be the induced automorphism of Lemma \ref{AutI} 
(i.e., $\delta_S(\beta(g))=\alpha(\delta_S(g))$ for $g\in{}S$).  Then 
$\beta(S_i)=S_i$ for each $i$ since $\delta_S$ is injective.  Also, by 
Lemma \ref{AutI}, $\alpha(P)=\beta(P)$ for each $P\in\Ob(\call)$, and 
$\pi\circ\alpha=c_\beta\circ\pi$, where $c_\beta\in\Aut(\calf)$ is 
conjugation by $\beta$ (and its restrictions).  

Fix $i\in\mm$, set $S_i^*=\prod_{j\ne{}i}S_j$ and 
$\call_i^*=\prod_{j\ne{}i}\call_j$, and identify $S=S_i\times{}S_i^*$ and 
$\widehat{\call}=\call_i\times\call_i^*$.  We claim the following:  
	\beqq \textup{$\forall$ $\psi\in\Mor(\call_i)$, \quad$\exists$ 
	$\alpha_i(\psi)\in\Mor(\call_i)$ such that 
	$\alpha(\psi,\Id_{S_i^*})=(\alpha_i(\psi),\Id_{S_i^*})$.} 
	\label{e:3.6b} \eeqq
For each $\psi\in\Mor(\call_i)$, 
	\[ \pi(\alpha(\psi,\Id_{S_i^*}))=c_\beta(\pi(\psi),\Id_{S_i^*})
	= (c_\beta(\pi(\psi)),\Id_{S_i^*}) \in \Mor(\calf) \]
since $\beta(S_j)=S_j$ for all $j$.  Hence by axiom (A), 
$\alpha(\psi,\Id_{S_i^*})=(\alpha_i(\psi),\delta_{S_i^*}(z))$ for some 
$\alpha_i(\psi)\in\Mor(\call_i)$ and some $z\in{}Z(S_i^*)$.  In 
particular, \eqref{e:3.6b} holds when $\psi$ is an automorphism of order 
prime to $p$.  Since $\alpha(\delta_P(x))=\delta_{\beta(P)}(\beta(x))$ for 
all $P\in\Ob(\call)$ and all $x\in{}N_S(P)$ (and since $\beta(S_i)=S_i$), 
\eqref{e:3.6b} also holds when $\psi=\delta_P(x)$ for $P\le{}S_i$ and 
$x\in{}N_{S_i}(P)$.  When $P\in\Ob(\call_i)$ is fully normalized in 
$\calf_i$, $\Aut_{\call_i}(P)$ is generated by elements of order prime to 
$p$ and by its Sylow $p$-subgroup $\delta_P(N_{S_i}(P))$ (Proposition 
\ref{L-prop}(d)), and hence \eqref{e:3.6b} holds for all 
$\psi\in\Aut_{\call_i}(P)$.  Finally, by Theorem \ref{AFT-L}, all morphisms 
in $\call_i$ are composites of restrictions of automorphisms of fully 
normalized subgroups, and hence \eqref{e:3.6b} holds for all 
$\psi\in\Mor(\call_i)$.  

Now let $\alpha_i\in\Aut(\call_i)$ be the automorphism defined by sending 
$P\in\Ob(\call_i)$ to $\beta(P)$, and $\psi\in\Mor(\call_i)$ to 
$\alpha_i(\psi)$ as defined in \eqref{e:3.6b}.  This is clearly a functor, 
it is isotypical since $\alpha$ is, and it preserves inclusions since 
$\alpha$ does.  Set $\Psi(\alpha)=(\alpha_1,\dots,\alpha_m)$.  Since each 
morphism in $\widehat{\call}$ is a composite of restrictions of morphisms 
of the form $(\psi_i,\Id_{S^*_i})$ for $\psi_i\in\Mor(\call_i)$, the 
restriction of $\alpha$ to $\widehat{\call}$ is $\prod_{i\in\mm}\alpha_i$.  

By construction, $\Psi$ is a homomorphism.  If 
$\Psi(\alpha)=(\Id_{\call_1},\ldots,\Id_{\call_m})$, then 
$\alpha|_{\widehat{\call}}=\Id$ by the above remarks, $\alpha$ is the 
identity on objects since $\alpha_S=\Id_{\Aut_\call(S)}$ (Lemma 
\ref{AutI}), and so $\alpha=\Id_\call$ by Theorem \ref{AFT-L} and 
since all $\calf$-centric $\calf$-radical subgroups are 
objects in $\widehat{\call}$
(Lemma \ref{rad<S1xS2}).  Hence $\Psi$ is injective.  Finally, since 
$\Aut_\call(S)\cong\prod_{i\in\mm}\Aut_{\call_i}(S_i)$, $\Psi$ induces a 
monomorphism
	\[ \widebar{\Psi}\: \Out\typ^0(\call) \defeq 
	\Aut\typ^0(\call)/\{c_\zeta\,|\,\zeta\in\Aut_\call(S)\} 
	\Right4{} \Out\typ(\call_1) \times \cdots \times \Out\typ(\call_m) 
	~. \]

Next consider the equivalence relation $\sim$ on $\mm$, where $i\sim{}j$ 
if $G_i\cong{}G_j$ (equivalently, $\calf_i\cong\calf_j$).  Fix 
isomorphisms $\tau_{ij}\in\Iso(G_i,G_j)$ for all pairs $i\sim{}j$ of 
elements in $\mm$, such that $\tau_{ij}(S_i)=S_j$, $\tau_{ii}=\Id_{G_i}$, 
$\tau_{ji}=\tau_{ij}^{-1}$, and $\tau_{ik}=\tau_{jk}\circ\tau_{ij}$ 
whenever $i\sim{}j\sim{}k$.  Let 
$\widehat{\tau}_{ij}\:\call_i\Right2{\cong}\call_j$ be the induced 
isomorphism of linking systems.  Then conjugation by $\widehat{\tau}_{ij}$ 
sends $\Out\typ(\call_i)$ to $\Out\typ(\call_j)$.  For each $i$, fix a 
splitting $s_i\:\Out\typ(\call_i)\Right2{}\Out(G_i)$ of $\kappa_{G_i}$, 
chosen so that $c_{\tau_{ij}}\circ{}s_i= 
s_j\circ{}c_{\widehat{\tau}_{ij}}$ if $i\sim{}j$. 

Let $\widebar{\Sigma}\le\Sigma_m$ be the group of permutations $\sigma$ of 
$\mm$ such that $\sigma(i)\sim{}i$ for each $i$.  For each 
$\sigma\in\widebar{\Sigma}$, let 
$\widehat{\sigma}_G\in\Aut(G)$ be the automorphism which sends $G_i$ to 
$G_{\sigma(i)}$ via $\tau_{i,\sigma(i)}$, and set 
$\widehat{\sigma}_\call=\til\kappa_G(\widehat{\sigma}_G)$.  Thus 
$\widehat{\sigma}_\call\in\Aut\typ^I(\call)$ sends each $\call_i$ to 
$\call_{\sigma(i)}$ via $\widehat{\tau}_{i,\sigma(i)}$.  

Fix $\alpha\in\Aut\typ^I(\call)$, and let $\beta\in\Aut(S,\calf)$ be the 
restriction of $\alpha_S\in\Aut(\Aut_\call(S))$ to $S\cong\delta_S(S)$.  
By Proposition \ref{uniq-decomp}, there is $\sigma\in\Sigma_m$ such that 
$\beta(S_i)=S_{\sigma(i)}$ for each $i$.  Since $\beta$ is fusion 
preserving, $\calf_i\cong\calf_{\sigma(i)}$, and hence $i\sim\sigma(i)$, 
for each $i$.  Thus $\sigma\in\widebar{\Sigma}$, and 
$\widehat{\sigma}_\call^{-1}\circ\alpha\in\Aut\typ^0(\call)$.  So 
$\Aut\typ^I(\call)$ is generated by $\Aut\typ^0(\call)$ and the 
$\widehat{\sigma}_\call$.  

Now let $s\:\Out\typ(\call)\Right2{}\Out(G)$ be the composite
	\begin{align*}  
	\Out\typ(\call) &= \Out\typ^0(\call) \sd{} 
	\{[\widehat{\sigma}_\call] \,|\, \sigma\in\widebar{\Sigma} \} 
	\RIGHT4{\widebar{\Psi}\sd{}}
	{([\widehat{\sigma}_\call]\mapsto\sigma)}
	\bigl( \Out\typ(\call_1) \times \cdots 
	\times \Out\typ(\call_m) \bigr) \sd{} \widebar{\Sigma} \\
	&\RIGHT7{(s_1,\dots,s_m)\sd{}}{(\sigma\mapsto[\widehat{\sigma}_G])}
	\bigl( \Out(G_1) \times \cdots \times \Out(G_m) \bigr) \sd{} 
	\{[\widehat{\sigma}_G] \,|\, \sigma\in\widebar{\Sigma} \} 
	\Right4{\incl} \Out(G) ~.
	\end{align*}
We must show $\kappa_G\circ{}s=\Id$.  Since 
$\kappa_G(s([\widehat{\sigma}_\call]))=\kappa_G([\widehat{\sigma}_G])
=[\widehat{\sigma}_\call]$ for $\sigma\in\widebar{\Sigma}$, it will suffice 
to show $\kappa_G(s([\alpha]))=[\alpha]$ for $\alpha\in\Aut\typ^0(\call)$.  
Let $\Out^0(G)\le\Out(G)$ be the subgroup of classes of automorphisms which 
leave each $G_i$ invariant, and consider the following composite:
	\begin{multline*}  
	\Out\typ^0(\call) \Right4{\widebar{\Psi}} 
	\prod_{i=1}^m\Out\typ(\call_i) \Right4{\prod s_i} 
	\prod_{i=1}^m\Out(G_i)\cong\Out^0(G) \\
	\Right8{\kappa_G|_{\Out^0(G)}} 
	\Out\typ^0(\call) \Right4{\widebar{\Psi}} 
	\prod_{i=1}^m\Out\typ(\call_i)~. 
	\end{multline*}
Here, $(\prod{}s_i)\circ\widebar{\Psi}=s|_{\Out\typ^0(\call)}$, 
$\widebar{\Psi}\circ\kappa_G|_{\Out^0(G)}=\prod\kappa_{G_i}$, and 
$(\prod\kappa_{G_i})\circ(\prod{}s_i)=\Id$.  This proves that 
$\widebar{\Psi}\circ\kappa_G|_{\Out^0(G)}\circ{}s|_{\Out\typ^0(\call)}
=\widebar{\Psi}$. 
Since $\widebar{\Psi}$ is injective, $\kappa_G\circ{}s=\Id$ on 
$\Out\typ^0(\call)$.  Thus $s$ is a splitting for $\kappa_G$, and this 
finishes the proof that $\calf$ is tame.

If each $\calf_i$ is strongly tame, then we can choose the $G_i$ to all be 
in the class $\gpclass{p}$.  Hence $G\in\gpclass{p}$ by Lemma 
\ref{lim2=0}(b), and $\calf$ is strongly tame.
\end{proof}

Theorem \ref{prod-tame} does \emph{not} say that an arbitrary product of 
reduced, tame fusion systems is tame:  such a product could conceivably 
have an indecomposable factor which is not tame.  However, at least when 
$p=2$, a theorem of Goldschmidt implies this is not possible.

\begin{Thm} 
Assume $p=2$, and let $\calf$ be a reduced fusion system over a $2$-group 
$S$. Assume $\calf=\calf_1\times\calf_2$, where $\calf_i$ is a fusion 
system over $S_i$ and $S=S_1\times{}S_2$.  Then $\calf$ is realizable, 
tame, or strongly tame if and only if $\calf_1$ and $\calf_2$ are both 
realizable, tame, or strongly tame, respectively.
\end{Thm}

\begin{proof}  Assume $\calf=\calf_S(G)$, where $G$ is a finite group and 
$S\in\syl2{G}$.  If $\calf$ is tame, we also assume $\kappa_G$ is split 
surjective, and if $\calf$ is strongly tame, we also assume 
$G\in\gpclass{2}$.  By Lemma \ref{Op'-Z(G)}, we can assume $O_{2'}(G)=1$.  

Let $G_i\nsg{}G$ be the normal closure of $S_i$ in $G$.  Since $\calf$ 
factors as a product $\calf_1\times\calf_2$, the subgroups $S_1$ and $S_2$ 
are strongly closed in $\calf$, and hence strongly closed in $G$ in the 
sense of \cite{Goldschmidt}.  So by Goldschmidt's 
theorem \cite[Corollary A1]{Goldschmidt}, $G_1\cap{}G_2=1$.  Thus 
$G_1\times{}G_2$ is a normal subgroup of odd index in $G$.  Since 
$\calf=\calf_S(G)$ has no proper normal subsystem of odd index (since it 
is reduced), $\calf_S(G)=\calf_S(G_1\times{}G_2)=\calf_{S_1}(G_1)\times 
\calf_{S_2}(G_2)$.  Hence $\calf_i=\calf_{S_i}(G_i)$ for $i=1,2$ 
(there can be at most one way to factor $\calf$ as a product of fusion 
systems over the $S_i$), and thus each $\calf_i$ is realizable.

Set $\call=\call_S^c(G)$ and $\call_i=\call_{S_i}^c(G_i)$.  
Define $\Phi\:\Aut\typ^I(\call_1)\times\Aut\typ^I(\call_2)\Right2{} 
\Aut\typ^I(\call)$ as follows.  Fix $\alpha_i\in\Aut\typ^I(\call_i)$ 
($i=1,2$).  Let $\beta_i\in\Aut(S_i,\calf_i)$ be the corresponding 
automorphisms (see Lemma \ref{AutI}), and set 
$\beta=(\beta_1,\beta_2)\in\Aut(S,\calf)$.  Thus 
$\alpha_i(P_i)=\beta_i(P_i)$ for each $P_i\in\Ob(\call_i)$ and 
$\pi(\alpha_i(\psi_i))=\beta_i\pi(\psi_i)\beta_i^{-1}$ for 
$\psi_i\in\Mor(\call_i)$.  Define $\alpha\in\Aut\typ^I(\call)$ on objects by 
setting $\alpha(P)=\beta(P)$ for $P\in\Ob(\call)$.  Fix 
$\psi\in\Mor_\call(P,Q)$, let $P_i,Q_i\le{}S_i$ be the images of $P$ and 
$Q$ under projection, and set $\widehat{P}=P_1\times{}P_2$ and 
$\widehat{Q}=Q_1\times{}Q_2$.  Since $G$ and $G_1\times{}G_2$ have the same 
fusion system over $S$, $\psi=[g]$ for some 
$g=(g_1,g_2)\in{}N_G(P,Q)$, where $g_i\in{}G_i$.  Then 
$g_i\in{}N_{G_i}(P_i,Q_i)$, and hence $\psi$ extends to 
$\widehat{\psi}=(\psi_1,\psi_2)\in\Mor_\call(\widehat{P},\widehat{Q})$ 
where $\psi_i=[g_i]\in\Mor_{\call_i}(P_i,Q_i)$.  Also, 
$\pi(\alpha_1(\psi_1),\alpha_2(\psi_2)) 
=\beta(\pi(\psi_1),\pi(\psi_2))\beta^{-1}$ sends $\beta(P)$ into 
$\beta(Q)$, and we define 
$\alpha(\psi)=(\alpha_1(\psi_1),\alpha_2(\psi_2))|_{\beta(P),\beta(Q)}$.  
Finally, $\alpha\in\Aut\typ^I(\call)$ since 
$\alpha_i\in\Aut\typ^I(\call_i)$, and we set 
$\Phi(\alpha_1,\alpha_2)=\alpha$.  

Assume $\calf$ is tamely realized by $G$, and let 
$s\:\Out\typ(\call)\Right2{}\Out(G)$ be a splitting for $\kappa_G$.  
For each $\alpha_1\in\Aut\typ^I(\call_1)$, 
$s([\Phi(\alpha_1,\Id_{\call_2})])=[\gamma]$ for some $\gamma\in\Aut(G,S)$ 
such that $\gamma|_{S_2}=\Id$.  Also, $\gamma(G_2)=G_2$ since $G_2$ is the 
normal closure of $S_2$ in $G$, and so $\gamma$ induces 
$\widebar{\gamma}\in\Aut(G/G_2,S_1)$.  The class 
$[\widebar{\gamma}]\in\Out(G/G_2)$ is independent of the choice of 
$\gamma$ modulo $\Inn(G)$, and hence this gives a well defined 
homomorphism $s_1$ from $\Out\typ^I(\call_1)$ to $\Out(G/G_2)$.  Also, 
$\calf_{S_1}(G/G_2)\cong\calf/S_2\cong\calf_1$, so 
$\call^c_{S_1}(G/G_2)\cong\call_1$; and $s_1$ is a splitting for 
$\kappa_{G/G_2}$ since $s$ is a splitting for $\kappa_G$.  Thus $\calf_1$ 
is tame, and $\calf_2$ is tame by a similar argument. If $\calf$ is 
strongly tame, then we can choose $G\in\gpclass{2}$, so 
$G/G_i\in\gpclass{2}$ ($i=1,2$) by Lemma \ref{lim2=0}(b), and hence 
$\calf_1$ and $\calf_2$ are strongly tame.  

This proves the ``only if'' part of the theorem.  Clearly, $\calf$ is 
realizable if both factors are.  If $\calf_1$ and $\calf_2$ are both 
(strongly) tame, then we have just shown that each of the indecomposable 
factors of $\calf_1$ and $\calf_2$ is (strongly) tame, and so $\calf$ is 
(strongly) tame by Theorem \ref{prod-tame}. 
\end{proof}

%%%%%%%%%%%%%%%%%%%%%%%%%%%%%%%%%%%%%%%%

\newsect{Examples}

We now give three families of examples, to illustrate some of the 
techniques which can be used to prove tameness of reduced fusion systems.  
As an introduction to these techniques, we first list the reduced fusion 
systems over dihedral and semidihedral groups and prove they are all tame. 
Next, we prove that certain fusion systems studied in \cite[\S\,4--5]{OV2} 
are reduced and tame; as a way of explaining how the information 
about these fusion systems given in \cite{OV2} is just what is needed to 
prove tameness.  As a third example, we prove that the fusion systems of 
all alternating groups are tame, and that they are reduced with certain 
obvious exceptions.  

In general, tameness is shown by examining, for a $p$-local finite group 
$\SFL$ realized by $G$, the homomorphisms
	\[ \Out(G) \Right4{\kappa_G} \Out\typ(\call) \Right4{\mu_G} 
	\Out(S,\calf) \]
defined in Sections \ref{s:tamedef} and \ref{s:Aut(L)}.  By definition, 
$\calf$ is tame if $\kappa_G$ is split surjective (for some choice of 
$G$).  However, the group $\Out(S,\calf)$ is usually much easier to 
describe than $\Out\typ(\call)$, and the composite $\mu_G\circ\kappa_G$ is 
induced by restriction to $S$.  So we need some way of describing 
$\Ker(\mu_G)$.

We first recall some definitions.  A proper subgroup $H\lneqq{}G$ of a 
finite group $G$ is \emph{strongly $p$-embedded} if $p\big||H|$, and 
for each $g\in{}G{\sminus}H$, $H\cap{}gHg^{-1}$ has order prime to $p$.  It 
is not hard to see that $G$ has a strongly $p$-embedded subgroup if and 
only if the poset $\cals_p(G)$ of nontrivial $p$-subgroups is disconnected 
(cf. \cite[Theorem X.4.11(b)]{HB3}), but we will not be using that here.

When $\calf$ is a saturated fusion system over a finite $p$-group $S$,
then a proper subgroup $P\lneqq{}S$ is \emph{$\calf$-essential} if it
is $\calf$-centric and fully normalized, and $\outf(P)$ contains a
strongly $p$-embedded subgroup. Thus each $\calf$-essential subgroup is 
fully normalized and $\calf$-centric by definition, and is $\calf$-radical
since $O_p(\Gamma)=1$ for any group $\Gamma$ which has a strongly
$p$-embedded subgroup.  See, e.g., \cite[Theorem 6.4.3]{Sz2} for a
proof of this last statement (it is  shown there only for $p=2$, but
the same proof works for odd primes). The following proposition is a
stronger version of Theorems \ref{AFT} and \ref{AFT-L}, and helps show
the importance of essential subgroups when working with fusion
systems.

\begin{Thm} \label{AFT-E}
Let $\calf$ be any saturated fusion system over a finite $p$-group $S$.  
Let $\cale$ be the set of $\calf$-essential subgroups of $S$, and set 
$\cale_+=\cale\cup\{S\}$.  Then each morphism in $\calf$ is a composite of 
restrictions of elements of $\autf(P)$ for $P\in\cale_+$.  If $\call$ is a 
linking system associated to $\calf$, then each morphism in $\call$ is a 
composite of restrictions of elements of $\Aut_\call(P)$ for $P\in\cale_+$. 
\end{Thm}

\begin{proof}  The statement about morphisms in $\calf$ is shown in 
\cite[\S\,5]{Puig}, and also in \cite[Corollary 2.6]{OV2}.  The second 
statement follows from this together with Proposition \ref{L-prop}(a) (and 
since $\Ob(\call)$ is closed under overgroups).  
\end{proof}

The following proposition will be useful when describing $\Ker(\mu_G)$, 
and for determining whether or not explicit elements in this group vanish. 
In fact, it applies to help describe $\Ker(\mu_\call)$, when $\call$ is an 
arbitrary linking system (not necessarily induced by a finite 
group).  For any fusion system $\calf$ over $S$ and any $P\le{}S$, we 
write 
	\[ C_{Z(P)}(\autf(P))=\{g\in{}Z(P)\,|\,\alpha(g)=g
	\textup{ for all } \alpha\in\autf(P)\} \]
and similarly for $C_{Z(P)}(\Aut_S(P))$ and $C_{Z(P)}(\Aut_\call(P))$.

\begin{Prop} \label{Ker(mu)}
Let $\calf$ be a saturated fusion system over the finite $p$-group $S$, 
and let $\call$ be a linking system associated to $\calf$.  Let 
$\call^c\subseteq\call$ be the full subcategory whose objects are the 
$\calf$-centric objects in $\call$.  Each element in $\Ker(\mu_\call)$ is 
represented by some $\alpha\in\Aut\typ^I(\call)$ such that 
$\alpha_S=\Id_{\Aut_\call(S)}$.  For each such $\alpha$, there are 
elements $g_P\in{}C_{Z(P)}(\Aut_S(P))$, defined for each fully normalized 
subgroup $P\in\Ob(\call^c)$, for which the following hold:
\begin{enuma}  
\item $\alpha_P\in\Aut(\Aut_\call(P))$ is conjugation by $\delta_P(g_P)$, 
and $g_P$ is uniquely determined by $\alpha$ modulo $C_{Z(P)}(\autf(P))$.
In particular, $\alpha_P=\Id_{\Aut_\call(P)}$ if and only if 
$g_P\in{}C_{Z(P)}(\autf(P))$.  

\item Assume $P,Q\in\Ob(\call^c)$ are both fully normalized in $\calf$.  
If $Q=aPa^{-1}$ for some $a\in{}S$, then we can choose $g_Q=ag_Pa^{-1}$.  
More generally, if $Q$ is $\calf$-conjugate to $P$, and there is 
$\zeta\in\Iso_\call(P,Q)$ such that $\alpha(\zeta)=\zeta$, then we can 
choose $g_Q=\pi(\zeta)(g_P)$. In either case, 
$\alpha_P=\Id_{\Aut_\call(P)}$ if and only if 
$\alpha_Q=\Id_{\Aut_\call(Q)}$.

\item If $Q\le{}P$ are both fully normalized objects in $\call^c$, then 
$g_P\equiv{}g_Q$ (mod $C_{Z(Q)}(\autf(P,Q))$), where $\autf(P,Q)$ is the 
group of those $\varphi\in\autf(P)$ such that $\varphi(Q)=Q$. 

\item Let $\cale$ be the set of all $\calf$-essential subgroups $P\lneqq{}S$ 
and let $\cale_0\subseteq\cale$ be the 
subset of those $P\in\cale$ such that $C_{Z(P)}(\autf(P))\lneqq 
C_{Z(P)}(\Aut_S(P))$.  Then $[\alpha]=1$ in $\Out\typ(\call)$ if and only 
if there is $g\in{}C_{Z(S)}(\autf(S))$ such that 
$g_P\in{}g{\cdot}C_{Z(P)}(\autf(P))$ for all $P\in\cale_0$.  

\item Let $\cale_0$ be as in \textup{(d)}, and let $\E_0$ be the set of 
all $P\in\cale_0$ such that $P=C_S(E)$ for some elementary abelian 
$p$-subgroup $E\le{}S$ which is fully centralized in $\calf$.  Then 
$[\alpha]=1$ in $\Out\typ(\call)$ if and only if there is 
$g\in{}C_{Z(S)}(\autf(S))$ such that $g_P\in{}g{\cdot}C_{Z(P)}(\autf(P))$ 
for all $P\in\E_0$.  In fact, it suffices that this hold for at least one 
subgroup $P\in\E_0$ in each $\calf$-conjugacy class intersecting $\E_0$ 
nontrivially.
\end{enuma}
\end{Prop}

\begin{proof}  We identify $S$ with $\delta_S(S)\le\Aut_\call(S)$ for 
short.  Fix $\alpha\in\Aut\typ^I(\call)$ such that 
$[\alpha]\in\Ker(\mu_\call)$.  Set $\beta=\til\mu_\call(\alpha)$; thus 
$\beta\in\autf(S)$.  Choose $\zeta\in\Aut_\call(S)$ such that 
$\pi(\zeta)=\beta$.  Then $\til\mu_\call(c_\zeta)=\beta$ by axiom (C) for 
the linking system $\call$, and so upon replacing $\alpha$ by 
$\alpha\circ{}c_{\zeta}{}^{-1}$, we can arrange that $\alpha_S$ is the 
identity on $\delta_S(S)\nsg\Aut_\call(S)$.  We will show in the proof of 
(a) how to arrange that $\alpha_S=\Id_{\Aut_\call(S)}$.  

\smallskip

\noindent\textbf{(a) } 
Fix a fully normalized subgroup $P\in\Ob(\call^c)$.  Set 
$\Gamma=\Aut_\call(P)$ for short, and identify $P$ with 
$\delta_P(P)\nsg\Gamma$.  Set 
$\Out(\Gamma,P)=\Aut(\Gamma,P)/\Inn(\Gamma)$, where 
$\Aut(\Gamma,P)\le\Aut(\Gamma)$ is the subgroup of automorphisms leaving 
$P$ invariant.  By \cite[Lemma 1.2]{OV2}, there is an exact sequence
	\[ 1 \Right2{} H^1(\Gamma/P;Z(P)) \Right4{\eta} \Out(\Gamma,P) 
	\Right4{R} N_{\Out(P)}(\Out_{\Gamma}(P))/\Out_{\Gamma}(P), \]
where $R$ is induced by restriction.  Since $\alpha_P\in\Aut(\Gamma)$ and 
$\alpha_P|_{\delta_P(N_S(P))}=\Id$, $[\alpha_P]\in\Ker(R)$, and 
$\eta^{-1}([\alpha_P])$ is trivial after restriction to 
$H^1(N_S(P)/P;Z(P))$.  The restriction map from $H^1(\Gamma/P;Z(P))$ to 
$H^1(N_S(P)/P;Z(P))$ is injective since $\delta_P(N_S(P))\in\sylp{\Gamma}$ 
(Proposition \ref{L-prop}(d)), and hence $[\alpha_P]=1$.  Thus 
$\alpha_P=c_{\delta_P(g_P)}$ for some $g_P\in{}Z(P)$ which is uniquely 
determined modulo $C_{Z(P)}(\Gamma)=C_{Z(P)}(\autf(P))$.  Also, 
$g_P\in{}C_{Z(P)}(\Aut_S(P))$, since $\alpha_P$ is the identity on 
$\delta_P(N_S(P))$.  

Set $\gamma=\delta_S(g_S)\in\Aut_\call(S)$.  Upon replacing $\alpha$ by 
$\alpha\circ{}c_{\gamma}^{-1}$, we can arrange that $\alpha_S=\Id$.

\smallskip

\noindent\textbf{(b) } Assume $\zeta\in\Iso_\call(P,Q)$ and 
$\alpha(\zeta)=\zeta$.  Fix $\psi\in\Aut_\call(Q)$, and set 
$\varphi=\zeta^{-1}\psi\zeta\in\Aut_\call(P)$.  Set $g=\pi(\zeta)(g_P)$; 
then $\zeta\circ\delta_P(g_P)\circ\zeta^{-1}=\delta_Q(g)$ by axiom (C) for a 
linking system.  Hence 
	\[ \alpha_Q(\psi) = \alpha_Q(\zeta\varphi\zeta^{-1}) = 
	\zeta\alpha_P(\varphi)\zeta^{-1} = 
	\zeta\delta_P(g_P)\varphi\delta_P(g_P)^{-1}\zeta^{-1} =
	\delta_Q(g)\psi\delta_Q(g)^{-1}~, \]
and we can choose $g_Q=g$.  

If $Q=aPa^{-1}$ and $\zeta=\delta_{P,Q}(a)$, then $\alpha(\zeta)=\zeta$ 
since $\alpha_S=\Id$  (and since $\alpha$ sends inclusions to inclusions). 
So again we can choose $g_Q=c_a(g_P)$.  

In either case, $g_P\in{}C_{Z(P)}(\autf(P))$ if and only if 
$g_Q\in{}C_{Z(Q)}(\autf(Q))$, and hence $\alpha_P=\Id$ if and only if 
$\alpha_Q=\Id$.

\smallskip

\noindent\textbf{(c) } Assume $Q\le{}P$, and let $\Aut_\call(P,Q)$ be the 
group of elements $\psi\in\Aut_\call(P)$ such that $\pi(\psi)(Q)=Q$.  Then 
$\alpha$ commutes with the restriction map 
	\[ \Res^P_Q\: \Aut_\call(P,Q) \Right4{} \Aut_\call(Q) \]
which is injective by Proposition \ref{L-prop}(f).  So if $\alpha$ acts on 
$\Aut_\call(P,Q)$ via conjugation by $\delta_P(g_P)$ and on 
$\Aut_\call(Q)$ via conjugation by $\delta_Q(g_Q)$, they must have the 
same action on $\Aut_\call(P,Q)$.  Since $g_Q$ and $g_P$ both lie in 
$Z(Q)\ge{}Z(P)$, we conclude $g_Q\equiv{}g_P$ (mod 
$C_{Z(Q)}(\autf(P,Q))$). 

\smallskip

\noindent\textbf{(d) }   
By Theorem \ref{AFT-E}, all 
morphisms in $\call$ are composites of restrictions of elements in 
$\Aut_\call(P)$ for $P$ $\calf$-essential or $P=S$.  
Hence if $\alpha\ne\Id_\call$, then since $\alpha_S=\Id$ by assumption, 
$\alpha_P\ne\Id$ for some $P\in\cale$.  By (a), 
$g_P\in{}C_{Z(P)}(\Aut_S(P))$ but $g_P\notin{}C_{Z(P)}(\autf(P))$, 
and so $P\in\cale_0$.  The converse is clear:  if $\alpha=\Id_\call$, 
then $\alpha_P=\Id$ and hence $g_P\in{}C_{Z(P)}(\autf(P))$ for all 
$P\in\cale_0$.

By definition, $[\alpha]=1$ in $\Out\typ(\call)$ if and only if 
$\alpha=c_\beta$ for some $\beta\in\Aut_\call(S)$.  Since $\alpha_S=\Id$, 
$\beta\in{}Z(\Aut_\call(S))$, and hence $\beta=\delta_S(g)$ for some 
$g\in{}C_{Z(S)}(\autf(S))$.  Thus $[\alpha]=1$ if and only if 
$\alpha=c_{\delta_S(g)}$ for some $g\in{}C_{Z(S)}(\autf(S))$, which we 
just saw is the case exactly when $g^{-1}g_P\in{}C_{Z(P)}(\autf(P))$ for 
all $P\in\cale_0$. 

\smallskip

\noindent\textbf{(e) }  We first prove the following statement:
	\beqq \begin{split}  
	\alpha=\Id_\call \ &\Longleftrightarrow 
	\textup{ $g_P\in{}C_{Z(P)}(\autf(P))$ for all $P\in\E_0$} \\
	&\Longleftrightarrow 
	\textup{ $\forall\,P\in\E_0$ $\exists\,Q\in\E_0$ 
	$\calf$-conjugate to $P$ with $g_Q\in{}C_{Z(Q)}(\autf(Q))$} ~.
	\end{split} \label{e:4.2a} \eeqq
The first statement implies the second by (a), and the second
implies the third tautologically.

Now assume $\alpha\ne\Id_\call$.  As was just seen in the proof of (d), 
there is $P\in\cale_0$ such that 
$g_P\notin{}C_{Z(P)}(\autf(P))$.  Assume $P$ is such that $|P|$ is maximal 
among orders of all such subgroups.  We will show that $P\in\E_0$ 
(possibly after replacing $P$ by another subgroup in its $\calf$-conjugacy 
class), and that $g_Q\notin{}C_{Z(Q)}(\autf(Q))$ for each $Q\in\E_0$ which 
is $\calf$-conjugate to $P$.  This will prove the remaining implication in 
\eqref{e:4.2a}.

We first check that 
	\beqq \textup{$T\in\Ob(\call)$ and $|T|>|P|$}
	\quad\Longrightarrow\quad
	\alpha_T=\Id~. \label{e:4.2aa} \eeqq
If $T=S$ or $T\in\cale_0$, this follows by assumption.  If 
$T\in\cale{\sminus}\cale_0$, then 
$g_{T}\in{}C_{Z(T)}(\Aut_S(T))=C_{Z(T)}(\autf(T))$, and hence 
$\alpha_T=\Id$ by definition of $g_T$.  Otherwise, each 
$\psi\in\Aut_\call(T)$ is a composite of restrictions of automorphisms of 
subgroups in $\cale\cup\{S\}$ (Theorem \ref{AFT-E}), each of those 
automorphisms and its restrictions are sent to themselves by $\alpha$, and 
hence $\alpha_T(\psi)=\psi$.

We next claim that 
	\beqq \textup{for all $Q$ $\calf$-conjugate to $P$, there is 
	$\zeta\in\Iso_\call(P,Q)$ such that $\alpha_{P,Q}(\zeta)=\zeta$.} 
	\label{e:4.2c} \eeqq
Choose any $\zeta_0\in\Iso_\call(P,Q)$.  By Theorem \ref{AFT-E} 
again, $\zeta_0$ is the composite of restrictions of automorphisms 
$\psi_i\in\Aut_\call(R_i)$ for subgroups $R_i\le{}S$ with $|R_i|\ge|P|$.  
If we remove from this composite all $\psi_i$ for which $|R_i|=|P|$, we 
get an isomorphism $\zeta\in\Iso_\call(P,Q)$ which is a composite of 
restrictions of automorphisms of strictly larger subgroups.  We just 
showed that $\alpha_{R_i}(\psi_i)=\psi_i$ whenever $|R_i|>|P|$, and thus 
$\alpha_{P,Q}(\zeta)=\zeta$.

Set $E=\Omega_1(Z(P))$:  the $p$-torsion subgroup of the center $Z(P)$.  
If $E$ is not fully normalized in $\calf$, then choose 
$\varphi\in\homf(N_S(E),S)$ such that $\varphi(E)$ is fully normalized 
(using \cite[Proposition A.2(b)]{BLO2}).  Then $\varphi(P)$ is fully 
normalized since $N_S(\varphi(P))\ge{}\varphi(N_S(P))$.  By 
\eqref{e:4.2c}, there is $\zeta\in\Iso_\call(P,\varphi(P))$ such that 
$\alpha_{P,\varphi(P)}(\zeta)=\zeta$.  So $\alpha_{\varphi(P)}\ne\Id$ by 
(b).  Upon replacing $P$ by $\varphi(P)$ and $E$ by $\varphi(E)$, we can 
now assume $E$ and $P$ are both fully normalized.  

Set $P^*=N_{C_S(E)}(P)\ge{}P$ and $\Gamma=\Aut_\call(P)$ for short.  To 
simplify notation, we identify $N_S(P)$ with $\delta_P(N_S(P))$.  Then 
$E\nsg\Gamma$, so $C_\Gamma(E)\nsg\Gamma$; and $P^*\in\sylp{C_\Gamma(E)}$ 
since $N_S(P)\in\sylp{\Gamma}$ (Proposition \ref{L-prop}(d)).  Also, 
$C_\Gamma(Z(P))\nsg\Gamma$, and has $p$-power index in $C_\Gamma(E)$ 
since each automorphism of $Z(P)$ which is the identity on its $p$-torsion 
subgroup $E$ has $p$-power order (cf. \cite[Theorem 5.2.4]{Gorenstein}).  
Hence each Sylow $p$-subgroup of $C_\Gamma(E)$ is 
$C_\Gamma(Z(P))$-conjugate to $P^*=C_{N_S(P)}(E)$.  By the Frattini 
argument, 
	\beqq \Gamma=N_\Gamma(P^*){\cdot}C_\Gamma(Z(P))~. \label{e:4.2b} 
	\eeqq

Since $\alpha_P\ne\Id_\Gamma$ is conjugation by $g_P\in{}Z(P)$, $\alpha_P$ 
is the identity on $C_\Gamma(Z(P))$.  Hence by \eqref{e:4.2b}, $\alpha_P$ 
is not the identity on $N_\Gamma(P^*)$.  By Proposition \ref{L-prop}(e), 
each $\alpha\in{}N_\Gamma(P^*)$ extends to 
$\widebar{\alpha}\in\Aut_\call(P^*)$, and thus 
$\alpha_{P^*}\ne\Id_{\Aut_\call(P^*)}$.  
If $C_S(E)\gneqq{}P$, then $P^*=N_{C_S(E)}(P)\gneqq{}P$ (cf. \cite[Theorem 
2.1.6]{Sz1}), which would imply $\alpha_{P^*}=\Id$ by \eqref{e:4.2aa}.  
We now conclude that $C_S(E)=P$, and hence that $P\in\E_0$.

Assume $Q\in\E_0$ is $\calf$-conjugate to $P$.  By \eqref{e:4.2c}, there 
is $\zeta\in\Iso_\call(P,Q)$ such that $\alpha(\zeta)=\zeta$.  So by (b), 
$\alpha_Q\ne\Id$ since $\alpha_P\ne\Id$, and this finishes the proof of 
\eqref{e:4.2a}.

The rest of the proof of (e) is identical to that of (d).
\end{proof}

As one simple application of Proposition \ref{Ker(mu)}, consider 
the group $G=A_6\cong{}PSL_2(9)$.  Set 
	\[ T_1=\gen{(1\,2)(3\,4),(1\,3)(2\,4)}\cong{}C_2^2~, \qquad
	T_2=\gen{(1\,2)(3\,4),(3\,4)(5\,6)}\cong{}C_2^2~, \] and 
$S=\gen{T_1,T_2}\in\syl2{G}$, and let $\calf=\calf_S(G)$ and 
$\call=\call_S^c(G)$.  Then $\cale=\cale_0=\widehat{\cale}_0= 
\{T_1,T_2\}$.  Set $g=(5\,6)$, and consider the automorphism 
$\alpha=\til\kappa_G(c_g)\in\Aut\typ^I(\call)$.  Then 
$\alpha\in\Ker(\til\mu_G)$, since $[g,S]=1$ (and since 
$\til\mu_G\circ\til\kappa_G$ sends $\beta\in\Aut(G,S)$ to $\beta|_S$).  
Since $[g,N_G(T_1)]=1$, $\alpha_{T_1}=\Id_{\Aut_\call(T_1)}$.  
Since $(1\,2)(3\,4)(5\,6)$ commutes with 
$N_G(T_2)=\gen{T_2,(1\,3)(2\,4),(1\,3\,5)(2\,4\,6)}$, $\alpha_{T_2}$ acts 
on $\Aut_\call(T_2)\cong{}N_G(T_2)$ as conjugation by 
$x=(1\,2)(3\,4)\in{}Z(S)$.  So in the notation of Proposition 
\ref{Ker(mu)}, $g_{T_1}=1$ and $g_{T_2}=x$.  In both cases, 
$C_{Z(T_i)}(\autf(T_i))=1$, so the $g_{T_i}$ are uniquely determined.  
Hence by Proposition \ref{Ker(mu)}(d), $[\alpha]=\kappa_G([c_g])$ 
represents a nontrivial element in $\Ker(\mu_G)$.

If $[\alpha]\in\Ker(\mu_G)$ is arbitrary, represented by 
$\alpha\in\Aut\typ^I(\call)$ such that $\alpha_S=\Id_{\Aut_\call(S)}$, 
then by Proposition \ref{Ker(mu)} again, $g_{T_i}\in{}Z(S)$ for $i=1,2$, 
and $[\alpha]=1$ if and only if $g_{T_1}=g_{T_2}$.  Thus 
$\Ker(\mu_G)\cong{}C_2$ is generated by $\kappa_G([c_g])$ as described 
above. Using this, and the well known description of 
$\Out(A_6)\cong{}C_2^2$, it is not hard to see that $\kappa_G$ is an 
isomorphism from $\Out(G)$ to $\Out\typ(\call)$.

This example will be generalized in two different ways below:  to 
other groups $PSL_2(q)$ for $q\equiv\pm1$ (mod $8$) in Proposition 
\ref{D(2^k)}, and to other alternating groups in Proposition 
\ref{F(An)-tame}.

\newsub{Dihedral and semidihedral 2-groups}

As our first examples, we list all reduced fusion systems over dihedral 
and semidihedral 2-groups, and prove they are all tame.  The list of all 
fusion systems over such groups is well known; it turns out that each of
them supports exactly one fusion system which is reduced.

As usual, $v_p(-)$ denotes the $p$-adic valuation:  $v_p(n)=k$ if $p^k|n$ 
but $p^{k+1}{\nmid}n$.

\begin{Prop} \label{D(2^k)}
Let $S$ be a dihedral group of order $2^k$ ($k\ge3$).  Then there is a 
unique reduced fusion system $\calf$ over $S$, and it is tame.  
Let $q$ be a prime power such that $v_2(q^2-1)=k+1$, set $G=PSL_2(q)$, and 
fix $S^*\in\syl2{G}$.  Then $S\cong{}S^*$ and $\calf\cong\calf_{S^*}(G)$; 
and $\kappa_G$ is an isomorphism if $q=p^{2^{k-2}}$ for some prime 
$p\equiv5$ (mod $8$).
\end{Prop}

\begin{proof}  Fix $a,b\in{}S$ such that $\gen{a}$ has index two and 
$S=\gen{a,b}$.  For each $i\in\Z$, set 
$T_i=\gen{a^{2^{k-2}},a^ib}\cong{}C_2^2$.  Two subgroups $T_i$ and $T_j$ 
are $S$-conjugate if and only if $i\equiv{}j$ (mod 2).  Set 
$\calp=\{T_i\,|\,i\in\Z\}$.  

If $P\le{}S$ is cyclic of order $2^m$, then $\Aut(P)\cong(\Z/2^m)^\times$ 
is a 2-group.  If $P\le{}S$ is dihedral of order $2^m\ge8$, then there is 
a unique cyclic subgroup of index two in $P$, and $\Aut(P)$ is a 2-group 
by Lemma \ref{mod-Fr}.  Thus the only subgroups $P\le{}S$ for which 
$\Aut(P)$ is not a $2$-group are the $T_i$.

Define $\calf$ to be the fusion system over $S$ generated by the 
automorphisms in $\Inn(S)$, $\Aut(P)$ for $P\in\calp$, and their 
restrictions.  Assume $\calf$ is saturated (this will be shown later).  
Then $\foc(\calf)=\gen{[S,S],\calp}=S$, and hence $O^2(\calf)=\calf$ 
(Theorem \ref{t:O^p(F)}(a)).  Also, $O^{2'}(\calf)=\calf$ since any normal 
subsystem of odd index would have to contain the same automorphism groups, 
and $O_2(\calf)=1$ by inspection.  Thus $\calf$ is reduced.

Let $\calf^*$ be an arbitrary saturated fusion system over $S$ such that 
$\foc(\calf^*)=S$.  Let $\cale$ be the set of all $\calf^*$-essential 
subgroups of $S$.  If $P\in\cale$, then $\Aut(P)$ must have 
elements of odd order, and hence $P\in\calp$.  For each 
$T_i\in\calp$, $\Aut(T_i)\cong\Sigma_3$ and $\Aut_S(T_i)\cong{}C_2$.  
Hence $\Aut_{\calf^*}(T_i)=\Aut(T_i)$ if $T_i\in\cale$.  Since $\Aut(S)$ 
is a 2-group, Theorem \ref{AFT-E} implies $\calf^*$ is generated by 
automorphisms in $\Aut_{\calf^*}(S)=\Inn(S)$, the $\Aut(P)$ for 
$P\in\cale\subseteq\calp$, and their restrictions.  In particular, 
$\foc(\calf^*)\le\gen{[S,S],\cale}$, and this has index at least two in 
$S$ if $\cale\subsetneqq\calp$.  Hence $\cale=\calp$, and so 
$\calf^*=\calf$.

Set $G=PSL_2(q)$ for any prime power $q\equiv\pm1$ (mod $8$), and fix 
$S^*\in\syl2{G}$.  As is well known, $S^*$ is a dihedral group and 
$|G|=\frac12q(q^2-1)$, so $S^*\cong{}D_{2^k}$ where $k=v_2(q^2-1)-1$.  So 
we identify $S^*=S$ for $S$ as above.  Since $G$ is simple, 
$\foc(\calf_{S}(G))=S\cap[G,G]=S$ by the focal subgroup theorem (cf. 
\cite[Theorem 7.3.4]{Gorenstein}), and we have just seen this implies 
$\calf_{S}(G)=\calf$.  In particular, $\calf$ is saturated, and hence 
reduced.

Now assume $q=p^{2^{k-2}}$, where $p\equiv5$ (mod $8$) (and $k\ge3$).  
The homomorphism $\kappa_G$ is an isomorphism in this case by 
\cite[Proposition 7.9]{BLO1}, where it is shown more generally for 
$p\equiv\pm3$ (mod $8$).  But we give a different proof here to illustrate 
how Proposition \ref{Ker(mu)} can be applied. 

Set $\til{G}=SL_2(q)$.  Fix $u\in\F_q^\times$ of order $2^k$.  Set 
$\til{a}=\mxtwo{u}00{u^{-1}}$ and $\til{b}=\mxtwo01{-1}0$, and let 
$a,b\in{}G$ be their images in the quotient.  Then 
$S\defeq\gen{a,b}\in\syl2{G}$.  Let $\delta\in\Aut(G)$ be conjugation by 
$\mxtwo{u}001$; then $\delta(a)=a$ and $\delta(b)=ab$.  Since $u$ is not a 
square in $\F_q^\times$, $[\delta]$ generates the subgroup (of order 2) 
of diagonal automorphisms in $\Out(G)$.  

By \cite[\S\,3]{Steinberg}, $\Out(G)=\gen{[\delta]}\times\gen{[\psi^p]} 
\cong C_2\times{}C_{2^{k-2}}$, where $\psi^p$ is the field automorphism 
which acts via $x\mapsto{}x^p$ on matrix elements.  Also, $\psi^p(a)=a^p$ 
and $\psi^p(b)=b$.  Since $p\equiv5$ (mod $8$), $[\delta|_S]$ and 
$[\psi^p|_S]$ generate $\Out(S)$.  Thus 
	\[ \mu_G\circ\kappa_G \: \Out(G) \Right5{} \Out(S,\calf) = \Out(S) 
	\]
is surjective with kernel generated by $[\alpha]$, where 
$\alpha=(\psi^p)^{2^{k-3}}$ is the field automorphism of order 2.

To prove that $\kappa_G$ is an isomorphism, it remains to show that 
$\Ker(\mu_G)$ has order 2 and is generated by $\kappa_G([\alpha])$.  
Set $w=a^{2^{k-2}}\in{}Z(S)$.  We refer to Proposition \ref{Ker(mu)}.  
Since $\alpha$ is the identity on $S=N_G(S)$ ($\alpha(a)=a$ since the 
field automorphism of order two sends $u$ to $-u$), there are elements 
$g_{T_i}\in{}C_{Z(T_i)}(\Aut_S(T_i))=\gen{w}$ for each $i$ such that 
$\til\kappa_G(\alpha)$ acts on $\Aut_\call(T_i)$ via conjugation by 
$g_{T_i}$. These elements are uniquely defined since 
$C_{Z(T_i)}(\autf(T_i))=1$.  

When $i$ is even, $T_i\le{}G_0\defeq{}PSL_2(\sqrt{q})$ (recall 
$T_i=\gen{a^{2^{k-2}},a^ib}$), and $N_{G_0}(T_i)$ has index at most two in 
$N_G(T_i)$.  Since $\alpha|_{G_0}=\Id$ and $\alpha|_S=\Id$, 
$\alpha$ is the identity on $N_G(T_i)\cong\Sigma_4$ 
in this case, and so $g_{T_i}=1$.

Now consider $T_i$ for odd $i$.  Let $\til{T}_i\cong{}Q_8$ be the inverse 
image in $\til{G}$ of $T_i\le{}G$, let $\til{w}$ be any lifting of $w$ to 
$\til{G}$, and set $z=\til{w}^2=\mxtwo{-1}00{-1}\in{}Z(\til{G})$.  
Since the field automorphism of order two sends $u$ to $-u$, it sends 
$\til{a}^i\til{b}$ to $z\til{a}^i\til{b}$.  If $\alpha$ acted on 
$N_G(T_i)\cong\Sigma_4$ via the identity, then its action on 
$N_{\til{G}}(\til{T}_i)$ would be the identity on a subgroup of index two, 
which necessarily would include $\til{T}_i$.  Since this is not the case, 
we conclude that $\alpha$ acts via conjugation by $w$, and thus that 
$g_{T_i}=w$ for $i$ odd.

By Proposition \ref{Ker(mu)}(d), since $g_{T_0}=1$ and $g_{T_1}=w$ (and 
$C_{Z(T_i)}(\autf(T_i))=1$), $\kappa_G([\alpha])\ne1$ in $\Out\typ(\call)$, 
and it is the only nontrivial element in $\Ker(\mu_G)$.  Thus 
$\Ker(\mu_G)\cong{}C_2$, which is what was left to prove.
\end{proof}

We now consider the semidihedral case.  

\begin{Prop} \label{SD(2^k)}
Let $S$ be a semidihedral group of order $2^k$ ($k\ge4$).  Then there is a 
unique reduced fusion system $\calf$ over $S$, and it is tame.  
Let $q$ be a prime power such that $v_2(q-1)=k-2$, set $G=PSU_3(q)$, and 
fix $S^*\in\syl2{G}$.  Then $S\cong{}S^*$ and $\calf\cong\calf_{S^*}(G)$, 
and $\kappa_G$ is an isomorphism if $3{\nmid}(q+1)$ and $q=p^{2^{k-4}}$ for 
some prime $p\equiv5$ (mod $8$).
\end{Prop}

\begin{proof}  Fix $a,b\in{}S$ such that $\gen{a}$ has index two, $b^2=1$, 
and $S=\gen{a,b}$.  Then $|a^ib|=2$ for $i$ even and $|a^ib|=4$ for $i$ 
odd.  For each $i\in\Z$, set $T_i=\gen{a^{2^{k-2}},a^{2i}b}\cong{}C_2^2$, 
and $R_i=\gen{a^{2^{k-3}},a^{2i+1}b}\cong{}Q_8$.  The $T_i$ are all 
$S$-conjugate to each other, and similarly for the $R_i$.  Set 
$\calp=\{T_i,R_i\,|\,i\in\Z\}$.

As shown in the proof of Proposition \ref{D(2^k)}, $\Aut(P)$ is a 2-group 
for each $P\le{}S$ which is cyclic, or dihedral of order $\ge8$.  The same 
argument applies when $P$ is quaternion of order $\ge16$, and also to $S$ 
itself.  Thus the only subgroups $P\le{}S$ for which $\Aut(P)$ is not a 
$2$-group are those in $\calp$.  

Define $\calf$ to be the fusion system over $S$ generated by the 
automorphisms in $\Inn(S)$, $\Aut(P)$ for $P\in\calp$, and their 
restrictions.  Assume $\calf$ is saturated (to be shown later).  Then 
$\foc(\calf)=\gen{[S,S],\calp}=S$, and hence $O^2(\calf)=\calf$ (Theorem 
\ref{t:O^p(F)}(a)).  Also, $O^{2'}(\calf)=\calf$ since any normal 
subsystem of odd index would have to contain the same automorphism groups, 
and $O_2(\calf)=1$ by inspection.  Thus $\calf$ is reduced.

Let $\calf^*$ be an arbitrary saturated fusion system over $S$ such that 
$\foc(\calf^*)=S$.  Let $\cale$ be the set of all $\calf^*$-essential 
subgroups of $S$.  If $P\in\cale$, then $\Aut(P)$ must have 
elements of odd order, and hence $P\in\calp$.  
For all $P\in\calp$, $[\Aut(P):\Aut_S(P)]=3$, and hence 
$\Aut_{\calf^*}(P)=\Aut(P)$ if $P\in\cale$.  Since $\Aut(S)$ is a 2-group, 
Theorem \ref{AFT-E} implies $\calf^*$ is generated by automorphisms in 
$\Aut_{\calf^*}(S)=\Inn(S)$, the $\Aut(P)$ for $P\in\cale$, and their 
restrictions.  In particular, $\foc(\calf^*)\le\gen{[S,S],\cale}$, and 
this has index at least two in $S$ if $\cale\subsetneqq\calp$.  Hence 
$\cale=\calp$, and so $\calf^*=\calf$.

Fix a prime power $q\equiv1$ (mod $4$), set $G=PSU_3(q)$, and fix 
$S^*\in\syl2{G}$.  Then $|G|=\frac1dq^3(q^2-1)(q^3+1)$ where 
$d=\textup{gcd}(3,q+1)$ \cite[p. 118]{Taylor}, and hence $|S^*|=2^k$ where 
$k=v_2(q-1)+2$.  Since $GU_2(q)$ has odd index in $SU_3(q)$, and the Sylow 
2-subgroups of $GU_2(q)$ are semidihedral by \cite[p.143]{CF}, the Sylow 
2-subgroups of $SU_3(q)$ and of $G$ are also semidihedral.  Thus 
$S^*\cong{}SD_{2^k}$, and we identify $S^*=S$ as above.  Since $G$ is 
simple, $\foc(\calf_S(G))=S$ (cf. \cite[Theorem 7.3.4]{Gorenstein}), and 
we just saw this implies $\calf_S(G)=\calf$.  In particular, $\calf$ is 
saturated.

Now assume $3{\nmid}(q+1)$ and $q=p^{2^{k-4}}$ for some prime $p\equiv5$ 
(mod $8$).  By \cite[\S\,3]{Steinberg}, $\Out(G)$ is generated by 
diagonal and field automorphisms; where the group of diagonal 
automorphisms has order $\textup{gcd}(3,q+1)=1$.  Thus 
$\Out(G)=\gen{[\psi^p]}$, generated by the class of the field automorphism 
$(x\mapsto{}x^p)$.  Since $G=PSU_3(q)$ is defined via matrices 
over $\F_{q^2}$, $\psi^p$ has order $2^{k-3}$.  

More explicitly, regard $G=PSU_3(q)=SU_3(q)$ as the group of 
matrices $M\in{}SL_3(q^2)$ such that $\psi^q(M^t)=M^{-1}$, where $M^t$ is 
the transpose $(a_{ij})\mapsto(a_{4-j,4-i})$.  Fix $u\in\F_{q^2}^\times$ of 
order $2^{k-1}$ (recall $v_2(q-1)=k-2$), and set 
$a=\textup{diag}(u,-1,u^{-q})$.  Since $u^{q-1}=-1$, $a\in{}SU_3(q)$.  Set 
$b=\mxthree001010100$.  Then $bab^{-1}=a^{-q}=a^{2^{k-2}-1}$, and so 
$S=\gen{a,b}$ is semidihedral.  Also, $\psi^p(a)=a^p$, $\psi^p(b)=b$, so 
$[\psi^p|_S]$ generates $\Out(S)$, and we conclude that
	\[ \mu_G \circ \kappa_G \: \Out(G) \Right5{\cong} \Out(S,\calf) = 
	\Out(S) \]
is an isomorphism. 

It remains to prove that $\Ker(\mu_G)=1$.  Fix $[\alpha]\in\Ker(\mu_G)$, 
and choose a representative $\alpha\in\Aut\typ^I(\call_S^c(G))$ for the 
class $[\alpha]$ such that $\alpha_S$ is the identity on 
$\Aut_{\call_S^c(G)}(S)$.  In the notation of Proposition \ref{Ker(mu)}, 
$\cale_0$ contains only the subgroups $T_i$, since $Z(R_i)\cong{}C_2$ (and 
hence $C_{Z(R_i)}(\Aut_S(R_i))=C_{Z(R_i)}(\autf(R_i))$).  If $\alpha$ is 
represented by elements $g_P$, then 
$g_{T_i}\in{}C_{Z(T_i)}(\Aut_S(T_i))=Z(S)$ for each $i$, and is uniquely 
determined since $C_{Z(T_i)}(\autf(T_i))=1$.  All of the $g_{T_i}$ are 
equal by point (b) in the proposition, and hence $[\alpha]=1$ by point 
(d).  
\end{proof}

\newsub{Tameness of some fusion systems studied in \cite{OV2}}

We next consider some fusion systems studied in \cite[\S4--5]{OV2}, and 
prove they are reduced and tame using the lists of essential subgroups and 
other information determined there. 

\begin{Prop} \label{MJMcL-tame}
The fusion systems at the prime $2$ of the group $PSL_4(5)$, and of the 
sporadic simple groups $M_{22}$, $M_{23}$, $\textup{McL}$, $J_2$, and 
$J_3$, are all reduced and tame.  Moreover, if $G$ is any of these groups, 
then $\kappa_G$ is an isomorphism.  
\end{Prop}

\begin{proof}  By \cite[\S1.5]{GL}, $\Out(G)\cong{}C_2$ when 
$G\cong{}M_{22}$, $\textup{McL}$, $J_2$, or $J_3$, while $\Out(M_{23})=1$. 
By \cite[(3.2)]{Steinberg}, when $G=PSL_4(5)$, $\Out(G)$ is generated by 
diagonal automorphisms (induced by conjugation by diagonal matrices in 
$GL_4(5)$) and a graph automorphism (induced by transpose inverse).  Since 
all multiples of the identity in $GL_4(5)$ have determinant one, the group 
of diagonal outer automorphisms is isomorphic to $\F_5^\times\cong{}C_4$.  
Since the graph automorphism inverts all diagonal matrices, we get
$\Out(G)\cong{}D_8$.  

Now let $G$ be any of the above six groups, fix $S\in\syl2{G}$, and set 
$\calf=\calf_S(G)$.  We prove below in each case that among the 
homomorphisms
	\[ \Out(G) \Right4{\kappa_G} \Out\typ(\call_S^c(G)) \Right4{\mu_G} 
	\Out(S,\calf) ~, \]
$\mu_G\circ\kappa_G$ is an isomorphism and $\mu_G$ is injective.  
It then follows that $\kappa_G$ is an isomorphism.

We show that $\mu_G\circ\kappa_G$ is injective for each of these groups, 
using arguments suggested to us by Richard Lyons.  These are based on the 
following statement, applied to certain subgroups $H\le{}G$:
	\beqq \left. \begin{array}{l}
	\alpha\in\Aut(H),\ S\in\syl2{H},\ \alpha|_S=\Id_S \\
	Q=O_2(H),\ C_H(Q)\le{}Q, \end{array} \right\} \ 
	\Longrightarrow\ \alpha\in\Aut_{Z(S)}(H). \label{e:Lyons} \eeqq
This follows, for example, from \cite[Lemma 1.2]{OV2}:  $\alpha\in\Inn(H)$ 
if a certain element in $H^1(H/Q;Z(Q))$ vanishes, and this element does 
vanish since its restriction to the Sylow subgroup $S/Q$ vanishes.  Thus 
$\alpha\in\Inn(H)$ and is the identity on $S$, so it must be conjugation 
by an element of $C_H(S)=Z(S)$.  

As in \cite{OV2}, we let $S_0=UT_3(4)$ denote the group of upper triangular 
$3\times3$ matrices over $\F_4$ with $1$'s on the diagonal.  For $x\in\F_4$ 
and $1\le{}i<j\le3$, $e_{ij}^x\in{}UT_3(4)$ is the matrix with entry $x$ in 
position $(i,j)$, $1$'s on the diagonal, and $0$'s elsewhere.  Set 
	\[ E_{ij}=\{e_{ij}^x\,|\,x\in\F_4\}~, \quad
	A_1=\gen{E_{12},E_{13}}~,\quad\textup{and}\quad 
	A_2=\gen{E_{13},E_{23}}~. \]
The field automorphism of $\F_4$ 
is denoted $x\mapsto\bar{x}$, and we write
$\F_4=\{0,1,\omega,\bar\omega\}$.  Also,
$\tau,\rho_1^*,\rho_2^*,\gamma_0,\gamma_1,c_\phi\in\Aut(S_0)$ are the 
automorphisms
	\begin{align*}  
	\tau\left(\mxthree1ab01c001\right) &= 
	\mxthree1cb01a001^{-1},  & 
	\rho_1^*\left(\mxthree1ab01c001\right) &= 
	\mxthree1a{b+\bar{a}}01c001,  & 
	\rho_2^*\left(\mxthree1ab01c001\right) &= 
	\mxthree1a{b+\bar{c}}01c001, \\
	\gamma_0\Bigl(\mxthree1ab01c001\Bigr) &= 
	\mxthree1{\omega a}{\bar{\omega}b}01{\omega c}001 , & 
	\gamma_1\Bigl(\mxthree1ab01c001\Bigr) &= 
	\mxthree1{\omega a}{b}01{\bar{\omega}c}001 , & 
	c_\phi\left(\mxthree1ab01c001\right) &= 
	\mxthree1{\bar{a}}{\bar{b}}01{\bar{c}}001 . 
	\end{align*}
The group $\Out(S_0)\cong{}C_2^4\sd{}(\Sigma_3\times\Sigma_3)$ is 
described precisely by \cite[Lemma 4.5]{OV2}.  In particular, the 
subgroups $\gen{\gamma_0,c_\phi\circ\tau}$ and $\gen{\gamma_1,\tau}$ are 
isomorphic to $\Sigma_3$ and commute with each other.

By the focal subgroup theorem (cf. \cite[Theorem 7.3.4]{Gorenstein}), 
$\foc(\calf)=S\cap[G,G]=S$ in each case, and hence $O^2(\calf)=\calf$.  
Assume $O_2(\calf)=1$, and set $\calf_0=O^{2'}(\calf)$.  By Lemma 
\ref{F0<|F}(e), $O_2(\calf_0)=1$, so $\calf_0$ is a centerfree, 
nonconstrained fusion system over $S$, and is included in the list of such 
fusion systems over $S$ in \cite[Theorems 4.8 \& 5.11]{OV2}.  By 
inspection of those lists, we see that $\calf_0=\calf$ in all cases.  So 
to prove $\calf$ is reduced, it remains only to show that $O_2(\calf)=1$.  

In each of Cases 1 and 2 below, we prove successively that (i) 
$\mu_G\circ\kappa_G$ is an isomorphism, (ii) $\mu_G$ is injective, and 
(iii) $O_2(\calf)=1$.

\smallskip

\noindent\textbf{Case 1: }  Assume first that 
$S=S_\phi=S_0\rtimes\gen{\phi}$:  the extension of $UT_3(4)$ by a field 
automorphism of $\F_4$.  Then $S_0=\gen{A_1,A_2}$ is characteristic in 
$S$, since $A_1$ and $A_2$ are the unique subgroups of $S$ isomorphic to 
$C_2^4$ (cf. \cite[Lemma 5.1(b)]{OV2}).  Since $c_\phi$ permutes freely a 
basis of $Z(S_0)=E_{13}$, \cite[Corollary 1.3 \& Lemma 4.5(a)]{OV2} imply 
there is an isomorphism
        \[ \Out(S) \RIGHT4{\Res}{\cong} 
        C_{\Out(S_0)}(\gen{[c_\phi]})/\gen{[c_\phi]} = 
        \gen{[\rho_1^*],[\rho_2^*],[\tau]} \cong D_8 ~. \]
Let $\dot\tau,\dot\rho_1^*,\dot\rho_2^*\in\Aut(S)$ be the extensions of 
$\tau,\rho_1^*,\rho_2^*\in\Aut(S_0)$ which send $\phi$ to itself.  

Set $H_i=\gen{A_i,\phi}$, and $N_i=\gen{H_i,\e12^1\e23^1}$.  By 
\cite[Theorem 5.11]{OV2} and Table 5.2 in its proof, in all cases, $S_0$ 
is $\calf$-essential, and for $i=1,2$ either $H_i$ or $N_i$ is 
$\calf$-essential but not both.  Also, $\outf(S)=1$ (since $\Out(S)$ is a 
2-group), and $\outf(S_0)=\gen{[\gamma_0],[c_\phi]}$ or 
$\gen{[\gamma_0],[\gamma_1],[c_\phi]}$.  By \cite[Lemma 5.8]{OV2}, there 
is a unique possibility for $\outf(N_i)$ if $N_i$ is essential, and hence 
this group is normalized by $[\dot\rho_1^*]$ and $[\dot\rho_2^*]$.  By 
\cite[Lemma 5.7]{OV2}, there are two possibilities for $\outf(H_i)$ (if 
$H_i$ is essential) which are exchanged under conjugation by 
$[\dot\rho_i^*]$ and invariant under conjugation by $[\dot\rho_{3-i}^*]$.  

By inspection, for $i=1,2$, $[\rho_i^*,\gamma_0]=1$ but 
$[\rho_i^*,\gamma_1]\ne1$.  Together with the above observations about the 
action of $\dot\rho_i^*$ on the possibilities for $\outf(H_i)$ and 
$\outf(N_i)$, this shows that $\dot\rho_i^*$ is fusion preserving 
(contained in $\Aut(S,\calf)$) exactly when $N_i$ is $\calf$-essential and 
$[\gamma_1]\notin\outf(S_0)$ (and $\dot\rho_1^*\dot\rho_2^*\in\Aut(S,\calf)$ 
only if $N_1$ and $N_2$ are both essential).  Also, $\dot\tau$ is fusion 
preserving if either the $N_i$ are both essential or the $H_i$ are both 
essential (and the $\outf(H_i)$ are chosen appropriately in the latter 
case), and otherwise $\Out(S,\calf)\le\gen{[\dot\rho_1^*],[\dot\rho_2^*]}$.  
Thus $\Out(S,\calf)$ is as described in Table \ref{tab:m22}, where we refer 
to \cite[Table 5.2]{OV2} for the information about the fusion systems.  

\begin{table}[ht]
\[ \renewcommand{\arraystretch}{1.5}
\begin{array}{|c||c|c|c|c|} \hline
G & \textup{$\calf$-essential} & \outf(S_0) & \Out(S,\calf) & \Out(G) \\
\hline\hline
M_{22} & S_0,H_1,N_2 & \gen{[\gamma_0],[c_{\phi}]} & 
\gen{[\dot\rho_2^*]}\cong{}C_2 
& C_2 \\ \hline
M_{23} & S_0,H_1,N_2 & \gen{[\gamma_0],[\gamma_1],[c_{\phi}]} & 1 & 1 \\ \hline
PSL_4(5) & S_0,N_1,N_2 & \gen{[\gamma_0],[c_{\phi}]} & \Out(S)\cong{}D_8 & D_8 
\\ \hline
\textup{McL} & S_0,N_1,N_2 & \gen{[\gamma_0],[\gamma_1],[c_{\phi}]} & 
\gen{[\dot\tau]}\cong{}C_2 & C_2 \\ \hline
\end{array}
\]
\caption{} \label{tab:m22}
\end{table}

%%We claim that $\mu_G\circ\kappa_G$ is an isomorphism in each case.  

\noindent\textbf{(i) }  Since $|\Out(G)|=|\Out(S,\calf)|$, it suffices to 
prove $\mu_G\circ\kappa_G$ is injective.  Fix $\alpha\in\Aut(G,S)$ such 
that $\mu_G(\kappa_G([\alpha]))=1$; thus $\alpha|_S=c_g$ for some 
$g\in{}N_G(S)$.  Upon replacing $\alpha$ by $c_g^{-1}\circ\alpha$, we can 
assume $\alpha|_S=\Id_S$.  When $G$ is one of the three sporadic groups, 
then by \cite[\S1.5]{GL}, $A_i$ is centric in $N_G(A_i)$ ($i=1,2$) and 
$G=\gen{N_G(A_1),N_G(A_2)}$.  When 
$G\cong{}PSL_4(5)\cong{}P\Omega_6^+(5)$, this is easily checked by 
identifying $S_0\le{}P\Omega_6^+(5)$ as the subgroup generated by classes 
of diagonal matrices (with respect to an orthonormal basis), together with 
permutation matrices for the permutations $(1\,2)(3\,4)$ and 
$(3\,4)(5\,6)$.  So by \eqref{e:Lyons}, there are elements 
$z_1,z_2\in{}Z(S)=\gen{\e13^1}$ such that $\alpha|_{N_G(A_i)}=c_{z_i}$ for 
$i=1,2$.  Let $g\in{}N_G(S_0)$ be such that $c_g=\gamma_0\in\autf(S_0)$.  
Then $g\in{}N_G(A_i)$ for $i=1,2$ since $\gamma_0$ leaves the $A_i$ 
invariant, so $\alpha(g)=c_{z_1}(g)=c_{z_2}(g)$, and hence $z_1=z_2$ since 
$[g,Z(S)]\ne1$.  Thus $\alpha\in\Aut_{Z(S)}(G)$.  

\noindent\textbf{(ii) }  Set $\call=\call_S^c(G)$.  By Proposition 
\ref{Ker(mu)}, each element of $\Ker(\mu_G)$ is represented by some 
$\alpha\in\Aut\typ^I(\call)$ which is the identity on objects and on 
$\Aut_\call(S)$, and such that for each fully normalized $P\in\Ob(\call)$, 
$\alpha_P\in\Aut(\Aut_\call(P))$ is conjugation by some element 
$g_P\in{}C_{Z(P)}(\Aut_S(P))$.  Since $Z(N_i)\cong{}C_2$ (so 
$C_{Z(N_i)}(\Aut_S(N_i))=C_{Z(N_i)}(\autf(N_i))$), the only 
$\calf$-essential subgroups which could be in the set $\cale_0$ defined in 
Proposition \ref{Ker(mu)}(d) are $S_0$, and $H_1$ and its $S$-conjugates 
if they are essential.

When $P=S_0$, 
	\beqq g_P \in C_{Z(P)}(\Aut_S(P)) = \gen{\e13^1} = Z(S) = 
	C_{Z(S)}(\autf(S))~. \label{e:m22a} \eeqq
So if $H_1$ is not $\calf$-essential, then $[\alpha]=1$ in 
$\Out\typ(\call)$ by Proposition \ref{Ker(mu)}(d).

Assume now that $H_1$ is $\calf$-essential. Then $H_1$ and $A_1$ are both 
$\calf$-centric and fully normalized in $\calf$, and \eqref{e:m22a} holds 
when $P$ is either of these subgroups.  By the description of $\autf(S_0)$ 
and $\autf(H_1)$ in Table \ref{tab:m22} and \cite[Lemma 5.7(a)]{OV2}, $A_1$ 
is invariant under all $\calf$-automorphisms of $S_0$ and of $H_1$, and 
hence $\autf(S_0,A_1)=\autf(S_0)$ and $\autf(H_1,A_1)=\autf(H_1)$.  Also, 
$C_{A_1}(\autf(S_0))=C_{A_1}(\autf(H_1))=1$.  Proposition \ref{Ker(mu)}(c) 
now implies $g_{H_1}=g_{A_1}=g_{S_0}$.  So $[\alpha]=1$ by Proposition 
\ref{Ker(mu)}(d) again; and thus $\mu_G$ is injective. 

\noindent\textbf{(iii) }  By \cite[Table 5.2]{OV2}, for each 
$i=1,2$, $\outf(A_i)$ is isomorphic to one of the groups $\Sigma_5$, 
$(C_3\times{}A_5){\rtimes}C_2$, $A_6$, or $A_7$.  Hence $A_1$ and $A_2$ are 
$\calf$-radical and $\calf$-centric, and 
$O_2(\calf)\le{}A_1\cap{}A_2=E_{13}$ by Proposition \ref{norm<=>}.  Since 
no proper nontrivial subgroup of $E_{13}$ is invariant under the action of 
$\gamma_0\in\autf(S_0)$, and $E_{13}$ itself is not invariant under the 
action of $\nu_2\in\autf(N_2)$ (see \cite[Lemma 5.8]{OV2}), we conclude 
that $O_2(\calf)=1$.

\smallskip

\noindent\textbf{Case 2: }  Now assume 
$S=S_\theta=S_0\rtimes\gen{\theta}$, where 
$c_\theta=\tau\circ{}c_\phi\in\Aut(S_0)$.  Thus $G=J_2$ or $J_3$.  Again 
in this case, $S_0$ is characteristic in $S$ (cf. \cite[Lemma 
4.1(d)]{OV2}).  Since $c_\theta$ permutes freely a basis of 
$Z(S_0)=E_{13}$, \cite[Corollary 1.3]{OV2}  together with the description 
of $\Out(S_0)$ in \cite[Lemma 4.5]{OV2}, imply there is an isomorphism
        \[ \Out(S) \RIGHT4{\Res}{\cong} 
        C_{\Out(S_0)}(\gen{[c_\theta]})/\gen{[c_\theta]} \cong \Sigma_4. \]

Set $Q=\gen{E_{13},\e12^1\e23^1,\e12^\omega\e23^{\bar\omega},\theta}$, an 
extraspecial group of type $D_8\times_{C_2}Q_8$ with 
$\Out(Q)\cong\Sigma_5$.  Let $\dot\gamma_1\in\Aut(S)$ be the extension of 
$\gamma_1\in\Aut(S_0)$ which sends $\theta$ to itself.  By 
results in \cite[\S4.2--3]{OV2}, $\calf=\calf_S(G)$ is isomorphic to the 
fusion system generated by automorphisms
        \[ \outf(S)=\gen{[\dot\gamma_1]},\qquad
        \outf(S_0)=\gen{[\gamma_0],[\gamma_1],[c_\theta]}\cong
        C_3\times\Sigma_3, \qquad
        \outf(Q)\cong{}A_5; \]
and by $\outf(A_i)\cong{}GL_2(4)$ if $G=J_3$.  Since 
$\Aut(S,\calf)/\Inn(S)$ normalizes $\outf(S)\cong{}C_3$, and the 
normalizer in $\Sigma_4$ of a subgroup of order $3$ has order $6$, 
$|\Out(S,\calf)|\le2$.  

\noindent\textbf{(i) }  
In both cases ($G\cong{}J_2$ or $J_3$), $\Out(G)\cong{}C_2$.  So to prove 
$\mu_G\circ\kappa_G$ is an isomorphism, it suffices to show it is 
injective.  Fix $\alpha\in\Aut(G,S)$ such that 
$\mu_G(\kappa_G([\alpha]))=1$; as before, we can assume $\alpha|_S=\Id_S$. 
By \cite[\S1.5]{GL}, $N_G(Z(S))$ and $N_G(E_{13})$ satisfy the hypotheses 
of \eqref{e:Lyons}, and they generate $G$ since both are maximal proper 
subgroups.  By \eqref{e:Lyons}, $\alpha|_{N_G(Z(S))}=\Id$, and 
$\alpha|_{N_G(E_{13})}=c_z$ for some $z\in{}Z(S)$.  Thus 
$\alpha\in\Aut_{Z(S)}(G)$.  

\noindent\textbf{(ii) }  Set $\call=\call_S^c(G)$.  By Proposition 
\ref{Ker(mu)}, each element of $\Ker(\mu_G)$ is represented by some 
$\alpha\in\Aut\typ^I(\call)$ which is the identity on objects, and such 
that for each fully normalized $P\in\Ob(\call)$, 
$\alpha_P\in\Aut(\Aut_\call(P))$ is conjugation by some element 
$g_P\in{}C_{Z(P)}(\Aut_S(P))$.  Since $Z(Q)\cong{}C_2$ (hence 
$C_{Z(Q)}(\Aut_S(Q))=C_{Z(Q)}(\autf(Q))$), $\cale_0$ contains at most the 
subgroups $S_0$, $A_1$, and $A_2$.  Note that in both cases, $A_1$ and 
$A_2$ are $\calf$-centric and fully normalized in $\calf$.

In both cases, $\gamma_0\in\autf(S_0)$ leaves $A_1$ and $A_2$ invariant, 
and acts on each of the groups $Z(S_0)=E_{13}$, $A_1$, and $A_2$ with 
trivial fixed subgroup.  Hence 
	\[ C_{Z(A_1)}(\autf(S_0,A_1)) = 
	C_{Z(A_2)}(\autf(S_0,A_2)) = 1 ~, \]
and so $g_{A_1}=g_{S_0}=g_{A_2}$ by Proposition \ref{Ker(mu)}(c).  Also,
$g_{S_0}\in{}C_{Z(S_0)}(\Aut_S(S_0))=\gen{\e13^1}=Z(S)= 
C_{Z(S)}(\autf(S))$, and Proposition \ref{Ker(mu)}(d) applies (with 
$g=g_{S_0}$) to show that $[\alpha]=1$ in $\Out\typ(\call)$.  Thus $\mu_G$ 
is injective.

\noindent\textbf{(iii) }  Since $S_0$ and $Q$ are $\calf$-centric and 
$\calf$-radical, $O_2(\calf)\le{}S_0\cap{}Q$.  Also, $\autf(Q)$ acts 
transitively on the set of elements of order four in $Q$, and on the set 
of noncentral elements of order two.  Since each of those sets contains 
elements in $S_0$ and elements not in $S_0$, this implies 
$O_2(\calf)\le{}Z(Q)=\gen{\e13^1}$.  Since $\gamma_0\in\autf(S_0)$ and 
$\gamma_0(\e13^1)\ne\e13^1$, it follows that $O_2(\calf)=1$.
\end{proof}

\newsub{Alternating groups}

We prove here that all fusion systems of alternating groups are tame, and 
are also (with the obvious exceptions) reduced.  Unlike the other examples 
given in this paper, we prove tameness without first determining the list 
of essential subgroups. 

We first fix some notation when working with alternating and symmetric 
groups.  We always regard $A_n\le\Sigma_n$ as groups of permutations of 
the set $\bfn=\{1,\ldots,n\}$.  For $\sigma\in\Sigma_n$, we set 
$\supp(\sigma)=\{i\in\bfn\,|\,\sigma(i)\ne{}i\}$ (the \emph{support} of 
$\sigma$).  Likewise, for $H\le\Sigma_n$, $\supp(H)$ is defined to be the 
union of the supports of its elements.

\begin{Lem} \label{Out(S,F(An))}
Fix a prime $p$ and $n\ge{}p^2$.  Assume $n\ge8$ if $p=2$.  Set 
$G=A_n$, and fix $S\in\sylp{G}$.  Set $q=p$ if $p$ is odd, and $q=4$ if 
$p=2$.  Then 
	\[ \Out(S,\calf_S(G)) \cong \begin{cases}  
	C_2 & \textup{if $n\equiv0,1$ (mod $q$)} \\
	1 & \textup{otherwise.} \end{cases} \]
In all cases, $\mu_G\circ\kappa_G$ sends 
$\Out(G)=\Out_{\Sigma_n}(G)\cong{}C_2$ onto $\Out(S,\calf_S(G))$.
\end{Lem}

\begin{proof}  Set $\calf=\calf_S(G)$ for short.  Set 
$E_*=\gen{(1\,2\,\cdots\,p)}\cong{}C_p$ if $p$ is odd, and 
$E_*=\gen{(1\,2)(3\,4),(1\,3)(2\,4)}\cong{}C_2^2$ if $p=2$.  Let $Q\le{}S$ 
be the subgroup generated by all subgroups of $S$ which are $G$-conjugate 
to $E_*$.  If $E_1$ and $E_2$ are $G$-conjugate to $E_*$, then either 
$E_1=E_2$, or $\supp(E_1)\cap\supp(E_2)=\emptyset$ and $[E_1,E_2]=1$, or 
$\gen{E_1,E_2}$ is not a $p$-group.  Since this last case is impossible 
when $E_1,E_2\le{}S$, we conclude that $Q=Q_1\times\cdots\times{}Q_k$, 
where $k=[n/q]$ and the $Q_i$ are pairwise commuting subgroups conjugate 
to $E_*$.  

Fix $\alpha\in\Aut(S,\calf)$, and set $R=\alpha(Q)$.  We first show that 
$R=Q$.  For $i\ge1$, let $r_i$ be the number of orbits of length $p^i$ 
under the action of $R$ on $\bfn$.  Thus 
	\beqq \sum_{i\ge1}p^ir_i = |\supp(R)| \le \begin{cases}  
	q{\cdot}[n/q] & \textup{if $p$ is odd or $r_1=0$} \\
	2{\cdot}[n/2] & \textup{if $p=2$ and $r_1\ge1$}
	\end{cases} \label{e:An1} \eeqq
since $\supp(R)$ has order a multiple of $p$, and a multiple of $4$ when 
$p=2$ and $r_1=0$.  Since $R\cong{}Q$ is elementary abelian, $R$ is 
contained in a product $\prod_{i\ge1}(B_i)^{r_i}$, where $B_i\cong{}C_p^i$ 
acts freely on a subset of $\bfn$ of order $p^i$, and hence 
	\beqq \rk(Q) = \rk(R) \le \begin{cases} \sum_{i\ge1}ir_i & \textup{if 
	$p$ is odd or $r_1=0$} \\
	\sum_{i\ge1}ir_i-1 & \textup{if $p=2$ and $r_1\ge1$~.}
	\end{cases} \label{e:An2} \eeqq
In the last case, ``$-1$'' appears since $R$ contains only even 
permutations, and since the only factors $B_i$ which act via odd 
permutations are those for $i=1$.

Thus if $p$ is odd or $r_1=0$, then by \eqref{e:An1} and \eqref{e:An2},
	\beqq \sum_{i\ge1}p^ir_i \le q\cdot[n/q] = qk = p\cdot \rk(Q) 
	\le \sum_{i\ge1}pir_i ~. \label{e:An2x} \eeqq
Also, $p^i\ge{}pi$, with equality only when $i=1$ or $p^i=4$.  Hence 
\eqref{e:An2x} is possible only when $p$ is odd, $r_1=k$, and 
$r_i=0$ for $i>1$; or when $p=2$, $r_2=k$, and $r_i=0$ for $i>2$.  In both 
cases, $R$ is a product of subgroups conjugate to $E_*$, and thus $R=Q$.  

Now assume $p=2$ and $r_1\ne0$.  By \eqref{e:An1} and \eqref{e:An2} again,
	\[ \sum_{i\ge1}2^ir_i - 2 \le 2\cdot ([n/2]-1) \le 4{\cdot}[n/4] 
	= 4k = 2\cdot \rk(Q) \le \sum_{i\ge1}2ir_i - 2~, \]
so $r_i=0$ for $i\ge3$, and the inequalities are equalities.  In 
particular, $r_1+2r_2=[n/2]=2k+1$, so $r_1$ and $[n/2]$ are both odd.  
Hence $R\cong(C_2^2)^{r_2}\times{}C_2^{r_1-1}$ (and $r_1\ge3$), where each 
element in $\Aut_G(R)$ permutes the $C_2^2$-factors and the $C_2$-factors. 
It follows that 
$\Aut_G(R)\cong(\Sigma_3\wr\Sigma_{r_2})\times\Sigma_{r_1}$.  Since 
$\alpha$ is fusion preserving, we have $\Aut_G(R)\cong\Aut_G(Q)$, where 
$\Aut_G(Q)=\Aut_{\Sigma_n}(Q)\cong\Sigma_3\wr\Sigma_k$ since $[n/2]$ is 
odd ($n-4k\ge2$ where $4k=|\supp(Q)|$, so there is a transposition which 
centralizes $Q$).  Thus 
$\Sigma_3\wr\Sigma_k\cong(\Sigma_3\wr\Sigma_{r_2})\times\Sigma_{r_1}$.  
Since $(\Sigma_3\wr\Sigma_\ell)\ab\cong{}C_2^2$ for all $\ell\ge2$, we get 
$r_2=1$, $\Sigma_3\wr\Sigma_k\cong\Sigma_3\times\Sigma_{r_1}$, and this is 
clearly impossible.  

Thus $\alpha(Q)=Q$.  Since $\alpha$ is fusion preserving, it permutes the 
$G$-conjugacy classes in $Q$.  For each $1\leq r\leq k$, there are 
$\binom{k}{r}{\cdot}(q-1)^r$ products of $r$ disjoint $p$-cycles in $Q$ if 
$p$ is odd, and $\binom{k}{r}{\cdot}(q-1)^r$ products of $2r$ disjoint 
$2$-cycles in $Q$ if $p=2$.  Clearly, $k(q-1)<\binom{k}r{\cdot}(q-1)^r$ 
for $1<r<k$, and $k(q-1)<(q-1)^k$ since $k>1$ and $(k,q)\ne(2,3)$ by 
assumption.  Hence $\alpha$ sends the set of $p$-cycles in $Q$ (products 
of two $2$-cycles in $Q$) to itself.  Since the $p$-cycles (products of 
two $2$-cycles) are precisely the nonidentity elements in 
$\bigcup_{i=1}^kQ_i$, and since $Q_1,\dots,Q_k$ are the maximal subgroups 
in this set, $\alpha$ permutes the $Q_i$.  

Thus there is $g\in{}N_{\Sigma_n}(Q)$ such that $c_g|_Q=\alpha|_Q$, and 
hence 
	\[ \Aut_{gSg^{-1}}(Q)=(\alpha|_Q)\Aut_S(Q)(\alpha|_Q)^{-1}
	=\Aut_S(Q) \]
since $\alpha\in\Aut(S)$.  Since $Q\nsg{}S$ by construction, 
this implies $gSg^{-1}\le{}S{\cdot}C_{\Sigma_n}(Q)$, where $S$ normalizes 
$C_{\Sigma_n}(Q)$ since it normalizes $Q$.  Hence there is 
$h\in{}C_{\Sigma_n}(Q)$ such that $hg\in{}N_{\Sigma_n}(S)$ (and 
$\alpha|_Q=c_{hg}|_Q$).  Upon replacing $\alpha$ by 
$\alpha\circ{}c_{hg}^{-1}$, we can now assume $\alpha|_Q=\Id$.

Set $\calf_0=N_\calf(Q)$.  Since $Q$ is fully normalized in $\calf$ and 
$Q\nsg{}S$, this is a saturated fusion system over $S$.  Also, 
$C_S(Q)\le{}Q$:  any permutation which centralizes $Q$ must leave each set 
$\supp(Q_i)$ invariant, and hence 
$C_G(Q)\cong(C_{A_q}(E_*))^k\times{}A_{n-qk}\cong{}Q\times{}A_{n-qk}$.  
Thus $Q$ is normal and 
centric in $\calf_0$, so $\calf_0$ is \emph{constrained} in the sense of 
\cite[Definition 4.1]{BCGLO1}. By \cite[Proposition 4.3]{BCGLO1}, there is 
a finite group $G_0$, unique up to isomorphism, such that $O_{p'}(G_0)=1$, 
$Q\nsg{}G_0$, $C_{G_0}(Q)\le{}Q$, $S\in\sylp{G_0}$, and 
$\calf_0=\calf_S(G_0)$.  Thus $Q\nsg{}G_0$ and $G_0/Q\cong\autf(Q)$.  The 
fusion preserving automorphism $\alpha$ induces an automorphism of 
$\calf_0=N_\calf(Q)$, and hence by the uniqueness of $G_0$ induces an 
automorphism $\beta\in\Aut(G_0)$ such that $\beta|_S=\alpha$.  Let 
$H\nsg{}G_0$ be the group of those $g\in{}G_0$ such that $c_g$ sends each 
$Q_i$ to itself via an automorphism of order prime to $p$.  Thus 
$H/Q\le(C_{p-1})^k$ (with index 1 or 2) when $p$ is odd, and 
$H/Q\cong{}C_3^k$ when $p=2$.  Since $\beta|_Q=\Id_Q$ and $H/Q$ has order 
prime to $p$, $\beta|_{H}$ is conjugation by an element $a\in{}Q$.  Upon 
replacing $\alpha$ and $\beta$ by $\alpha\circ{}c_a^{-1}$ and 
$\beta\circ{}c_a^{-1}$, we can assume $\beta|_{H}=\Id_{H}$.  But now, 
$Z(H)=1$, so distinct elements of $G_0$ have distinct conjugation actions 
on $H$, and hence $\beta=\Id_{G_0}$.  Thus $\alpha=\beta|_S=\Id_S$.  

We have now shown that each element of $\Aut(S,\calf)$ is conjugation by 
some element of $\Sigma_n$.  Since $n>6$, $\Out(G)=\Out_{\Sigma_n}(G)$ 
by, e.g., \cite[Theorem 3.2.17]{Sz1}. Thus $\mu_G\circ\kappa_G$ sends 
$\Out(G)\cong{}C_2$ onto $\Out(S,\calf)$.  This last group is trivial 
exactly when there is $g\in{}N_{\Sigma_n}(S){\sminus}A_n$ such that 
$c_g|_S\in\autf(S)$; i.e., when $c_g|_S=c_h|_S$ for some $h\in{}N_G(S)$.  
Upon replacing $g$ by $gh^{-1}$, we see that $\Out(S,\calf)=1$ if and only 
if some odd permutation $g\in\Sigma_n{\sminus}A_n$ centralizes $S$.  

If $n\not\equiv0,1$ (mod $q$), then there is a transposition $(i\,j)$ 
which centralizes $S$:  when $p$ is odd because one can choose 
$i,j\in\bfn{\sminus}\supp(S)$, and when $p=2$ because the $S$-action on 
$\bfn$ has an orbit 
$\{i,j\}$ of order $2$.  Thus $\Out(S,\calf)=1$ in this case.  If 
$n\equiv0,1$ (mod $q$), then $|\bfn{\sminus}\supp(Q)|\le1$, and so 
	\[ C_{\Sigma_n}(S) \le C_{\Sigma_n}(Q) = Q \le A_n~. \]
Thus $\Out(S,\calf)=\Out_{\Sigma_n}(S)$ has order two in this case.  
\end{proof}

The following well known lemma will be needed when working with 
elementary abelian subgroups of symmetric groups.

\begin{Lem} \label{cent-Sigma}
Fix $n\ge1$ and an abelian subgroup $G\le\Sigma_n$.  Let 
$H_1,\dots,H_m\le{}G$ be the distinct stabilizer subgroups for the action 
of $G$ on $\bfn$, and let $X_i\subseteq\bfn$ be the set of elements with 
stabilizer subgroup $H_i$ (so $\bfn$ is the disjoint union of the $X_i$).  
Then each $X_i$ is $G$-invariant.  Let $k_i$ be the number of $G$-orbits 
in $X_i$.  Then 
	\[ C_{\Sigma_n}(G) \cong \prod_{i=1}^m (G/H_i)\wr\Sigma_{k_i}~, \]
where each factor $(G/H_i)\wr\Sigma_{k_i}$ has support $X_i$, 
$\Sigma_{k_i}$ permutes the $k_i$ $G$-orbits in $X_i$, and each factor 
$G/H_i$ in $(G/H_i)^{k_i}$ has as support one of those $G$-orbits.
\end{Lem}

\begin{proof}  Let $Y_1,\dots,Y_t$ be the $G$-orbits in $\bfn$, and let 
$C_0\le C_{\Sigma_n}(G)$ be the subgroup of elements which leave each 
of the $Y_i$ invariant.  Since $G$ is abelian, $y$ and $g(y)$ have 
the same stabilizer subgroup for each $g\in{}G$ and each $y\in\bfn$.  Let 
$K_i$ be the stabilizer subgroup of the elements in $Y_i$.  Then the 
homomorphism 
	\[ \chi\: \prod_{i=1}^t(G/K_i) \Right5{} C_0~, \]
defined by setting $\chi(g_1K_1,\dots,g_tK_t)(y)=g_i(y)$ for $y\in{}Y_i$, 
is an isomorphism.

Since all elements in each orbit have the same stabilizer subgroup, each 
set $X_i$ is a union of orbits $Y_j$ (i.e., is $G$-invariant).  Also, 
$C_0$ is normal in $C_{\Sigma_n}(G)$:  it is the kernel of the 
homomorphism to $\Sigma_t$ which describes how an element $\sigma$ 
permutes the orbits.  Each $\sigma\in{}C_{\Sigma_n}(G)$ 
sends each orbit in $\bfn$ to another orbit with the same 
stabilizer subgroup, and thus leaves each $X_i$ invariant.  Since $X_i$ 
contains $k_i$ orbits, 
$C_{\Sigma_n}(G)/C_0\cong\prod_{i=1}^m\Sigma_{k_i}$, and 
$C_{\Sigma_n}(G)$ is isomorphic to the product of the wreath products 
$(G/H_i)\wr\Sigma_{k_i}$.  
\end{proof}

We are now ready to prove that all fusion systems of alternating groups are 
tame.

\begin{Prop} \label{F(An)-tame}
Fix a prime $p$ and $n\ge2$, set $G=A_n$, and choose $S\in\sylp{G}$.  Then 
$\calf_S(G)$ is tame.  If $p=2$ and $n\ge8$; or if $p$ is odd, 
$n\ge{}p^2$ and $n\equiv0,1$ (mod $p$); then 
	\[ \kappa_G\:\Out(G)\Right4{}\Out\typ(\call_S^c(G)) \cong C_2 \] 
is an isomorphism.  
\end{Prop}

\begin{proof}  Set $\calf=\calf_S(G)$ and $\call=\call_S^c(G)$.  If 
$n<p^2$, or if $p=2$ and $n<6$, then the Sylow $p$-subgroups of $A_n$ are 
abelian, so $\calf$ is constrained, $\red(\calf)=1$ by Proposition 
\ref{red(constr)}, and so $\calf$ is tame by Theorem \ref{ThA}.  If $p=2$ 
and $n=6,7$, then since $A_6\cong{}PSL_2(9)$ and $A_7$ has the same fusion 
system as $A_6$, $\calf$ is tame by Proposition \ref{D(2^k)}.  

If $p$ is odd and $n\ge{}p^2$, then 
$\mu_G\:\Out\typ(\call)\Right2{\cong}\Out(S,\calf)$ is an isomorphism 
by \cite[Theorem E]{BLO1} and 
\cite[Theorem A]{limz-odd}.  So by Lemma \ref{Out(S,F(An))}, either 
$n\equiv0,1$ (mod $p$) and $\kappa_G$ is an isomorphism, or 
$\Out\typ(\call)=1$ and hence $\kappa_G$ is split surjective.  Thus 
$\calf$ is tame in these cases.

It remains to handle the case $p=2$ and $n\ge8$.  By Lemma 
\ref{Out(S,F(An))} again, it suffices to prove 
	\beqq \textup{$\Ker(\mu_G)=1$ if $n\equiv0,1$ (mod $4$) 
	\quad\textup{and}\quad $|\Ker(\mu_G)|\le2$ if $n\equiv2,3$ (mod 
	$4$)~,} \label{e:An3} \eeqq
and also 
	\beqq \textup{$n\equiv2,3$ (mod $4$) \quad$\Longrightarrow$\quad 
	there is $x\in{}C_{\Sigma_n}(S){\sminus}G$ such that 
	$\kappa_G([c_x])\ne1$~.} 
	\label{e:An4} \eeqq

Let $Q\le{}S$ be as in the proof of Lemma \ref{Out(S,F(An))}:  the 
subgroup generated by all subgroups of $S$ $G$-conjugate to 
$E_*=\gen{(1\,2)(3\,4),(1\,3)(2\,4)}$.  We saw in the proof of the lemma 
that $Q=Q_1\times\cdots\times{}Q_k$, where $k=[n/4]$, the $Q_i$ are the 
only subgroups of $S$ $G$-conjugate to $E_*$, and they have pairwise 
disjoint support.  Thus $Q$ is weakly closed: the unique subgroup of $S$ 
in its $G$-conjugacy class.

Fix $[\alpha]\in\Ker(\mu_G)$.  By Proposition \ref{Ker(mu)}, we can assume 
$[\alpha]$ is the class of $\alpha\in\Aut\typ^I(\call)$ for which 
$\alpha_S=\Id_{\Aut_\call(S)}$.  Let $g_P\in{}C_{Z(P)}(\Aut_S(P))$, for 
$P\le{}S$ $\calf$-centric and fully normalized, be the elements defined in 
Proposition \ref{Ker(mu)}.  Set $g=g_Q\in{}C_Q(\Aut_S(Q))=Z(S)$ (the last 
equality since $Q$ is normal and centric in $S$).  For each fully 
normalized $P\ge{}Q$ 
(including $P=S$), all automorphisms in $\autf(P)$ leave $Q$ invariant 
since it is weakly closed, so $g_P\equiv{}g_Q=g$ (mod 
$C_{Z(Q)}(\autf(P))=C_{Z(P)}(\autf(P))$) by Proposition \ref{Ker(mu)}(c).  
So upon replacing $\alpha$ by $\alpha\circ{}c_{[g]}^{-1}$, we can assume 
$g=1$, and $\alpha_P=\Id_{\Aut_\call(P)}$ (and $g_P=1$) for all fully 
normalized $P\ge{}Q$. 

For each $1\le{}i\le{}k$, there is a $3$-cycle $h_i\in{}N_G(Q)$ which 
permutes transitively the involutions in $Q_i$ and centralizes the other 
$Q_j$.  Thus $C_Q(\autf(Q))\le\bigcap_{i=1}^kC_Q(h_i)=1$.  
So by Proposition \ref{Ker(mu)}(e), $[\alpha]=1$ if and only if 
for each $\calf$-conjugacy class $\calp$ of subgroups which do not contain 
$Q$, if $\calp\cap\E_0\ne\emptyset$, then there is at least one 
subgroup $P\in\calp\cap\E_0$ such that $g_P\in{}C_{Z(P)}(\autf(P))$.  
Recall $P\in\E_0$ if $P$ 
is $\calf$-essential, $P=C_S(E)$ for some elementary abelian subgroup $E$ 
fully centralized in $\calf$, and 
$C_{Z(P)}(\Aut_S(P))\gneqq{}C_{Z(P)}(\autf(P))$.  Let $\E_0^{\ngeq{}Q}$ be 
the set of subgroups $P\in\E_0$ which do not contain $Q$.

Fix $P=C_S(E)\in\E_0^{\ngeq{}Q}$.  Since $E$ is fully centralized, 
$P\in\syl2{C_G(E)}$.  Since $P$ is $\calf$-essential, $\outf(P)$ has a 
strongly $2$-embedded subgroup, and hence all involutions in any Sylow 
$2$-subgroup of $\outf(P)$ are in its center (cf. \cite[Propositions 
3.3(a) \& 3.2]{OV2}).  In particular, $\outf(P)$ contains no subgroup 
isomorphic to $D_8$.

Fix $\widebar{P}\in\syl2{C_{\Sigma_n}(E)}$ which contains $P$.  Thus 
$P=\widebar{P}\cap{}A_n$.  Also, $E\nsg\widebar{P}$, so 
$\widebar{P}\le\widebar{P}{\cdot}C_{\Sigma_n}(\widebar{P})\le 
C_{\Sigma_n}(E)$, and hence 
	\beqq \widebar{P}{\cdot}C_{\Sigma_n}(\widebar{P})\big/\widebar{P} 
	\textup{ has odd order.} \label{e:An5} \eeqq
By Lemma \ref{cent-Sigma}, each union of $m$ $E$-orbits of order $q=2^i$ 
which have the same stabilizer subgroup contributes a factor 
$E_q\wr\Sigma_m$ to $C_{\Sigma_n}(E)$, where $E_q\cong(C_2)^i$ is acting 
freely on an orbit of order $q$ in $\bfn$.  Since a Sylow $2$-subgroup of 
$\Sigma_m$ is a product of wreath products $C_2\wr\cdots\wr{}C_2$, 
$\widebar{P}\in\syl2{C_{\Sigma_n}(E)}$ is a product of subgroups of the 
form $E_q\wr{}C_2\wr\cdots\wr{}C_2$ (or $E_q$) with pairwise disjoint 
support.  If $\widebar{P}$ contains a factor 
$E_q\wr{}C_2\wr\cdots\wr{}C_2$ for $q=2^r\ge8$, then $\outf(P)$ contains 
$GL_r(2)\ge{}D_8$, which we just saw is impossible.  

Write $\bfn=X_0\amalg X_1\amalg X_2$, where $X_0$ is the set of points 
fixed by $\widebar{P}$, $X_1$ is the union of $\widebar{P}$-orbits of 
length 2, and $X_2$ is the union of $\widebar{P}$-orbits of length $\ge4$. 
By the above description of $\widebar{P}$, $\widebar{P}=P_1\times{}P_2$, 
where $\supp(P_i)=X_i$ for $i=1,2$, $P_1\cong{}C_2^m$ where $2m=|X_1|$, 
and $P_2$ is a product of subgroups $E_4\wr{}C_2\wr\cdots\wr{}C_2$ and 
$C_2\wr\cdots\wr{}C_2$ (the latter of order $\ge8$).  By \eqref{e:An5}, 
$|X_0|\le1$, since otherwise there would be a 2-cycle in 
$C_{\Sigma_n}(\widebar{P})$ not in $\widebar{P}$.  

Each factor $E_4$ or $C_2\wr{}C_2$ (with support of order 4) contains a 
subgroup conjugate to $E_*$ (thus one of the factors $Q_i$ in $Q$).  Thus 
$X_2\subseteq\supp(Q\cap\widebar{P})$.  If $n-|X_2|\le3$, then 
$X_2=\supp(Q)$, so $Q\le\widebar{P}\cap{}A_n=P$, contradicting the 
original assumption on $P$.  Thus $|X_0\cup{}X_1|>3$. Since $|X_0|\le1$ 
and $|X_1|=2m$ is even, we have $m\ge2$.  

If $\{i,j\}$ is any of the $m$ orbits of order $2$ in $X_1$, then 
$(i\,j)\in{}C_{\Sigma_n}(P){\sminus}A_n$ and $\widebar{P}=\gen{P,(i\,j)}$.  
Thus $N_{\Sigma_n}(P)=N_{\Sigma_n}(\widebar{P})$, 
$C_{\Sigma_n}(P)=C_{\Sigma_n}(\widebar{P})$, 
$P{\cdot}C_{\Sigma_n}(P)=\widebar{P}{\cdot}C_{\Sigma_n}(\widebar{P})$, and 
so $N_{\Sigma_n}(P)/P{\cdot}C_{\Sigma_n}(P)\cong N_G(P)/P{\cdot}C_G(P)$.  
This proves that
	\[ \Out_G(P) = \Out_{\Sigma_n}(P) \cong 
	\Out_{\Sigma_n}(\widebar{P}) 
	\cong \Sigma_m \times \Out_{\Sigma_{X_2}}(P_2), \]
where the first isomorphism is induced by restriction.  Here, 
$\Sigma_{X_2}$ is the group of permutations of $X_2$.  

If $m=2$, then $O_2(\Out_G(P))\ne1$, and if $m\ge4$, then 
$\Out_G(P)\ge{}D_8$.  Either of these would contradict the assumption that 
$\Out_G(P)$ contains a strongly $2$-embedded subgroup.  Thus $m=3$, and 
$X_1=\supp(P_1)$ has order 6.  A group with a strongly $2$-embedded 
subgroup cannot split as a product of two groups of even order, so 
$|\Out_{\Sigma_{X_2}}(P_2)|$ is odd.  Since 
$P_2{\cdot}C_{\Sigma_{X_2}}(P_2)/P_2$ is isomorphic to a subgroup of 
$\widebar{P}{\cdot}C_{\Sigma_n}(\widebar{P})/\widebar{P}$, it has odd 
order by \eqref{e:An5}, and hence 
	\[ \bigl|N_{\Sigma_{X_2}}(P_2)/P_2\bigr| = \left| 
	\frac{N_{\Sigma_{X_2}}(P_2)}{P_2{\cdot}C_{\Sigma_{X_2}}(P_2)} 
	\right| \cdot \left| \frac{P_2{\cdot}C_{\Sigma_{X_2}}(P_2)}{P_2} 
	\right| = \bigl|\Out_{\Sigma_{X_2}}(P_2)\bigr| \cdot 
	\bigl| P_2{\cdot}C_{\Sigma_{X_2}}(P_2)/P_2 \bigr| \]
is also odd.  If $P_2\le{}T\in\syl2{\Sigma_{X_2}}$, then $N_T(P_2)/P_2$ 
has odd order, so $P_2=T$ (cf. \cite[Theorem 2.1.6]{Sz1}), and thus 
$P_2\in\syl2{\Sigma_{X_2}}$.   

Since $P_2$ is a Sylow 2-subgroup of a symmetric group and has no orbits 
of order 2, it is a product of subgroups $C_2\wr\cdots\wr{}C_2$ of order 
$\ge8$.  Since $4\big||X_2|$ (a union of orbits 
of order $2^i\ge4$) and $|X_0|\le1$, 
	\[ n=|X_0|+6+|X_2|\equiv2,3 \pmod{4}~.  \]

If $R$ is any other subgroup in $\E_0^{\ngeq{}Q}$, then 
$R=\widebar{R}\cap{}G$, $\bfn=Y_0\amalg{}Y_1\amalg{}Y_2$ where $Y_0$ is 
the set of elements fixed by $\widebar{R}$ and $Y_1$ is the union of 
$R$-orbits of order $2$, $\widebar{R}=R_1\times{}R_2$ where 
$\supp(R_i)=Y_i$, $R_2\in\syl2{\Sigma_{Y_2}}$, $|Y_1|=6=|X_1|$, and 
$|Y_2|=|X_2|$ (the largest multiple of $4$ which is $\le{}n{-}6$).  Thus 
$R$ is $\Sigma_n$-conjugate to $P$, and is $A_n$-conjugate to $P$ since 
there are odd permutations which centralize $P$ (the transpositions in 
$P_1$).  

Now, $Z(\widebar{P})=P_1\times{}Z(P_2)$, where $Z(P_2)$ is a product of 
one copy of $C_2$ for each factor $C_2\wr\cdots\wr{}C_2$ in $P_2$ 
(equivalently, for each $P_2$-orbit in $X_2$).  Also, each of these 
factors $C_2$ has support the corresponding $P_2$-orbit, hence of order a 
multiple of $4$, and hence contained in $A_n$.  Thus $Z(P_2)\le{}G=A_n$.  
Also, $\autf(P)$ acts via the identity on $Z(P_2)$, since all of the 
factors $C_2\wr\cdots\wr{}C_2$ in $P_2$ have different orders (hence their 
supports have different orders).  Since 
$\Aut_{A_n}(P_1\cap{}A_n)\cong\Sigma_3$ acts on $P_1$ by permuting the 
three transpositions, $\autf(P)$ acts on $P_1\cap{}A_n\cong{}C_2^2$ with 
trivial fixed set.  Since $Z(P) = (P_1\cap A_n)\times Z(P_2)$, it now 
follows that $C_{Z(P)}(\Aut_S(P))/C_{Z(P)}(\autf(P))$ has order two.

To summarize, every class in $\Ker(\mu_G)$ is represented by some $\alpha$ 
such that $\alpha_P=\Id$ when $P\ge{}Q$, and for such $\alpha$, 
$[\alpha]=1$ if and only if $g_P\in{}C_{Z(P)}(\autf(P))$ for some 
representative in each $\calf$-conjugacy class in $\E_0^{\ngeq{}Q}$.  When 
$n\equiv0,1$ (mod $4$), $\E_0^{\ngeq{}Q}=\emptyset$, so $\Ker(\mu_G)=1$.  
When $n\equiv2,3$ (mod $4$), all subgroups in $\E_0^{\ngeq{}Q}$ are 
$\calf$-conjugate to some fixed $P$, and so $|\Ker(\mu_G)|\le 
|C_{Z(P)}(\Aut_S(P))/C_{Z(P)}(\autf(P))|=2$.  This proves \eqref{e:An3}.  

Assume $n\equiv2,3$ (mod $4$), and 
$P=\widebar{P}\cap{}A_n\in\E_0^{\ngeq{}Q}$ as above.  Set $k=[n/4]$ as 
before.  Assume $S$ was chosen so that 
$\supp(S)=X_1\amalg{}X_2=\{1,\ldots,4k+2\}$ and 
$\supp(Q)=\{3,\ldots,4k+2\}$.  We have shown that 
$|X_2|=|\supp(S)|-6=4k-4$ and $X_2\subseteq\supp(Q)$.  Thus, after 
possibly rearranging the elements of $\bfn$, we can assume 
$X_2=\{7,8,\ldots,4k+2\}$ and $P_1=\gen{(1\,2),(3\,4),(5\,6)}$.

Set $x=(1\,2)$.  Then $\Out(G)=\gen{[c_x]}\cong{}C_2$, $[x,S]=1$, and 
$c_x$ is the identity on $N_G(Q)/C'_G(Q)$ and hence on $\Aut_\call(Q)$.  
(Note that if $n=4k+3$, then $C'_G(Q)=\gen{(1\,2\,n)}$ does not commute 
with $x$.)  Also, $(1\,2)(3\,4)(5\,6)$ centralizes $N_G(P)$,  and hence 
$c_x$ acts on $\Aut_\call(P)$ (or on $N_G(P)$) via conjugation  by 
$g_P\defeq (3\,4)(5\,6)\in C_{Z(P)}(\Aut_S(P))$. Since $g_P\notin 
C_{Z(P)}(\autf(P))$, $[c_x]$ is sent to a nontrivial element in  
$\Ker(\mu_G)$. This proves \eqref{e:An4}, and finishes the proof of the 
proposition. 
\end{proof}

We finish by proving that with the obvious exceptions, most fusion systems 
of alternating groups are reduced.

\begin{Prop} \label{F(An)-red}
Fix a prime $p$ and $n\ge{}p^2$ such that $n\equiv0,1$ (mod $p$).  Assume 
$n\ge8$ if $p=2$.  Set $G=A_n$, and choose $S\in\sylp{G}$.  Then the fusion 
system $\calf_S(G)$ is reduced.
\end{Prop}

\begin{proof}  Set $\calf=\calf_S(G)$.  By the focal subgroup theorem (cf. 
\cite[Theorem 7.3.4]{Gorenstein}), $\foc(\calf)=S\cap[G,G]=S$, so 
$O^p(\calf)=\calf$.  

Let $Q\le{}S$ be as in the proof of Lemma \ref{Out(S,F(An))}:  the 
subgroup generated by all subgroups of $S$ $G$-conjugate to $E_*$, where 
$E_*=\gen{(1\,2\,\cdots\,p)}\cong{}C_p$ if $p$ is odd, and 
$E_*=\gen{(1\,2)(3\,4),(1\,3)(2\,4)}\cong{}C_2^2$ if $p=2$.  We saw in the 
proof of the lemma that $Q=Q_1\times\cdots\times{}Q_k$, where $k=[n/p]$ 
($p>2$) or $[n/4]$ ($p=2$), the $Q_i$ are the only subgroups of $S$ 
$G$-conjugate to $E_*$, and they have pairwise disjoint support.  Thus $Q$ 
is $\autf(S)$-invariant.  We also saw that $C_S(Q)=Q$, and hence $Q$ is 
$\calf$-centric (since it is the only subgroup in its $\calf$-conjugacy 
class by construction).  Finally,
	\beqq \Aut_{\Sigma_n}(Q) \cong \Aut(E_*)\wr\Sigma_k 
	\qquad\textup{where}\qquad \Aut(E_*)\cong \begin{cases} 
	C_{p-1} & \textup{if $p>2$} \\
	\Sigma_3 & \textup{if $p=2$~,} \end{cases} \label{e:aaa} \eeqq
and hence $\autf(Q)$ has index at most two in this wreath product.  When 
$p=2$, since $\Sigma_k\le\Aut_{\Sigma_n}(Q)$ permutes the $Q_i$ with 
support of order $4$, it is contained in $\autf(Q)$.

Set $R=O_p(\calf)$.  Since $Q$ is $\calf$-centric, and is $\calf$-radical 
by \eqref{e:aaa}, $R\le{}Q$ by Proposition \ref{norm<=>}.  Assume $R\ne1$, 
and fix $g\in{}R$ of order $p$.  There is $h\in{}Q$ which is $G$-conjugate 
to $g$ (a product of the same number of $p$-cycles) such that $gh$ is a 
$p$-cycle (or a product of two 2-cycles if $p=2$).  Then  $h\in{}R$ since 
$R\nsg\calf$, and so $gh\in{}R$.  Since each $Q_i$ is generated by 
elements $G$-conjugate to $gh$, this would imply that $R=Q$.  But in all 
cases, there are elements both in $Q$ and in $S{\sminus}Q$ which are 
products of $p$ disjoint $p$-cycles, so $Q$ is not strongly closed in 
$\calf$.  We conclude that $R=O_p(\calf)=1$.

Now set $\calf_0=O^{p'}(\calf)$; we must show $\calf_0=\calf$.  By 
\cite[Theorem 5.4]{BCGLO2}, it suffices to show that 
$\Aut_{\calf_0}(S)=\autf(S)$.  Also, by the same theorem, 
	\begin{multline*} 
	\Aut_{\calf_0}(S)=\autf^0(S) \ge \bigl\langle\alpha\in\autf(S) 
	\,\big|\, \alpha|_P\in O^{p'}(\autf(P)), \\
	\textup{ some $\calf$-centric subgroup $P\le{}S$ with 
	$\alpha(P)=P$} \bigr\rangle~.
	\end{multline*}
For $\alpha\in\autf(S)$, if $\alpha|_Q\in O^{p'}(\autf(Q))$, then 
$\alpha\in\Aut_{\calf_0}(S)$.  If $p=2$, then 
$O^{2'}(\autf(Q))=\autf(Q)$ by the description in \eqref{e:aaa}, so 
$\calf_0=\calf$ in this case.  

Assume $p$ is odd.  Let $p^\ell$ be the largest power of $p$ such that 
$p^\ell\le{}n$.  Write $S=S_1\times{}S_2$, where 
$\supp(S_1)\cap\supp(S_2)=\emptyset$ and $|\supp(S_1)|=p^\ell$.  Fix 
$T\in\sylp{\Sigma_p}$, and identify 
	\[ S_1 = T\wr T\wr \cdots\wr T \le 
	\Sigma_p\wr \Sigma_p\wr \cdots\wr \Sigma_p \le 
	\Sigma_{p^\ell} \le \Sigma_n. \]
Let $\Phi\:(\Sigma_p)^\ell\Right2{} 
\Sigma_p\wr\cdots\wr\Sigma_p\le\Sigma_{p^\ell}$ be the monomorphism which 
sends the first factor diagonally to $(\Sigma_p)^{p^{\ell-1}}$, the second 
factor diagonally to $(1\wr\Sigma_p)^{p^{\ell-2}}$, etc.  Set 
$P_1=\Phi(T^\ell)$ and $P=P_1\times{}S_2\le{}S$.  Fix $u\in\F_p^\times$ of 
order $p-1$, and choose $h\in{}N_{\Sigma_p}(T)$ such that $hgh^{-1}=g^u$ 
for $g\in{}T$.  Let $\alpha\in\autf(S)$ be conjugation by 
$\Phi(h,h^{-1},1,\dots,1)$.  Then $\alpha|_{P_1}$ has matrix 
$\diag(u,u^{-1},1,\dots,1)\in \text{SL}_\ell(p)$ with respect to the 
canonical basis. Since $\Aut_\calf(P_1)$ has index at most two in 
$\Aut_{\Sigma_n}(P_1)\cong \text{GL}_\ell(p)$, we get $\alpha|_P\in 
O^{p'}(\autf(P))$, and so $\alpha\in\Aut_{\calf_0}(S)$ since $P$ is 
$\calf$-centric.  Also, $\alpha|_Q$ represents a generator of 
$\autf(Q)/O^{p'}(\autf(Q))\cong\F_p^\times$, so this finishes the proof 
that $\Aut_{\calf_0}(S)=\autf(S)$ and hence that $\calf_0=\calf$.  Thus 
$\calf$ is reduced.
\end{proof}

%%%%%%%%%%%%%%%%%%%%%%%%%%%%%%%%%%%%%%%%%%%%%%%%%%%

\bigskip\bigskip
%%\newpage

\end{document}